\documentclass[12pt,reqno]{amsart}
\usepackage{graphicx}
\usepackage{amssymb}
\usepackage{amsfonts}
\usepackage{amsmath,latexsym,amssymb,amsfonts,amsbsy, amsthm}
\usepackage[usenames]{color}
\usepackage{xspace,colortbl}
\usepackage{mathrsfs}
\usepackage[colorlinks,linkcolor=blue,citecolor=blue]{hyperref}
\usepackage[titletoc]{appendix}
\allowdisplaybreaks

\newcommand{\beq}{\begin{equation}}
\newcommand{\eeq}{\end{equation}}
\newcommand{\ben}{\begin{eqnarray}}
\newcommand{\een}{\end{eqnarray}}
\newcommand{\beno}{\begin{eqnarray*}}
\newcommand{\eeno}{\end{eqnarray*}}
\newcommand{\R}{\mathbb{R}}
\newtheorem{thm}{Theorem}[section]
\newtheorem{defi}[thm]{Definition}
\newtheorem{lem}[thm]{Lemma}
\newtheorem{prop}[thm]{Proposition}
\newtheorem{coro}[thm]{Corollary}
\newtheorem{rmk}[thm]{Remark}

\newtheorem*{sublem}{Sub-Lemma}

\title[Nonlinear heat equation]{Refined blowup analysis and nonexistence of Type II blowups  for   an energy critical nonlinear heat equation}

\author[K. Wang]{Kelei Wang$^\dag$}
\address{$^\dag$School of Mathematics and Statistics \\ Wuhan University\\
Wuhan 430072, China}
\email{wangkelei@whu.edu.cn}

\author[J. Wei]{Juncheng Wei$^\ast$}
\address{$^\ast$Department of Mathematics\\ University of British
Columbia\\ Vancouver, B.C., Canada, V6T 1Z2.
 }
\email{jcwei@math.ubc.ca}

\thanks{K. Wang's research was supported by the National Natural Science Foundation of China (No.~11871381).J. Wei's research is partially supported by NSERC of Canada.}



\begin{document}

\begin{abstract}
We consider the energy critical semilinear heat equation
$$
\left\{\begin{aligned}
&\partial_t u-\Delta u =|u|^{\frac{4}{n-2}}u   &\mbox{in } \R^n\times(0,T),\\
&u(x,0)=u_0(x),
\end{aligned}\right.
$$
where $ n\geq 3$,  $u_0\in L^\infty(\R^n)$, and  $T\in\R^+$ is the first blow up time. We prove that if $ n \geq 7$ and $ u_0 \geq 0$,  then  any blowup must be of Type I, i.e.,
\[\|u(\cdot, t)\|_{L^\infty(\R^n)}\leq C(T-t)^{-\frac{1}{p-1}}.\]
A similar result holds for bounded convex domains. The proof relies on a reverse inner-outer gluing mechanism and delicate analysis of bubbling behavior (bubbling tower/cluster).
\end{abstract}

\maketitle

\tableofcontents


\section{Introduction}\label{sec introduction}
\setcounter{equation}{0}

In this paper we consider the blowup problem for  the nonlinear heat equation
\begin{equation}\label{eqn}
    \partial_tu-\Delta u=|u|^{p-1}u
\end{equation}
where $ p>1$.

\subsection{Cauchy problem}

We first consider the Cauchy problem
\begin{equation}\label{Cauchy}
  \left\{\begin{aligned}
&\partial_t u-\Delta u =|u|^{p-1}u   &\mbox{in } \R^n\times(0,T),\\
&u(x,0)=u_0(x).
\end{aligned}\right.
\end{equation}
Here $u_0\in L^\infty(\R^n)$,  and  $T\in\R^+$ is the first blow up time, that is, $u\in L^\infty(\R^n\times[0,t])$ for any $t<T$, while
\[\lim_{t\to T}\|u(t)\|_{L^\infty(\R^n)}=+\infty.\]

Problem (\ref{Cauchy}) is one of classical nonlinear parabolic equations which has been studied extensively in recent decades. See the monograph by Quitter and Souplet \cite{Quittner-Souplet} for backgrounds and references therein.  Following the seminal work of Fujita \cite{Fujita}, it is well-known that finite time blow-up must occur when the initial datum is sufficiently
large (in some suitable sense). An important and fundamental  question is the classification of the blowup near the first blowup time $T$.  Since the equation \eqref{eqn} is invariant under the scaling
\begin{equation}\label{scaling invariance}
  u^\lambda(x,t):=\lambda^{\frac{2}{p-1}}u(\lambda x,\lambda^2 t),
\end{equation}
the blow up is divided into two types: it is {\bf Type I} if there exists a constant $C$ such that for any $t<T$,
\begin{equation}\label{Type I}
  \|u(\cdot, t)\|_{L^\infty(\R^n)}\leq C(T-t)^{-\frac{1}{p-1}},
\end{equation}
otherwise it is called {\bf Type II}.

When $p$ is subcritical (i.e. $1<p<+\infty$ for $n=1,2$, or $1<p<\frac{n+2}{n-2}$ for $n\geq 3$), in classical works of Giga-Kohn \cite{Giga-Kohn2} (in the case of $ u_0 \geq 0$)  and   Giga-Matsui-Sasayama \cite{Giga04} (in the case of sign-changing $u_0$),  it is shown   that any
blow up is of Type I.

In contrast, less is known  about the energy {\bf critical} case, i.e. $n\geq 3$ and  $p=\frac{n+2}{n-2}$, see e.g. \cite[Open Problem 2.8 in Appendix I]{Quittner-Souplet}. In this paper we establish the following
\begin{thm}\label{main result for Cauchy}
  If $n\geq 7$, $p=\frac{n+2}{n-2}$ and $u_0\geq0$, then any blow up to (\ref{Cauchy}) is of Type I.
\end{thm}

Some remarks on Theorem \ref{main result for Cauchy} are in order.  First we remark that the dimension restriction in Theorem \ref{main result for Cauchy} is {\bf optimal}: when $n \leq 6$ and  $ p=\frac{n+2}{n-2}$, Type II blowup solutions to \eqref{Cauchy} do exist. This was first predicted and proved formally in the pioneering work of  Filippas-Herrero-Vel\'{a}zquez in \cite{Filippas-H-V} via method of matched asymptotics. The first rigorous proof of Type II blowup solutions in the radial setting is given  by Schweyer in \cite{Schweyer2012TypeII} for  $n=4$. (For nonradial setting and multiply blowups in dimension $n=4$ we refer to \cite{dP-M-W-Z-4D}.)  For the remaining dimensions rigorous construction of Type II blowups are established recently: for $n=3$, see  \cite{wei-3Dfinite}, and for $n=5$, we refer to \cite{dP-M-W-5D, Harada-5D}. For $n=6$, see Harada \cite{harada-6D}. We should mention that when $p=\frac{n+2}{n-2}$  all the type II blowup solutions to  \eqref{Cauchy}  constructed so far are sign-changing. It is an open question if there are Type II blowups for positive solutions in low dimensions $n=3,4,5,6$.

Second we remark that  the exponent restriction in Theorem \ref{main result for Cauchy} is  also {\bf optimal}: when $ p>\frac{n+2}{n-2}$ many types of Type II blowup solutions have been found. The first example was discovered in the radial setting by Herrero and Velazquez in \cite{HerreroV-1992} for $p>p_{JL}$ where    $p_{JL}$ is the Joseph-Lundgreen exponent,
\[p_{JL}=
\begin{cases}
1+\frac{4}{n-4-2\sqrt{n-1}} & \text{if $n\ge11$}\\
 + \infty, & \text{if $ n\le 10$}.\\
\end{cases} \]
 See also Mizoguchi \cite{Mizoguchi2005} for the case of a ball, Seki \cite{Seki-2018} for the case of $p=p_{JL}$, and Collot \cite{Collot2017} for the case of general domains with the restriction
that $p>p_{JL}$  and $p$ is an odd integer. A new  and interesting anisotropic Type II blow-up solutions is also constructed recently for $ p>p_{JL} (n-1), n\geq 14$ by Collot, Merle and Raphael in \cite{Collot2020anisotropic}.
In the intermediate exponent regime $ \frac{n+2}{n-2} <p<p_{JL}$,  Matano and Merle \cite{Matano-Merle04, Matano-Merle09} proved that no type II blow-up is present for radial solutions (under some extra technical conditions). However type II blow-up  also exists in this intermediate regime.  In \cite{wei-geometry}, the authors successfully constructed non-radial type II blow-up solution  in the Matano-Merle range $p =\frac{n+1}{n-3} \in (\frac{n+2}{n-2}, p_{JL} (n))$. Another type of non-radial Type blow-up with shrinking spheres is recently found for $ n\geq 5, p=4$ in \cite{wei-tube}.

Theorem \ref{main result for Cauchy} is the first instance of classification in the critical exponent case for general initial datum. In the pioneering work of Filippas-Herrero-Vel\'{a}zquez \cite{Filippas-H-V} it is shown that blowup is type I if the initial data is positive and radially decreasing, when $n\geq 3$.  For initial datum with low energy, we mention that in \cite{Collot2017ground}, Collot, Merle and Raphael proved a result similar to Theorem \ref{main result for Cauchy} under the condition that  $ \| u_0 -W\|_{\dot{H}^1 (\R^n)} \ll 1$, where $W$ is a positive Aubin-Talenti solution satisfying \eqref{stationary eqn}, i.e. positive  steady state of  \eqref{Cauchy}. More precisely assume that $ \| u_0 -W\|_{\dot{H}^1 (\R^n)} \ll 1$ and the dimension $ n\geq 7$, then there is a trichotomy to solution to \eqref{Cauchy}: it either dissipates to zero, or approaches to a rescaled Aubin-Talenti solution, or blows up in finite time and the blow-up is of Type I.  Note that in Theorem \ref{main result for Cauchy}, we have assumed neither decaying nor energy condition.

\subsection{Outline of proof}

The proof of  Theorem \ref{main result for Cauchy} consists of the following five steps.
\begin{enumerate}
  \item First we perform tangent flow analysis at a possible blow up point, see Part \ref{part first time singularity}. This is in the same spirit of Giga-Kohn \cite{Giga-Kohn1,Giga-Kohn2,Giga-Kohn3}, but we rewrite it in the form which is more familiar in geometric measure theory such as the  tangent cone analysis/ blowing up analysis/tangent flow analysis used in the study of minimal surfaces, mean curvature flows and many other geometric variational problems.
  \item Bubbles may appear during this blow up procedure. A general theory on the energy concentration phenomena in \eqref{eqn} is then developed in Part \ref{part energy concentration}, where we mainly follow the treatment (on harmonic map heat flows) in Lin-Wang \cite{Lin-Wang1,Lin-Wang2,Lin-Wang3}, see also their monograph \cite{Ling-Wang-book}.
  \item We then perform a refined analysis of this bubbling phenomena, first in the case of one bubble (\emph{multiplicity one case}), see Part \ref{part one bubble}. Here we mainly use \emph{the inner-outer gluing mechanism} developed by the second author with Davila, del Pino and Musso in \cite{delPino-JEMS}, \cite{Davila-dP-Wei}, \cite{dP-M-W-5D}, \cite{dP-M-W-3D}, \cite{wei-signchanging}, \cite{wei-3Dfinite} (see also \cite{DelPino-survey} for a survey);
  \item In Part \ref{part many bubbles}, we combine the analysis in Part \ref{part energy concentration} and Part \ref{part one bubble} to establish the refined analysis in the general case (\emph{higher multiplicity}), where in particular we will exclude the {\bf bubble towering} formation in this energy concentration phenomena. Arguments used here are motivated by  those used in the study of Yamabe problems through pioneering work of Schoen \cite{Schoen1989}, including secondary blow ups, construction of Green functions, next order expansion of Pohozaev identities, see for example Schoen \cite{Schoen-course, Schoen1989}, Kuhri-Marques-Schoen \cite{K-Marques-Schoen}, Li-Zhu \cite{Li-Zhu}, Li-Zhang \cite{Li-Zhang2}.
  \item Finally, the refined analysis in Part \ref{part many bubbles} and Part \ref{part one bubble} are applied to the first time singularity problem. Results in   Part \ref{part many bubbles} are used to exclude blow ups with {\bf bubble clustering}, while results in Part \ref{part one bubble} are used to exclude blow ups with only one bubble.
\end{enumerate}

\subsection{Cauchy-Dirichlet problem}

Our method can also be applied to the initial-boundary value problem
\begin{equation}\label{Cauchy-Dirichlet}
  \left\{\begin{aligned}
&\partial_t u-\Delta u =|u|^{p-1}u   &\mbox{in } \Omega^T:=\Omega\times(0,T),\\
&u=0                  &\mbox{on } \partial\Omega\times(0,T),\\
&u(x,0)=u_0(x).
\end{aligned}\right.
\end{equation}
Here $\Omega\subset\R^n$   is a bounded domain with $C^2$ boundary, $u_0\in L^q(\Omega)$ for some $q\geq n(p-1)/2$. (The exponent $\frac{n(p-1)}{2}$ is optimal. See \cite{Mizoguchi-Souplet}.)  By \cite[Theorem 15.2]{Quittner-Souplet} or \cite{Brezis-Cazenave} and \cite{Weissler1980},
there exists a $T>0$ such that $u\in L^\infty(\overline{\Omega}\times(0,t])$ for any $t<T$.  Here we assume $u$ blows up in finite time, and $T$ is the first blow up time, that is,
\[\lim_{t\to T}\|u(\cdot, t)\|_{L^\infty(\Omega)}=+\infty.\]

For this Cauchy-Dirichlet problem we prove
\begin{thm}\label{main result for Cauchy-Dirichlet}
  If $n\geq 7$, $p=\frac{n+2}{n-2}$ and $u_0\geq0$, the first time singularity in the interior  must be  Type I, that is, for any $\Omega^\prime\Subset \Omega$, there exists a constant $C$ such that
\[\|u(\cdot, t)\|_{L^\infty(\Omega^\prime)}\leq C(T-t)^{-\frac{1}{p-1}}, \quad \forall 0<t<T.\]
\end{thm}
The proof of this theorem is similar to the one for Theorem \ref{main result for Cauchy}.
It is also generally believed that there is no boundary blow up. This, however, is only known in some special cases, e.g. when $\Omega$ is convex (cf. \cite[Theorem 5.3]{Giga-Kohn3}, \cite{Giga04convex}).

Once we have this Type I blow up bound, it will be interesting to know if the set of blow up points enjoy further regularities, and if the blow up profiles satisfy the uniform, refined estimates as  in the subcritical case, see for example Liu \cite{Liu}, Filippas-Kohn \cite{Filippas-Kohn}, Filippas-Liu \cite{Filippas-Liu}, Vel\'{a}zquez \cite{Velazquez-higher, Velazquez-classification, Velazquez-Hausdorff}, Merle-Zaag \cite{Merle-Zaag-refined, Merle-Zaag-stability, Merle-Zaag-optimal, Merle-Zaag-Liouville}, Zaag \cite{Zaag-first,Zaag-onedimensional, Zaag-second} and Kammerer-Merle-Zaag \cite{Merle-Zaag-dynamical}.

Our proof also implies that the energy collapses and the $L^{p+1}$ norm blows up at $T$, cf.  Giga \cite{Giga1986bound}, Zaag \cite{Zaag-energy} and Du \cite{Du3}.
\begin{coro}\label{coro energy collapsing}
  Under the assumptions of Theorem \ref{main result for Cauchy-Dirichlet}, if there exists an interior blow up point, then
  \begin{equation}\label{energy collapsing}
  \lim_{t\to T}\int_{\Omega}\left[\frac{1}{2}|\nabla u(x,t)|^2-\frac{1}{p+1}u(x,t)^{p+1}\right]dx=-\infty,
\end{equation}
and
\begin{equation}\label{blow up of p+1 norm}
   \lim_{t\to T}\int_{\Omega}u(x,t)^{p+1}dx=+\infty.
\end{equation}
\end{coro}
In fact, we prove that
\begin{equation}\label{divergence of time deriavatives}
 \lim_{t\to T}\int_{0}^{t}\int_{\Omega}|\partial_tu(x,t)|^2dx=+\infty.
\end{equation}
Then \eqref{energy collapsing} follows from the standard energy identity for $u$.  A local version of \eqref{divergence of time deriavatives} and \eqref{energy collapsing} also hold for the solution of the Cauchy problem \eqref{Cauchy}.

We do not know if the $L^2(\Omega)$ norm of $\nabla u(t)$ blows up as $t\to T$. We  conjecture that the blow up must be complete, i.e. the solution cannot be extended beyond $T$ as a weak solution, cf. Baras-Cohen \cite{Baras-Cohen}, Galaktinov-Vazquez \cite{Galaktionov-Vazquez}, Martel \cite{Martel1998} and \cite[Section 27]{Quittner-Souplet}.

\subsection{List of notations and conventions used throughout the paper.}

\begin{itemize}

\item The open ball in $\R^n$ is denoted by $B_r(x)$, and by $B_r$ if the center is the origin $0$.

\item The parabolic cylinder is $Q_r(x,t):=B_r(x)\times(t-r^2,t+r^2)$, the forward parabolic cylinder is $Q_r^+(x,t):=B_r(x)\times(t,t+r^2)$, and the backward parabolic cylinder is $Q_r^-(x,t):=B_r(x)\times(t-r^2,t)$. If the center is the origin $(0,0)$, it will not be written down explicitly.

\item The parabolic distance is
\[ \mbox{dist}_P\left((x,t),(y,s)\right):= \max\{|x-y|, \sqrt{|t-s|}\}.\]
H\"{o}lder, Lipschitz continuous functions with respect to this distance is defined as usual.

\item Given a domain $\Omega\subset\R^n$, $H^1(\Omega)$ is the Sobolev space endowed with the norm
\[ \left[\int_{\Omega}\left(|\nabla u|^2+u^2\right)dx\right]^{1/2}.\]
Given an interval $\mathcal{I}\subset\R$, we use $L^2_tH^1_x$ to denote the space $L^2(\mathcal{I}; H^1)$.

\item A bubble is an entire solution of the stationary equation
\begin{equation}\label{stationary eqn}
  -\Delta W=|W|^{\frac{4}{n-2}}W    \quad \mbox{in } \R^n
\end{equation}
with finite energy
\[ \int_{\R^n}|\nabla W|^2=\int_{\R^n}|W|^{\frac{2n}{n-2}}<+\infty.\]
By the translational and scaling invariance, if $W$ is a bubble, so is
\[ W_{\xi,\lambda}(x):=\lambda^{\frac{n-2}{2}}W\left(\frac{x-\xi}{\lambda}\right), \quad \xi\in\R^n, ~~ \lambda\in\R^+.\]

\item
By Caffarelli-Gidas-Spruck \cite{Caffarelli-Gidas-Spruck}, all entire positive solutions to the stationary equation \eqref{stationary eqn}
are given by Aubin-Talenti bubbles
\begin{equation}\label{bubble}
  W_{\xi,\lambda}(x):= \left(\frac{\lambda}{\lambda^2+\frac{|x-\xi|^2}{n(n-2)}}\right)^{\frac{n-2}{2}}, \quad \quad \quad \lambda>0, \quad \xi\in\R^n.
\end{equation}
They have finite energy, which are always equal to
\begin{equation}\label{total energy}
  \Lambda:=\int_{\R^n}|\nabla W_{\xi,\lambda}|^2=\int_{\R^n}W_{\xi,\lambda}^{p+1}.
\end{equation}
For simplicity, denote $W:=W_{0,1}$.

\item We use $C$ (large) and $c$ (small) to denote universal constants.  They could be different from line to line.  Given two quantities $A$ and $B$, if  $A\leq CB$ for some universal constant $C$, we simply write  $A\lesssim B$. If the constant $C$ depends on a quantity $D$, it will be written as $C(D)$ or $\lesssim_D$.
\end{itemize}

\newpage
\part{Energy concentration behavior}\label{part energy concentration}

\section{Setting}
\setcounter{equation}{0}

In this part  we assume that $p>1$ (which may not  necessarily be the critical exponent), and that the solutions could be  sign-changing. Here we study the energy concentration behavior of  the nonlinear heat equation \eqref{eqn}.

First we need to define what we call as a solution. 

\begin{defi}[Suitable weak solution]
  A function $u$ is a suitable weak solution of \eqref{eqn} in $Q_1$, if $\partial_tu,\nabla u\in L^2(Q_1)$, $u\in L^{p+1}(Q_1)$, and
  \begin{itemize}
    \item  $u$ satisfies \eqref{eqn} in the weak sense, that is,  for any $\eta\in C_0^\infty(Q_1)$,
  \begin{equation}\label{weak solution I}
    \int_{Q_1}\left[\partial_t u\eta+\nabla u\cdot\nabla \eta-|u|^{p-1}u\eta\right]=0;
  \end{equation}
    \item    $u$ satisfies the localized energy inequality: for any $\eta\in C_0^\infty(Q_1)$,
  \begin{equation}\label{energy inequality I}
  \int_{Q_1}\left[ \left(\frac{|\nabla u|^2}{2}-\frac{|u|^{p+1}}{p+1}\right)\partial_t\eta^2-|\partial_tu|^2\eta^2-2 \eta\partial_tu\nabla u\cdot\nabla\eta\right]\geq 0;
\end{equation}
    \item $u$ satisfies the stationary condition: for any $Y\in C_0^\infty(Q_1, \R^n)$,
    \begin{equation}\label{stationary condition I}
    \int_{Q_1} \left[\left(\frac{|\nabla u|^2}{2}-\frac{|u|^{p+1}}{p+1}\right)\mbox{div}Y-DY(\nabla u,\nabla u)+\partial_tu \nabla u\cdot Y\right]=0.
    \end{equation}
  \end{itemize}
\end{defi}

A smooth solution satisfies all of these conditions. (In fact, \eqref{energy inequality I} will become an equality, which is just the standard energy identity.) But   a suitable weak solution need  not to be smooth everywhere.
 \begin{defi}
  Given a weak solution $u$ of \eqref{eqn}, a point $(x,t)$ is a regular point of $u$ if there exists an $r>0$ such that $u\in C^\infty(Q_r(x,t))$, otherwise it is a singular point of $u$. The corresponding sets are denoted by $\mathcal{R}(u)$ and $\mathcal{S}(u)$.
 \end{defi}
 By definition,  $\mathcal{R}(u)$ is open and $\mathcal{S}(u)$ is closed.

In this part,  $u_i$ denotes a sequence of suitable weak solutions to \eqref{eqn} in $Q_1$, satisfying a uniform energy bound
\begin{equation}\label{energy bound I}
   \limsup_{i\to+\infty}\int_{Q_1}\left[|\nabla u_i|^2+|u_i|^{p+1}+|\partial_tu_i|^2\right]dxdt<+\infty.
\end{equation}
The integral bound on $\partial_tu_i$ in \eqref{energy bound I} can be deduced by substituting the bounds on the first two integrands into \eqref{energy inequality I}, if we shrink the domain $Q_1$ a little.

We will state the results about the energy concentration behavior   after introducing some necessary notations. The main results in this part are
\begin{enumerate}
  \item  Theorem \ref{thm energy concentration set}, where we establish basic properties about this energy concentration behavior;
  \item  Theorem \ref{thm Marstrand}  about the case $2\frac{p+1}{p-1}$ is not an integer, where we prove a strong convergence result;
  \item Theorem \ref{thm Brakke flow} about the case $2\frac{p+1}{p-1}$ is  an integer, where we show that the limiting problem is a generalized Brakke's flow;
\item  Theorem \ref{thm energy quantization I}, which is about the energy quantization result in the critical case $p=\frac{n+2}{n-2}$.
\end{enumerate}
Our treatment in this part mainly follows the work of Lin and Wang \cite{Lin-Wang1}, \cite{Lin-Wang2}, \cite{Lin-Wang3}. See also their monograph \cite{Ling-Wang-book}.

There are still many problems remaining open about this energy concentration behavior, such as the energy quantization result in the general case (cf. Lin-Riviere \cite{Lin-Riviere} for the corresponding results for harmonic maps, and Naber-Valtorta \cite{Naber-Valtorta} for Yang-Mills fields), Brakke type regularity result for the limiting problem (cf. Brakke \cite{Brakke}, Kasai-Tonegawa \cite{Tonegawa-MCF}). But we will content with such a preliminary analysis of this energy concentration phenomena, because the main goal of this part is to providing  a setting for  later parts.

\subsection{List of notations and conventions used in this part.}

\begin{itemize}

\item For any $\lambda>0$ and each set $A\subset \R^n\times\R$, we use $\lambda A$ to denote the parabolic scaling of $A$, which is the set
\[ \left\{(\lambda x, \lambda^2 t): (x,t)\in A \right\}.\]

\item Given a set $A\subset\R^n\times \R$, its time slices are $A_t:=A\cap \left(\R^n\times\{t\}\right)$.

  \item Denote $m:=\frac{2 (p+1)}{p-1}$. Note that $m>2$, and $p=\frac{m+2}{m-2}$. If $m$ is an integer, then $p$ is the Sobolev critical exponent in dimension $m$.


  \item For $s\geq 0$, $\mathcal{P}^s$ denotes the $s$-dimensional parabolic Hausdorff measure. The dimension of a subset of $\R^n\times\R$ is always understood to be the parabolic Hausdorff dimension.

\end{itemize}

\section{Monotonicity formula and $\varepsilon$-regularity}\label{sec tools}
\setcounter{equation}{0}

In this section, $u$ denotes a fixed suitable weak solution of \eqref{eqn} in $Q_1$. Here we recall the monotonicity formula and $\varepsilon$-regularity theorems. We also derive a Morrey space estimate, which will be used below in Section \ref{sec tangent flow} to preform the tangent flow analysis. Our use of Morrey space estimates follows closely Chou-Du-Zheng \cite{Du1} and Du \cite{Du2, Du3}, which is different from the ones used in Blatt-Struwe \cite{Struwe3} and Souplet \cite{Souplet-Morrey}.

Fix a function $\psi\in C^\infty_0(B_1)$  such that $0\leq\psi\leq 1$, $\psi\equiv 1$ in $B_{3/4}$ and $|\nabla\psi|+|\nabla^2\psi|\leq 100$.
For $t>0$ and $x\in\R^n$, let
\[G(x,t):=\left(4\pi t\right)^{-\frac{n}{2}}e^{-\frac{|x|^2}{4t}}\]
be the standard heat kernel on $\R^n$.

Take a large constant $C$ and a small constant $c$. For each $(x,t)\in B_{3/4}\times(-3/4,1]$ and $s\in(0, 1/4)$, define
\begin{eqnarray*}
 \Theta_s(x,t;u)&:= &s^{\frac{p+1}{p-1}}\int_{B_1}\left[\frac{|\nabla u(y,t-s)|^2 }{2}-\frac{|u(y,t-s)|^{p+1}}{p+1}\right]G(y-x,s)\psi(y)^2 dy\\
   &+&\frac{1}{2(p-1)}s^{\frac{2}{p-1}}\int_{B_1}u(y,t-s)^2 G(y-x,s)\psi(y)^2 dy+Ce^{-cs^{-1}}.
\end{eqnarray*}
\begin{rmk}\label{rmk on monotonicity quantity}
Because $u$ and $\nabla u$ are only integrable in space-time, rigorously we should integrate one more time  in $s$,  e.g. to consider the quantity
\[ \overline{\Theta}_s(x,t;u):=s^{-1}\int_s^{2s}\Theta_\tau(x,t;u)d\tau.\]
However, to simplify notations we will not use this quantity.
\end{rmk}

The following is a localized version of the monotonicity formula of Giga and Kohn, see \cite{Giga-Kohn1, Giga-Kohn3} and \cite{Giga04}.
\begin{prop}[Localized monotonicity formula]\label{prop monotoncity formula}
If $C$ is universally large and $c$ is universally small, then for any $(x,t)\in B_{3/4}\times(-3/4,1]$  and $0<s_1<s_2<1/4$,
\begin{eqnarray*}
 && \Theta_{s_2}(x,t)-\Theta_{s_1}(x,t)\\
  &\geq & \int_{s_1 }^{s_2} \tau^{\frac{2}{p-1}-1}\int_{B_1}\Big|(t-\tau)\partial_tu(y,t-\tau)+\frac{u(y,t-\tau)}{p-1}+\frac{y}{2}\cdot\nabla u(y,t-\tau)\Big|^2\\
&&  \quad \quad \quad \quad  \quad \quad  \times G(y-x, \tau)\psi(y)^2 dyd\tau.
\end{eqnarray*}
\end{prop}

This almost monotonicity allows us to define
\[\Theta(x,t):=\lim_{s\to 0^+}\Theta_s(x,t).\]
For each $s>0$, from the definition we see $s^{-1}\int_{s}^{2s}\Theta_\tau(x,t) d\tau$ is a continuous function of $(x,t)$. Combining this fact with the  monotonicity formula, we get
\begin{lem}\label{lem u.s.c.}
  $\Theta(x,t)$ is  upper semi-continuous in $(x,t)$.
\end{lem}

The following Morrey space bound is essentially Giga-Kohn's \cite[Proposition 2.2]{Giga-Kohn2} or \cite[Proposition 3.1]{Giga-Kohn3}.
\begin{prop}\label{prop Morrey}
  For any $(x,t)\in Q_{1/2}$ and $r\in(0,1/4)$,
\begin{eqnarray}\label{Morrey 1}
 && r^{m-2-n}\int_{Q_r^-(x,t-r^2)}\left(|\nabla u|^2+|u|^{p+1}\right)+r^{m-n}\int_{Q_r^-(x,t-r^2)}|\partial_t u|^2 \nonumber\\
 &\leq& C\max\left\{\Theta_{4r^2}(x,t),\Theta_{4r^2}(x,t)^{\frac{2}{p+1}}\right\}.
\end{eqnarray}
\end{prop}

Combining this estimate with the monotonicity formula, we obtain a Morrey space estimate of $u$.
\begin{coro}[Morrey space bound]\label{coro Morrey}
There exists a universal constant $M$ depending on $\int_{Q_1}\left[|\nabla u|^2+|u|^{p+1}\right]$  such that for any  $(x,t)\in Q_{1/2}$ and $r\in(0,1/4)$,
\begin{equation}\label{Morrey I}
  r^{m-2-n}\int_{Q_r(x,t)}\left(|\nabla u|^2+|u|^{p+1}\right)+r^{m-n}\int_{Q_r(x,t)}|\partial_t u|^2 \leq M.
\end{equation}
\end{coro}
These Morrey space estimates are invariant under the scaling \eqref{scaling invariance}.

Combining \eqref{Morrey 1} and \eqref{Morrey I}, we also get
\begin{coro}[Change of base point]\label{coro change of base point}
For any $\varepsilon>0$, there exist $0<\delta(\varepsilon)\ll \theta(\varepsilon) \ll 1$ and $C(\varepsilon)$ so that the following holds.
  For any $(x,t)\in Q_{1/2}$, $r<1/4$ and $(y,s)\in  Q_{\delta r}(x,t)$,
\[\Theta_{\theta r^2}(y,s)\leq C\max\left\{\Theta_{r^2}(x,t),\Theta_{r^2}(x,t)^{\frac{2}{p+1}}\right\}+\varepsilon.\]
\end{coro}

Next we recall two standard $\varepsilon$-regularity theorems. The first one is a reformulation of Du \cite[Theorem 3.1]{Du3}.
\begin{thm}[$\varepsilon$-regularity I]\label{thm ep regularity I}
  There exist three universal constants $\varepsilon_{\ast}$, $\delta_\ast$ and $c_{\ast}$ so that the following holds.
 If
 \[ \Theta_{r^2}(x,t)\leq \varepsilon_{\ast},\]
 then
 \[\sup_{Q_{\delta_\ast r}(x,t)}|u|\leq c_{\ast} r^{-\frac{2}{p-1}}.\]
\end{thm}

Once we have an $L^\infty$ bound on $u$, higher order regularity follows by applying standard parabolic estimates to \eqref{eqn}. As a consequence, we get
\begin{coro}
  $\mathcal{R}(u)=\{\Theta=0\}$ and $\mathcal{S}(u)=\{\Theta\geq \varepsilon_{\ast}\}$.
\end{coro}

The second $\varepsilon$-regularity theorem is Chou-Du-Zheng \cite[Theorem 2]{Du1} (see also Du \cite[Proposition 4.1]{Du3}).
\begin{thm}[$\varepsilon$-regularity II]\label{thm ep regularity II}
After decreasing $\varepsilon_{\ast }$, $\delta_{\ast}$, and enlarging $c_{\ast }$ further,
the following holds.
There exists a constant $\theta_\ast>0$ such that, if
 \[ \int_{Q_r^-(x,t-\theta_\ast r^2)}\left(|\nabla u|^2+|u|^{p+1}\right)\leq \varepsilon_{\ast }r^{n+2-m},\]
 then
 \[\sup_{Q_{\delta_{\ast } r}(x,t)}|u|\leq c_{\ast } r^{-\frac{2}{p-1}}.\]
\end{thm}
A standard covering argument then gives
\begin{coro}\label{coro dim of singular set I}
  If $u$ is a suitable weak solution of \eqref{eqn} in $Q_1$, then
  \[\mathcal{P}^{n+2-m}\left(\mathcal{S}(u)\right)=0.\]
  In particular, $\mbox{dim}\left(\mathcal{S}(u)\right)\leq n+2-m$.
\end{coro}

\section{Defect measures}\label{sec defect measure}
\setcounter{equation}{0}

From now on until Section \ref{sec critical exponent}, we will be concerned with the energy concentration behavior in \eqref{eqn}. In this section we define the defect measure and prove some basic properties about it.

\subsection{Definition and basic properties} By \eqref{energy bound I}, after passing to a subsequence, we may assume $u_i$ converges weakly to a limit $u_\infty$ in $L^{p+1}(Q_1)$, $\nabla u_i$ and $\partial_tu_i$ converges weakly to $\nabla u_\infty$ and $\partial_tu_\infty$ respectively in $L^2(Q_1)$.

By Sobolev embedding theorem and an interpolation argument, for any $1\leq q<p+1$, $u_i$ converges to $u_\infty$ strongly in $L^q_{loc}(Q_1)$. As a consequence, $u_\infty$ is a weak solution of \eqref{eqn}.

There exist three Radon measures $\mu_1,\mu_2$ and $\nu$ such that
\[
 \left\{\begin{aligned}
   & |u_i|^{p+1}dxdt\rightharpoonup |u_\infty|^{p+1}dxdt+\mu_1, \\
  & |\nabla u_i|^2dxdt\rightharpoonup |\nabla u_\infty|^2dxdt+\mu_2, \\
  & |\partial_tu_i|^2dxdt\rightharpoonup |\partial_t u_\infty|^2dxdt+\nu.
\end{aligned}\right.
\]
These measures, called \emph{defect measures}, characterize the failure of corresponding strong convergence.

We also let
\[
 \nabla u_i\otimes \nabla u_i dxdt\rightharpoonup \nabla u_\infty\otimes \nabla u_\infty dxdt+\mathcal{T} d\mu_2,
\]
where $\mathcal{T}$ is a matrix valued $\mu_2$-measurable functions. Furthermore, $\mathcal{T}$ is symmetric, semi-positive definite and its trace equals $1$ $\mu_2$-a.e.

\begin{lem}[Energy partition]
  $\mu_1=\mu_2$.
\end{lem}
\begin{proof}
For any $\eta\in C_0^\infty(Q_1)$, multiplying the equation \eqref{eqn} by $u_i\eta^2$ and integrating by parts, we obtain
\begin{equation}\label{energy partition 1}
  \int_{Q_1}\left[-u_i^2\eta\partial_t\eta+|\nabla u_i^2|\eta^2-|u_i|^{p+1}\eta^2+2\eta u_i\nabla u_i\nabla\eta\right]=0.
\end{equation}
Letting $i\to\infty$, and noting that $u_\infty$ also satisfies \eqref{energy partition 1}, we obtain
\[\int_{Q_1}\eta^2 d\left(\mu_1-\mu_2\right)=0.\]
Since this holds for any $\eta\in C_0^\infty(Q_1)$, $\mu_1=\mu_2$ as Radon measures.
\end{proof}
By this lemma,  we can write $\mu_1$ and $\mu_2$ just as $\mu$. Denote $\Sigma:=\mbox{spt}(\mu)$, and define
\[\Sigma^\ast:=\left\{(x,t):  ~~ \forall r>0, ~~ \limsup_{i\to\infty}r^{m-2-n}\int_{Q_r^-(x,t)}\left(|\nabla u_i|^2+|u_i|^{p+1}\right)\geq\varepsilon_\ast/2\right\}\]
to be \emph{the blow up locus}.

Now we state some basic properties on this energy concentration phenomena.
\begin{thm}\label{thm energy concentration set}
Suppose $u_i$ is a sequence of suitable weak solutions of \eqref{eqn} in $Q_1$, satisfying \eqref{energy bound I}. Define $\mu,\nu$, $\Sigma$ and $\Sigma^\ast$ as above. Then the following holds.
\begin{enumerate}
  \item (Morrey space bound) For any $(x,t)\in Q_{1/2}$ and  $r\in(0,1/4)$,
\begin{equation}\label{Morrey bound for the limit}
\left\{\begin{aligned}
   & \mu(Q_r(x,t))+\int_{Q_r(x,t)}\left(|\nabla u_\infty|^2+|u_\infty|^{p+1}\right)\leq M r^{n+2-m},  \\
  & \nu(Q_r(x,t)) +\int_{Q_r(x,t)}|\partial_t u_\infty|^2\leq M r^{n-m}.
\end{aligned}\right.
\end{equation}
Here $M$ is the constant in Corollary \ref{coro Morrey}.

  \item   The blow up locus $\Sigma^\ast$ is closed.

  \item (Smooth convergence)  $u_i$ converges to $u_\infty$ in $C_{loc}^\infty(Q_1\setminus \Sigma^\ast)$. As a consequence,
  \begin{itemize}
    \item $u_\infty\in C^\infty(Q_1\setminus \Sigma^\ast)$, that is, $\mathcal{S}(u_\infty)\subset \Sigma^\ast$.
    \item  $\Sigma\subset\Sigma^\ast$ and $\mbox{spt}(\nu)\subset \Sigma^\ast$.
  \end{itemize}

 \item  (Measure estimate of the blow up locus)  For any $(x,t)\in \Sigma^\ast\cap Q_{1/2}$ and $0<r<1/2$,
\[\frac{\varepsilon_\ast}{C} r^{n+2-m}\leq \mathcal{P}^{n+2-m}(\Sigma^\ast\cap Q_r(x,t))\leq \frac{C}{\varepsilon_\ast} r^{n+2-m}.\]

\item  (Lower density bound)  For $\mu$-a.e. $(x,t)\in\Sigma\cap Q_{1/2}$ and $r\in(0,1/2)$,
  \begin{equation}\label{lower density bound}
    \mu(Q_r(x,t))\geq\mu(Q_{r/2}^-(x,t-\theta_\ast r^2))\geq\varepsilon_\ast r^{n+2-m}.
  \end{equation}

\item  $\mathcal{P}^{n+2-m}\left(\Sigma^\ast \setminus \Sigma\right)=0 $.

\item  There exists a measurable function $\theta$ on $\Sigma$ such that
\[ \mu= \theta(x,t) \mathcal{P}^{n+2-m}\lfloor_{\Sigma}.\]
Moreover,
\[ \frac{\varepsilon_\ast}{C}\leq \theta \leq C \quad \quad \mathcal{P}^{n+2-m}-\mbox{a.e. in} ~~ \Sigma.\]

\end{enumerate}
\end{thm}

Before presenting the proof, we first note the following Federer-Ziemer type result \cite{Federer-Ziemer}. It can be proved by a   Vitali covering argument.
\begin{lem}\label{lem vanishing of Morrey}
\begin{enumerate}
  \item For $\mathcal{P}^{n+2-m}$ a.e. $(x,t)\in Q_{1/2}$,
  \[ \lim_{r\to0}r^{m-2-n}\int_{Q_r(x,t)}\left(|\nabla u_\infty|^2+|u_\infty|^{p+1}\right)=0.\]
  \item For $\mathcal{P}^{n-m}$ a.e. $(x,t)\in Q_{1/2}$,
  \[ \lim_{r\to0}r^{m-n}\int_{Q_r(x,t)}|\partial_t u_\infty|^2=0.\]
\end{enumerate}
\end{lem}

\begin{proof}[Proof of Theorem \ref{thm energy concentration set}]
\begin{enumerate}
\item  This Morrey space bound  follows directly by passing to the limit in \eqref{Morrey 1} (for $u_i$).

  \item  By definition,  for any $(x,t)\notin \Sigma^\ast$, there exists an $r>0$ such that for all $i$ large,
  \[ \int_{Q_r^-(x,t)}\left(|\nabla u_i|^2+|u_i|^{p+1}\right)<\varepsilon_\ast r^{n+2-m}.\]
  By Theorem \ref{thm ep regularity II} and standard parabolic regularity theory, $u_i$ are uniformly bound in $C^k(Q_{\delta_\ast r}(x,t))$ for each $k\in\mathbb{N}$. By the weak convergence of $u_i$ and Arzela-Ascolli theorem, they converge  to  $u_\infty$ in  $C^\infty(Q_{\delta_\ast r}(x,t))$.
  \item  This follows directly from the previous point.

  \item This follows from a standard Vitali type covering argument, by utilizing the upper density bound in \eqref{Morrey bound for the limit} and the lower density bound coming from the definition of $\Sigma^\ast$, that is, for any $(x,t)\in\Sigma^\ast\cap Q_{1/2}$ and $0<r<1/2$,
      \begin{equation}\label{lower density bound full}
        \mu(Q_r(x,t))+\int_{Q_r(x,t)}\left(|\nabla u_\infty|^2+|u_\infty|^{p+1}\right)\geq \frac{\varepsilon_\ast}{2} r^{n+2-m}.
      \end{equation}
  \item This follows by combining \eqref{lower density bound full} and Lemma \ref{lem vanishing of Morrey}.
  \item This follows from a Vitali type covering argument, by using \eqref{lower density bound full} and Lemma \ref{lem vanishing of Morrey}.
  \item This follows from differentiation theorem for measures, and an application of \eqref{Morrey bound for the limit} and \eqref{lower density bound}. \qedhere
\end{enumerate}
\end{proof}

\subsection{Definition of the mean curvature} In this subsection we define a mean curvature type term for defect measures.
\begin{lem}\label{lem decomposition of mean curvature}
There exists an $\R^n$ valued function $\mathbf{H}\in L^2(Q_1,d\mu)$ such that
  \[ \partial_tu_i\nabla u_i dxdt \rightharpoonup \partial_tu_\infty\nabla u_\infty dxdt+\frac{1}{m}\mathbf{H} d\mu\]
  weakly as Radon measures.
\end{lem}
\begin{proof}
By Cauchy-Schwarz inequality,
\begin{equation}\label{control of mean cuvature 0}
  \int_{Q_1}|\partial_tu_i||\nabla u_i|\leq \left(\int_{Q_1}|\partial_tu_i|^2\right)^{1/2}\left(\int_{Q_1}|\nabla u_i|^2\right)^{1/2}
\end{equation}
are bounded as $i\to+\infty$. Therefore we may assume
\[\partial_tu_i\nabla u_i dxdt \rightharpoonup \xi \quad \mbox{weakly as vector valued Radon measures},\]
where $\xi$ is an $\R^n$-valued Radon measure.

Take the Radon-Nikodym decomposition of $\xi$ with respect to the  Lebesgue measure, $\xi=\xi^a+\xi^s$, where $\xi^a$ is the absolute continuous part, and $\xi^s$ is the singular part.

By Point (3) in  Theorem \ref{thm energy concentration set},
\[ \xi= \partial_tu_\infty \nabla u_\infty dxdt \quad \mbox{outside} ~~ \Sigma^\ast.\]
In view of Point (4) in  Theorem \ref{thm energy concentration set}, this is just the absolutely continuous part $\xi^a$. In other words,
\begin{equation}\label{absolute part of mean curvature measure}
  \xi^a=\partial_tu_\infty \nabla u_\infty dxdt .
\end{equation}

On the other hand, we can also define
\[ \mathbf{H}_i:=
\begin{cases}
   \frac{\partial_tu_i \nabla u_i}{|\nabla u_i|^2}, & \mbox{if } |\nabla u_i|\neq 0 \\
  0, & \mbox{otherwise}.
\end{cases}
\]
It satisfies
\[ \int_{Q_1}|\mathbf{H}_i|^2 |\nabla u_i|^2 dxdt\leq \int_{Q_1}|\partial_tu_i|^2dxdt.\]
By Hutchinson \cite{Hutchinson}, we may assume $(\mathbf{H}_i,|\nabla u_i|^2dxdt)$ converges to $(\widetilde{\mathbf{H}},|\nabla u_\infty|^2dxdt+\mu)$ weakly as measure-functions pairs.
By Fatou lemma,
\begin{equation}\label{application of Fatou lemma}
 \int_{Q_1}|\widetilde{\mathbf{H}}|^2|\nabla u_\infty|^2dxdt+\int_{Q_1}|\widetilde{\mathbf{H}}|^2d\mu<+\infty.
\end{equation}

Since $\xi$ is the weak limit of $\partial_tu_i\nabla u_idxdt$, we have
\begin{equation}\label{representation of mean curvature measure}
  \xi=\widetilde{\mathbf{H}}\left(|\nabla u_\infty|^2dxdt+\mu\right).
\end{equation}
By  \eqref{absolute part of mean curvature measure},
\[ \widetilde{\mathbf{H}}=\begin{cases}
   \frac{\partial_tu_\infty \nabla u_\infty}{|\nabla u_\infty|^2}, & \mbox{if } |\nabla u_\infty|\neq 0 \\
  0, & \mbox{otherwise},
\end{cases}
\]
a.e. with respect to the Lebesgue measure in $Q_1$. (In fact, this holds everywhere in $Q_1\setminus \Sigma^\ast$.) Hence in view of \eqref{absolute part of mean curvature measure}, we get
\[\xi^s=\widetilde{\mathbf{H}}d\mu.\]
In particular, $\widetilde{\mathbf{H}}$  is the Radon-Nikodym derivative $d\xi^s/d\mu$.
The proof is complete by defining $\mathbf{H}:=m\widetilde{\mathbf{H}}$.
\end{proof}
Similar to \eqref{control of mean cuvature 0}, we obtain
\begin{coro}\label{coro coarse estimate on mean curvature}
  For each $Q_r(x,t)\subset Q_1$,
  \[ \frac{1}{m^2}\int_{Q_r(x,t)}|\mathbf{H}|^2d\mu \leq  \nu(Q_r(x,t))\mu(Q_r(x,t)).\]
\end{coro}
A more precise estimate will be given in Lemma \ref{lem relation between H and nu} below.

Passing to the limit in \eqref{energy inequality I} gives the limiting energy inequality:
\begin{eqnarray}\label{energy identity limit}
0&\leq & \int_{Q_1} \left[\left(\frac{|\nabla u_\infty|^2}{2}-\frac{|u_\infty|^{p+1}}{p+1}\right)\partial_t\eta^2-|\partial_tu_\infty|^2\eta^2-2 \eta\partial_tu_\infty\nabla u_\infty\cdot\nabla\eta \right] \nonumber\\
&+& \frac{1}{m}\int_{Q_1}\partial_t\eta^2 d\mu-\int_{Q_1}\eta^2d\nu-\frac{1}{m}\int_{Q_1}\nabla\eta^2\cdot \mathbf{H}d\mu.
\end{eqnarray}
Passing to the limit in \eqref{stationary condition I} gives the limiting stationary condition: for any $Y\in C_0^\infty(Q_1,\R^n)$,
\begin{eqnarray}\label{stationary condition limit}
  0 &=&  \int_{Q_1}\left[\left(\frac{|\nabla u_\infty|^2}{2}-\frac{|u_\infty|^{p+1}}{p+1}\right)\mbox{div}Y-DY(\nabla u_\infty,\nabla u_\infty)+\partial_tu_\infty \nabla u_\infty\cdot Y \right] \nonumber \\
  &+&\frac{1}{m}\int_{Q_1}\left[\left(I-m\mathcal{T}\right)\cdot DY +\mathbf{H}\cdot Y\right] d\mu.
    \end{eqnarray}

\subsection{Limiting monotonicity formula} In this section we establish the monotonicity formula for the limit $(u_\infty,\mu)$. For this purpose, we need first to define the time slices of $\mu$.
\begin{lem}\label{lem disintegration of defect measure}
There exists a family of Radon measures$\mu_t$ on $B_1$ (defined for a.e. $t\in(-1,1)$) such that
 \begin{equation}\label{disintegration}
   \mu=\mu_t  dt.
 \end{equation}
\end{lem}
\begin{proof}
Denote the projection onto the time axis by $\pi$.
For each $r<1$, let $\mu^r$ be the restriction of $\mu$ to $B_r\times(-1,1)$.

Take a function $\eta_r\in C_0^\infty(B_1)$, $0\leq \eta_r \leq 1$ and $\eta_r\equiv 1$ in $B_r$.
The limiting energy inequality \eqref{energy identity limit} implies that
\[ \int_{B_1}\left(\frac{1}{2}|\nabla u_\infty(x,t)|^2-\frac{1}{p+1}|u_\infty(x,t)|^{p+1}\right)\eta_r(x)^2dx+\frac{1}{m}\pi_\sharp\left(\eta_r \mu\right)\]
is a BV function on $(-1,1)$. Hence there exists a function $\overline{e}_r(t)\in L^1(-1,1)$ such that
\[\pi_\sharp\left(\eta_r \mu\right)=\overline{e}_r(t)dt,\]

Because
\[0\leq \pi_\sharp\mu^r\leq \pi_\sharp\left(\eta_r \mu\right),\]
we find another function $e_r(t)\in L^1(-1,1)$ such that
\[\pi_\sharp\mu^r=e^r(t)dt.\]
By the disintegration theorem, there exists a family of probability measure $\overline{\mu}_t$ on $(-1,1)$ such that
\[ \mu^r=  \overline{\mu}_t  d\left(\pi_\sharp\mu^r\right).\]
By defining
\[\mu_t:=e_r(t)\overline{\mu}_t,\]
we get \eqref{disintegration}.
\end{proof}
\begin{rmk}
Unlike harmonic map heat flows, because the energy density for \eqref{eqn} is sign-changing, we do not know if
 \[ \left[\frac{1}{2}|\nabla u_\infty(x,t)|^2-\frac{1}{p+1}|u_\infty(x,t)|^{p+1}\right]dx+\mu_t\]
 is a well-defined measure \emph{for all $t$}. Similarly, we also do not know if there is an estimate on the Hausdorff measure of $\Sigma^\ast_t$.
\end{rmk}

Define $\Theta_s(x,t;u_\infty,\mu)$ to be
\begin{eqnarray*}
 &&s^{\frac{p+1}{p-1}}\int_{B_1}\left[\frac{|\nabla u_\infty(y,t-s)|^2 }{2}-\frac{|u_\infty(y,t-s)|^{p+1}}{p+1}\right]G(y-x,s)\psi(y)^2 dy \\
   &+&\frac{1}{2(p-1)}s^{\frac{2}{p-1}}\int_{B_1}u_\infty(y,t-s)^2 G(y-x,s)\psi(y)^2 dy \\
   &+&\frac{1}{m}s^{\frac{p+1}{p-1}}\int_{\R^n} G(y-x,s)\psi(y)^2 d\mu_{t-s}(y)+Ce^{-cs^{-1}}.
\end{eqnarray*}
As in Remark \ref{rmk on monotonicity quantity}, rigourously we should integrate one more time in $s$.

Passing to the limit in the monotonicity formula for $u_i$, we obtain
\begin{prop}[Limiting, localized monotonicity formula]\label{prop monotoncity formula II}
For any $(x,t)\in Q_{1/2}$ and a.a. $0<s_1<s_2<1/4$,
\begin{eqnarray*}
 && \Theta_{s_2}(x,t;u_\infty,\mu)-\Theta_{s_1}(x,t;u_\infty,\mu)\\
  &\geq & \int_{s_1 }^{s_2} \tau^{\frac{2}{p-1}-1}\int_{B_1}\Big|(t-\tau)\partial_tu_\infty(y,t-\tau)+\frac{u_\infty(y,t-\tau)}{p-1}+\frac{y}{2}\cdot\nabla u_\infty(y,t-\tau)\Big|^2\\
&&  \quad \quad \quad \quad  \quad \quad  \times G(y-x, \tau)\psi(y)^2 dyd\tau  \\
&+& \int_{s_1 }^{s_2}\int_{B_1}\tau^{\frac{2}{p-1}-1}(t-\tau)^2 G(y-x, \tau)\psi(y)^2 d\nu(y,\tau)\\
&+&\int_{s_1 }^{s_2}\int_{B_1}\tau^{\frac{2}{p-1}-1}\frac{|y|^2}{4} G(y-x, \tau)\psi(y)^2 d\mu(y,\tau)\\
&+&\frac{1}{m}\int_{s_1 }^{s_2}\int_{B_1}\tau^{\frac{2}{p-1}-1}(t-\tau) y\cdot\mathbf{ H}(y,\tau)G(y-x, \tau)\psi(y)^2d\mu(y,\tau).
\end{eqnarray*}
\end{prop}

\section{Tangent flow analysis, I}\label{sec tangent flow}
\setcounter{equation}{0}

In this section we perform the tangent flow analysis for $(u_\infty,\mu)$.

For any $(x,t)\in \Sigma^\ast$ and  a sequence $\lambda_i\to 0$, define the blowing up sequence
\[
  \left\{\begin{aligned}
&u_\infty^{\lambda_i}(y,s):=\lambda_i^{\frac{2}{p-1}}u_\infty(x+\lambda_i y, t+\lambda_i^2s),\\
&\mu^{\lambda_i}(A):=\lambda_i^{m-2-n}\mu(\lambda_i A), \quad \nu^{\lambda_i}(A):=\lambda_i^{m-n}\nu(\lambda_i A), \quad \mbox{for any}~~ A\subset \R^n\times\R.
\end{aligned}\right.
\]
By scaling  \eqref{Morrey bound for the limit}, we see that with the same constant $M$ in Corollary \ref{coro Morrey}, for any $Q_R\subset\R^n\times\R$, we have
\begin{equation}\label{Morrey scaled}
\left\{\begin{aligned}
   & \mu^{\lambda_i}(Q_R)+\int_{Q_R}\left(|\nabla u^{\lambda_i}_\infty|^2+|u_\infty^{\lambda_i}|^{p+1}\right)\leq M R^{n+2-m},  \\
  & \nu^{\lambda_i}(Q_R) +\int_{Q_R}|\partial_t u_\infty^{\lambda_i}|^2\leq M R^{n-m}.
\end{aligned}\right.
\end{equation}
Therefore there exists a subsequence (still denoted by $\lambda_i$) such that $u_\infty^{\lambda_i}\rightharpoonup u_\infty^0$ weakly in $ L^2_tH^1_{x,loc}(\R^n\times\R)\cap L^{p+1}_{loc}(\R^n\times\R)$, and
\[
 \left\{\begin{aligned}
& |\nabla u_\infty^{\lambda_i}|^2dxdt+\mu^{\lambda_i}\rightharpoonup |\nabla u_\infty^0|^2dxdt+\mu^0,\\
& |\partial_tu_\infty^{\lambda_i}|^2dxdt+\nu^{\lambda_i}\rightharpoonup |\partial_t u_\infty^0|^2dxdt+\nu^0
\end{aligned}\right.
 \]
 weakly as Radon measures on any compact set of $\R^n\times\R$.

\begin{rmk}\label{rmk tangent flow at special pt}
By Lemma \ref{lem vanishing of Morrey}, we can avoid a set of zero $\mathcal{P}^{n+2-m}$ measure so that
\begin{equation}\label{vanishing Morrey condition}
  \lim_{r\to0} \left[r^{m-2-n}\int_{Q_r(x,t)}\left(|\nabla u_\infty|^2+|u_\infty|^{p+1}\right)+ r^{m-n}\int_{Q_r(x,t)}|\partial_tu_\infty|^2\right]=0.
\end{equation}
Under this assumption, $u_\infty^{\lambda_i}$ converges to $0$ strongly in $L^2_tH^1_{x,loc}(\R^n\times\R)\cap L^{p+1}_{loc}(\R^n\times\R)$, $\partial_tu_\infty^{\lambda_i}$ converges to $0$ strongly in $L^2_{loc}(\R^n\times\R)$, so
\[
  \mu^{\lambda_i}\rightharpoonup  \mu^0, \quad  \nu^{\lambda_i}\rightharpoonup  \nu^0.
 \]
This special choice will be used in Section \ref{sec Marstrand theorem}.
\end{rmk}

Because $\mu$ is the defect measure coming from $u_i$, we can take a further subsequence of $u_i$  so that, by defining
\[ u_i^{\lambda_i}(y,s):=\lambda_i^{\frac{2}{p-1}}u_i\left(x+\lambda_i y, t+\lambda_i^2s\right),\]
we have
\[
\left\{\begin{aligned}
&|\nabla u_i^{\lambda_i}|^2dyds\rightharpoonup |\nabla u_\infty^0|^2dxdt+\mu^0,\\
&|\partial_t u_i^{\lambda_i}|^2dyds\rightharpoonup |\partial_tu_\infty^0|^2dxdt+\nu^0.
\end{aligned}\right.
\]
As a consequence, all results in Section \ref{sec defect measure} hold for $u_\infty^0$, $\mu^0$ and $\nu^0$. In particular,
\begin{enumerate}
  \item  by Lemma \ref{lem disintegration of defect measure}, there exists a family of Radon measures $\mu^0_t$ on $\R^n$ (for a.a. $t\in\R$) such that
\begin{equation}\label{disintegration 2}
  \mu^0=\mu^0_t  dt;
\end{equation}
  \item  there exists an $\mathbf{H}^0\in L^2_{loc}(\R^n\times\R, d\mu^0)$ such that
  \[\partial_tu_i^{\lambda_i}\nabla u_i^{\lambda_i}dxdt\rightharpoonup \partial_tu_\infty^0\nabla u_\infty^0dxdt+\frac{1}{m}\mathbf{H}^0d\mu^0;\]
  \item  for any $(x,t)\in\R^n\times\R$ and $s>0$, we  define $\Theta_s(x,t;u_\infty^0,\mu^0)$ to be
\begin{eqnarray*}
 &&s^{\frac{p+1}{p-1}}\int_{\R^n}\left[\frac{|\nabla u_\infty^0(y,t-s)|^2 }{2}-\frac{|u_\infty^0(y,t-s)|^{p+1}}{p+1}\right]G(y-x,s)dy \\
   &+&\frac{1}{2(p-1)}s^{\frac{2}{p-1}}\int_{\R^n}u_\infty^0(y,t-s)^2 G(y-x,s)dy +\frac{1}{m}s^{\frac{p+1}{p-1}}\int_{\R^n} G(y-x,s)d\mu_{t-s}^0(y),
\end{eqnarray*}
which is still non-decreasing in $s>0$.
\end{enumerate}
For any $(x,t)\in\R^n\times\R$, we still define
\[\Theta(x,t;u_\infty^0,\mu^0):=\lim_{s\to0}\Theta_s(x,t;u_\infty^0,\mu^0).\]

By the scaling invariance of $\Theta$, we obtain
\begin{equation}\label{constancy of monotonocity quantity}
  \Theta_s(0,0;u_\infty^0,\mu^0)=\Theta(x,t; u_\infty,\mu), \quad \forall s>0.
\end{equation}
Then an application of Proposition \ref{prop monotoncity formula II} gives
\begin{lem}\label{lem backwardly self-similar}
\begin{enumerate}
  \item The function $u_\infty^0$ is backwardly self-similar in the sense that
  \[ u_\infty^0(\lambda x, \lambda^2 t)=\lambda^{-\frac{2}{p-1}} u_\infty(x,t), \quad \forall (x,t)\in\R^n\times\R^-, ~~ \lambda>0.\]
  \item
The measure $\mu^0$ is backwardly self-similar in the sense that
\[ \mu^0(\lambda A)=\lambda^{n+2-m}\mu^0(A), \quad \forall\lambda>0, ~~~A\subset \R^n\times\R^-.\]
\end{enumerate}
\end{lem}

By \eqref{constancy of monotonocity quantity}, following the proof of Lin-Wang \cite[Lemma 8.3.3]{Ling-Wang-book}, we obtain
\begin{lem}\label{lem invariant subspace}
For any $(x,t)\in\R^n\times\R$,
\[\Theta(x,t;u_\infty^0,\mu^0)\leq \Theta(0,0;u_\infty^0,\mu^0).\]
Moreover, if the equality is attained at $(x,t)$, then $u_\infty^0$ and $\mu^0$ are translational invariant in the $(x,t)$-direction.
\end{lem}

By this lemma,
\[\mathcal{L}(u_\infty^0,\mu^0):=\left\{(x,t): ~~ \Theta(x,t;u_\infty^0,\mu^0)=\Theta(0,0;u_\infty^0, \mu^0)\right\}\]
is a linear subspace of $\R^n\times\R$, which is called {\it the invariant subspace of $(u_\infty^0,\mu^0)$}.

\begin{defi}
The invariant dimension of $(u_\infty^0,\mu^0)$ is
\[
\begin{cases}
  k+2, & \mbox{if } \mathcal{L}_{\mu_0}=\R^k\times\R \\
  k, & \mbox{otherwise}.
\end{cases}
\]
\end{defi}
Using these notations we can define a stratification of $\Sigma^\ast$, and give a dimension estimate on these stratifications as in White \cite{White-stratification}, see also \cite[Section 8.3]{Ling-Wang-book}.

\section{The case $m$ is not an integer}\label{sec Marstrand theorem}
\setcounter{equation}{0}

In this section, we use Marstrand theorem (\cite{Marstrand}, see also  \cite[Theorem 1.3.12]{Lin-book}) to study the case when  $m$ is not an integer. This is similar to the elliptic case studied in Du \cite{Du-singular} and the authors \cite{Wang-Wei2015}. The main result in this case is
\begin{thm}\label{thm Marstrand}
  Suppose $m$ is not an integer. Then under the assumptions of Theorem \ref{thm energy concentration set},  $u_i$ converges strongly in $L^{p+1}_{loc}(Q_1)$ and $\nabla u_i$ converges strongly to $\nabla u_\infty$ in $L^2_{loc}(Q_1)$. In other words, $\mu=0$.
   As a consequence, $u_\infty$ is a suitable weak solution of \eqref{eqn} in $Q_1$.
\end{thm}
Here we do not claim the strong convergence of $\partial_tu_i$.

To prove this theorem, we use the tangent flow analysis in the previous section, but only performed at a point $(x,t)$ satisfying the condition \eqref{vanishing Morrey condition}. Hence by Remark \ref{rmk tangent flow at special pt}, $u_\infty^0=0$ and
\[ \mu^\lambda\rightharpoonup\mu^0, \quad \nu^\lambda\rightharpoonup \nu^0.\]

By \eqref{lower density bound} in Theorem \ref{thm energy concentration set}, $\mu^0$ is not identically zero in $\R^n\times\R^-$.
Then by the self-similarity of $\mu^0$,  there exists a point $(x_0,-1)\in\mbox{spt}(\mu^0)$. We blow up $\mu^0$ again at $(x_0,-1)$ as in Section \ref{sec tangent flow}, producing a tangent measure $\mu^1$ to $\mu^0$ at this point.
\begin{lem}
The measure  $\mu^1$ is static, that is,
\[\partial_t\mu^1=0 \quad \mbox{in the distributional sense}.\]
\end{lem}
\begin{proof}
By Lemma \ref{lem backwardly self-similar} and the scaling invariance of $\Theta_s$, for any $\lambda>0$,
\[\Theta_{\lambda^2s}(\lambda x_1,-\lambda^2;\mu^0)=\Theta_s(x_1,-1;\mu^0).\]
Letting $s\to0$, we obtain
\[\Theta(\lambda x_1,-\lambda^2;\mu^0)=\Theta(x_1,-1;\mu^0).\]
After blowing up to $\mu^1$, this equality implies that
\[\Theta(0,t;\mu^1)=\Theta(0,0;\mu^1), \quad \forall t\in\R.\]
By Lemma \ref{lem invariant subspace}, $\mu^1$ is invariant under translations in the time direction.
\end{proof}
\begin{coro}
  $\nu^1=0$ and $\mathbf{H}^1=0$ $\mu^1$-a.e..
\end{coro}
\begin{proof}
This follows from the energy identity for $\mu^1$,
\[ \frac{d}{dt}\int_{\R^n}\eta d\mu^1_t=-m\int \eta d\nu^1-\int_{\R^n}\nabla \eta\cdot \mathbf{H}^1 d\mu^1_t.\]
Note that $\mathbf{H}^1$ can be controlled by $\nu^1$ as in Corollary \ref{coro coarse estimate on mean curvature}.
\end{proof}

By the previous lemma, we can view $\mu^1$ as a Radon measure on $\R^n$. Now the stationary condition \eqref{stationary condition limit} reads as
\begin{equation}\label{stationary condition limit 2}
  \int_{\R^n} \left( I-m\mathcal{T}\right)\cdot DY  d\mu^1=0, \quad \forall Y\in C_0^\infty(\R^n,\R^n).
\end{equation}
Following Moser  \cite{Moser}, this is called a stationary measure.
Similar to  \cite[Lemma 2.1]{Moser}, we have
\begin{lem}
  For any $x\in\R^n$,
\[ \Theta_r(x;\mu^1):=r^{m-n}\mu_1(B_r(x))\]
is non-decreasing in $r>0$.
\end{lem}
As a consequence,
\[ \Theta(x;\mu^1):=\lim_{r\to0} r^{m-n}\mu_1(B_r(x))\]
exists. By \eqref{Morrey bound for the limit}, \eqref{lower density bound} and  Lemma \ref{lem vanishing of Morrey},
\[ \frac{\varepsilon_\ast}{2}\leq \Theta(x;\mu^1)\leq C, \quad \mu^1-\mbox{a.e. in} ~~ \R^n.\]
By Marstrand theorem, $m$ must be an integer. In other words, if $m$ is not an integer, then $\mu^1$, and consequently, $\mu^0$ and $\mu$, must be trivial. This finishes the proof of Theorem \ref{thm Marstrand}.

\section{Partial regularity of suitable weak solutions}\label{sec partial regularity}
\setcounter{equation}{0}

In this section, as an application of Theorem \ref{thm Marstrand},  we prove the following partial regularity result for a fixed, suitable weak solution.
This is almost the same with the one obtained in \cite{Du1}, with a small improvement.
\begin{thm}\label{thm partial regularity for suitable weak solutions}
Suppose $u$ is a suitable weak solution of \eqref{eqn} in $Q_1$. Then
  \begin{itemize}
    \item If $1<p<\frac{n+2}{n-2}$, $u$ is smooth, that is, $\mathcal{S}(u)=\emptyset$.
    \item If there exists an integer $k$ with $3\leq k<n$ such that $\frac{k+3}{k-1}< p<\frac{k+2}{k-2}$, then $\mbox{dim}_{\mathcal{P}}(\mathcal{S}(u))\leq n-k+2$.
    \item If $m$ is an integer, then $\mathcal{P}^{n-m+2}(\mathcal{S}(u))=0$.
  \end{itemize}
\end{thm}
We have already obtained a dimension bound on the singular set $\mathcal{S}(u)$ in Corollary \ref{coro dim of singular set I}. Here we need only to improve this bound when $m$ is not an integer. This follows from a standard dimension reduction argument, cf. \cite{Lin-mapping}, \cite{Wang-Wei2015}.

Suppose $(x,t)\in \mathcal{S}(u)\cap Q_{1/2}$. By Theorem \ref{thm ep regularity II}, for any $r\in(0,1/2)$,
\begin{equation}\label{nondegenaracy condition for blow up}
  \int_{Q_r^-(x,t)}\left(|\nabla u|^2+|u|^{p+1}\right)\geq \varepsilon_\ast r^{n+2-m}.
\end{equation}
For  $\lambda\to0$, define the blowing up sequence
\[ u^\lambda(y,s):=\lambda^{\frac{2}{p-1}}u(x+\lambda y, t+\lambda^2s).\]
As in Section \ref{sec tangent flow}, we can take a subsequence so that $u^\lambda\rightharpoonup u^0$ weakly in $ L^2_tH^1_{x,loc}(\R^n\times\R)\cap L^{p+1}_{loc}(\R^n\times\R)$, and
\[
 \left\{\begin{aligned}
& |\nabla u^\lambda|^2dxdt\rightharpoonup |\nabla u^0|^2dxdt+\mu^0,\\
& |\partial_tu^\lambda|^2dxdt \rightharpoonup |\partial_t u^0|^2dxdt+\nu^0
\end{aligned}\right.
 \]
 weakly as Radon measures on any compact set of $\R^n\times\R$.

Furthermore, because $m$ is not an integer, by Theorem \ref{thm Marstrand}, we deduce that $\mu^0=0$ and $u^\lambda$ in fact converges strongly. Because tangent flows are always nontrivial,  $u^0\neq 0$.

By Lemma \ref{lem backwardly self-similar}, $u^0$ is backwardly self-similar.
For these solutions, the following Liouville theorem was established by Giga-Kohn in \cite{Giga-Kohn1}.
\begin{thm}\label{thm Liouville for backward self-similar sol.}
When $p\leq (n+2)/(n-2)$, if $u$ is a backward self-similar solution of \eqref{eqn} in $\R^n\times\R^-$, satisfying
\begin{equation}\label{bounded integral condition}
 \sup_{x\in\R^n} \int_{Q_1^-(x,-1)}\left(|\nabla u|^2+|u|^{p+1}\right)<+\infty,
\end{equation}
 then either $u\equiv 0$, or $u\equiv \pm (p-1)^{-\frac{1}{p-1}}(-t)^{-\frac{1}{p-1}}$.
\end{thm}
Here the original $L^\infty(\R^n)$ assumption in \cite[Theorem 1]{Giga-Kohn1} is replaced by an integral one \eqref{bounded integral condition}.  With this condition, we are still able to do the same computation in \cite{Giga-Kohn1} to deduce the Pohozaev identity \cite[Proposition 1]{Giga-Kohn1}. For $u^0$, the condition \eqref{bounded integral condition} follows by scaling \eqref{Morrey I} and passing to the limit.

By this theorem, if $p<(n+2)/(n-2)$ is subcritical, because $u^0$ also satisfies
\[ \int_{Q_1}|u^0|^{p+1}<+\infty,\]
we must have $u^0=0$. This contradiction implies that there is no singular point of $u$. In other words, $u$ is smooth.

If $p>(n+2)/(n-2)$ is supcritical, by Point (3) in Theorem \ref{thm energy concentration set}, the strong convergence of $u^\lambda$ implies that if $\lambda$ is small enough, then $\lambda^{-1}\left(\mathcal{S}(u)-(x,t)\right)$ is contained in a small neighborhood of $\mathcal{S}(u^0)$. Since $u^0$ is backwardly self-similar (see Lemma \ref{lem backwardly self-similar}), applying White's stratification theorem in \cite{White-stratification}, we  conclude the proof of Theorem \ref{thm partial regularity for suitable weak solutions}.

\section{The case $m$ is an integer}\label{sec rectifiability}
\setcounter{equation}{0}

In this section we continue the analysis of energy concentration behavior, now under the assumption that $m$ is an integer. For this case we prove
\begin{thm}\label{thm Brakke flow}
If $m$ is an integer, then $(u_\infty,\mu)$ is a generalized Brakke's flow in the following sense:
  For any $\varphi\in C_0^\infty(B_1;\R^+)$ and $-1<t_1<t_2<1$,
  \begin{eqnarray}\label{Brakke flow}
     &&\left[\int_{B_1}\left(\frac{1}{2}|\nabla u_\infty(t_2)|^2-\frac{1}{p+1}|u_\infty(t_2)|^{p+1}\right)\varphi dx+\frac{1}{m}\int_{B_1}\varphi d\mu_{t_2}\right] \nonumber\\
     && -\left[\int_{B_1}\left(\frac{1}{2}|\nabla u_\infty(t_1)|^2-\frac{1}{p+1}|u_\infty(t_1)|^{p+1}\right)\varphi dx+\frac{1}{m}\int_{B_1}\varphi d\mu_{t_1}\right]  \nonumber\\
     &\leq& -\int_{t_1}^{t_2}\int_{B_1}\left[|\partial_tu_\infty|^2\varphi+\partial_tu_\infty\nabla u_\infty\cdot\nabla\varphi \right]dxdt\\
     && -\frac{1}{m}\int_{t_1}^{t_2}\int_{\Sigma_t}\left[|\mathbf{H}_t|^2\varphi-\nabla\varphi\cdot \mathbf{H}_t\right]d\mu_tdt. \nonumber
  \end{eqnarray}
\end{thm}

By Lemma \ref{lem disintegration of defect measure}, for a.e. $t\in(-1,1)$ and any $Y\in C_0^\infty(B_1,\R^n)$,
\begin{eqnarray}\label{stationary condition limit time slice}
  0 &=&  \int_{B_1}\left[\left(\frac{|\nabla u_\infty(x,t)|^2}{2}-\frac{|u_\infty(x,t)|^{p+1}}{p+1}\right)\mbox{div}Y-DY\left(\nabla u_\infty(x,t),\nabla u_\infty(x,t)\right)\right] dx \nonumber\\
  &+&\int_{B_1}\partial_tu_\infty(x,t) \nabla u_\infty(x,t)\cdot Y dx   \\
  &+&\int_{B_1}\left[\left(\frac{1}{m}I-\mathcal{T}(x,t)\right)\cdot DY +\frac{1}{m}\mathbf{H}(x,t)\cdot Y\right] d\mu_t(x). \nonumber
    \end{eqnarray}
In view of the lower density bound in \eqref{lower density bound},  as in \cite{Ambrosio-Soner} or \cite{Lin-annals} (see also \cite{Moser} or \cite[Proposition 3.1]{DePhilippis}), we deduce that $\mu_t$ is countably $(n-m)$-rectifiable. In other words, $\Sigma_t$ is countably $(n-m)$-rectifiable, and
\begin{equation}\label{representation of stress-energy tensor}
 I-m\mathcal{T}=T_x\Sigma_t, \quad \mathcal{H}^{n-m} ~ \mbox{a.e.  on} ~ \Sigma_t,
\end{equation}
where $T_x\Sigma_t$ is the weak tangent space (identified with the projection map onto it) of $\Sigma_t$ at $x$.

Similar to Lin-Wang \cite[Lemma 9.2.2]{Ling-Wang-book}, we get
\begin{lem}\label{lem mean curvature vector}
  For a.e. $t\in(-1,1)$,
  \[\mathbf{H}(x,t)\bot T_x\Sigma_t, \quad \mbox{for} ~ \mathcal{H}^{n-m} ~ \mbox{a.e.  } ~ x\in\Sigma_t.\]
\end{lem}
Similar to Lin-Wang \cite[Lemma 9.2.7]{Ling-Wang-book}, we also get
\begin{lem}\label{lem relation between H and nu}
  For any $\eta\in C_0^\infty(Q_1)$,
  \[\int_{Q_1}\eta^2|\mathbf{H}(x,t)|^2d\mu_t(x)dt\leq \frac{1}{m}\int_{Q_1}\eta^2 d\nu.\]
\end{lem}
Plugging this estimate into \eqref{energy identity limit}, we obtain \eqref{Brakke flow}. This finishes the proof of Theorem \ref{thm Brakke flow}.

\section{The case $p=(n+2)/(n-2)$}\label{sec critical exponent}
\setcounter{equation}{0}

In this section we assume $p=(n+2)/(n-2)$ is the Sobolev critical exponent, and all of the solutions $u_i$ are  \emph{smooth}. Note that now $m=2(p+1)/(p-1)=n$.
The main result of this section is  about the quantization of energy.
\begin{thm}\label{thm energy quantization I}
  For any $(x,t)\in\Sigma$, there exist finitely many bubbles $W^k$, $k=1,\cdots, N$, such that
  \[\Theta(x,t)=\frac{1}{n\left(4\pi\right)^{n/2}}\sum_{k=1}^{N}\int_{\R^n}|\nabla W^k|^2.\]
  Furthermore, if all solutions are positive, then there exists an $N\in\mathbb{N}$ such that
  \[ \Theta(x,t)=N\frac{\Lambda }{n\left(4\pi\right)^{n/2}}.\]
\end{thm}
 We first prove a local version of this proposition in Subsection \ref{subsec local quantization}, then prove this theorem in Subsection \ref{subsec proof of energy quantization}. During this course, some further properties of defect measures will also be established in Subsection \ref{subsec proof of energy quantization} and Subsection \ref{subsec further properties of defect measure}, in particular, for applications in Part \ref{part one bubble}, Part \ref{part many bubbles} and Part \ref{part first time singularity}, the special case of positive solutions will be discussed.

\subsection{A local quantization result}\label{subsec local quantization}

In this subsection we prove the following
\begin{lem}\label{lem energy quantization I}
Given a constant $M>0$, suppose a sequence of smooth solutions $u_i$ to \eqref{eqn} in $Q_1$ satisfies
 \[
 \left\{\begin{aligned}
& |\nabla u_i|^2dxdt\rightharpoonup M\delta_0\otimes dt,\\
& |u_i|^{p+1}dxdt\rightharpoonup M\delta_0\otimes dt,\\
& \int_{Q_1}|\partial_tu_i|^2dxdt\to0.
\end{aligned}\right.
\]
Then there exist finitely many bubbles $W^k$ such that
\[ M=\sum_{k}\int_{\R^n}|\nabla W^k|^2dx.\]
\end{lem}

First let us present some immediate consequences of the assumptions in this lemma.
\begin{itemize}
\item An application of the $\varepsilon$-regularity theorem (Theorem \ref{thm ep regularity II}) implies that
\begin{equation}\label{smooth convergence outside blow up locus, localized}
  u_i\to 0 \quad \mbox{in} ~ C^\infty_{loc}\left((B_1\setminus\{0\})\times(-1,1)\right).
\end{equation}
  \item By the $L^2(Q_1)$ bound on $\partial_tu_i$, we deduce that $u_i\to 0$ in $C_{loc}(-1,1;L^2_{loc}(B_1))$.
  \item  By  Fatou lemma,
\[ \int_{-1}^{1}\lim_{i\to+\infty}\left(\int_{B_1}|\partial_tu_i(x,t)|^2dx\right) dt\leq \lim_{i\to+\infty}\int_{Q_1}|\partial_tu_i|^2dxdt=0.\]
Hence for a.e. $t\in(-1,1)$,
\begin{equation}\label{bound on time derivative I}
\lim_{i\to+\infty}\int_{B_1}|\partial_tu_i(x,t)|^2dx=0.
\end{equation}
\end{itemize}

The following lemma describes the energy concentration behavior for a.e. time slice of $u_i$, under the assumptions in Lemma \ref{lem energy quantization I}.
\begin{lem}\label{lem concentration of time slice}
 For a.e. $t\in(-1,1)$,
  \[ \left\{\begin{aligned}
& |\nabla u_i(x,t)|^2dx\rightharpoonup  M\delta_0,\\
& |u_i(x,t)|^{p+1}dx\rightharpoonup  M\delta_0.
\end{aligned}\right.
\]
 \end{lem}
\begin{proof}
For any $\eta\in C_0^\infty(B_1)$, by the energy identity for $u_i$, the function
\[ E_{\eta,i}(t):= \int_{B_1}\left(\frac{|\nabla u_i(x,t)|^2}{2}-\frac{|u_i(x,t)|^{p+1}}{p+1}\right)\eta(x)^2 dx\]
is uniformly bounded in $BV_{loc}(-1,1)$. After passing to a subsequence, they converge in $L^1_{loc}(-1,1)$ and a.e. in $(-1,1)$ to the limit
$(M/n)\eta(0)^2$.

Another consequence of this uniform BV bound is, $E_{\eta,i}$ are uniformly bounded in any compact set of $(-1,1)$. If \eqref{bound on time derivative I} holds at $t$, then
\[ \int_{B_1}\left(|\nabla u_i(t)|^2-|u_i(t)|^{p+1}\right)\eta^2=-\int_{B_1}\left[\partial_t u_i(t) u_i(t) \eta^2+2\eta u_i(t)\nabla u_i(t)\cdot\nabla\eta\right]\]
are also bounded as $i\to+\infty$. Therefore
\begin{equation}\label{H1 bound for time slice}
 \limsup_{i\to+\infty}\int_{B_1}\left[|\nabla u_i(t)|^2+|u_i(t)|^{p+1}\right]dx<+\infty.
\end{equation}
Because $u_i(t)\to 0$ in $L^2(B_1)$, with the help of \eqref{smooth convergence outside blow up locus, localized}, we find a nonnegative constant $M(t)$ such that
\[|\nabla u_i(x,t)|^2dx\rightharpoonup  M(t)\delta_0, \quad  |u_i(x,t)|^{p+1}dx\rightharpoonup  M(t)\delta_0.\]
Then using the above a.e. convergence of $E_{\eta,i}(t)$, and letting $\eta$ vary in $C_0^\infty(B_1)$, we deduce that $M(t)=M$.
\end{proof}

For those $t$ satisfying \eqref{bound on time derivative I}, we have the uniform $H^1(B_1)$ bound \eqref{H1 bound for time slice}, hence by Struwe's global compactness theorem (\cite{Struwe1984compactness}), the following bubble tree convergence holds for $u_i(t)$.
\begin{prop}[Bubble tree convergence]\label{prop bubble tree convergence I}
There exist $N(t)$ points $\xi_{ik}^\ast(t)$, positive constants $\lambda_{ik}^\ast(t)$, $k=1,\cdots, N(t)$, all converging to $0$  as $i\to+\infty$, and  $N(t)$ bubbles $W^k$, such that
   \[ u_i(x,t)=\sum_{k=1}^{N(t)}W^k_{\xi_{ik}^\ast(t),\lambda_{ik}^\ast(t)}(x)+o_i(1),\]
   where $o_i(1)$ are measured in $H^1(B_1)$.

As a consequence,
\begin{equation}\label{energy quantization for time slice}
 \int_{B_1}|\nabla u_i(x,t)|^2dx=\sum_{k=1}^{N(t)}\int_{\R^n}|\nabla W^k|^2+o_i(1).
\end{equation}
\end{prop}
\begin{rmk}\label{rmk bubble tree for positive solutions}
  If $u_i$ are positive, then all bubbles constructed in this proposition are positive, see \cite[Section 3.2]{Hebey}. In view of the Liouville theorem of Caffarelli-Gidas-Spruck \cite{Caffarelli-Gidas-Spruck}, we can take all of these $W^k$ to be the standard Aubin-Talenti bubble $W$. As a consequence, \eqref{energy quantization for time slice} reads as
   \begin{equation}\label{energy quantization for time slice, positive}
 \int_{B_1}|\nabla u_i(x,t)|^2dx=N(t)\Lambda+o_i(1).
\end{equation}
\end{rmk}
\begin{rmk}\label{rmk bubble clustering and towering}
About bubble tree convergence, there are two phenomena  that we will investigate more closely latter.
\begin{description}
\item [Bubble towering] Two bubbles at $\xi_{ik}^\ast(t)$ and $\xi_{i\ell}^\ast(t)$ are towering  if
       \[\limsup_{i\to+\infty}\frac{|\xi_{ik}^\ast(t)-\xi_{i\ell}^\ast(t)|}
 {\max\{\lambda_{ik}^\ast(t),\lambda_{i\ell}^\ast(t)\}}<+\infty.\]
       In this case, either $\frac{\lambda_{ik}^\ast(t)}{\lambda_{i\ell}^\ast(t)}\to+\infty$ or $\frac{\lambda_{i\ell}^\ast(t)}{\lambda_{ik}^\ast(t)}\to+\infty$. Hence these two bubbles are located at almost the same point (with respect to the bubble scales), but the height of one bubble is far larger than the other one's.
\item [Bubble clustering] If  for some $k\neq \ell$,
\[ \lim_{i\to+\infty}|\xi_{ik}^\ast(t)-\xi_{i\ell}^\ast(t)|=0\]
but
       \[\lim_{i\to+\infty}\frac{|\xi_{ik}^\ast(t)-\xi_{i\ell}^\ast(t)|}
 {\max\{\lambda_{ik}^\ast(t),\lambda_{i\ell}^\ast(t)\}}=+\infty,\]
       we say these bubbles are clustering.
\end{description}
Using the terminology introduced in the study of Yamabe problem (see Schoen \cite{Schoen-course}), if there is no bubble clustering and towering, the blow up is called \emph{isolated and simple}.
\end{rmk}

Combining Lemma \ref{lem concentration of time slice} with Proposition \ref{prop bubble tree convergence I}, we conclude the proof of Lemma \ref{lem energy quantization I}. Furthermore, if all solutions are positive, by Remark \ref{rmk bubble tree for positive solutions}, there exists an $N\in\mathbb{N}$ such that
\begin{equation}\label{energy quantization I, positive}
    M=N\Lambda, \quad \mbox{and}\quad N(t)=N \quad \mbox{a.e. in }~~ (-1,1).
\end{equation}

\subsection{Proof of Theorem \ref{thm energy quantization I}}\label{subsec proof of energy quantization}

Under the critical exponent assumption,   the tangent flow analysis in Section \ref{sec tangent flow} can give more information.

First, by noting that $m=n$, in the definition of the blowing up sequence, we have
  \[\nu^\lambda(A):= \nu\left((x,t)+\lambda A\right), \quad \forall A\subset \R^n\times\R.\]
  In the same way, for any $R>0$,
\begin{equation}\label{time derivative to 0 I}
 \int_{Q_R}|\partial_tu_\infty^\lambda|^2=\int_{Q_{\lambda R}}|\partial_tu_\infty|^2\to 0 \quad \mbox{as} ~ \lambda\to0.
\end{equation}
  Therefore by defining
  \begin{equation}\label{definition of atom}
   Q :=\lim_{\lambda\to 0}\nu\left(Q_\lambda(x,t)\right),
  \end{equation}
  we get
  \begin{equation}\label{limiting Dirac}
   \nu^0=Q \delta_{(0,0)}.
  \end{equation}
There is an atom of $\nu$ at $(x,t)$ if and only if  $Q >0$.

By Corollary \ref{coro coarse estimate on mean curvature}, we find that
\begin{equation}\label{limiting Dirac of mean curvature}
 \mathbf{H}^0d\mu= \mathbf{H}^0(0,0)\delta_{(0,0)}.
\end{equation}

Recall that $u_\infty$ is backwardly self-similar (Lemma \ref{lem backwardly self-similar}). By combining Theorem \ref{thm Liouville for backward self-similar sol.}  with \eqref{time derivative to 0 I}, we deduce that
 \[u_\infty^0\equiv 0 \quad \mbox{in} ~~ \R^n\times\R.\]
  Hence in view of \eqref{limiting Dirac} and \eqref{limiting Dirac of mean curvature}, the energy identity \eqref{energy identity limit} for $(u_\infty^0,\mu^0)$ reads as
\begin{equation}\label{energy identity limit+critical 1}
\frac{1}{m}\int_{Q_1}\partial_t\eta  d\mu^0=Q\eta(0,0) +\frac{1}{m}\nabla\eta(0,0)\cdot \mathbf{H}^0(0,0), \quad \forall \eta\in C_0^\infty(\R^n\times\R).
\end{equation}
In particular,
\[\partial_t\mu^0=0 \quad \mbox{in the distributional sense in } ~ \left(\R^n\times\R\right)\setminus\{(0,0)\}.\]
Combining this fact with the backward self-similarity of $\mu^0$ (see Lemma \ref{lem backwardly self-similar}), we deduce that there exists a constant $M\geq Q$ such that
\begin{equation}\label{tangent measure critical}
  \mu^0=M\delta_0 \otimes dt\lfloor_{\R^-}+\left(M-Q\right)\delta_0 \otimes dt\lfloor_{\R^+}.
\end{equation}
By choosing $\eta(x,t)=\varphi(x)\psi(t)x$ in \eqref{energy identity limit+critical 1}, where $\varphi\in C_0^\infty(\R^n)$ and $\psi\in C_0^\infty(\R)$, we also deduce that
\begin{equation}\label{blow up of mean curvature}
  \mathbf{H}^0(0,0)=0.
\end{equation}

Now we come to the proof of Theorem \ref{thm energy quantization I}.
\begin{proof}

As in Section \ref{sec tangent flow}, there exist two sequences of solutions to \eqref{eqn} (with $p=(n+2)/(n-2)$) satisfying the assumptions in Lemma \ref{lem energy quantization I} in $Q_1(0,-2)$ and $Q_1(0,2)$. Therefore there exist two groups of finitely many bubbles, $\{W^k\}$ and $\{W^\ell\}$, such that
\begin{equation}\label{energy quantization general}
 M=\sum_{k}\int_{\R^n}|\nabla W^k|^2dx,  \quad M-Q=\sum_{\ell}\int_{\R^n}|\nabla W^\ell|^2dx.
\end{equation}

Furthermore, if all solutions are positive, then there exist  $N_1, N_2\in\mathbb{N}$ such that
\begin{equation}\label{energy quantization positive}
  M=N_1\Lambda, \quad  M-Q=N_2\Lambda.
\end{equation}

By the weak convergence of $|\nabla u_\infty^\lambda|^2dxdt+\mu^\lambda$ etc., we get
\begin{eqnarray*}
 \Theta(x,t;u_\infty,\mu)&=&  \lim_{\lambda\to0}\int_{1}^{2}\Theta_{\lambda^2s}(x,t;u_\infty,\mu)ds \\
   &=&  \lim_{\lambda\to 0}\int_{1}^{2}\Theta_s(0,0; u_\infty^\lambda, \mu^\lambda)ds\\
   &=& \frac{1}{n}\int_{1}^{2}s^{\frac{n}{2}}\left[\int_{\R^n} G(y,1)d\mu^0_{-s}(y)dy\right]ds\\
   &=&\frac{M}{n\left(4\pi\right)^{n/2}}.
\end{eqnarray*}
Substituting \eqref{energy quantization general} into this equality, we conclude the proof.
\end{proof}

A consequence of Theorem \ref{thm energy quantization I} is the following relation between $\Theta$ and $\theta$.
\begin{coro}\label{coro relation between densities}
  For $\mu$-a.e. $(x,t)$,
  \[\Theta(x,t;u_\infty,\mu)= \frac{1}{n\left(4\pi\right)^{\frac{n}{2}}} \theta(x,t),\]
and consequently, there exist finitely many bubbles $W^k$, $k=1,\cdots, N$, such that
  \[\theta(x,t)=\sum_{k=1}^{N}\int_{\R^n}|\nabla W^k|^2.\]
Moreover, if all solutions are positive, there exist an $N(x,t)\in \mathbb{N}$ such that
  \[\theta(x,t)= N(x,t)\Lambda \quad \mu-a.e..\]
\end{coro}
\begin{proof}
For those $(x,t)$ satisfying the condition \eqref{vanishing Morrey condition}, we have
\[ \theta(x,t)= \lim_{\lambda\to 0}\frac{\mu\left(Q_\lambda(x,t)\right)}{2\lambda^2}=\frac{1}{2}\mu^0(Q_1).\]
We can also assume there is no atom for $\nu$ at $(x,t)$ by avoiding a at most countable set. Then by the above structure theory of tangent flows and Theorem \ref{thm energy quantization I}, we get
\[ \mu^0(Q_1)=2M=2 n\left(4\pi\right)^{\frac{n}{2}} \Theta(x,t). \qedhere\]
\end{proof}

\subsection{Further properties of  defect measures}\label{subsec further properties of defect measure}

Some consequences follow from the above structure result on tangent flows. The first one is
\begin{lem}\label{lem time slice isolated}
For each $t$, $\Sigma_t^\ast$ is isolated.
\end{lem}
\begin{proof}
Assume by the contrary, there exists a sequence of points $(x_j,t)\in \Sigma_{t}^\ast$ converging to a limit point $(x,t)$. Because $\Sigma^\ast$ is closed, $(x,t)\in\Sigma^\ast$. Denote
\[
  \lambda_j:=|x_j-x|\to 0.
\]
Define the blow up sequence $(u_\infty^{\lambda_j}, \mu^{\lambda_j})$ with respect to the bast point $(x,t)$ as before. Assume without loss of generality that
\[ \frac{x_j-x}{\lambda_j}\to x_\infty\in\partial B_1.\]
By the upper semi-continuity of $\Theta$, we get
\begin{equation}\label{singular pt on unit sphere}
 \Theta(x_\infty,0;\mu^0)\geq \limsup_{j\to+\infty}\Theta(x_j,t;u_\infty,\mu)\geq \varepsilon_\ast.
\end{equation}
On the other hand,the tangent flow analysis  shows that $\mbox{spt}(\mu^0)=\{0\}\times\R$. This fact, together with the $\varepsilon$-regularity theorem and the fact that $u_\infty^0=0$, implies that
\begin{itemize}
  \item  $u_\infty^{\lambda_j}\to 0$ in $C^\infty_{loc}\left((\R^n\setminus\{0\})\times\R\right)$,
  \item and  for all $\lambda_j$ small, $\mbox{spt}(\mu^{\lambda_j})$ is contained in $B_{1/2}\times \R$.
\end{itemize}
Therefore,
\[ \Theta(x_\infty,0;\mu^0)=\lim_{j\to+\infty}\Theta\left(\frac{x_j-x}{\lambda_j},0;u_\infty^{\lambda_j}, \mu^{\lambda_j}\right)=0.\]
This is a contradiction with \eqref{singular pt on unit sphere}. In other words, there does not exist converging sequences in $\Sigma_t^\ast$.
\end{proof}

The next one is about the form of the stress-energy tensor $\mathcal{T}$.
\begin{lem}\label{lem stress-energy tensor}
$\mathcal{T}=I/n$ $\mu$-a.e..
\end{lem}
\begin{proof}
This is \eqref{representation of stress-energy tensor} in this special case. Here we explain briefly a direct proof. First, because $\mathcal{T}$ is $\mu$-measurable, it is approximately continuous $\mu$-a.e.. Choose a point $(x,t)$ so that $\mathcal{T}$  is approximate continuous and  Lemma \ref{lem vanishing of Morrey} holds
at this point. Take a sequence $\lambda\to0$ and define the tangent flow at this point as before, denoted by $\mu^0$. Let
\[\mathcal{T}^\lambda(y,s):=\mathcal{T}(x+\lambda y, t+\lambda^2s).\]
By the approximate continuity of $\mathcal{T}$ at $(x,t)$, we see
\[\mathcal{T}^\lambda d\mu^\lambda \rightharpoonup \mathcal{T}(x,t)d\mu^0.\]
Then we obtain the stationary condition for the tangent flow,
\[ \int_{\R^n\times\R} \left[I- n\mathcal{T}(x,t)\right]\cdot DY d\mu^0=0, \quad \forall Y\in C_0^\infty(\R^n\times \R, \R^n).\]
In view of the form of $\mu^0$ in \eqref{tangent measure critical}, substituting suitable $Y$ as test functions into this identity we deduce that $\mathcal{T}(x,t)=I/n$.
\end{proof}

The following lemma is a rigourous statement that  zero dimensional objects should have zero mean curvatures.
\begin{lem}\label{lem mean curvature zero}
  $\mathbf{H}=0$ $\mu$-a.e..
\end{lem}
\begin{proof}
We will show that, for a.e. $t\in(-,1,1)$,
\begin{equation}\label{mean curvature at time slice}
  \mathbf{H}=0  \quad \mu_t-\mbox{everywhere}.
\end{equation}

For a.e. $t$,
\begin{equation}\label{time slice in H1}
  \int_{B_1}\left[|\nabla u_\infty(x,t)|^2+|u_\infty(x,t)|^{p+1} +|\partial_tu_\infty(x,t)|^2\right]dx<+\infty,
\end{equation}
and by Lemma \ref{lem disintegration of defect measure}, $\mu_t$ is a well defined Radon measure.

By the form of $\mathcal{T}$, now the stationary condition \eqref{stationary condition limit} reads as
\begin{eqnarray}\label{stationary condition limit+critical}
  0 &=&  \int_{Q_1}\left[\left(\frac{|\nabla u_\infty|^2}{2}-\frac{|u_\infty|^{p+1}}{p+1}\right)\mbox{div}Y-DY(\nabla u_\infty,\nabla u_\infty)+\partial_tu_\infty \nabla u_\infty\cdot Y \right]\nonumber \\
  &+&\frac{1}{m}\int_{Q_1}\ \mathbf{H}\cdot Y d\mu, \quad \quad \forall Y\in C_0^\infty(Q_1,\R^n).
    \end{eqnarray}
Hence for a.e. $t$, we have the stationary condition
\begin{eqnarray}\label{stationary condition limit+critical+time slice}
  0 &=&  \int_{B_1}\left[\left(\frac{|\nabla u_\infty|^2}{2}-\frac{|u_\infty|^{p+1}}{p+1}\right)\mbox{div}X-DX(\nabla u_\infty,\nabla u_\infty)+\partial_tu_\infty \nabla u_\infty\cdot X\right] \nonumber \\
  &+&\frac{1}{m}\int_{B_1}\ \mathbf{H}\cdot X d\mu_t, \quad \quad \forall X\in C_0^\infty(B_1,\R^n).
    \end{eqnarray}
Now we claim that \\
{\bf Claim.} $u_\infty(t)\in W^{2,2}_{loc}(B_1)\cap L^{2p}_{loc}(B_1)$.

Once we have this claim, by integration by parts, we can show that $u_\infty(t)$ satisfies the stationary condition
\[0 =  \int_{B_1}\left[\left(\frac{|\nabla u_\infty|^2}{2}-\frac{|u_\infty|^{p+1}}{p+1}\right)\mbox{div}X-DX(\nabla u_\infty,\nabla u_\infty)+\partial_tu_\infty \nabla u_\infty\cdot X\right].\]
Combining this identity with \eqref{stationary condition limit+critical+time slice} we get
\[\int_{B_1}\ \mathbf{H}\cdot X d\mu_t=0, \quad \quad \forall X\in C_0^\infty(B_1,\R^n),\]
from which \eqref{mean curvature at time slice} follows.

{\bf Proof of the claim.} By Lemma \ref{lem time slice isolated}, $\Sigma_t^\ast$ is isolated. Since $u_\infty$ is smooth outside $\Sigma^\ast$, we only need to show that  for each $\xi\in\Sigma_t^\ast$, there exists a ball $B_r(\xi)$ such that $u_\infty(t)\in W^{2,2}(B_r(\xi))\cap L^{2p}(B_r(\xi))$.

We take a sufficiently small $\sigma$ and choose this ball so that
\[ \int_{B_r(\xi)}|u_\infty(t)|^{\frac{2n}{n-2}}\leq \sigma.\]
Take a standard cut-off function $\eta$ in $B_r(\xi)$ such that $\eta\equiv 1$ in $B_{r/2}(\xi)$. A direct calculation gives
\[-\Delta\left(u_\infty(t)\eta\right)= |u_\infty(t)|^{\frac{n-2}{2}} u_\infty \eta+ f_\eta,\]
where $f_\eta\in L^2(B_r(\xi))$. Then by the $W^{2,2}$ estimate for Laplacian operator and an application of H\"{o}lder inequality, if $\sigma$ is sufficiently small, we obtain
\[\|u_\infty(t)\eta\|_{W^{2,2}(B_r(\xi))}\lesssim \|u(t)\|_{L^2(B_r(\xi))}+\|f_\eta\|_{L^2(B_r(\xi))}.\]
The $L^{2p}$ estimate on $u_\infty(t)$ follows by combining this estimate with the equation for $u_\infty(t)$. The proof of the claim is complete.
\end{proof}

\begin{coro}
  $u_\infty$ satisfies the stationary condition \eqref{stationary condition I}.
\end{coro}
We do not know if the energy inequality \eqref{energy identity limit} can be decoupled in the same way.

Finally, for positive solutions, we note the following fact as a consequence of \eqref{energy quantization positive}:
\begin{itemize}
  \item  either $Q=0$, which implies that there is no atom at $(x,t)$,
  \item or $Q\geq \Lambda$, which implies that there is an atom in $\nu$ at $(x,t)$ and its mass is at least $\Lambda$.
\end{itemize}
In conclusion, we get
\begin{lem}\label{lem atoms in nu}
  If all solutions are positive, then the mass of each atom in $\nu$ is at least $\Lambda$. Consequently, there are at most finitely many atoms in $\nu$.
\end{lem}
The atoms of $\nu$ correspond to the singular point of $\Sigma$. Microscopically, such an atom comes from scalings of a connecting orbit (or a terrace of connecting orbits)
 \[ \left\{\begin{aligned}
& \partial_tu-\Delta u=u^p, \quad \mbox{in} ~~ \R^n\times\R,\\
& \int_{\R^n}\left[\frac{1}{2}|\nabla u|^2-\frac{1}{p+1}u^{p+1}\right]\Bigg|_{+\infty}^{-\infty}= Q.
\end{aligned}\right.
\]
However, we do not know if $\Sigma$ is smooth outside this singular set.

\newpage


\part{Energy concentration with only one bubble}\label{part one bubble}

\section{Setting}\label{sec setting I}
\setcounter{equation}{0}

In this part, $p=(n+2)/(n-2)$ is the Sobolev critical exponent, $u_i$ denotes a sequence of \emph{smooth, positive solutions} of \eqref{eqn} in $Q_1$, satisfying the following three assumptions:
  \begin{description}

  \item[(II.a) Weak limit] $u_i$ converges weakly to $u_\infty$ in $L^{p+1}(Q_1)$, and  $\nabla u_i$ converges weakly to $\nabla u_\infty$ in $L^2(Q_1)$. Here $u_\infty$ is a \emph{smooth solution} of \eqref{eqn} in $Q_1$.

  \item[(II.b) Energy concentration set] weakly as Radon measures
  \[ \left\{\begin{aligned}
& |\nabla u_i|^2dxdt\rightharpoonup |\nabla u_\infty|^2dxdt+\Lambda\delta_0\otimes dt,\\
& u_i^{p+1}dxdt\rightharpoonup u_\infty^{p+1}dxdt+\Lambda\delta_0\otimes dt.
\end{aligned}\right.
\]

 \item[(II.c) Convergence of time derivatives] as $i\to\infty$, $\partial_tu_i$ converges to $\partial_tu_\infty$ strongly in $L^2(Q_1)$.
 \end{description}
The assumption {\bf(II.b)} says there is only one bubble. From the above assumptions,  it is also seen that
  $u_i$ converges to $u_\infty$ in $C(-1,1;L^2(B_1))$.

The main result of this part is the following theorem, which can be viewed as a weak form of Schoen's Harnack inequality for Yamabe problem (see \cite{Schoen-course}, \cite{LiYanYan-Harnack}). As in Yamabe problem, this will be used to prove that there is no bubble towering.
\begin{thm}\label{thm no bubble towering}
Under the above assumptions, we must have $u_\infty\equiv 0$.
\end{thm}

In the following,  we also assume there exists a constant $L$  such that for all $i$,
\begin{equation}\label{Lip assumption}
 \sup_{-1<t<1}\int_{B_1}|\partial_tu_i(x,t)|^2dx\leq L.
\end{equation}
This assumption is in fact a consequence of {\bf(II.a-II.c)}, see Section \ref{sec Lipschitz hypothesis} in Part \ref{part many bubbles} for the proof.

The proof of Theorem \ref{thm no bubble towering} uses mainly a reverse version of \emph{the inner-outer gluing mechanism}.

\begin{figure}[htbp]
    \centering
        \includegraphics[width=0.9\textwidth]{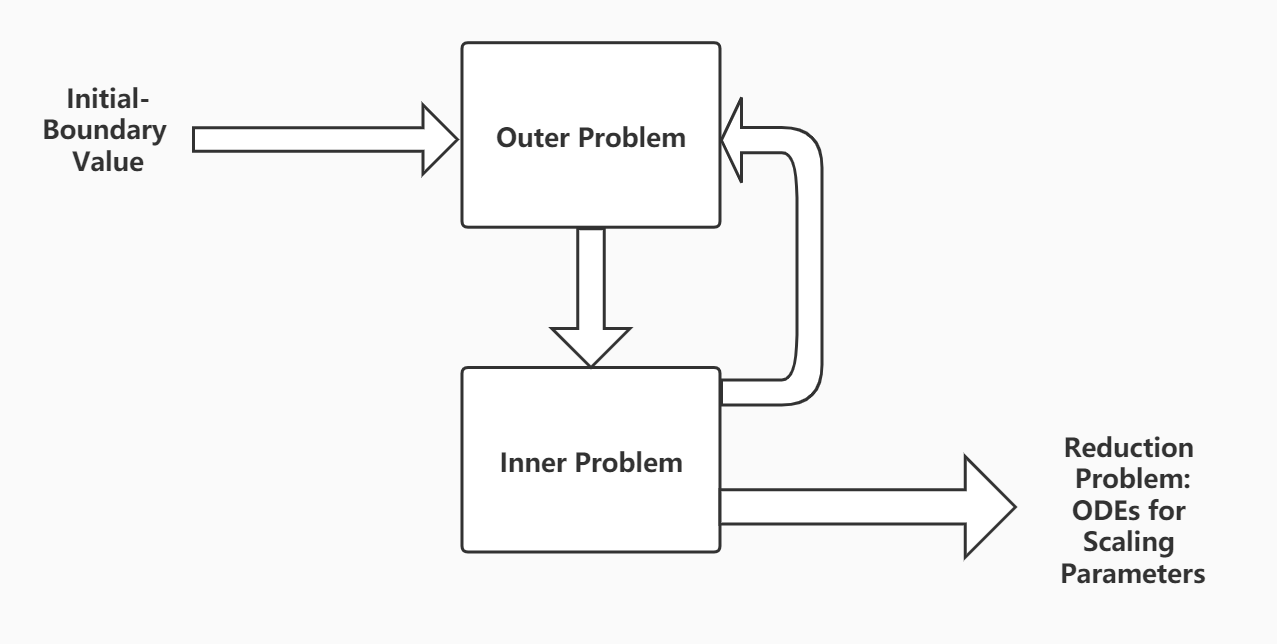}
         \caption{Inner-outer gluing mechanism}
      \end{figure}

\begin{enumerate}
  \item In Section \ref{sec blow up profile}, we describe the blow up profile of $u_i$, i.e. the form of bubbles when energy concentration phenomena appears, see Proposition \ref{prop blow up profile I}. This is the main order term of $u_i$, and it provides us the starting point for the decomposition in the next step.
  \item In Section \ref{sec decomposition}, we take two decompositions for $u_i$: the first one is an orthogonal decomposition where we decompose $u_i$ into a standard bubble (which is the main order term) and an error function (which is the next order term); the second one is the inner-outer decomposition, where we divide the error function into two further parts, the first one (the inner part) is localized near the bubble, and the second one (the outer part) is on the original scale.
  \item In Section \ref{sec inner}, we establish an estimate  for the inner problem, where we mainly use the nondegenaracy of the bubble (see Section \ref{sec bubbles}). Roughly speaking, this estimate reads as
\[ \mathcal{I}\leq A\mathcal{O}+ \mbox{higher order terms from scaling parameters etc.},\]
where $\mathcal{I}$ is a quantity measuring the inner component, $\mathcal{O}$ is  a quantity measuring the outer component, and $A$ is a constant.
  \item In Section \ref{sec outer}, we establish an estimate for the outer problem. This estimate reads  as
\begin{eqnarray*}
  \mathcal{O}&\leq&  B\mathcal{I}+\mbox{effect from initial-boundary value} \\
 &+&\mbox{higher order terms from scaling parameters etc.},
\end{eqnarray*}
where   $B$ is a constant.

The \emph{inner-outer gluing mechanism} works thanks to the fact that
\[ AB<1.\]
This follows from a fast decay estimate away from the bubble domains, where we mainly rely on a Gaussian bound on heat kernels associated to a parabolic operator with small Hardy term, see Moschini-Tesei \cite{Moschini}.

  \item In Section \ref{sec Harnack for lambda}, we combine these two estimates on inner and outer problems to establish an Harnack inequality for the scaling parameter. This gives a uniform in time control on the height of bubbles.
  \item In Section \ref{sec inner and outer 2}, by the estimate in Section \ref{sec Harnack for lambda}, we improve the estimates on inner and outer problems.
  \item In Section \ref{sec regularity on phi}, we improve further the estimates on the error function to an optimal one, such as uniform $L^\infty$ estimate, first order gradient H\"{o}lder estimate and Schauder estimate.
  \item In Section \ref{sec Pohozaev}, with the help of these estimates on the error function, we are able to linearize the Pohozaev identity. This is because a Pohozaev identity holds for $u$, and $u$ is approximated by the bubble at main order, which also satisfies a Pohozaev identity,  then   the next order term in the Pohozaev identity of $u$ gives further information.
  \item In Section \ref{sec Schoen's Harnack inequality}, we use this linearization of Pohozaev identity to establish a weak form of Schoen's Harnack inequality, which also finishes the proof of Theorem \ref{thm no bubble towering}.
  \item In Section \ref{sec bubble cluster}, we use this weak form of Schoen's Harnack inequality to exclude a special case of bubble clustering, where we use the same idea of constructing Green functions in the study of Yamabe problems. This result will be used in Section \ref{sec bubble cluster IV} of Part \ref{part first time singularity}.
\end{enumerate}

\subsection{ List of notations and conventions used in this part}
\begin{itemize}

  \item Given $\theta\in[0,1)$, $\alpha\geq 0$ and $k\in\mathbb{N}$, let $\mathcal{X}_{\alpha}^{k+\theta}$ be the space of  functions $\phi\in C^{k,\theta}(\R^n)$ with  the weighted $C^{k,\theta}$ norm
  \[\|\phi\|_{\alpha,k+\theta}:= \sup_{x\in\R^n}\sum_{\ell=0}^{k-1}\left(1+|x|\right)^{\ell+\alpha}|\nabla^\ell\phi(x)|+\sup_{x\in\R^n}\left(1+|x|\right)^{k+\alpha} \|\nabla^k\phi\|_{C^{\theta}(B_1(x))}.\]
        If $\theta=0$ and $k=0$, this space is written as $\mathcal{X}_\alpha$.

\item Throughout this part, an even function $\eta \in C_0^\infty (-2,2)$ will be fixed, which  satisfies $\eta \equiv 1$ in $(-1,1)$ and $|\eta^\prime|+|\eta^{\prime\prime}|\leq 10$. For any $R>0$, denote
\[\eta_R(y):=\eta \left(\frac{|y|}{R}\right).\]

\item $\alpha:=\frac{n-2}{2}$.

  \item $\bar{p}:=\min\{p,2\}$.

\item Two large constants $K\gg L \gg 1$ will be chosen in Section \ref{sec decomposition}.

\item In this part, unless otherwise stated, it is always assumed that $n\geq 7$, which implies $\alpha>2$.

\end{itemize}

\section{Blow up profile}\label{sec blow up profile}
\setcounter{equation}{0}

In this section we only assume $n\geq 3$. Here we prove
\begin{prop}[Blow up profile]\label{prop blow up profile I}
For any $t\in[-81/100,81/100]$, there exists a  unique maxima point of $u_i(\cdot,t)$ in the interior of $B_1(0)$. Denote this point by $\xi_i^\ast(t)$ and let
$\lambda_i^\ast(t):=u_i(\xi_i^\ast(t),t)^{-\frac{2}{n-2}}$.

As $i\to +\infty$,
\[\lambda_i^\ast(t)\to0, \quad  \xi_i^\ast(t)\to 0,  \quad \mbox{uniformly in } C([-81/100,81/100]),\]
and  the function
\[u_i^t(y,s):=\lambda_i^\ast(t)^{\frac{n-2}{2}}u_i\left(\xi_i^\ast(t)+\lambda_i^\ast(t)y , t+\lambda_i^\ast(t)^2s\right),\]
converges to $W(y)$ in $C^\infty_{loc}( \R^n\times\R)$.
\end{prop}
Let us first present some preliminary results, which are needed for the proof of this proposition.
The first one is a direct consequence of Theorem \ref{thm ep regularity II} (with the help of {\bf (II.b)}).
\begin{coro}\label{coro smooth convergence outside blow up locus}
As $i\to+\infty$,  $u_i$ are uniformly bounded in $C^\infty_{loc}\left(Q_1\setminus\Sigma\right)$.
\end{coro}

We also note that there is no concentration  for $\partial_tu_i$.
\begin{lem}\label{lem nonconcentration of time derivative}
  For any $\sigma>0$, there exists an $r(\sigma)>0$ such that for any $(x,t)\in Q_{9/10}$,
  \[\limsup_{i\to+\infty}\int_{Q_{r(\sigma)}(x,t)}|\partial_tu_i|^2<\sigma.\]
\end{lem}
\begin{proof}
  Assume by the contrary, there exists a $\sigma>0$, a sequence of points $(x_i,t_i)\in Q_{9/10}$ and a sequence of $r_i\to0$ such that
  \begin{equation}\label{absurd assumption I.0}
    \int_{Q_{r_i}(x_i,t_i)}|\partial_tu_i|^2\geq \sigma.
  \end{equation}
  Without loss of generality, assume $(x_i,t_i)\to (x_\infty,t_\infty)$. For any $r>0$ fixed, by {\bf(II.c)}, if $i$ is large enough,
\[ \int_{Q_r(x_\infty,t_\infty)}|\partial_tu_\infty|^2\geq  \int_{Q_{r_i}(x_i,t_i)}|\partial_tu_i|^2\geq \sigma.\]
This is a contradiction with the fact that $\partial_tu_\infty\in L^2(Q_1)$.
\end{proof}

Next is a result about the energy concentration behavior of each time slice $u_i(t)$.
\begin{lem}\label{lem close to one bubble}
  For any $\delta>0$, there exists an $r(\delta)\in(0,1)$ such that for any $t\in[-81/100,81/100]$,
\begin{equation}\label{close to one bubble}
\limsup_{i\to+\infty}\left( \left| \int_{B_{r(\delta)}}|\nabla u_i(x,t)|^2dx-\Lambda\right|+\left| \int_{B_{r(\delta)}} u_i(x,t)^{p+1}dx-\Lambda\right|\right)<\delta.
\end{equation}
\end{lem}
\begin{proof}
Take a sufficiently small  $\sigma>0$, and then choose $r(\sigma)$ according to Lemma \ref{lem nonconcentration of time derivative}. Without loss of generality, assume  $t=0$. By the scaling invariance of energy, in the following we need only to prove the corresponding result for the rescaling of $u_i$,
\[ \widetilde{u}_i(x,t):=r(\sigma)^{\frac{n-2}{2}}u_i\left(r(\sigma)x, r(\sigma)^2t\right).\]
However, in order not to complicate notations, we still use   $u_i$ to denote  $\widetilde{u}_i$. In particular, now  $u_i$ also satisfyies
\begin{equation}\label{smallness condition}
  \int_{Q_1}|\partial_t u_i|^2\leq \sigma.
\end{equation}

For each $r\in(0,1)$, recall that $\eta_r(x):=\eta(|x|/r)$ is a standard cut-off function in $B_{2r}$.
 Because $u_i$ is smooth,  we have the standard localized energy identity
 \begin{eqnarray}\label{energy identity I}
 &&\frac{d}{dt}\int_{B_1}\left[\frac{1}{2}|\nabla u_i(x,t)|^2-\frac{1}{p+1}u_i(x,t)^{p+1}\right]\eta_r(x)^2dx\\
  &= & -\int_{B_1}|\partial_tu_i(x,t)|^2\eta_r(x)^2dx+2\int_{B_1}\eta_r(x)\partial_tu_i(x,t)\nabla u_i(x,t)\cdot\nabla\eta_r(x)dx. \nonumber
 \end{eqnarray}
Integrating this identity and applying Cauchy-Schwarz inequality, we get
\begin{eqnarray}\label{oscillation of energy}
\mbox{osc}_{t\in(-1,1)}\int_{B_1}\left[\frac{1}{2}|\nabla u_i(x,t)|^2-\frac{1}{p+1}u_i(x,t)^{p+1}\right]\eta_r(x)^2 &\lesssim & \left(\int_{Q_1}|\partial_tu_i|^2\right)^{1/2} \nonumber\\
&\lesssim& \sigma^{1/2},
\end{eqnarray}
where we have used   \eqref{smallness condition}.

By \eqref{oscillation of energy}, {\bf (II.b)} and the smoothness of $u_\infty$,   we get
\begin{eqnarray}\label{energy level}
 && \int_{B_1}\left[\frac{1}{2}|\nabla u_i(x,t)|^2-\frac{1}{p+1}u_i(x,t)^{p+1}\right] \eta_r(x)^2 dxdt\\
 &=&\frac{\Lambda}{n}+O\left(\sigma^{1/2}\right)+O\left(r^n\right)+o_i(1), \quad \quad \quad \forall t\in(-1,1).\nonumber
\end{eqnarray}

Next, for each $t\in(-1,1)$, multiplying \eqref{eqn} by $u_i(t)\eta_r^2$ and integrating on $B_1$, we get
\begin{eqnarray}\label{integration by parts for eqn 1}
   &&  \int_{B_1}\left[|\nabla u_i(x,t)|^2-u_i(x,t)^{p+1}\right]\eta_r(x)^2dx  \\
    &=&  \int_{B_1}\left[\partial_tu_i(x,t)u_i(x,t)\eta_r(x)^2-2\eta_r(x)u_i(x,t)\nabla\eta_r(x)\cdot\nabla u_i(x,t)\right]dx. \nonumber
\end{eqnarray}
Concerning the right hand side of this equation, the following estimates hold.
\begin{itemize}
  \item  By Cauchy-Schwarz inequality and \eqref{Lip assumption}, we have
\begin{eqnarray*}
  &&\left| \int_{B_1}\partial_tu_i(x,t)u_i(x,t)\eta_r(x)^2 dx\right|\\
   &\lesssim&    \left(\int_{B_{ r}}\partial_tu_i(x,t)^2dx\right)^{\frac{1}{2}}\left(\int_{B_{ r}}u_i(x,t)^2dx\right)^{\frac{1}{2}}\\
  &\lesssim& L^{1/2} \left[ \left(\int_{B_{ r}}u_\infty(x,t)^2dx\right)^{\frac{1}{2}}+o_i(1)\right],
\end{eqnarray*}
where in the second inequality we have used the Lipschitz assumption \eqref{Lip assumption} and the uniform convergence of $u_i(t)$ in $C_{loc}(-1,1;L^2(B_1))$.

  \item Still by the uniform convergence of $u_i(t)$ in $C_{loc}(-1,1;L^2(B_1))$, we obtain
\begin{eqnarray*}
 &&\int_{B_1} \eta_r(x)u_i(x,t)\nabla\eta_r(x)\cdot\nabla u_i(x,t)dx\\
   &=& -\frac{1}{2}\int_{B_1}u_i(x,t)^2 \Delta\eta_r(x)^2dx \\
   &=&-\frac{1}{2} \int_{B_1}u_\infty(x,t)^2 \Delta\eta_r(x)^2dx+o_i(1).
\end{eqnarray*}
\end{itemize}

Substituting these two estimates into \eqref{integration by parts for eqn 1} and noting the smoothness of $u_\infty$, we get
\begin{equation}\label{integration by parts for eqn 2}
  \int_{B_1}\left[|\nabla u_i(x,t)|^2-u_i(x,t)^{p+1}\right]\eta_r(x)^2dx=o_r(1)+o_i(1).
\end{equation}
Plugging this relation into \eqref{energy level}, we obtain
\begin{equation}\label{uniform convergence of local energy}
  \left\{\begin{aligned}
& \int_{B_1} |\nabla u_i(x,t)|^2 \eta_r(x)^2 = \Lambda+O\left(\sigma^{1/2}\right)+o_r(1)+o_i(1),\\
& \int_{B_1}  u_i(x,t)^{p+1} \eta_r(x)^2 = \Lambda+O\left(\sigma^{1/2}\right)+o_r(1)+o_i(1).
\end{aligned}\right.
\end{equation}
Then \eqref{close to one bubble} follows by noting that $\eta_{r}\leq \chi_{B_r}\leq \eta_{2r}$ and choosing a suitable $r$.
\end{proof}
Some consequences follow directly from this lemma.
\begin{enumerate}
  \item   There exists a constant $C(u_\infty,L)$ such that
\begin{equation}\label{uniform energy bound}
  \int_{B_{4/5}}\left[|\nabla u_i(t)|^2+u_i(t)^{p+1}\right] \leq C(u_\infty,L), \quad \forall t\in[-81/100,81/100].
\end{equation}
  \item By the  convergence of $u_i(t)$ in $C^\infty_{loc}(B_1\setminus\{0\})$, we deduce that for any $t\in[-81/100,81/100]$,
\begin{equation}\label{measure convergence of time slice}
   |\nabla u_i(t)|^2dx\rightharpoonup |\nabla u_\infty|^2dx+ \Lambda\delta_0, \quad  u_i(t)^{p+1}dx\rightharpoonup u_\infty^{p+1}dx+ \Lambda\delta_0.
\end{equation}
\end{enumerate}

 Equation \eqref{measure convergence of time slice} implies that for any $t\in[-81/100,81/100]$,
\[\sup_{B_1}u_i(x,t)\to+\infty \quad \mbox{as} \quad i\to+\infty.\]
Because $u_i(t)$ are bounded in $C_{loc}(B_1\setminus\{0\})$, this supremum must be a maxima, and it is attained near $0$. Take one such maximal point $\xi_i^\ast(t)$, and define $\lambda_i^\ast(t)$, $u_i^t$ as in the statement of Proposition \ref{prop blow up profile I}. By definition,
\begin{equation}\label{normalization condition I}
 u_i^t(0,0)=\max_{y\in B_{\lambda_i^\ast(t)^{-1}/2}}u_i^t(y,0)=1.
\end{equation}
Scaling \eqref{Lip assumption} gives
\begin{equation}\label{scaling integral of time derivative}
  \lim_{i\to+\infty}\int_{B_R}\partial_tu_i^t(y,s)^2dy=0, \quad \mbox{for any } R>0, ~~~ s\in\R.
\end{equation}
By scaling \eqref{close to one bubble}, we see for any $R>0$ and $s\in\R$,
\begin{equation}\label{scaling integral of energy}
  \left\{\begin{aligned}
& \limsup_{i\to+\infty}\int_{B_R} |\nabla u_i^t(y,s)|^2dy\leq \Lambda,\\
& \limsup_{i\to+\infty}\int_{B_R}  u_i^t(y,s)^{p+1} dy\leq \Lambda.
\end{aligned}\right.
\end{equation}
Then by Struwe's global compactness theorem (see \cite{Struwe1984compactness}), we deduce that for any $s\in\R$, $u_i^t(\cdot,s)$ converges strongly in $W^{1,2}_{loc}(\R^n)$ and $L^{p+1}_{loc}(\R^n)$. Denote its limit by $u_\infty^t$. By \eqref{scaling integral of time derivative} and \eqref{scaling integral of energy}, $u_\infty^t$ is a   nonnegative solution of \eqref{stationary eqn}. In view of its $H^1$ regularity, it must be smooth. Hence for any $(y,s)\in\R^n\times\R$, there exists an $r>0$ such that
\[ \int_{Q_r^{-}(y,s)}\left[|\nabla u_\infty^t|^2+\left(u_\infty^t\right)^{p+1}\right]\leq \frac{\varepsilon_\ast}{2}r^2.\]
By the above strong convergence of $u_i^t$, we get
\[ \limsup_{i\to+\infty}\int_{Q_r^{-}(y,s)}\left[|\nabla u_i^t|^2+\left(u_i^t\right)^{p+1}\right]<\varepsilon_\ast r^2.\]
Therefore Theorem \ref{thm ep regularity I} is applicable, which implies that $u_i^t$ are uniformly bounded in $C^\infty(Q_{\delta_\ast r}(y,s))$. In conclusion, $u_i^t$ are uniformly bounded in $C^\infty_{loc}(\R^n\times\R)$. Hence it converges to $u_\infty^t$ in $C^\infty_{loc}(\R^n\times\R)$. Then we can let $i\to+\infty$ in \eqref{normalization condition I}, which leads to
\[ u_\infty^t(0)=\max_{y\in \R^n}u_\infty^t(y)=1.\]
By the Liouville theorem of Caffarelli-Gidas-Spruck \cite{Caffarelli-Gidas-Spruck}, we get
\[u_\infty^t\equiv W \quad \mbox{in  } \R^n.\]

Next we give a scaling invariant estimate for $u_i$ away from $\xi_i^\ast(t)$.
\begin{lem}\label{lem scaling invariant estimate I}
For any $R$ large enough, there exist two constants  $\sigma(R)\ll 1$ and $C(R)>0$ such that, outside $ B_{R\lambda_i^\ast(t)}(\xi_i^\ast(t))$ we have
 \begin{equation}\label{scaling invariant estimate I}
  \left\{
\begin{aligned}
&  u_i(x,t)\leq \sigma(R)|x-\xi_i^\ast(t)|^{-\frac{n-2}{2}}+C(R), \\
&|\nabla u_i(x,t)|\leq \sigma(R)|x-\xi_i^\ast(t)|^{-\frac{n}{2}}+C(R) ,\\
&|\nabla^2u_i(x,t)|+|\partial_tu_i(x,t)|\leq \sigma(R) |x-\xi_i^\ast(t)|^{-\frac{n+2}{2}}+C(R).
\end{aligned}\right.
  \end{equation}
\end{lem}
\begin{proof}

For any $\delta>0$, by the smooth convergence of $u_i^t$ obtained above, there exists an $R(\delta)\gg 1$ such that for any $t\in[-81/100,81/100]$,
\begin{equation}\label{bubble energy}
  \left\{\begin{aligned}
& \int_{B_{R(\delta)\lambda_i^\ast(t)}(\xi_i^\ast(t))}|\nabla u_i(x,t)|^2\geq \Lambda-\delta/4,\\
& \int_{B_{R(\delta)\lambda_i^\ast(t)}(\xi_i^\ast(t))}  u_i(x,t)^{p+1}\geq \Lambda-\delta/4.
\end{aligned}\right.
\end{equation}
By Lemma \ref{lem close to one bubble}, there exists an $r(\delta)\ll 1$ such that for any $t\in[-81/100,81/100]$,
\begin{equation}\label{close to one bubble 2}
  \left\{\begin{aligned}
& \int_{B_{r(\delta)}}|\nabla u_i(x,t)|^2\leq \Lambda+\delta/4,\\
& \int_{B_{r(\delta)}} u_i(x,t)^{p+1}\leq \Lambda+\delta/4.
\end{aligned}\right.
\end{equation}
Combining \eqref{bubble energy} with \eqref{close to one bubble 2}, we get
\begin{equation}\label{small energy on neck}
\sup_{|t| \leq 81/100} \int_{B_{r(\delta)}\setminus B_{R(\delta)\lambda_i^\ast(t)}(\xi_i^\ast(t))}\left[|\nabla u_i(x,t)|^2+u_i(x,t)^{p+1}\right]dx \leq\delta.
\end{equation}

By Corollary \ref{coro smooth convergence outside blow up locus}, there   exists a constant $C(\delta)$ such that
\[u_i(x,t)+|\nabla u_i(x,t)|+|\nabla^2u_i(x,t)|+|\partial_tu_i(x,t)| \leq C(\delta), \quad \mbox{if} ~ x\notin B_{r(\delta)}(\xi_i^\ast(t)). \]

It remains to prove \eqref{scaling invariant estimate I} when  $x\in B_{r(\delta)}\setminus B_{R(\delta)\lambda_i^\ast(t)}(\xi_i^\ast(t))$. We argue by contradiction, so assume there exists a sequence $(x_i,t_i)$, with $t_i\in[-81/100,81/100]$ and $x_i\in  B_{r(\delta)}(0)\setminus B_{R(\delta)\lambda_i^\ast(t_i)}(\xi_i^\ast(t_i))$, but for some $\sigma$ (to be determined below),
\begin{equation}\label{absurd assumption I}
 u_i(x_i,t_i)\geq \sigma |x_i-\xi_i^\ast(t_i)|^{-\frac{n-2}{2}}.
\end{equation}

Let $r_i:=|x_i-\xi_i^\ast(t_i)|$ and
\[\widetilde{u}_i(y,s):=r_i^{\frac{n-2}{2}}u_i\left(\xi_i^\ast(t_i)+r_i y, t_i+r_i^2s\right).\]
Denote
\[\widetilde{x}_i:=\frac{x_i-\xi_i^\ast(t_i)}{r_i}.\]
It lies on $\partial B_1$. Assume it converges to a limit point $\widetilde{x}_\infty\in\partial B_1$.

Scaling \eqref{bubble energy} and \eqref{close to one bubble} leads to
\begin{equation}\label{bubble energy 2}
  \left\{\begin{aligned}
& \int_{B_{R(\delta)^{-1}}}|\nabla \widetilde{u}_i(y,0)|^2\geq \Lambda-\delta/4,\\
& \int_{B_{R(\delta)^{-1}}} \widetilde{u}_i(y,0)^{p+1}\geq \Lambda-\delta/4
\end{aligned}\right.
\end{equation}
and for any $R>0$ and $s\in\R$,
\begin{equation}\label{close to one bubble 3}
  \left\{\begin{aligned}
& \int_{B_R(0)}|\nabla \widetilde{u}_i(y,s)|^2\leq \Lambda+\delta/4,\\
& \int_{B_R(0)} \widetilde{u}_i(y,s)^{p+1}\leq \Lambda+\delta/4.
\end{aligned}\right.
\end{equation}
Then by noting that for any $s$, $\partial_s\widetilde{u}_i(\cdot,s)$ converges to $0$ strongly in $L^2_{loc}(\R^n)$, we can apply Lemma \ref{lem close to one bubble} to find an $\rho(\delta)>0$ such that for any $s\in[-1,1]$,
\[\left| \int_{B_{\rho(\delta)}}|\nabla \widetilde{u}_i(y,s)|^2dy-\Lambda\right|+\left| \int_{B_{\rho(\delta)}} \widetilde{u}_i(y,s)^{p+1}dy-\Lambda\right|<\delta.\]
Combining this estimate with \eqref{close to one bubble 3}, we get
\[ \int_{-1}^{0}\int_{B_{1/2}(\widetilde{x}_i)}\left(|\nabla\widetilde{u}_i|^2+\widetilde{u}_i^{p+1}\right)\leq C\delta.\]
By choosing $C\delta<\varepsilon_\ast$, we can apply Theorem \ref{thm ep regularity I} to deduce that
\[\widetilde{u}_i(\widetilde{x}_i,0)+|\nabla \widetilde{u}_i(\widetilde{x}_i,0)|  + |\nabla^2\widetilde{u}_i(\widetilde{x}_i,0)|+|\partial_s\widetilde{u}_i(\widetilde{x}_i,0)| \leq c(\delta),\]
where $c(\delta)$ is a small constant depending on $\delta$. Scaling back to $u_i$ we get a contradiction with \eqref{absurd assumption I}. In other words, \eqref{scaling invariant estimate I} must hold in $B_{r(\delta)}\setminus B_{R(\delta)\lambda_i^\ast(t)}(\xi_i^\ast(t))$.
\end{proof}

Combining Proposition \ref{prop blow up profile I} and Lemma \ref{lem scaling invariant estimate I}. we get
\begin{coro}\label{coro unique maximal point}
  For each $t\in[-81/100,81/100]$, $\xi_i^\ast(t)$ is the unique maximal point of $u_i(t)$ in $B_1$.
\end{coro}
This completes the proof of Proposition \ref{prop blow up profile I}.

\section{Orthogonal and inner-outer decompositions}\label{sec decomposition}
\setcounter{equation}{0}

For simplicity of notations, from here to Section \ref{sec inner and outer 2}  we are concerned with a fixed solution $u_i$ with sufficiently large index $i$, and denote it  by $u$. In this section we define the orthogonal decomposition and inner-outer decomposition for $u$.

\subsection{Orthogonal decomposition}

From now on, until Section \ref{sec Schoen's Harnack inequality}, a constant $K\gg 1$ will be fixed. In the following we will use the notations about bubbles and kernels to their linearized equations, $W_{\xi,\lambda}$ and $Z_{i,\xi,\lambda}$, see Appendix \ref{sec bubbles}.

\begin{prop}[Orthogonal condition]\label{prop orthogonal decompostion}

For any $t\in[-81/100,81/100]$, there exists a unique $(a(t),\xi(t),\lambda(t))\in\R\times\R^n\times \R^+$ with
\begin{equation}\label{deviation form max pt}
\frac{|\xi(t)-\xi^\ast(t)|}{\lambda(t)}+\Big|\frac{\lambda(t)}{\lambda^\ast(t)}-1\Big|+\Big|\frac{a(t)}{\lambda(t)}\Big|=o(1),
\end{equation}
such that for each $i=0,\cdots, n+1$,
\begin{equation}\label{orthogonal condition}
  \int_{B_1}\left[u(x,t)-W_{\xi(t),\lambda(t)}(x)-a(t)Z_{0,\xi(t),\lambda(t)}(x)\right]\eta_K\left(\frac{x-\xi(t)}{\lambda(t)}\right) Z_{i,\xi(t),\lambda(t)}(x)dx =0.
\end{equation}
\end{prop}
\begin{proof}
For these $t$, set
\[\widetilde{u}(y):=\lambda^\ast(t)^{\frac{n-2}{2}}u(\xi^\ast(t)+\lambda^\ast(t)y,t).\]

Define a smooth map $F$ from  $B_{1/2}^{n+2}\subset\R^{n+2}$ into $\R^{n+2}$ as
\[F(\xi,\lambda, a):= \left(\int_{\R^n}\left[\widetilde{u}(y)-W_{\xi,1+\lambda}(y)-aZ_{0,\xi,1+\lambda}(y)\right]\eta_K\left(\frac{y-\xi}{1+\lambda}\right) Z_{i,\xi,1+\lambda}(y) dy\right)_{i=0}^{n+1}.\]
The task is reduced to find a solution of $F(\xi,\lambda, a)=0$.

By Proposition \ref{prop blow up profile I}, $|F(0,0,0)|\ll 1$. By a direct calculation and
applying Proposition \ref{prop blow up profile I} once again, we find a fixed small constant $\rho_\ast$ such that, for any $(\xi,\lambda,a)\in B_{\rho_\ast}^{n+2}$, $DF(\xi,\lambda,a)$ is diagonally dominated, and hence invertible. Then the inverse function theorem implies
the existence and uniqueness of the solution $F=0$ in $B_{\rho_\ast}^{n+2}$.
\end{proof}
In the above, in order to show that $DF(\xi,\lambda,a)$ is diagonally dominated, we need $n\geq 5$ to make sure $Z_{n+1}\in L^2(\R^{n})$, and we also need the $n\geq 7$ assumption  because at first we only know that $u$ decays like $|y|^{-(n-2)/2}$. (This term appears in the integral containing $\nabla\eta$.)

For later purpose, we notice a Lipschitz  bound on the parameters $(a(t),\xi(t),\lambda(t))$.
\begin{lem}
For any $t\in[-81/100,81/100]$,
\begin{equation}\label{Lip for parameters}
  \left|a^\prime(t)\right|+\left|\xi^\prime(t)\right|+\left|\lambda^\prime(t)\right|\lesssim \left(\int_{B_1}\partial_tu(x,t)^2dx\right)^{1/2}.
\end{equation}
\end{lem}
\begin{proof}
For each $i=1,\cdots,n$, because $Z_i\eta_K$ is orthogonal to $Z_0$ in $L^2(\R^n)$, the orthogonal condition \eqref{orthogonal condition} reads as
\[
  \int_{B_1}\left[u(x,t)-W_{\xi(t),\lambda(t)}(x)\right]\eta_K\left(\frac{x-\xi(t)}{\lambda(t)}\right) Z_{i,\xi(t),\lambda(t)}(x)dx =0.
\]
Differentiating this equality in $t$, by the $L^2$ orthogonal relations between different  $Z_i$, we obtain
\begin{eqnarray}\label{representation of xi prime}
   &&  \left[\int_{B_1}\eta_K\left(\frac{x-\xi(t)}{\lambda(t)}\right) Z_{i,\xi(t),\lambda(t)}(x)^2dx\right]\xi_i^\prime(t) \nonumber\\
    &=&   \int_{B_1} \partial_tu(x,t)\eta_K\left(\frac{x-\xi(t)}{\lambda(t)}\right) Z_{i,\xi(t),\lambda(t)}(x)dx\\
    &+&  \int_{B_1}\left[u(x,t)-W_{\xi(t),\lambda(t)}(x)\right]\frac{\partial}{\partial t}\left[\eta_K\left(\frac{x-\xi(t)}{\lambda(t)}\right) Z_{i,\xi(t),\lambda(t)}(x)\right]dx. \nonumber
\end{eqnarray}

First, by the definition of $Z_{i,\xi(t),\lambda(t)}$, we have
\begin{equation*}\label{12.1}
  \int_{B_1}\eta_K\left(\frac{x-\xi(t)}{\lambda(t)}\right) Z_{i,\xi(t),\lambda(t)}(x)^2dx=\int_{\R^n}\eta_K(y)Z_i(y)^2dy\geq c.
\end{equation*}
Next, by Cauchy-Schwarz inequality and \eqref{Lip assumption}, we get
\begin{eqnarray*}\label{12.2}
 &&\left| \int_{B_1} \partial_tu(x,t)\eta_K\left(\frac{x-\xi(t)}{\lambda(t)}\right) Z_{i,\xi(t),\lambda(t)}(x)dx \right|  \\
 &\leq& \left(\int_{B_1}\partial_tu(x,t)^2dx\right)^{1/2}\left(\int_{B_1}\eta_K\left(\frac{x-\xi(t)}{\lambda(t)}\right)^2 Z_{i,\xi(t),\lambda(t)}(x)^2dx\right)^{1/2}\\
 &\lesssim& \left(\int_{B_1}\partial_tu(x,t)^2dx\right)^{1/2}.
\end{eqnarray*}
Finally, because in $B_{K\lambda(t)}(\xi(t))$,
\[
  \left\{\begin{aligned}
&\left|u(x,t)-W_{\xi(t),\lambda(t)}(x)\right|=o(1) \lambda(t)^{-\frac{n-2}{2}}\left(1+\frac{|x-\xi(t)|}{\lambda(t)}\right)^{-\frac{n-2}{2}},\\
&\left|\frac{\partial}{\partial t} \eta_K\left(\frac{x-\xi(t)}{\lambda(t)}\right) Z_{i,\xi(t),\lambda(t)}(x)\right|\lesssim \frac{\left|\xi^\prime(t)\right|+\left|\lambda^\prime(t)\right|}{\lambda(t)^{\frac{n+2}{2}}}\left(1+\frac{|x-\xi(t)|}{\lambda(t)}\right)^{1-n},\\
& \left| \eta_K\left(\frac{x-\xi(t)}{\lambda(t)}\right)\frac{\partial}{\partial t} Z_{i,\xi(t),\lambda(t)}(x)\right|\lesssim \frac{\left|\xi^\prime(t)\right|+\left|\lambda^\prime(t)\right|}{\lambda(t)^{\frac{n+2}{2}}}\left(1+\frac{|x-\xi(t)|}{\lambda(t)}\right)^{1-n},
\end{aligned}\right.
\]
we get
\begin{eqnarray*}
   && \left|\int_{B_1}\left[u(x,t)-W_{\xi(t),\lambda(t)}(x)\right]\frac{\partial}{\partial t}\left[\eta_K\left(\frac{x-\xi(t)}{\lambda(t)}\right) Z_{i,\xi(t),\lambda(t)}(x)\right]dx \right| \\
  &\ll & \lambda(t)^{-n}\left(\left|\xi^\prime(t)\right|+\left|\lambda^\prime(t)\right|\right)\int_{B_1}\left(1+\frac{|x-\xi(t)|}{\lambda(t)}\right)^{2-\frac{3n}{2}}\\
  &\lesssim& \left|\xi^\prime(t)\right|+\left|\lambda^\prime(t)\right|.
\end{eqnarray*}

Plugging these three estimates into \eqref{representation of xi prime}, we obtain
\begin{equation}\label{estimate of xi prime}
  |\xi^\prime(t)|\lesssim \left(\int_{B_1}\partial_tu(x,t)^2dx\right)^{1/2}+o\left(\left|\xi^\prime(t)\right|+\left|\lambda^\prime(t)\right|\right).
\end{equation}

In the same way, we get
\begin{equation}\label{estimate of lambda prime}
  |\lambda^\prime(t)|\lesssim \left(\int_{B_1}\partial_tu(x,t)^2dx\right)^{1/2}+o\left(\left|\xi^\prime(t)\right|+\left|\lambda^\prime(t)\right|+\left|a^\prime(t)\right|\right)
\end{equation}
and
\begin{equation}\label{estimate of a prime}
  |a^\prime(t)|\lesssim \left(\int_{B_1}\partial_tu(x,t)^2dx\right)^{1/2}+o\left(\left|\xi^\prime(t)\right|+\left|\lambda^\prime(t)\right|+\left|a^\prime(t)\right|\right).
\end{equation}
Adding these three estimates together we get \eqref{Lip for parameters}.
\end{proof}

\subsection{Inner-outer decomposition}

In the following we denote the error function
\[\phi(x,t):=u(x,t)-W_{\xi(t),\lambda(t)}(x)-a(t)Z_{0,\xi(t),\lambda(t)}(x),\]
\[W_\ast(x,t):=W_{\xi(t),\lambda(t)}(x), \quad Z_{i,\ast}(x,t):=Z_{i,\xi(t),\lambda(t)}(x)\]
and $Z_\ast:=(Z_{0,\ast},Z_{1,\ast},\cdots,Z_{n+1,\ast})$.

Combining Lemma \ref{lem scaling invariant estimate I} and Proposition \ref{prop orthogonal decompostion}, we obtain
\begin{prop}\label{prop scaling invariant estimate for error}
In $B_1$,
  \[
  \left\{
\begin{aligned}
& |\phi(x,t)|=o\left(\left(\lambda(t)+|x-\xi_\ast(t)|\right)^{-\frac{n-2}{2}}\right)+O(1), \\
&|\nabla \phi(x,t)|=o\left(\left(\lambda(t)+|x-\xi_\ast(t)|\right)^{-\frac{n}{2}}\right)+O(1),\\
&|\nabla^2\phi(x,t)|+|\phi_t(x,t)|=o\left(\left(\lambda(t)+|x-\xi_\ast(t)|\right)^{-\frac{n+2}{2}}\right)+O(1).
\end{aligned}\right.
  \]
\end{prop}

The error function $\phi$ satisfies
\begin{equation}\label{error eqn}
  \partial_t\phi-\Delta\phi=pW_\ast^{p-1}\phi+\mathcal{N}+\left(-a^\prime+\mu_0\frac{a}{\lambda^2},\xi^\prime ,\lambda^\prime \right)\cdot Z_\ast-a\partial_tZ_{0,\ast},
\end{equation}
where  $\{\}^\prime$ denotes differentiation in $t$, and the nonlinear term is defined by
\[\mathcal{N}:=\left(W_\ast+\phi+aZ_{0,\ast}\right)^p- W_\ast^p-p W_\ast^{p-1}\left(\phi+aZ_{0,\ast}\right).\]

Now take a further decomposition of $\phi$, \emph{the inner-outer decomposition}, as follows.
Keep $K$ as the large constant used in Proposition \ref{prop orthogonal decompostion}. Take another constant $L$ satisfying $1\ll L \ll K$.
Denote
\begin{equation}\label{definition of cut-off}
   \eta_{in}(x,t):=\eta\left(\frac{x-\xi(t)}{K\lambda(t)}\right), \quad \eta_{out}(x,t):=\eta\left(\frac{x-\xi(t)}{L\lambda(t)}\right).
\end{equation}
Set
\[\phi_{in}(x,t):=\phi(x,t)\eta_K(x,t), \quad \phi_{out}(x,t):=\phi(x,t)\left[1-\eta_{out}(x,t)\right].\]

For later purpose, we introduce two quantities:
\[\mathcal{I}(t):=\left[L\lambda(t)\right]^\alpha\sup_{B_{2L\lambda(t)}(\xi(t))\setminus B_{L\lambda(t)}(\xi(t))}|\phi|+\left[L\lambda(t)\right]^{1+\alpha}\sup_{B_{2L\lambda(t)}(\xi(t))\setminus B_{L\lambda(t)}(\xi(t))}|\nabla\phi|\]
and
\[\mathcal{O}(t):=\left[K\lambda(t)\right]^\alpha\sup_{B_{2K\lambda(t)}(\xi(t))\setminus B_{K\lambda(t)}(\xi(t))}|\phi|+\left[K\lambda(t)\right]^{1+\alpha}\sup_{B_{2K\lambda(t)}(\xi(t))\setminus B_{K\lambda(t)}(\xi(t))}|\nabla\phi|.\]
By Proposition \ref{prop scaling invariant estimate for error}, for any $t\in[-81/100,81/100]$,
\[\mathcal{I}(t)+\mathcal{O}(t)\ll 1.\]
Starting from this smallness, we will improve it to an explicit bound in the following sections.

\section{Inner problem}\label{sec inner}
\setcounter{equation}{0}

In this section we give an $C^{1,\theta}$ estimate on the inner component $\phi_{in}$, see Proposition \ref{prop decay estimate}.

Define a new coordinate system around $\xi(t)$ by
\[y:=\frac{x-\xi(t)}{\lambda(t)}, \quad \quad \tau=\tau(t).\]
Here the new time variable $\tau$ is determined by the relation
\[\tau^\prime(t)=\lambda(t)^{-2}, \quad \quad \tau(0)=0.\]
Because there is a one to one correspondence between $\tau$ and $t$, in the following we will not distinguish between them. For example, we will use $\mathcal{O}(\tau)$ instead of the notation $\mathcal{O}(t(\tau))$.

It is convenient to write $\phi$ in these new coordinates. By denoting
\[\varphi(y,\tau):=\lambda(t)^{\frac{n-2}{2}}\phi(x,t),\]
we get
\begin{eqnarray}\label{inner eqn}
\partial_\tau\varphi-\Delta \varphi
 & =&p W^{p-1}\varphi+\left(-\frac{\dot{a}}{\lambda}+\mu_0\frac{a}{\lambda}, \frac{\dot{\xi}}{ \lambda},\frac{\dot{\lambda}}{\lambda} \right)\cdot Z  \nonumber\\
 &+&\mathcal{N}+\frac{\dot{\xi}}{\lambda}\cdot\nabla\varphi+\frac{\dot{\lambda}}{\lambda}\left(y\cdot\nabla\varphi+\frac{n-2}{2}\varphi\right)\\
 &+&\frac{a}{\lambda}\left[\frac{\dot{\xi}}{\lambda}\cdot\nabla Z_0+\frac{\dot{\lambda}}{\lambda}\left(y\cdot\nabla Z_0+\frac{n}{2}Z_0\right)\right]. \nonumber
\end{eqnarray}
Here  $ \dot{\{\}} $ denotes differentiation in $\tau$, and by abusing notations,
\[\mathcal{N}:=\left( W+\varphi +\frac{a}{\lambda}Z_0\right)^p- W^p-p W^{p-1} \left(\varphi+\frac{a}{\lambda}Z_0\right).\]

We will need the following pointwise estimate on $\mathcal{N}$. Recall that $\bar{p}=\min\{p,2\}$.
\begin{lem}\label{lem estimate on nonlinearity}
The following point-wise inequality holds:
 \[|\mathcal{N}|\lesssim  |\varphi|^{\bar{p}}+\Big|\frac{a}{\lambda}\Big|^{\bar{p}} Z_0^{\bar{p}}.\]
\end{lem}
\begin{proof}
There exists a $\vartheta\in[0,1]$ such that
\begin{eqnarray*}
     && \left( W+\varphi +\frac{a}{\lambda}Z_0\right)^p- W^p-p W^{p-1} \left(\varphi+\frac{a}{\lambda}Z_0\right) \\
&=& p\left\{\left[ W+\vartheta\left(\varphi +\frac{a}{\lambda}Z_0\right)\right]^{p-1}-W^{p-1}\right\}\left(\varphi+\frac{a}{\lambda}Z_0\right).
  \end{eqnarray*}
Since both $W$ and $W+\varphi +\frac{a}{\lambda}Z_0=u$ are bounded by a universal constant, and $\varphi+\frac{a}{\lambda}Z_0$ is small, by considering the two cases $W\geq |\varphi +\frac{a}{\lambda}Z_0|$ and $W<|\varphi +\frac{a}{\lambda}Z_0|$ separately, we get
\[
| \mathcal{N} |   \lesssim \Big|\varphi+\frac{a}{\lambda}Z_0\Big|^{\bar{p}}
     \lesssim  |\varphi|^{\bar{p}}+\Big|\frac{a}{\lambda}\Big|^{\bar{p}} Z_0^{\bar{p}}. \qedhere
\]
\end{proof}

By scaling the estimates in Proposition \ref{prop scaling invariant estimate for error}, we obtain
\begin{prop}\label{prop initial smallness on error}
For each $\tau$,
  \[K^\alpha\sup_{y\in B_{2K}}\left[|\varphi(y,\tau)|+K|\nabla\varphi(y,\tau)|+K^2\left(|\nabla^2\varphi(y,\tau)|+|\partial_\tau\varphi(y,\tau)|\right)\right]
  =o(1).
  \]
\end{prop}

In $(y,\tau)$ coordinates, the inner error function $\phi_{in}$ is
\[\varphi_K(y,\tau):=\varphi(y,\tau)\eta_K(y).\]
For each $\tau$, the support of $\varphi_K(\cdot,\tau)$ is contained in $B_K$.
 By \eqref{orthogonal condition}, for each $ i=0,\cdots,n+1$,
  \begin{equation}\label{orthogonal in inner}
  \int_{\R^n} \varphi_K(y,\tau)Z_i(y)dy =0.
  \end{equation}

The equation for $\varphi_K$ is
\begin{equation}\label{inner eqn with cut-off}
 \partial_\tau\varphi_K-\Delta_y\varphi_K
  =pW^{p-1}\varphi_K +\lambda^{-1}\left(-\dot{a}+\mu_0a,\dot{\xi},\dot{\lambda} \right)\cdot Z+E_K.
\end{equation}
Here
\begin{eqnarray*}
  E_K &:=& -2\nabla\varphi\cdot\nabla\eta_K-\varphi\Delta\eta_K+\mathcal{N} \eta_K \\
  &+&  \lambda^{-1}\left(-\dot{a}+\mu_0a,\dot{\xi},\dot{\lambda} \right)\cdot Z\left(\eta_K-1\right)\\
  &+&\lambda^{-1}\left[\dot{\lambda} \left(\frac{n-2}{2}\varphi+  y\cdot\nabla_y\varphi\right)+\dot{\xi} \cdot\nabla_y\varphi\right]\eta_K \\
  &+& \frac{a}{\lambda}\left[\frac{\dot{\xi}}{\lambda}\cdot\nabla Z_0+\frac{\dot{\lambda}}{\lambda}\left(y\cdot\nabla Z_0+\frac{n}{2}Z_0\right)\right]\eta_K.
\end{eqnarray*}

The following estimate holds for $E_K$.
\begin{lem}\label{lem estimates on E}
For any $\tau$, we have
\begin{eqnarray*}
 \|E_K(\tau)\|_{2+\alpha} & \lesssim &  \mathcal{O}(\tau) +K^{2+\alpha-\bar{p}\alpha}\|\varphi_K(\tau)\|_{\alpha}^{\bar{p}}+\Big|\frac{a(\tau)}{\lambda(\tau)}\Big|^{\bar{p}} \\
   &+& K^{3+\alpha-n}\left|\frac{\dot{\xi}(\tau)}{\lambda(\tau)}\right|
   +K^{4+\alpha-n}\left|\frac{\dot{\lambda}(\tau)}{\lambda(\tau)}\right|+e^{-cK}\Big|\frac{\dot{a}(\tau)-\mu_0a(\tau)}{\lambda(\tau)}\Big|.
\end{eqnarray*}
\end{lem}
\begin{proof}
  \begin{enumerate}
    \item  Because $\alpha=(n-2)/2$, by the definition of $\mathcal{O}$, we have
    \[K^{2+\alpha} \sup_y\left|2\nabla\varphi(y,\tau)\cdot\nabla\eta_K(y)+\varphi(y,\tau)\Delta\eta_K(y) \right|\lesssim \mathcal{O}(\tau).\]

    \item By Lemma \ref{lem estimate on nonlinearity}, we obtain
    \begin{eqnarray*}
       |\mathcal{N}\eta_K|& \lesssim& |\varphi|^{\bar{p}}\eta_K+\left|\frac{a}{\lambda}\right|^{\bar{p}}Z_0^{\bar{p}} \\
        &=& |\varphi|^{\bar{p}}\left(\eta_K-\eta_K^{\bar{p}}\right)+ |\varphi_K|^{\bar{p}}+\left|\frac{a}{\lambda}\right|^{\bar{p}}Z_0^{\bar{p}}.
    \end{eqnarray*}
    First, because $\mathcal{O}(\tau)\ll 1$,
    \[ \sup_{y\in\R^n}K^{2+\alpha}|\varphi(y,\tau)|^{\bar{p}}\left[\eta_K(y)-\eta_K(y)^{\bar{p}}\right]\lesssim K^{2+\alpha}\mathcal{O}(\tau)^{\bar{p}}\ll \mathcal{O}(\tau).\]
    Secondly, because the support of $\varphi_K$ is contained in $B_K$,
    \[\sup_{y\in\R^n}\left(1+|y|\right)^{2+\alpha}|\varphi_K(y,\tau)|^{\bar{p}}\lesssim K^{2+\alpha-\bar{p}\alpha}\|\varphi_K(\tau)\|_{\alpha}^{\bar{p}}.\]
    Finally, by the exponential decay of $Z_0$ at infinity,
    \[\sup_{y\in\R^n}\left(1+|y|\right)^{2+\alpha}\left|\frac{a(\tau)}{\lambda(\tau)}\right|^{\bar{p}}|Z_0(y)|^{\bar{p}}\lesssim \left|\frac{a(\tau)}{\lambda(\tau)}\right|^{\bar{p}}.\]

    \item For $i=1,\cdots,n$, $Z_i(y)=O\left(|y|^{1-n}\right)$. Hence
    \[\sup_{y\in\R^n}\left(1+|y|\right)^{2+\alpha}\left|\frac{\dot{\xi}(\tau)}{\lambda(\tau)} \cdot Z(y) \left[\eta_K(y)-1\right]\right|\lesssim K^{3+\alpha-n}\left|\frac{\dot{\xi}(\tau)}{\lambda(\tau)}\right|.\]

    \item As in the previous case, because $Z_{n+1}(y)=O\left(|y|^{2-n}\right)$, we get
    \[\sup_{y\in\R^n}\left(1+|y|\right)^{2+\alpha}\left|\frac{\dot{\lambda}(\tau)}{\lambda(\tau)} Z_{n+1}(y) \left[\eta_K(y)-1\right]\right|\lesssim K^{4+\alpha-n}\left|\frac{\dot{\lambda}(\tau)}{\lambda(\tau)}\right|.\]

\item By the exponential decay of $Z_0$, we get
 \[\sup_{y\in\R^n}\left(1+|y|\right)^{2+\alpha}\left|\frac{\dot{a}(\tau)-\mu_0a(\tau)}{\lambda(\tau)} Z_0(y) \left[\eta_K(y)-1\right]\right|\lesssim  e^{-cK}\frac{\left|\dot{a}(\tau)-\mu_0a(\tau)\right|}{\lambda(\tau)}.\]

 \item By Proposition \ref{prop initial smallness on error}, we have
 \[ \sup_{y\in\R^n}\left(1+|y|\right)^{2+\alpha}\left|\frac{\dot{\lambda}(\tau)}{\lambda(\tau)} \left[\frac{n-2}{2}\varphi(y)+  y\cdot\nabla_y\varphi(y)\right]\eta_K(y)\right|\ll K^{4+\alpha-n} \frac{|\dot{\lambda}(\tau)|}{\lambda(\tau)}.\]

 \item As in the previous case,
 \begin{eqnarray*}
 && \sup_{y\in\R^n}\left(1+|y|\right)^{2+\alpha}\left|\frac{a(\tau)}{\lambda(\tau)}\left[\frac{\dot{\xi}(\tau)}{\lambda(\tau)}\cdot\nabla Z_0(y)+\frac{\dot{\lambda}(\tau)}{\lambda(\tau)}\left(y\cdot\nabla Z_0(y)+\frac{n}{2}Z_0(y)\right)\right]\eta_K(y)\right|\\
 &\ll & K^{4+\alpha-n}\left(\frac{|\dot{\xi}(\tau)|}{\lambda(\tau)}+ \frac{|\dot{\lambda}(\tau)|}{\lambda(\tau)}\right).
 \end{eqnarray*}
   \end{enumerate}
Putting these estimates together we finish the proof.
\end{proof}
Because $n\geq 7$ and $\alpha=(n-2)/2$,
\[ 4+\alpha-n<0.\]
Furthermore, by noting that those terms like $\|\varphi_K(\tau)\|_{\alpha}$  are small, this lemma can be restated as
\begin{eqnarray}\label{main order term estimate in E}
 \|E_K(\tau)\|_{2+\alpha}   &\lesssim&   o\left(\|\varphi_K(\tau)\|_{\alpha}+\Big|\frac{\dot{\xi}(\tau)}{\lambda(\tau)}\Big|+\Big|\frac{\dot{\lambda}(\tau)}{\lambda(\tau)}\Big|
 +\Big|\frac{a(\tau)}{\lambda(\tau)}\Big|+\Big|\frac{\dot{a}(\tau)-\mu_0a(\tau)}{\lambda(\tau)}\Big|\right) \nonumber\\
 &+&\mathcal{O}(\tau).
\end{eqnarray}
In particular, the main order term in $E_K$ is $\mathcal{O}(\tau)$, which comes from the outer component.

Now we prove our main estimate in this section.
\begin{prop}[$C^{1,\theta}$ estimates]\label{prop decay estimate}
\begin{itemize}
  \item
For any $\tau$,
   \begin{equation}\label{bound on parameters}
  \left|\frac{\dot{\xi}(\tau)}{\lambda(\tau)}\right|+\left|\frac{\dot{\lambda}(\tau)}{\lambda(\tau)}\right|+\left|\frac{\dot{a}(\tau)-\mu_0a(\tau)}{\lambda(\tau)}
       \right| \lesssim    \mathcal{O}(\tau) +K^{2+\alpha-\bar{p}\alpha}\|\varphi_K(\tau)\|_{\alpha}^{\bar{p}}+\Big|\frac{a(\tau)}{\lambda(\tau)}\Big|^{\bar{p}}.
  \end{equation}

  \item  For any $\sigma>0$, there exist two constants $C$ (universal) and $T(\sigma,K)$ such that for any $\tau_1<\tau_2$,
     \begin{eqnarray}\label{decay of inner component first order}
      \|\varphi_K\|_{C^{(1+\theta)/2}(\tau_1,\tau_2;\mathcal{X}_\alpha^{1+\theta})} & \leq &  \sigma  \|\varphi_K(\tau_1-T(\sigma,K))\|_{\mathcal{X}_\alpha } \\
     &&     +C\left\|\left(\mathcal{O},\left|\frac{a}{\lambda}\right|^{\bar{p}}\right)\right\|_{L^\infty(\tau_1-T(\sigma,K), \tau_2)}. \nonumber
  \end{eqnarray}
 \end{itemize}
 \end{prop}
 \begin{proof}
 Take a decomposition $\varphi_K=\varphi_{K,1}+\varphi_{K,2}$, where $\varphi_{K,1}$ is the solution of \eqref{linear eqn 1} with initial value (at $\tau_1-T(\sigma)$, with $T(\sigma,K)$ to be determined below) $\varphi_K$, and $\varphi_{K,2}$ is the solution of \eqref{linear eqn 2} with non-homogeneous term $E_K$.

 To apply Lemma \ref{lem linear decay estimate 1}, take an $\alpha^\prime>n/2$.
 Because $\varphi_K$ is supported in $B_K$, we have
 \begin{equation}\label{transfer between different weights}
   \left\|\varphi_K(\tau_1-T(\sigma,K))\right\|_{\alpha^\prime}\leq \left( 2K\right)^{\alpha^\prime-\alpha}\left\|\varphi_K(\tau_1-T(\sigma,K))\right\|_{\alpha}.
 \end{equation}
By Lemma \ref{lem linear decay estimate 1}, there exists a $T(\sigma,K)>0$ such that
\begin{equation}\label{decay estimate in time 1}
  \|\varphi_{K,1}\|_{C^{(1+\theta)/2}(\tau_1,\tau_2;C^{1+\theta}(B_{2K}))} \leq \sigma \left(4K\right)^{-1-\alpha}\left(2K\right)^{\alpha-\alpha^\prime}.
\end{equation}

For $\varphi_{K,2}$,
in view of the estimate in Lemma \ref{lem estimates on E},  a direct application of Lemma \ref{lem linear decay estimate 2} gives
\begin{equation}\label{decay estimate in time 2}
\|\varphi_{K,2}\|_{C^{(1+\theta)/2}(\tau_1,\tau_2;\mathcal{X}_\alpha^{1+\theta})}  \leq  C(\alpha)\left\|\left(\mathcal{O},\left|\frac{a}{\lambda}\right|^{\bar{p}}\right)\right\|_{L^\infty(\tau_1-T(\sigma,K), \tau_2)}.
\end{equation}
 Adding \eqref{decay estimate in time 1} and \eqref{decay estimate in time 2}, by using again the fact that $\varphi_K(\tau)$ is supported in $B_K$, we get \eqref{decay of inner component first order}.
 \end{proof}

For applications in Section \ref{sec Pohozaev}, we give another estimate on the parameters $\dot{\lambda}$ etc. in terms of the $L^\infty$ norm of $\varphi$.
\begin{lem}\label{lem estiamte of parameters}
  For each $\tau$,
\begin{equation}\label{estimate of parameters}
\left|\frac{\dot{a}(\tau)-\mu_0a(\tau)}{\lambda(\tau)}
       \right|+\left|\frac{\dot{\xi}(\tau)}{\lambda(\tau)}\right|+\left|\frac{\dot{\lambda}(\tau)}{\lambda(\tau)}\right| \lesssim \sup_{B_{2K}\setminus B_K}|\varphi(\tau)|+\sup_{B_{2K}}|\varphi(\tau)|^{\bar{p}}+\left|\frac{a(\tau)}{\lambda(\tau)}\right|^{\bar{p}}.
\end{equation}
\end{lem}
\begin{proof}
 For each $j=0,\cdots, n+1$, multiplying \eqref{inner eqn with cut-off} by $Z_j$,  utilizing  the orthogonal condition \eqref{orthogonal in inner}, we get
   \[
  \left\{\begin{aligned}
&\frac{\dot{a}(\tau)-\mu_0a(\tau)}{\lambda(\tau)}=\int_{\R^n}Z_0(y)E_{K}(y,\tau)dy,\\
&\frac{\dot{\xi}_j(\tau)}{\lambda(\tau)}=-\int_{\R^n}Z_j(y)E_{K}(y,\tau)dy, \quad j=1,\cdots, n,\\
& \frac{\dot{\lambda}(\tau)}{\lambda(\tau)}=-\int_{\R^n}Z_{n+1}(y)E_{K}(y,\tau)dy.
\end{aligned}\right.
\]
An integration by parts gives
\begin{eqnarray}\label{main order term in parameters}
 \left|\int_{\R^n}\left(2\nabla\varphi\cdot\nabla\eta_K+\varphi\Delta\eta_K\right)Z_j \right|
  &=& \left| \int_{\R^n}\varphi\left(2\nabla\eta_K\cdot\nabla Z_j+\Delta\eta_K Z_j\right) \right| \nonumber\\
  &\lesssim& \sup_{B_{2K}\setminus B_K}|\varphi|.
\end{eqnarray}
By the at most $|y|^{2-n}$ decay of $Z_j$, Lemma \ref{lem estimate on nonlinearity} and Proposition \ref{prop initial smallness on error}, we find the contribution from other terms in $E_K$ is of the order
\[O\left(\sup_{y\in B_{2K}}|\varphi(y,\tau)|^{\bar{p}}+\left|\frac{a(\tau)}{\lambda(\tau)}\right|^{\bar{p}}\right)+ o\left(\left|\frac{\dot{\xi}(\tau)}{\lambda(\tau)}\right|+\left|\frac{\dot{\lambda}(\tau)}{\lambda(\tau)}\right|+\left|\frac{\dot{a}(\tau)-\mu_0a(\tau)}{\lambda(\tau)}
       \right|\right).\]
       Adding these estimates for $\frac{\dot{a}(\tau)-\mu_0a(\tau)}{\lambda(\tau)}$, $\frac{\dot{\xi}_j(\tau)}{\lambda(\tau)}$ and $\frac{\dot{\lambda}(\tau)}{\lambda(\tau)}$ together, we get \eqref{estimate of parameters}.
 \end{proof}

Finally, we will need the following backward estimate on $a/\lambda$.
\begin{lem}\label{lem representation of a}
For any $\tau_1<\tau_2$, we have
\begin{eqnarray}\label{representation of a}
  \frac{ |a(\tau_1)|}{\lambda(\tau_1)}  \lesssim   e^{-\frac{\mu_0}{2}(\tau_2-\tau_1)} \frac{ |a(\tau_2)|}{\lambda(\tau_2)} + \int_{\tau_1}^{\tau_2}e^{-\frac{\mu_0}{2}(\tau_2-s)}  \left[ \mathcal{O}(s)  + \|\varphi_K(s)\|_{\alpha}^{\bar{p}}\right]ds.
\end{eqnarray}
\end{lem}
\begin{proof}
By a direct differentiation and then applying \eqref{bound on parameters}, we obtain
  \begin{eqnarray*}
    \frac{d}{d\tau}\frac{a(\tau)}{\lambda(\tau)}  &=&  \frac{\dot{a}(\tau)}{\lambda(\tau)}-\frac{\dot{\lambda}(\tau)}{\lambda(\tau)}\frac{a(\tau)}{\lambda(\tau)}
  \\
     &=& \mu_0 \frac{a(\tau)}{\lambda(\tau)}+ O\left(\mathcal{O}(\tau) +\|\varphi_K(\tau)\|_{\alpha}^{\bar{p}}+\left|\frac{a(\tau)}{\lambda(\tau)}\right|^{\bar{p}}\right)\\
      &=& \left[\mu_0+o(1)\right] \frac{a(\tau)}{\lambda(\tau)}+ O\left(\mathcal{O}(\tau) +\|\varphi_K(\tau)\|_{\alpha}^{\bar{p}}\right).
  \end{eqnarray*}
Integrating this differential inequality gives \eqref{representation of a}.
\end{proof}

\section{Outer problem}\label{sec outer}
\setcounter{equation}{0}

In this section we establish an estimate on the outer component $\phi_{out}$. Our main goal is to obtain an estimate of $\mathcal{O}$ in terms of $\mathcal{I}$ and the parameters $\lambda^\prime$ etc..

\subsection{Decomposition of the outer equation} First we need to take a further decomposition of $\phi_{out}$.
Recall that $\phi_{out}$ satisfies
\begin{eqnarray}\label{outer eqn}
  \partial_t\phi_{out}-\Delta\phi_{out} &=&\left(u^p- W_\ast^p\right)\left(1-\eta_{out}\right) \nonumber\\
  &-&\phi\left(\partial_t\eta_{out}-\Delta\eta_{out}\right)+2\nabla\phi\cdot\nabla\eta_{out}\nonumber\\
  &+&\left(-a^\prime+\mu_0\frac{a}{\lambda^2},\xi^\prime,\lambda^\prime\right)Z_\ast\left(1-\eta_{out}\right)\\
   &-&a\partial_tZ_0^\ast \left(1-\eta_{out}\right).  \nonumber
\end{eqnarray}
Introducing
   \[
    \left\{\begin{aligned}
& V_\ast:=
\begin{cases}
  \frac{u^p- W_\ast^p}{u-W_\ast}\chi_{B_{L\lambda(t)}(\xi(t))^c}, & \mbox{if } u\neq W_\ast,\\
  pu^{p-1}\chi_{B_{L\lambda(t)}(\xi(t))^c}, & \mbox{otherwise},
\end{cases}
\\
& F_2:=-\phi\left(\partial_t\eta_{out}-\Delta\eta_{out}\right)+2\nabla\phi\cdot\nabla\eta_{out},\\
& F_3:=\lambda^\prime Z_{n+1,\ast}\left(1-\eta_{out}\right),\\
& F_4:=\sum_{j=1}^{n}\xi_j^\prime Z_{j,\ast}\left(1-\eta_{out}\right),\\
& F_5:=\left[aV_\ast Z_0^\ast- \left(a^\prime-\mu_0\frac{a}{\lambda^2}\right)Z_0^\ast -a\partial_tZ_0^\ast\right] \left(1-\eta_{out}\right),
\end{aligned}\right.
 \]
 then \eqref{outer eqn} is written as
 \begin{equation}\label{outer eqn simplified}
   \partial_t\phi_{out}-\Delta\phi_{out} = V_\ast\phi_{out}+F_2+F_3+F_4+F_5.
 \end{equation}

Corresponding to this decomposition of the right hand side, we have $\phi_{out}=\phi_1+\phi_2+\phi_3+\phi_4+\phi_5$, with these five functions solving the following five equations respectively: first,
\[
\left\{
\begin{aligned}
& \partial_t\phi_1-\Delta\phi_1=V_\ast \phi_1  \quad & \mbox{in } Q_{9/10},\\
&\phi_1=\phi_{out}, \quad & \mbox{on } \partial^pQ_{9/10};
\end{aligned}\right.
\]
next,
\[
\left\{
\begin{aligned}
& \partial_t\phi_2-\Delta\phi_2=V_\ast \phi_2+F_2  \quad & \mbox{in } Q_{9/10},\\
&\phi_2=0, \quad & \mbox{on } \partial^pQ_{9/10};
\end{aligned}\right.
\]
and $\phi_3,\phi_4,  \phi_5$ solve the same equation with $\phi_2$, but with the right hand side term replaced by
$F_3$, $F_4$ and $F_5$ respectively.

The following estimates hold for $V_\ast$ and $F_2$-$F_5$.
\begin{enumerate}

\item[(i)] Because $u>0$ and $W_\ast>0$, by convexity of the function $x\mapsto x^p$ and Lemma \ref{lem scaling invariant estimate I}, for any $\delta>0$ and $C>0$, there exists a constant $R(\delta)>0$ such that, if $L\geq R(\delta)$, then
\begin{equation}\label{esimate on V}
|V_\ast| \leq    p\left(u^{p-1}+W_\ast^{p-1}\right)
 \leq \frac{\delta}{|x-\xi(t)|^2}+C.
\end{equation}

\item[(ii)]
 By the definition of $\eta_{out}$ in \eqref{definition of cut-off}, we get
\begin{eqnarray}\label{estimate on F2, 1}
 && \big|\phi\left(\partial_t\eta_{out}-\Delta\eta_{out}\right)\big| \nonumber\\
 &\lesssim&  |\phi|\left(\frac{|\lambda^\prime(t)|}{\lambda(t)}+\frac{|\xi^\prime(t)|}{L\lambda(t)}+\frac{1}{L^2\lambda(t)^2}\right)\chi_{B_{2L\lambda(t)}(\xi(t))\setminus B_{L\lambda(t)}(\xi(t))}\\
  &\lesssim& \frac{\mathcal{I}(t)}{ L^\alpha\lambda(t)^\alpha} \left(\frac{|\lambda^\prime(t)|}{\lambda(t)}+\frac{|\xi^\prime(t)|}{L\lambda(t)}+\frac{1}{L^2\lambda(t)^2}\right)\chi_{B_{2L\lambda(t)}(\xi(t))\setminus B_{L\lambda(t)}(\xi(t))}.\nonumber
\end{eqnarray}

In the same way, we get
\begin{eqnarray}\label{estimate on F2, 2}
  2|\nabla\phi||\nabla\eta_{out}| &\lesssim& \frac{ |\nabla\phi|}{L \lambda(t) }\chi_{B_{2L\lambda(t)}(\xi(t))\setminus B_{L\lambda(t)}(\xi(t))} \\
  &\lesssim& \frac{\mathcal{I}(t)}{ L^{2+\alpha}\lambda(t)^{2+\alpha}} \chi_{B_{2L\lambda(t)}(\xi(t))\setminus B_{L\lambda(t)}(\xi(t))}. \nonumber
\end{eqnarray}

Combining \eqref{estimate on F2, 1} and \eqref{estimate on F2, 2}, we obtain
\begin{equation}\label{estimate on F2}
  |F_2|\lesssim \frac{\mathcal{I}(t)}{ L^\alpha\lambda(t)^\alpha} \left[\frac{|\lambda^\prime(t)|}{\lambda(t)}+\frac{|\xi^\prime(t)|}{L\lambda(t)}+\frac{1}{L^2\lambda(t)^2}\right]\chi_{B_{2L\lambda(t)}(\xi(t))\setminus B_{L\lambda(t)}(\xi(t))}.
\end{equation}

 \item[(iii)] By the decay of $Z_{n+1}$ at infinity, we get
  \begin{equation}\label{estimate on F4}
    | F_3|\lesssim  \lambda(t)^{\frac{n-4}{2}}|\lambda^\prime(t)||x-\xi(t)|^{2-n}\chi_{B_{L\lambda(t)}(\xi(t))^c}.
  \end{equation}

  \item[(iv)] By the decay of $Z_1,\cdots, Z_n$ at infinity, we get
  \begin{equation}\label{estimate on F3}
    |F_4|\lesssim   \lambda(t)^{\frac{n-2}{2}}|\xi^\prime(t)||x-\xi(t)|^{1-n}\chi_{B_{L\lambda(t)}(\xi(t))^c}.
  \end{equation}

\item[(v)] By the decay of $Z_0$ at infinity, we get
\begin{equation}\label{estimate on F5,1}
    \Big|\left(a^\prime-\mu_0\frac{a}{\lambda^2}\right)Z_0^\ast\left(1-\eta_{out}\right)\Big|\lesssim \Big|a^\prime(t)-\mu_0\frac{a(t)}{\lambda(t)^2}\Big|\lambda(t)^{-\frac{n}{2}}e^{-c\frac{|x-\xi(t)|}{\lambda(t)}}\chi_{B_{L\lambda(t)}(\xi(t))^c}.
  \end{equation}
Similarly, we have
\begin{equation}\label{estimate on F5,2}
  \Big|a\partial_tZ_0^\ast \left(1-\eta_{out}\right)\Big|\lesssim |a(t)|\lambda(t)^{-\frac{n+2}{2}}\left(|\lambda^\prime(t)|+|\xi^\prime(t)|\right)e^{-c\frac{|x-\xi(t)|}{\lambda(t)}}\chi_{B_{L\lambda(t)}(\xi(t))^c}.
\end{equation}
Finally, by \eqref{esimate on V}, we obtain
\begin{eqnarray}\label{estimate on F5,3}
  \left|aV_\ast Z_0^\ast \left(1-\eta_{out}\right)\right| &\lesssim &\frac{|a(t)|}{|x-\xi(t)|^2}
  \lambda(t)^{-\frac{n}{2}}e^{-c\frac{|x-\xi(t)|}{\lambda(t)}}\chi_{B_{L\lambda(t)}(\xi(t))^c}\\
  &\lesssim & \frac{|a(t)|}{\lambda(t)^2}
  \lambda(t)^{-\frac{n}{2}}e^{-c\frac{|x-\xi(t)|}{\lambda(t)}}\chi_{B_{L\lambda(t)}(\xi(t))^c}. \nonumber
\end{eqnarray}
Combining \eqref{estimate on F5,1}-\eqref{estimate on F5,3}, we get
\begin{eqnarray}\label{estimate on F5}
  |F_5|&\lesssim& \left[\frac{|a(t)|}{\lambda(t)^2}+ \Big|a^\prime(t)-\mu_0\frac{a(t)}{\lambda(t)^2}\Big|  + \frac{|a(t)|}{\lambda(t)}\left(|\lambda^\prime(t)|+|\xi^\prime(t)|\right)\right] \\
  && \quad \quad \times \lambda(t)^{-\frac{n}{2}}e^{-c\frac{|x-\xi(t)|}{\lambda(t)}}\chi_{B_{L\lambda(t)}(\xi(t))^c}. \nonumber
\end{eqnarray}
\end{enumerate}

\medskip

For each $i=1,\cdots,5$, let
\begin{equation}\label{transformation of outer components}
  \overline{\phi}_i(x,t):=\left|\phi_i\left(x-\xi(t),t\right)\right|.
\end{equation}
By the Kato inequality, we obtain
\begin{equation}\label{outer eqn 1}
\partial_t\overline{\phi}_1-\Delta\overline{\phi}_1+\xi^\prime(t)\cdot\nabla\overline{\phi}_1\leq |V_\ast|\overline{\phi}_1
\end{equation}
and for $i=2,3,4,5$,
\begin{equation}\label{outer eqn 3}
\partial_t\overline{\phi}_i-\Delta\overline{\phi}_i+\xi^\prime(t)\cdot\nabla\overline{\phi}_i\leq |V_\ast|\overline{\phi}_i+|F_i|.
\end{equation}

\subsection{Estimates for the outer equation} In this subsection, we establish some pointwise estimates on $\overline{\phi}_1, \cdots, \overline{\phi}_5$.

Denote the Dirichlet heat kernel for the operator $\mathcal{H}:=\partial_t-\Delta-\left(\delta|x|^{-2}+C\right)+\xi^\prime(t)\cdot\nabla$ in $Q_1$  by $ G(x,t;y,s)$ ($t>s$).

Let
\[\gamma:=\frac{n-2}{2}-\sqrt{\frac{(n-2)^2}{4}-4\delta}.\]
Then $w(x):=|x|^{-\gamma}$ is a weak solution of the elliptic equation
\[-\Delta w(x)=\frac{\delta}{|x|^2}w(x).\]

\begin{lem}[Estimate  on $\overline{\phi}_1$]\label{lem estimate outer 1}
$\overline{\phi}_1(x,t)\lesssim |x|^{-\gamma}$ in $Q_{8/9}$.
\end{lem}
\begin{proof}
Set $\varphi_1:=\overline{\phi}_1/w$. It is a sub-solution to the parabolic equation
\begin{equation}\label{degenerate eqn}
 \partial_t \varphi=w^{-2}\mbox{div}\left(w^2\nabla \varphi\right)-\xi^\prime(t)\cdot\nabla \varphi+\left[C+\gamma \xi^\prime(t)\cdot\frac{ x}{|x|^2}\right]\varphi.
\end{equation}
We note the following three facts about  the coefficients of this equation.
\begin{enumerate}
  \item Because $\gamma\ll 1$, by \cite[Theorem 3.1]{Moschini}, the volume doubling property and the scaling invariant Poincare inequality holds on $\R^n$ with the weighted measure $w^2dx$. By \cite[Theorem 5.2.3]{Saloff-Coste}, for some $q>2$ (independent of $\gamma$), there is a local Sobolev embedding from $W^{1,2}(B_1,w^2dx)$ into $L^q(B_1,w^2dx)$.

  \item By the Lipschitz hypothesis \eqref{Lip assumption}, $\xi^\prime(t)\in L^\infty(-1,1)$.
  \item Similarly, the zeroth order term $C+\gamma \xi^\prime(t)\cdot\frac{ x}{|x|^2}\in L^\infty(-1,1;L^{n-3\gamma}(B_1, w^2dx))$.
\end{enumerate}
Then by standard  De Giorgi-Nash-Moser estimate (cf. \cite[Section 3]{Moschini} and \cite[Chapter 5]{Saloff-Coste}), we  deduce that $\varphi_1$ is bounded  in $Q_{8/9}$. (The lower order terms in this parabolic operator do not affect this argument by noting their higher integrability.)
\end{proof}

In fact, Moser's Harnack inequality holds for positive solutions of \eqref{degenerate eqn}. This is similar to \cite[Theorem 3.5]{Moschini}. Then by   \cite[Section 5.4.7]{Saloff-Coste} (see also \cite[Theorem 4.3]{Moschini}),  the heat kernel  $G$ satisfies a Gaussian bound
\begin{equation}\label{heat kernel bound}
 G(x,t;y,s)\leq C\left(t-s\right)^{-\frac{n}{2}}e^{-c\frac{|x-y|^2}{t-s}}\left(1+\frac{\sqrt{t-s}}{|x|}\right)^\gamma \left(1+\frac{\sqrt{t-s}}{|y|}\right)^\gamma.
\end{equation}

Hereafter we fix two constants $\beta>1$ (but sufficiently close to $1$) and $\mu\in(0,1)$ (to be determined below). We  choose $\mu$ so   that
\[ \left(\frac{n-2}{2}-2\gamma\right)\left(1-\mu\right)>2.\]
which is denoted by $2+2\kappa$. This inequality is guaranteed  by the assumption that $n\geq 7$.

\begin{lem}[Estimate  on $\overline{\phi}_2$]\label{lem estimate outer 2}
If $K\lambda(t)/2\leq |x|\leq 4K\lambda(t)$, then
\begin{eqnarray*}
  \overline{\phi}_2(x,t) &\lesssim &  \left[K^{-\left(\frac{n-2}{2}-\gamma\right)(\beta-1)}+ \frac{L^{ \frac{n-2}{2}-\gamma}}{K^{\frac{n-2}{2}-(2\beta-1)\gamma}}\right]\left[K\lambda(t)\right]^{-\frac{n-2}{2}}
\|\mathcal{I}\|_{L^\infty(t-\lambda(t)^{2\mu},t)}\\
   &+& \lambda(t)^{-\gamma(1-\mu)-\frac{n-2}{2}\mu}.
\end{eqnarray*}
\end{lem}
\begin{proof}
By the heat kernel representation formula, for any $(x,t)$ we have
\[\overline{\phi}_2(x,t)\leq\int_{-81/100}^{t}\int_{B_{2L\lambda(s)}\setminus B_{L\lambda(s)}}G(x,t;y,s)\frac{\mathcal{I}(s)}{ L^\alpha\lambda(s)^\alpha} \left(\frac{|\lambda^\prime(s)|}{\lambda(s)}+\frac{|\xi^\prime(s)|}{L\lambda(s)}+\frac{1}{L^2\lambda(s)^2}\right).\]
Divide this integral into three parts, the first part $\mathrm{I}$ being on $(-81/100, t-\lambda(t)^{2\mu})$, the second part $\mathrm{II}$ involving the integral on $(t-\lambda(t)^{2\mu}, t-K^{2\beta}\lambda(t)^2)$, and the third part $\mathrm{III}$ involving the integral on $(t-K^{2\beta}\lambda(t)^2,t)$.

{\bf Estimate of $\mathrm{I}$.} By \eqref{bound on parameters},
\begin{equation}\label{differential Harnack at lambda scale}
 |\lambda^\prime(s)|+|\xi^\prime(s)|\ll \lambda(s)^{-1}.
\end{equation}
Hence
\[\left\|\frac{\mathcal{I}(s)}{ L^\alpha\lambda(s)^\alpha} \left(\frac{|\lambda^\prime(s)|}{\lambda(s)}+\frac{|\xi^\prime(s)|}{L\lambda(s)}+\frac{1}{L^2\lambda(s)^2}\right)\chi_{B_{2L\lambda(s)}\setminus B_{L\lambda(s)}}\right\|_{L^{\frac{2n}{n+2}}(B_1)}
\lesssim \mathcal{I}(s).\]
By the Guassian bound on $G(x,y;t,s)$ in \eqref{heat kernel bound},
\begin{equation}\label{Gaussian bound}
   \left\|G(x,t;\cdot,s)\right\|_{L^{\frac{2n}{n-2}}(B_1)}\lesssim (t-s)^{-\frac{n+2}{4}}\left(1+\frac{\sqrt{t-s}}{|x|}\right)^\gamma.
\end{equation}
Then by  H\"{o}lder inequality we get
\begin{eqnarray}\label{2I}
\mathrm{I} &\lesssim& |x|^{-\gamma}\int_{-81/100}^{t-\lambda(t)^{2\mu}}(t-s)^{-\frac{n+2}{4}+\frac{\gamma}{2}}\mathcal{I}(s)ds\\
&\lesssim& K^{-\gamma}\lambda(t)^{-(1-\mu)\gamma-\frac{n-2}{2}\mu}, \nonumber
\end{eqnarray}
where in the last step we used only the estimate $|\mathcal{I}(s)|\leq C$.

{\bf Estimate of $\mathrm{II}$.} This case differs from the previous one only in the last step. Now we have
\begin{eqnarray*}
\mathrm{II} &\lesssim& |x|^{-\gamma}\int_{t-\lambda(t)^{2\mu}}^{t-K^{2\beta}\lambda(t)^2}(t-s)^{-\frac{n+2}{4}+\frac{\gamma}{2}}\mathcal{I}(s)\\
&\lesssim& K^{-\frac{n-2}{2}\beta+(\beta-1)\gamma}\lambda(t)^{-\frac{n-2}{2}} \|\mathcal{I}\|_{L^\infty(t-\lambda(t)^{2\mu},t-K^{2\beta}\lambda(t)^2)}.
\end{eqnarray*}

{\bf Estimate of $\mathrm{III}$.} Still by the heat kernel representation formula, $\mathrm{III}$ is bounded by
 \begin{eqnarray*}
  & & \int_{t-K^{2\beta}\lambda(t)^2}^{t}\int_{B_{2L\lambda(s)}\setminus B_{L\lambda(s)}} \left(t-s\right)^{-\frac{n}{2}}e^{-c\frac{|x-y|^2}{t-s}}\left(1+\frac{\sqrt{t-s}}{|x|}\right)^\gamma \left(1+\frac{\sqrt{t-s}}{|y|}\right)^\gamma\\
    &&\quad \quad \quad \quad\quad \quad \quad \quad\quad\quad\times \frac{\mathcal{I}(s)}{ L^\alpha\lambda(s)^\alpha} \left(\frac{|\lambda^\prime(s)|}{\lambda(s)}+\frac{|\xi^\prime(s)|}{L\lambda(s)}+\frac{1}{L^2\lambda(s)^2}\right)dyds\\
  &=& \int_{t-K^{2\beta}\lambda(t)^2}^{t}\frac{\mathcal{I}(s)}{L^\alpha\lambda(s)^\alpha} \left(\frac{|\lambda^\prime(s)|}{\lambda(s)}+\frac{|\xi^\prime(s)|}{L\lambda(s)}+\frac{1}{L^2\lambda(s)^2}\right) \left(1+\frac{\sqrt{t-s}}{|x|}\right)^\gamma \\
  &&\times\left[\int_{B_{2L\lambda(s)}\setminus B_{L\lambda(s)}} \left(t-s\right)^{-\frac{n}{2}}e^{-c\frac{|x-y|^2}{t-s}} \left(1+\frac{\sqrt{t-s}}{|y|}\right)^\gamma dy \right]ds.
\end{eqnarray*}
For $s\in[t-K^{2\beta}\lambda(t)^2,t]$, by Proposition \ref{prop blow up profile I} we have
\begin{equation}\label{differential Harnack at lambda scale 2}
 \lambda(s)\sim\lambda(t).
\end{equation}
Thus by noting that   $K\lambda(t)/2\leq |x| \leq 4K\lambda(t)$ and $K\gg L$, we have
\[|x|\gg |y|  \quad \mbox{for any} ~~  y\in B_{2L\lambda(s)}.\]
Therefore in the integral above we can replace $e^{-c\frac{|x-y|^2}{t-s}} $ by $e^{-c\frac{|x|^2}{t-s}} $.

Then in view of \eqref{differential Harnack at lambda scale}, $\mathrm{III}$ is controlled by
\begin{eqnarray*}
  & & K^{(\beta-1)\gamma} \|\mathcal{I}\|_{L^\infty(t-K^{2\beta}\lambda(t)^2,t)}L^{\frac{n-2}{2}}\lambda(t)^{\frac{n-2}{2}}\\
  && \quad \quad\quad \quad \times \int_{t-K^{2\beta}\lambda(t)^2}^{t}\left(t-s\right)^{-\frac{n}{2}}e^{-c\frac{|x|^2}{t-s}} \left(1+\frac{\sqrt{t-s}}{L\lambda(t)}\right)^\gamma ds\\
  &\lesssim&  K^{(\beta-1)\gamma} \|\mathcal{I}\|_{L^\infty(t-K^{2\beta}\lambda(t)^2,t)}L^{\frac{n-2}{2}}\lambda(t)^{\frac{n-2}{2}} \\ && \quad \quad\quad \quad \left(1+\frac{K^\beta}{L}\right)^\gamma\int_{t-K^{2\beta}\lambda(t)^2}^{t}\left(t-s\right)^{-\frac{n}{2}}e^{-c\frac{|x|^2}{t-s}}  ds\\
 &\lesssim& K^{(2\beta-1)\gamma}L^{\frac{n-2}{2}-\gamma}\|\mathcal{I}\|_{L^\infty(t-K^{2\beta}\lambda(t)^2,t)}\lambda(t)^{\frac{n-2}{2}}
  |x|^{2-n} \\
 &\lesssim & \frac{L^{\frac{n-2}{2}-\gamma}}{K^{\frac{n-2}{2}-(2\beta-1)\gamma}} \left[K\lambda(t)\right]^{-\frac{n-2}{2}}\|\mathcal{I}\|_{L^\infty(t-K^2\lambda(t)^2,t)}.
\end{eqnarray*}

Adding up the estimates for $\mathrm{I}$, $\mathrm{II}$ and $\mathrm{II}$, we finish the proof.
\end{proof}

\begin{lem}[Estimate  on $\overline{\phi}_3$]\label{lem estimate outer 3}
For $K\lambda(t)/2\leq |x|\leq 4K\lambda(t)$,
\[\overline{\phi}_3(x,t)\lesssim \lambda(t)^{-\frac{n-2}{2}\mu+2(\mu-1)\gamma}+ K^{-\frac{n-2}{2}(\beta-1)+2(\beta-1)\gamma }\left[K\lambda(t)\right]^{-\frac{n-2}{2}} \|\lambda\lambda^\prime\|_{L^\infty(t-\lambda(t)^{2\mu},t)}.\]
\end{lem}
\begin{proof}
By the heat kernel representation,
 \begin{eqnarray*}
\overline{\phi}_3(x,t)&\lesssim & \int_{-81/100}^{t}\int_{B_{L\lambda(s)}^c} \left(t-s\right)^{-\frac{n}{2}}e^{-c\frac{|x-y|^2}{t-s}}\left(1+\frac{\sqrt{t-s}}{|x|}\right)^\gamma \left(1+\frac{\sqrt{t-s}}{|y|}\right)^\gamma \\
&&\quad \quad \quad \quad  \times  \lambda(s)^{\frac{n-4}{2}}|\lambda^\prime(s)||y|^{2-n}dyds\\
    &=&  \int_{-81/100}^{t}\left(1+\frac{\sqrt{t-s}}{|x|}\right)^\gamma \lambda(s)^{\frac{n-4}{2}}|\lambda^\prime(s)| \\
     && \quad \times \left[\int_{B_{L\lambda(s)}^c} \left(t-s\right)^{-\frac{n}{2}}e^{-c\frac{|x-y|^2}{t-s}} \left(1+\frac{\sqrt{t-s}}{|y|}\right)^\gamma|y|^{2-n}dy\right]ds.
\end{eqnarray*}
We still divide this integral into three parts, $\mathrm{I}$ being on the interval $(-81/100,t- \lambda(t)^{2\mu})$, $\mathrm{II}$ being on the interval $(t-\lambda(t)^{2\mu},t- K^{2\beta}\lambda(t)^2)$, and $\mathrm{III}$ on $(t- K^{2\beta}\lambda(t)^2,t)$.

{\bf Estimate of $\mathrm{I}$.} Direct calculation gives
\begin{eqnarray*}
   \mathcal{P}&:=& \int_{B_{L\lambda(s)}^c} \left(t-s\right)^{-\frac{n}{2}}e^{-c\frac{|x-y|^2}{t-s}} \left(1+\frac{\sqrt{t-s}}{|y|}\right)^\gamma|y|^{2-n}dy \\
 &\lesssim& \left(t-s\right)^{-\frac{n-2}{2}}\left[1+\frac{\sqrt{t-s}}{L\lambda(s)}\right]^\gamma
 \int_{\frac{L\lambda(s)}{\sqrt{t-s}}}^{+\infty}e^{-c\left(\frac{|x|}{\sqrt{t-s}}-r\right)^2}rdr.
\end{eqnarray*}
Recall that $K\lambda(t)/2\leq |x|\leq 4K\lambda(t)$, $t-s\geq K^{2\beta}\lambda(t)^2$. There are two cases.
\begin{itemize}
  \item If $L\lambda(s)\geq 6K\lambda(t)$,
  \[\mathcal{P}\lesssim \left(t-s\right)^{-\frac{n-2}{2}}\left[1+\frac{\sqrt{t-s}}{L\lambda(s)}\right]^\gamma
 e^{-c\frac{L^2\lambda(s)^2}{t-s}}.\]
  \item  If $L\lambda(s)\leq 6K\lambda(t)$,
  \begin{eqnarray*}
  \mathcal{P}&\lesssim & \left(t-s\right)^{-\frac{n-2}{2}}\left[1+\frac{\sqrt{t-s}}{L\lambda(s)}\right]^\gamma\\
  &\lesssim&\left(t-s\right)^{-\frac{n-2-\gamma}{2}}L^{-\gamma}\lambda(s)^{-\gamma}.
  \end{eqnarray*}
\end{itemize}
Hence
 \begin{eqnarray}\label{3I}
 \mathrm{I} &\lesssim &  |x|^{-\gamma}\int_{(-81/100,t-\lambda(t)^{2\mu})\cap\{L\lambda(s)\geq 6K\lambda(t)\}} (t-s)^{-\frac{n-2-\gamma}{2}} \lambda(s)^{\frac{n-4}{2}}|\lambda^\prime(s)| \nonumber\\
 &&\quad \quad \quad \quad \quad \quad \quad \times \left[1+\frac{\sqrt{t-s}}{L\lambda(s)}\right]^\gamma
 e^{-c\frac{L^2\lambda(s)^2}{t-s}}  \\
    &+&L^{-\gamma}|x|^{-\gamma} \int_{(-81/100,t-\lambda(t)^{2\mu})\cap\{L\lambda(s)\leq 6K\lambda(t)\}} (t-s)^{-\frac{n-2}{2}+\gamma} \lambda(s)^{\frac{n-4}{2}-\gamma} |\lambda^\prime(s)| . \nonumber
\end{eqnarray}

Because
\begin{eqnarray*}
   &&   (t-s)^{-\frac{n-2-\gamma}{2}} \lambda(s)^{\frac{n-6}{2}}  \left[1+\frac{\sqrt{t-s}}{L\lambda(s)}\right]^\gamma
 e^{-c\frac{L^2\lambda(s)^2}{t-s}}  \\
  &=& (t-s)^{-\frac{n+2-2\gamma}{4}} \left[\frac{\lambda(s)}{\sqrt{t-s}}\right]^{\frac{n-6}{2}}\left[1+\frac{\sqrt{t-s}}{L\lambda(s)}\right]^\gamma e^{-c\frac{L^2\lambda(s)^2}{t-s}}\\
  &\lesssim&L^{-\frac{n-6}{2}-\gamma}\lambda(s)^{-\gamma}(t-s)^{-\frac{n+2}{4}+\gamma},
\end{eqnarray*}
by noting that we have assumed $L\lambda(s)\geq 6K\lambda(t)$, the first integral is controlled by
\begin{eqnarray*}
    & & L^{-\frac{n-6}{2}} \left[K\lambda(t)\right]^{-2\gamma}\int_{-81/100}^{t-\lambda(t)^{2\mu}} (t-s)^{-\frac{n+2}{4}+\gamma}|\lambda(s)\lambda^\prime(s)| \\
   & \lesssim & L^{-\frac{n-6}{2}} K^{-2\gamma }\lambda(t)^{-\frac{n-2}{2}\mu+2(\mu-1)\gamma}.
\end{eqnarray*}
In the last step  we   used only the estimate $|\lambda\lambda^\prime |\leq C$.

For the second integral, by the estimate $|\lambda\lambda^\prime|\leq C$ again, we obtain
\begin{eqnarray*}
   && L^{-\gamma}|x|^{-\gamma} \int_{(-81/100,t-\lambda(t)^{2\mu})\cap\{L\lambda(s)\leq 6K\lambda(t)\}} (t-s)^{-\frac{n-2}{2}+\gamma} \lambda(s)^{\frac{n-4}{2}-\gamma} |\lambda^\prime(s)|  \\
 &\lesssim & L^{-\frac{n-6}{2}}K^{\frac{n-6}{2}-2\gamma}\lambda(t)^{\frac{n-6}{2}-2\gamma}\int_{-81/100}^{t-\lambda(t)^{2\mu}} (t-s)^{-\frac{n-2}{2}+\gamma}\\
 &\lesssim & L^{-\frac{n-6}{2}}K^{\frac{n-6}{2}-2\gamma}\lambda(t)^{-(n-4)\mu+\frac{n}{2}-3+2(\mu-1)\gamma}.
\end{eqnarray*}

Because
\[-(n-4)\mu+\frac{n}{2}-3+2(\mu-1)\gamma>-\frac{n-2}{2}\mu+2(\mu-1)\gamma,\]
combining these two estimates, we get
\[\mathrm{I}\lesssim   L^{-\frac{n-6}{2}} K^{-2\gamma }\lambda(t)^{-\frac{n-2}{2}\mu+2(\mu-1)\gamma}.\]

{\bf Estimate of $\mathrm{II}$.}
As in the previous case, we have
 \begin{eqnarray*}
 \mathrm{II} &\lesssim &  |x|^{-\gamma}\int_{(t-\lambda(t)^{2\mu},t-K^{2\beta}\lambda(t)^2)\cap\{L\lambda(s)\geq 6K\lambda(t)\}} (t-s)^{-\frac{n-2-\gamma}{2}} \lambda(s)^{\frac{n-4}{2}}|\lambda^\prime(s)|\\
 &&\quad \quad \quad \quad \quad \quad \quad \times \left[1+\frac{\sqrt{t-s}}{L\lambda(s)}\right]^\gamma
 e^{-c\frac{L^2\lambda(s)^2}{t-s}}  \\
    &+&L^{-\gamma}|x|^{-\gamma} \int_{(t-\lambda(t)^{2\mu},t-K^{2\beta}\lambda(t)^2)\cap\{L\lambda(s)\leq 6K\lambda(t)\}} (t-s)^{-\frac{n-2}{2}+\gamma} \lambda(s)^{\frac{n-4}{2}-\gamma} |\lambda^\prime(s)|\\
 &\lesssim&
 L^{-\frac{n-6}{2}} \left[K\lambda(t)\right]^{-2\gamma}\int_{t-\lambda(t)^{2\mu}}^{t-K^{2\beta}\lambda(t)^2} (t-s)^{-\frac{n+2}{4}+\gamma}|\lambda(s)\lambda^\prime(s)| \\
 &+& L^{-\frac{n-6}{2}}K^{\frac{n-6}{2}-2\gamma}\lambda(t)^{\frac{n-6}{2}-2\gamma}\int_{t-\lambda(t)^{2\mu}}^{t-K^{2\beta}\lambda(t)^2} (t-s)^{-\frac{n-2}{2}+\gamma}|\lambda(s)\lambda^\prime(s)|\\
   & \lesssim & L^{-\frac{n-6}{2}}\left[ K^{-\frac{n-2}{2}\beta+2(\beta-1)\gamma }+K^{-\frac{n-2}{2}\beta-\left(\frac{n-6}{2}-2\gamma\right)\left(\beta-1\right)}\right]\\
   && \quad \times\lambda(t)^{-\frac{n-2}{2}} \|\lambda\lambda^\prime\|_{L^\infty(t-\lambda(t)^{2\mu},t-K^{2\beta}\lambda(t)^2)}\\
 &  \lesssim &  K^{-\frac{n-2}{2}\beta+2(\beta-1)\gamma }\lambda(t)^{-\frac{n-2}{2}}  \|\lambda\lambda^\prime\|_{L^\infty(t-\lambda(t)^{2\mu},t-K^{2\beta}\lambda(t)^2)}.
\end{eqnarray*}

{\bf Estimate of $\mathrm{III}$.} By \eqref{differential Harnack at lambda scale 2},
 \begin{eqnarray*}
 \mathrm{III} &\lesssim & \int_{t-K^{2\beta}\lambda(t)^2}^{t}\int_{B_1} \left(t-s\right)^{-\frac{n}{2}}e^{-c\frac{|x-y|^2}{t-s}}\left(1+\frac{\sqrt{t-s}}{|x|}\right)^\gamma \left(1+\frac{\sqrt{t-s}}{|y|}\right)^\gamma\\
 && \quad \quad \quad \quad \quad \times \lambda(s)^{\frac{n-4}{2}}|\lambda^\prime(s)||y|^{2-n}  \\
       &=& \int_{t-K^{2\beta}\lambda(t)^2}^{t} \left(1+\frac{\sqrt{t-s}}{|x|}\right)^\gamma \lambda(s)^{\frac{n-4}{2}}|\lambda^\prime(s)|\\
    &&   \quad \quad \quad \quad \times \left[\underbrace{\int_{B_1}\left(t-s\right)^{-\frac{n}{2}}e^{-c\frac{|x-y|^2}{t-s}} \left(1+\frac{\sqrt{t-s}}{|y|}\right)^\gamma|y|^{2-n}dy}_{\mbox{Lemma \ref{lem heat kernel integral 1}}}\right]ds\\
      &\lesssim&\int_{t-K^{2\beta}\lambda(t)^2}^{t} \left(|x|+\sqrt{t-s}\right)^{2-n}\left(1+\frac{\sqrt{t-s}}{|x|}\right)^{2\gamma} \lambda(s)^{\frac{n-4}{2}}|\lambda^\prime(s)| ds \\
    &\lesssim &\left[K\lambda(t)\right]^{2-n} K^{2(\beta-1)\gamma}\int_{t-K^{2\beta}\lambda(t)^2}^{t} \lambda(s)^{\frac{n-4}{2}}|\lambda^\prime(s)|ds\\
    &\lesssim&  K^{2-n+2\beta+2(\beta-1)\gamma}\lambda(t)^{-\frac{n-2}{2}} \|\lambda\lambda^\prime\|_{L^\infty(t-K^{2\beta}\lambda(t)^2,t)} .
\end{eqnarray*}

Because $n>6$ and $\beta$ is sufficiently close to $1$, we have
\[2-n+2\beta+2(\beta-1)\gamma<-\frac{n-2}{2}\beta+2(\beta-1)\gamma.\]
Adding up estimates for $\mathrm{I}$, $\mathrm{II}$ and $\mathrm{III}$ we finish the proof.
\end{proof}

\begin{lem}[Estimates on $\overline{\phi}_4$]\label{lem estimate outer 4}
For $K\lambda(t)/2\leq |x|\leq 4K\lambda(t)$,
\[\overline{\phi}_4(x,t)\lesssim  \lambda(t)^{-\frac{n-2}{2}\mu+2(\mu-1)\gamma}+K^{-\frac{n-2}{2}(\beta-1)+2(\beta-1)\gamma}\left[K\lambda(t)\right]^{-\frac{n-2}{2}} \|\lambda\xi^\prime\|_{L^\infty(t-\lambda(t)^{2\mu},t)}.\]
\end{lem}
\begin{proof}
We still have the heat kernel representation
\begin{eqnarray*}
\overline{\phi}_4(x,t) &\lesssim & \int_{-81/100}^{t}\int_{B_{L\lambda(s)}^c} \left(t-s\right)^{-\frac{n}{2}}e^{-c\frac{|x-y|^2}{t-s}}\left(1+\frac{\sqrt{t-s}}{|x|}\right)^\gamma \left(1+\frac{\sqrt{t-s}}{|y|}\right)^\gamma \\
&& \quad \quad\quad\quad\times  \lambda(s)^{\frac{n-2}{2}}|\xi^\prime(s)||y|^{1-n}dyds\\
    &=&  \int_{-81/100}^{t}\left(1+\frac{\sqrt{t-s}}{|x|}\right)^\gamma \lambda(s)^{\frac{n-2}{2}}|\xi^\prime(s)| \\
    && \quad \times \left[\int_{B_{L\lambda(s)}^c} \left(t-s\right)^{-\frac{n}{2}}e^{-c\frac{|x-y|^2}{t-s}} \left(1+\frac{\sqrt{t-s}}{|y|}\right)^\gamma|y|^{1-n}dy\right]ds.
\end{eqnarray*}
We still divide this integral into three parts, $\mathrm{I}$ being on the interval $(-81/100,t-  \lambda(t)^{2\mu})$, $\mathrm{II}$ being on the interval $(t-\lambda(t)^{2\mu},t- K^{2\beta}\lambda(t)^2)$, and $\mathrm{III}$ on the interval $(t-K^{2\beta}\lambda(t)^2,t)$.

{\bf Estimate of $\mathrm{I}$.} Direct calculation gives
\begin{eqnarray*}
   \mathcal{P}&:=& \int_{B_{L\lambda(s)}^c} \left(t-s\right)^{-\frac{n}{2}}e^{-c\frac{|x-y|^2}{t-s}} \left(1+\frac{\sqrt{t-s}}{|y|}\right)^\gamma|y|^{1-n}dy \\
 &\lesssim& \left(t-s\right)^{-\frac{n-1}{2}}\left[1+\frac{\sqrt{t-s}}{L\lambda(s)}\right]^\gamma
 \int_{\frac{L\lambda(s)}{\sqrt{t-s}}}^{+\infty}e^{-c\left(\frac{|x|}{\sqrt{t-s}}-r\right)^2}dr.
\end{eqnarray*}
Recall that $K\lambda(t)/2\leq |x|\leq 4K\lambda(t)$ and $t-s\geq K^{2\beta}\lambda(t)^2$.
There are two cases.
\begin{itemize}
  \item If $L\lambda(s)\geq 6K\lambda(t)$,
  \[\mathcal{P}\lesssim \left(t-s\right)^{-\frac{n-1}{2}}\left[1+\frac{\sqrt{t-s}}{L\lambda(s)}\right]^\gamma
 e^{-c\frac{L^2\lambda(s)^2}{t-s}}.\]
  \item  If $L\lambda(s)\leq 6K\lambda(t)$,
  \[
  \mathcal{P} \lesssim  \left(t-s\right)^{-\frac{n-1-\gamma}{2}}L^{-\gamma}\lambda(s)^{-\gamma}.
  \]
\end{itemize}

Plugging these estimates into the formula for $\mathrm{I}$, we obtain
 \begin{eqnarray}\label{4I}
  \mathrm{I}
    &\lesssim&  |x|^{-\gamma}\int_{(-81/100,t-\lambda(t)^{2\mu})\cap\{L\lambda(s)\geq 6K\lambda(t)\}} (t-s)^{-\frac{n-1-\gamma}{2}} \lambda(s)^{\frac{n-2}{2}}|\xi^\prime(s)| \nonumber\\
   && \quad \quad \quad \quad \quad \quad  \times   \left[1+\frac{\sqrt{t-s}}{L\lambda(s)}\right]^\gamma
 e^{-c\frac{L^2\lambda(s)^2}{t-s}}  \\
    &+& L^{-\gamma}|x|^{-\gamma}\int_{(-81/100,t- \lambda(t)^{2\mu})\cap\{L\lambda(s)\leq 6K\lambda(t)\}} (t-s)^{-\frac{n-1}{2}+\gamma}  \lambda(s)^{\frac{n-2}{2}-\gamma} |\xi^\prime(s)| . \nonumber
\end{eqnarray}

Because
\begin{eqnarray*}
   &&  (t-s)^{-\frac{n-1-\gamma}{2}} \lambda(s)^{\frac{n-4}{2}} \left[1+\frac{\sqrt{t-s}}{L\lambda(s)}\right]^\gamma  e^{-c\frac{L^2\lambda(s)^2}{t-s}} \\
  &=& (t-s)^{-\frac{n+2}{4}+\frac{\gamma}{2}} \left[\frac{\lambda(s)}{\sqrt{t-s}}\right]^{\frac{n-4}{2}}\left[1+\frac{\sqrt{t-s}}{L\lambda(s)}\right]^\gamma e^{-c\frac{L^2\lambda(s)^2}{t-s}}\\
  &\lesssim& L^{-\gamma}\lambda(s)^{-\gamma}(t-s)^{-\frac{n+2}{4}+\gamma},
\end{eqnarray*}
the first integral is controlled by
\begin{eqnarray*}
&& \left[K\lambda(t)\right]^{-2\gamma}\int_{-81/100}^{t- \lambda(t)^{2\mu}}(t-s)^{-\frac{n+2}{4}+\gamma}|
\lambda(s)\xi^\prime(s)| \\
& \lesssim & K^{-2\gamma}\lambda(t)^{-\frac{n-2}{2}\mu+2(\mu-1)\gamma}.
\end{eqnarray*}
In the last step we used only the estimate $| \lambda\xi^\prime|\leq C$.

Similarly, for the second integral,   we have
\begin{eqnarray*}
&& L^{-\gamma}|x|^{-\gamma}\int_{(-81/100,t- \lambda(t)^{2\mu})\cap\{L\lambda(s)\leq 4K\lambda(t)\}} (t-s)^{-\frac{n-1}{2}+\gamma}  \lambda(s)^{\frac{n-2}{2}-\gamma} |\xi^\prime(s)|\\
&\lesssim &L^{-\frac{n-4}{2}}K^{ \frac{n-4}{2}-2\gamma}\lambda(t)^{\frac{n-4}{2}-2\gamma}\int_{-81/100}^{t-\lambda(t)^{2\mu}} (t-s)^{-\frac{n-1}{2}+\gamma}\\
&\lesssim& L^{-\frac{n-4}{2}}K^{ \frac{n-4}{2}-2\gamma}\lambda(t)^{\frac{n-4}{2}-(n-3)\mu+2(\mu-1)\gamma}.
\end{eqnarray*}
Because
\[-\frac{n-2}{2}\mu+2(\mu-1)\gamma<\frac{n-4}{2}-(n-3)\mu+2(\mu-1)\gamma,\]
we get
\[\mathrm{I}\lesssim \lambda(t)^{-\frac{n-2}{2}\mu+2(\mu-1)\gamma}.\]

{\bf Estimate of $\mathrm{II}$.} We have
 \begin{eqnarray*}
  \mathrm{II}
    &\lesssim&  |x|^{-\gamma}\int_{(t-\lambda(t)^{2\mu},t-K^{2\beta}\lambda(t)^2)\cap\{L\lambda(s)\geq 6K\lambda(t)\}} (t-s)^{-\frac{n-1-\gamma}{2}} \lambda(s)^{\frac{n-2}{2}}|\xi^\prime(s)| \\
   && \quad \quad \quad \quad \quad \quad  \times   \left[1+\frac{\sqrt{t-s}}{L\lambda(s)}\right]^\gamma
 e^{-c\frac{L^2\lambda(s)^2}{t-s}}  \\
    &+& L^{-\gamma}|x|^{-\gamma}\int_{(t-\lambda(t)^{2\mu},t-K^{2\beta}\lambda(t)^2)\cap\{L\lambda(s)\leq 6K\lambda(t)\}} (t-s)^{-\frac{n-1}{2}+\gamma}  \lambda(s)^{\frac{n-2}{2}-\gamma} |\xi^\prime(s)| \\
&\lesssim & \left[K\lambda(t)\right]^{-2\gamma}\int_{t-\lambda(t)^{2\mu}}^{t-K^{2\beta}\lambda(t)^2}(t-s)^{-\frac{n+2}{4}+\gamma}|
\lambda(s)\xi^\prime(s)| \\
&+& L^{-\frac{n-4}{2}}K^{\frac{n-4}{2}-2\gamma}\lambda(t)^{\frac{n-4}{2}-2\gamma}\int_{t-\lambda(t)^{2\mu}}^{t-K^{2\beta}\lambda(t)^2}  (t-s)^{-\frac{n-1}{2}+\gamma}  |\lambda(s)\xi^\prime(s)|\\
& \lesssim & \left[K^{-\frac{n-2}{2}\beta+2(\beta-1)\gamma}+L^{-\frac{n-4}{2}}K^{-\beta(n-3-2\gamma)+\frac{n-4}{2}-2\gamma}\right]\\
&& \quad \quad \times \lambda(t)^{-\frac{n-2}{2}} \|\lambda\xi^\prime\|_{L^\infty(t-\lambda(t)^{2\mu},t-K^{2\beta}\lambda(t)^2)}\\
&\lesssim & K^{-\frac{n-2}{2}\beta+2(\beta-1)\gamma}\lambda(t)^{-\frac{n-2}{2}}\|\lambda\xi^\prime\|_{L^\infty(t-\lambda(t)^{2\mu},t-K^{2\beta}\lambda(t)^2)}.
\end{eqnarray*}

{\bf Estimate of $\mathrm{III}$.} As in the $\overline{\phi}_3$ case, we have
 \begin{eqnarray*}
\mathrm{ II} &=& \int_{t-K^{2\beta}\lambda(t)^2}^{t} \left(1+\frac{\sqrt{t-s}}{|x|}\right)^\gamma \lambda(s)^{\frac{n-2}{2}}|\xi^\prime(s)|\\
&&\quad \quad \quad \quad \times \left[\int_{B_1}\left(t-s\right)^{-\frac{n}{2}}e^{-c\frac{|x-y|^2}{t-s}} \left(1+\frac{\sqrt{t-s}}{|y|}\right)^\gamma|y|^{1-n}dy\right]ds\\
      &\lesssim&  |x|^{1-n}\int_{t-K^{2\beta}\lambda(t)^2}^{t}  \left(1+\frac{\sqrt{t-s}}{|x|}\right)^{2\gamma} \lambda(s)^{\frac{n-2}{2}}|\xi^\prime(s)|ds\\
    &\lesssim&   K^{1-n+2\beta+2(\beta-1)\gamma}\lambda(t)^{-\frac{n-2}{2}}\|\lambda\xi^\prime\|_{L^\infty(t-K^{2\beta}\lambda(t)^2,t)}.
\end{eqnarray*}
Because
\[ 1-n+2\beta+2(\beta-1)\gamma<-\frac{n-2}{2}\beta+2(\beta-1)\gamma,\]
adding up estimates for $\mathrm{I}$, $\mathrm{II}$ and $\mathrm{III}$ we finish the proof.
\end{proof}

\begin{lem}[Estimates on $\overline{\phi}_5$]\label{lem estimate outer 5}
For $K\lambda(t)/2\leq |x|\leq 4K\lambda(t)$,
\begin{eqnarray*}
\overline{\phi}_5(x,t) &\lesssim &  \left[K^{-\left(\frac{n-2}{2}-\gamma\right)(\beta-1)}+K^{\frac{n-2}{2}+2\beta+2(\beta-1)\gamma}e^{-cK}+\frac{K^{\frac{n-2}{2}+2\beta+(\beta-1)\gamma}}
{e^{cL}}\right] \\
&& \quad\quad \quad\quad \times \left[K\lambda(t)\right]^{-\frac{n-2}{2}}\left\|\left(\frac{a}{\lambda},\lambda a^\prime-\mu_0\frac{a}{\lambda}, \lambda \lambda^\prime,\lambda\xi^\prime \right)\right\|_{L^\infty(t-\lambda(t)^{2\mu},t)}\\
&+& \lambda(t)^{-\frac{n-2}{2}\mu+(\mu-1)\gamma}.
\end{eqnarray*}
\end{lem}
\begin{proof}
By the heat kernel representation, we have
\begin{eqnarray*}
  \overline{\phi}_5(x,t) &\leq & \int_{-81/100}^{t}\int_{B_{L\lambda(s)}^c}G(x,y;t,s)\lambda(s)^{-\frac{n+2}{2} }
 e^{-c\frac{|y|}{\lambda(s)}} \\
 &&  \times\left[\frac{|a(s)|}{\lambda(s)}+\left|\lambda(s)a^\prime(s)-\mu_0\frac{a(s)}{\lambda(s)}\right|+ \frac{|a(s)|}{\lambda(s)}\left(|\lambda(s)\lambda^\prime(s)|+|\lambda(s)\xi^\prime(s)|\right)\right].
\end{eqnarray*}
We still divide this integral into three parts, $\mathrm{I}$ being on the interval $(-81/100,t-\lambda(t)^{2\mu})$, $\mathrm{II}$ being on the interval $(t-\lambda(t)^{2\mu},t- K^{2\beta}\lambda(t)^2)$, and $\mathrm{III}$ on $(t-K^{2\beta}\lambda(t)^2,t)$.

{\bf Estimate of $\mathrm{I}$.}
 For $\mathrm{I}$, just using the estimate
 \[ \left\|\left(\frac{a}{\lambda},\lambda a^\prime-\mu_0\frac{a}{\lambda}, \frac{a}{\lambda}\left(|\lambda \lambda^\prime |+|\lambda \xi^\prime |\right)\right)\right\|_{L^\infty(-1,t)}\leq C,\]
by the operator bound on $G$ (see \eqref{Gaussian bound}) and H\"{o}lder inequality, we get
 \begin{eqnarray*}
 \mathrm{I} &\lesssim& \int_{-81/100}^{t-\lambda(t)^{2\mu}}\int_{B_1}G(x,t;y,s)\lambda(s)^{-\frac{n+2}{2} }
 e^{-c\frac{|y|}{\lambda(s)}} dyds\\
 &\lesssim&   \int_{-81/100}^{t-\lambda(t)^{2\mu}}(t-s)^{-\frac{n+2}{4}}\left(1+\frac{\sqrt{t-s}}{|x|}\right)^\gamma
ds  \\
  &\lesssim& \lambda(t)^{-\frac{n-2}{2}\mu+(\mu-1)\gamma}.
\end{eqnarray*}

{\bf Estimate of $\mathrm{II}$.}
 For $\mathrm{II}$, in the same way we obtain
 \begin{eqnarray*}
 \mathrm{II} &\lesssim& \left\|\left(\frac{a}{\lambda},\lambda a^\prime-\mu_0\frac{a}{\lambda}, \frac{a}{\lambda}\left(|\lambda \lambda^\prime |+|\lambda \xi^\prime |\right)\right)\right\|_{L^\infty(t-\lambda(t)^{2\mu},t-K^{2\beta}\lambda(t)^2)}\\
  &&\quad \times\int_{t-\lambda(t)^{2\mu}}^{t-K^{2\beta}\lambda(t)^2}\int_{B_1}G(x,y;t,s)\lambda(s)^{-\frac{n+2}{2} }
 e^{-c\frac{|y|}{\lambda(s)}} dyds\\
 &\lesssim& \left\|\left(\frac{a}{\lambda},\lambda a^\prime-\mu_0\frac{a}{\lambda}, \lambda \lambda^\prime,\lambda\xi^\prime \right)\right\|_{L^\infty(t-\lambda(t)^{2\mu},t-K^{2\beta}\lambda(t)^2)}\\
  &&\quad \times\int_{-1}^{t-K^{2\beta}\lambda(t)^2}(t-s)^{-\frac{n+2}{4}}\left(1+\frac{\sqrt{t-s}}{|x|}\right)^\gamma ds\\
  &\lesssim& K^{-\frac{n-2}{2}\beta+(\beta-1)\gamma}\lambda(t)^{-\frac{n-2}{2}}  \left\|\left(\frac{a}{\lambda},\lambda a^\prime-\mu_0\frac{a}{\lambda}, \lambda \lambda^\prime,\lambda\xi^\prime \right)\right\|_{L^\infty(t-\lambda(t)^{2\mu},t-K^{2\beta}\lambda(t)^2)}.
\end{eqnarray*}

{\bf Estimate of $\mathrm{III}$.} In this case,  first note that
\begin{eqnarray*}
   && \int_{t-K^{2\beta}\lambda(t)^2}^{t}\int_{B_{L\lambda(s)}^c}G(x,t;y,s)\lambda(s)^{-\frac{n+2}{2} }
 e^{-c\frac{|y|}{\lambda(s)}} dy ds \\
  &\lesssim&\int_{t-K^{2\beta}\lambda(t)^2}^{t}\int_{B_{L\lambda(s)}^c} \left(t-s\right)^{-\frac{n}{2}}e^{-c\frac{|x-y|^2}{t-s}}\left(1+\frac{\sqrt{t-s}}{|x|}\right)^\gamma \left(1+\frac{\sqrt{t-s}}{|y|}\right)^\gamma\\
  && \quad\quad \quad \quad\quad \quad \quad \quad \times \lambda(s)^{-\frac{n+2}{2} }
 e^{-c\frac{|y|}{\lambda(s)}} dy ds \\
 &\lesssim &\int_{t-K^{2\beta}\lambda(t)^2}^{t} \left(1+\frac{\sqrt{t-s}}{|x|}\right)^\gamma \lambda(s)^{-\frac{n+2}{2} }
\\
 &&\quad \times \left[ \underbrace{\int_{B_{L\lambda(s)}^c}\left(t-s\right)^{-\frac{n}{2}}e^{-c\frac{|x-y|^2}{t-s}}\left(1+\frac{\sqrt{t-s}}{|y|}\right)^\gamma e^{-c\frac{|y|}{\lambda(s)}} dy}_{\mbox{ Lemma \ref{lem heat kernel integral 2}}} \right]ds\\
  &\lesssim &\int_{t-K^{2\beta}\lambda(t)^2}^{t} \left(1+\frac{\sqrt{t-s}}{|x|}\right)^\gamma \lambda(s)^{-\frac{n+2}{2} }
\\
 &&\quad \times \left[e^{-c\frac{|x|}{\lambda(s)}}\left(1+\frac{\sqrt{t-s}}{|x|}\right)^\gamma+
 \left(1+\frac{\lambda(s)}{\sqrt{t-s}}\right)^{n+\gamma} e^{-cL-c\frac{|x|^2}{t-s}}   \right] ds\\
 &\lesssim & \left[K^{2\beta+2(\beta-1)\gamma}e^{-cK}+K^{2\beta+(\beta-1)\gamma}e^{-cL}\right]\lambda(t)^{-\frac{n-2}{2}},
\end{eqnarray*}
where we have used   \eqref{differential Harnack at lambda scale 2} to deduce that last inequality.

Then similar to the estimate of $\mathrm{II}$, we get
\begin{eqnarray*}
\mathrm{III}&\lesssim &\left[K^{2\beta+2(\beta-1)\gamma}e^{-cK}+K^{2\beta+(\beta-1)\gamma}e^{-cL}\right]\lambda(t)^{-\frac{n-2}{2}} \\
&& \quad\quad \quad\quad \times \left\|\left(\frac{a}{\lambda},\lambda a^\prime-\mu_0\frac{a}{\lambda}, \lambda \lambda^\prime,\lambda\xi^\prime \right)\right\|_{L^\infty(t-K^{2\beta}\lambda(t)^2,t)}.
\end{eqnarray*}

Adding up estimates for $\mathrm{I}$, $\mathrm{II}$ and $\mathrm{III}$ we finish the proof.
\end{proof}

\subsection{Estimate of $\mathcal{O}(t)$}
As an application of the pointwise estimates on $\overline{\phi}_1, \cdots, \overline{\phi}_5$ obtained in the previous subsection, we give an estimate on $\mathcal{O}(t)$.
\begin{prop}\label{prop inner to outer}
 There exist two constants $C(K,L)$ and $\sigma(K,L)\ll 1$ (depending on $K$ and $L$) such that for any $t\in(-81/100,81/100)$,
\begin{equation}\label{inner to outer 0}
  \mathcal{O}(t) \leq  C(K,L)\lambda(t)^{2+2\kappa}+ \sigma(K,L)  \left\|\left(\mathcal{I} ,\frac{a}{\lambda},\lambda a^\prime-\mu_0\frac{a}{\lambda}, \lambda \lambda^\prime,\lambda\xi^\prime \right)\right\|_{L^\infty(t-\lambda(t)^{2\mu},t)}.
\end{equation}
\end{prop}
\begin{proof}
  First by the five lemmas in the previous subsection, we obtain
  \begin{eqnarray}\label{estimate of outer 1}
   &&\left[K\lambda(t)\right]^{\frac{n-2}{2}}\sup_{B_{2K\lambda(t)}(\xi(t))\setminus B_{K\lambda(t)}(\xi(t))}|\phi(x,t)|\\
   &\leq&  \left[K\lambda(t)\right]^{\frac{n-2}{2}}\sup_{B_{2K\lambda(t)} \setminus B_{K\lambda(t)} } \left[\overline{\phi}_1(x,t)+\overline{\phi}_2(x,t)+\overline{\phi}_3(x,t)+\overline{\phi}_4(x,t)+\overline{\phi}_5(x,t)\right]\nonumber\\
    &\leq & C(K,L)\lambda(t)^{2+2\kappa}+ \sigma(K,L)  \left\|\left(\mathcal{I} ,\frac{a}{\lambda},\lambda a^\prime-\mu_0\frac{a}{\lambda}, \lambda \lambda^\prime,\lambda\xi^\prime \right)\right\|_{L^\infty(t-\lambda(t)^{2\mu},t)}. \nonumber
  \end{eqnarray}

Next, by a rescaling and standard interior gradient estimates for heat equation, we get
 \begin{eqnarray*}
   &&\left[K\lambda(t)\right]^{\frac{n}{2}}\sup_{B_{2K\lambda(t)}(\xi(t))\setminus B_{K\lambda(t)}(\xi(t))}|\nabla\phi_1(x,t)|\\
   &\lesssim&\left[K\lambda(t)\right]^{\frac{n-2}{2}}\sup_{\left(B_{4K\lambda(t)(\xi(t))} \setminus B_{K\lambda(t)/2(\xi(t))}\right)\times(t-K^2\lambda(t)^2,t) } |\phi_1(y,s)|\\
   &\lesssim&  K^{\frac{n-2}{2}-\gamma}\lambda(t)^{\frac{n-2}{2}-\gamma}.
  \end{eqnarray*}
 The same estimate holds for $\phi_2$, if we note that $F_2\equiv 0$ in $B_{4K\lambda(t)(\xi(t))} \setminus B_{K\lambda(t)/2(\xi(t))}$.

 In the same way, for $\phi_3$ we have
   \begin{eqnarray*}
   &&\left[K\lambda(t)\right]^{\frac{n}{2}}\sup_{B_{2K\lambda(t)}(\xi(t))\setminus B_{K\lambda(t)}(\xi(t))}|\nabla\phi_3(x,t)|\\
   &\lesssim&\left[K\lambda(t)\right]^{\frac{n-2}{2}}\sup_{\left(B_{4K\lambda(t)(\xi(t))} \setminus B_{K\lambda(t)/2(\xi(t))}\right)\times(t-K^2\lambda(t)^2,t) } \left[|\phi_3(y,s)|+\left(K\lambda(t)\right)^2|F_3(y,s)|\right]\\
   &\lesssim&    K^{-\frac{n-2}{2}(\beta-1)+2(\beta-1)\gamma } \|\lambda\lambda^\prime\|_{L^\infty(t-\lambda(t)^{2\mu},t)}+K^{\frac{n-2}{2}} \lambda(t)^{2+2\kappa}\\
    &+& K^{-\frac{n-6}{2}} \left\|\lambda\lambda^\prime\right\|_{L^\infty(t-K^2\lambda(t)^2,t)}.
  \end{eqnarray*}
 The same estimates hold for $\phi_4$ and $\phi_5$.

Putting these gradient estimates together and arguing as in \eqref{estimate of outer 1}, we get the estimate on the remaining terms in $\mathcal{O}(t)$.
\end{proof}

\section{A Harnack inequality for $\lambda$}\label{sec Harnack for lambda}
\setcounter{equation}{0}

In this section, we combine the estimates in the previous two sections to establish an Harnack inequality for $\lambda$.

For any $t\in(-81/100,81/100)$, denote
\[\mathcal{D}(t):= \|\phi_{in}(t)\|_{\alpha;1+\theta}+ \left|\lambda(t)\xi^\prime(t)\right|+\left|\lambda(t)\lambda^\prime(t)\right|+\left|\lambda(t)a^\prime(t)-\mu_0\frac{a(t)}{\lambda(t)}
\right|+\left|\frac{ a(t) }{\lambda(t)}\right|.
\]
For   $r\in(0,81/100)$,  let
\[g(r):=\sup_{-r<t<r}\mathcal{D}(t)\]
and
\[ M(r):=\sup_{-r<t<r}\lambda(t)^2, \quad  m(r):=\inf_{-r<t<r}\lambda(t)^2. \]

By Proposition \ref{prop decay estimate} and Lemma \ref{lem representation of a}, there exist three universal constants $C$, $T\gg 1$ and $\sigma\ll 1$ such that
   \begin{equation}\label{outer to inner}
  \mathcal{D}(t)\leq \sigma \sup_{[t-T\lambda(t)^2,t+T\lambda(t)^2]}\mathcal{D}(s)+C\sup_{[t-T\lambda(t)^2,t+T\lambda(t)^2]}\mathcal{O}(s).
  \end{equation}
By Proposition \ref{prop inner to outer},
\begin{equation}\label{inner to outer}
  \mathcal{O}(t) \leq  C(K,L)\lambda(t)^{2+2\kappa}+ \sigma(K,L)\sup_{[t-\lambda(t)^{2\mu},t]} \mathcal{D}(s).
\end{equation}
Combining these two estimates we get
\begin{equation}\label{decay relation 1}
  \mathcal{D}(t)\leq \frac{1}{2} \sup_{[t-2\lambda(t)^{2\mu},t+2K\lambda(t)^2]}\mathcal{D}(s)+C\sup_{[t-2\lambda(t)^{2\mu},t+2K\lambda(t)^2]} \lambda(s)^{2+2\kappa}.
\end{equation}

For any $r\in(0,1)$, taking supremum over $t\in\left(-r+2M(r)^\mu,r-2M(r)^\mu\right)$ in \eqref{decay relation 1}, we obtain
\begin{equation}\label{decay relation 4}
 g\left(r-2M(r)^{\mu}\right)\leq   \frac{1}{2}g(r)+CM(r)^{1+\kappa}.
\end{equation}
For any $r_1<r_2$, an iteration of this inequality from $r_2$ to $r_1$ in $\lfloor(r_2-r_1)/M(r_2)^\mu\rfloor$ steps leads to
\[g(r_1)\leq g(r_2) e^{-c(r_2-r_1)M(r_2)^{-\mu}}+CM(r_2)^{1+\kappa}.\]
For any $r\in(0,1)$, $M(r)\ll 1$, so
\[ e^{-cM(r)^{-\mu/2}}\lesssim M(r)^{1+\kappa}.\]
Hence by choosing $r_2=r$ and $r_1=r-M(r)^{\mu/2}$, we get
\[g\left(r-M(r)^{\mu/2}\right)\lesssim M(r)^{1+\kappa}.\]
By this estimate, integrating $\lambda\lambda^\prime$ on $[-r+M(r)^{\mu/2}, r-M(r)^{\mu/2}]$, we find a constant $C_H$ such that
\begin{equation}\label{Harnack 1}
  M\left(r-M(r)^{\mu/2}\right)\leq m\left(r-M(r)^{\mu/2}\right)+C_HM(r)^{1+\kappa}.
\end{equation}

\begin{lem}
  There exists an $r_H\in[1/2,3/4]$ such that
\[M\left(r_H-M(r_H)^{\mu/2}\right) \geq 2C_HM(r_H)^{1+\kappa}.\]
\end{lem}
\begin{proof}
  Assume  that  for any $r\in[1/2,3/4]$,
\begin{equation}\label{fast decay}
  M\left(r-M(r)^{\mu/2}\right) \leq 2C_HM(r)^{1+\kappa}.
\end{equation}
Set
\[r_0:=3/4, \quad a_0:=M(r_0)\]
and for $k\in \mathbb{N}$,
\[r_{k+1}:=r_k-a_k^{\mu/2}, \quad  a_{k+1}:=M(r_{k+1}).\]
Then \eqref{fast decay} says
\[ a_{k+1}\leq 2C_Ha_k^{1+\kappa}.\]
By our assumption, $a_0\ll 1$. An induction gives
\[a_k\leq a_0^{k+1}, \quad  r_k\geq r_0-\sum_{k=0}^{+\infty}a_0^{(k+1)\frac{\mu}{2}}\geq 1/2.\]
As a consequence, $M(1/2)=0$. This is impossible.
\end{proof}
For this $r_H$, \eqref{Harnack 1} can be written as a  Harnack inequality
\begin{equation}\label{Harnack}
  m(r_H)\geq \left[1-C_HM(r_H)^\kappa\right] M(r_H)\geq \frac{1}{2}M(r_H).
\end{equation}

After a scaling $u(x,t)\mapsto r_H^{(n-2)/2}u(r_Hx, r_H^2t)$, from here to Section \ref{sec Pohozaev} we work in the following setting:
\begin{itemize}
  \item we denote
  \begin{equation}\label{determination of the scaling parameter size}
  \varepsilon:=\lambda(0),
\end{equation}
and after a translation in the spatial direction, we   assume $\xi(0)=0$;
  \item
by the above Harnack inequality \eqref{Harnack}, for any $t\in[-1,1]$,
\begin{equation}\label{bound on parameters 1}
  \lambda(t)=\varepsilon+O\left(\varepsilon^{1+\kappa}\right)
  \end{equation}
and
\begin{equation}\label{bound on parameters 2}
   \|\phi_{in}(t)\|+\left|\lambda(t)\lambda^\prime(t)\right|+\left|\lambda(t)\xi^\prime(t)\right|+\left|\lambda(t)a^\prime(t)-\mu_0\frac{a(t)}{\lambda(t)}\right|
+\left|\frac{a(t)}{\lambda(t)}\right| \lesssim \varepsilon^{1+\kappa}.
\end{equation}

  \item
integrating $\xi^\prime$ and using the above  estimate, we obtain
\begin{equation}\label{bound on parameters 3}
  |\xi(t)|\lesssim \varepsilon^{\kappa}, \quad \mbox{for any} ~~ t\in[-1,1].
\end{equation}
\end{itemize}

\section{Inner problem and outer problem  again}\label{sec inner and outer 2}
\setcounter{equation}{0}

In this section, under the assumptions \eqref{determination of the scaling parameter size}-\eqref{bound on parameters 2}, we prove
\begin{prop}\label{prop inner estimate 2}
For any $\gamma>0$, there exists a constant $C(\gamma)$ such that
  \begin{equation}\label{inner estimate 2}
 g\left(\frac{8}{9}\right)\leq C(\gamma) \varepsilon^{\frac{n-2}{2}-\gamma}.
  \end{equation}
\end{prop}
This gives an improvement on the estimate of scaling parameters etc.. A more precise estimate on $\phi$   also follows:
\begin{prop}\label{prop outer estimate 2}
For any $\gamma>0$, there exists  a constant $C(\gamma)>0$ such that
  \begin{equation}\label{outer estimate 3}
   |\phi(x,t)|\leq C(\gamma)\left(\varepsilon+|x-\xi(t)|\right)^{-\gamma} \quad \mbox{in} \quad  Q_{7/8}.
  \end{equation}
\end{prop}

To prove these propositions, estimates on $\overline{\phi}_2$, $\overline{\phi}_3$, $\overline{\phi}_4$ and $\overline{\phi}_5$ in Section \ref{sec outer} are not sufficient. We need to upgrade them.
\begin{lem}[Estimate on $\overline{\phi}_2$]\label{lem estimate on phi 2,2}
If $|x|\geq  K\lambda(t)/2$, then
\begin{equation}\label{estimate on phi 2,1}
 \overline{\phi}_2(x,t)\lesssim   \left[K^{-\left(\frac{n-2}{2}-\gamma\right)\beta}\frac{\varepsilon^{\gamma}}{|x|^\gamma}+ L^{\frac{n-2}{2}-\gamma} K^{\beta\gamma}\frac{\varepsilon^{n-2}}{ |x|^{n-2}}\right] \varepsilon^{-\frac{n-2}{2}}\|\mathcal{I}\|_{L^\infty(-1,t)}.
\end{equation}
\end{lem}
\begin{proof}
As in the proof of Lemma \ref{lem estimate outer 2}, we have the heat kernel representation for $\overline{\phi}_2$. But now we just divide the integral into two parts, $\mathrm{I}$ being on $(-1,t-K^{2\beta}\lambda(t)^2)$ and $\mathrm{II}$  on $(t-K^{2\beta}\lambda(t)^2,t)$.

For $\mathrm{I}$, similar to \eqref{2I}, we have
\begin{eqnarray*}
  \mathrm{I} &\lesssim & |x|^{-\gamma}\int_{-1}^{t-K^{2\beta}\lambda(t)^2}(t-s)^{-\frac{n+2}{4}+\frac{\gamma}{2}}\mathcal{I}(s)ds \\
  &\lesssim & K^{-\left(\frac{n-2}{2}-\gamma\right)\beta}\lambda(t)^{-\frac{n-2}{2}+\gamma}|x|^{-\gamma}\|\mathcal{I}\|_{L^\infty(-1,t-K^{2\beta}\lambda(t)^2)}\\
  &\lesssim& K^{-\left(\frac{n-2}{2}-\gamma\right)\beta}\varepsilon^{-\frac{n-2}{2}+\gamma}|x|^{-\gamma}\|\mathcal{I}\|_{L^\infty(-1,t-K^{2\beta}\lambda(t)^2)}
\end{eqnarray*}

The estimate for $\mathrm{II}$ is almost the same as the one for $\mathrm{III}$ in the proof of Lemma \ref{lem estimate outer 2}, that is,
\begin{eqnarray*}
 \mathrm{II} &\lesssim &  \int_{t-K^{2\beta}\lambda(t)^2}^{t}\frac{\mathcal{I}(s)}{L^\alpha\lambda(s)^\alpha} \left(\frac{|\lambda^\prime(s)|}{\lambda(s)}+\frac{|\xi^\prime(s)|}{L\lambda(s)}+\frac{1}{L^2\lambda(s)^2}\right) \left(1+\frac{\sqrt{t-s}}{|x|}\right)^\gamma \\
  &&\times\left[\int_{B_{2L\lambda(s)}\setminus B_{L\lambda(s)}} \left(t-s\right)^{-\frac{n}{2}}e^{-c\frac{|x-y|^2}{t-s}} \left(1+\frac{\sqrt{t-s}}{|y|}\right)^\gamma dy \right]\\
   &\lesssim&  \|\mathcal{I}\|_{L^\infty(t-K^{2\beta}\lambda(t)^2,t)}L^{\frac{n-2}{2}}\lambda(t)^{\frac{n-2}{2}}\int_{t-K^{2\beta}\lambda(t)^2}^{t}\left(t-s\right)^{-\frac{n}{2}}e^{-c\frac{|x|^2}{t-s}} \left(1+\frac{\sqrt{t-s}}{L\lambda(t)}\right)^\gamma \\
  &\lesssim& \|\mathcal{I}\|_{L^\infty(t-K^{2\beta}\lambda(t)^2,t)}L^{\frac{n-2}{2}}\lambda(t)^{\frac{n-2}{2}} \left(1+\frac{K^\beta}{L}\right)^\gamma\int_{t-K^{2\beta}\lambda(t)^2}^{t}\left(t-s\right)^{-\frac{n}{2}}e^{-c\frac{|x|^2}{t-s}}\\
 &\lesssim& \|\mathcal{I}\|_{L^\infty(t-K^{2\beta}\lambda(t)^2,t)}L^{\frac{n-2}{2}}\lambda(t)^{\frac{n-2}{2}} \left(1+\frac{K^\beta}{L}\right)^\gamma |x|^{2-n} \\
 &\lesssim &  L^{\frac{n-2}{2}-\gamma} K^{\beta\gamma}\varepsilon^{\frac{n-2}{2}} |x|^{2-n}\|\mathcal{I}\|_{L^\infty(t-K^2\lambda(t)^2,t)}.
\end{eqnarray*}

Putting these two estimates together, we obtain \eqref{estimate on phi 2,2}.
\end{proof}
\begin{coro}\label{coro outer 2}
  If $ |x|\geq   K\lambda(t)/2$, then
\begin{equation}\label{estimate on phi 2,2}
 \overline{\phi}_2(x,t)\lesssim   \left[K^{-\frac{n-2}{2}(\beta-1)+\beta\gamma}+\frac{L^{\frac{n-2}{2}-\gamma}}{K^{\frac{n-2}{2}-\beta\gamma}} \right] \left(K\varepsilon\right)^{-\frac{n-2}{2}}\|\mathcal{I}\|_{L^\infty(-1,t)}.
\end{equation}
\end{coro}

\begin{lem}[Estimate on $\overline{\phi}_3$]\label{lem estimate on phi 3,2}
If $|x|\geq K\lambda(t)/2$, then
\begin{equation}\label{estimate on phi 3,1}
 \overline{\phi}_3(x,t)\lesssim  \varepsilon^{\frac{n-6}{2}}|x|^{4-n}\|\lambda\lambda^\prime\|_{L^\infty(-1,t)}.
\end{equation}
\end{lem}
\begin{proof}
  As in the proof of Lemma \ref{lem estimate outer 3}, we still have the heat kernel representation
\begin{eqnarray*}
    \overline{\phi}_3(x,t) &\lesssim &  \int_{-81/100}^{t} \left(1+\frac{\sqrt{t-s}}{|x|}\right)^\gamma \lambda(s)^{\frac{n-4}{2}}\left|\lambda^\prime(s)\right|\\
   && \times \left[\int_{B_1}(t-s)^{-\frac{n}{2}}  \left(1+\frac{\sqrt{t-s}}{|y|}\right)^\gamma e^{-c\frac{|x-y|^2}{t-s}}|y|^{2-n}dy\right]ds\\
&\lesssim& \varepsilon^{\frac{n-6}{2}} \|\lambda\lambda^\prime\|_{L^\infty(-1,t)}\\
&&\times  \int_{-81/100}^{t} \left(1+\frac{\sqrt{t-s}}{|x|}\right)^{2\gamma} \left(\sqrt{t-s}+|x|\right)^{2-n}ds  \\
&\lesssim & \varepsilon^{\frac{n-6}{2}}\|\lambda\lambda^\prime\|_{L^\infty(-1,t)}|x|^{4-n}.
\end{eqnarray*}
Here we have used  Lemma \ref{lem heat kernel integral 1} to deduce the second inequality.
\end{proof}
\begin{coro}\label{coro outer 3}
  If $|x|\geq  K\lambda(t)/2$, then
\begin{equation}\label{estimate on phi 3,2}
 \overline{\phi}_3(x,t)\lesssim   K^{-\frac{n-6}{2}} \left(K\varepsilon\right)^{-\frac{n-2}{2}}\|\lambda\lambda^\prime\|_{L^\infty(-1,t)}.
\end{equation}
\end{coro}

\begin{lem}[Estimate on $\overline{\phi}_4$]\label{lem estimate on phi 4,2}
If $|x|\geq K\lambda(t)/2$, then
\begin{equation}\label{estimate on phi 4,1}
 \overline{\phi}_4(x,t)\lesssim  \varepsilon^{\frac{n-4}{2}}|x|^{3-n}\|\lambda\xi^\prime\|_{L^\infty(-1,t)}.
\end{equation}
\end{lem}
\begin{proof}
  As in the proof of Lemma \ref{lem estimate outer 4}, we still have the heat kernel representation
\begin{eqnarray*}
    \overline{\phi}_4(x,t) &\lesssim &  \int_{-81/100}^{t} \left(1+\frac{\sqrt{t-s}}{|x|}\right)^\gamma \lambda(s)^{\frac{n-2}{2}}\left|\xi^\prime(s)\right|\\
   && \times \left[\int_{B_1}(t-s)^{-\frac{n}{2}}  \left(1+\frac{\sqrt{t-s}}{|y|}\right)^\gamma e^{-c\frac{|x-y|^2}{t-s}}|y|^{1-n}dy\right]ds\\
&\lesssim& \varepsilon^{\frac{n-4}{2}} \|\lambda\xi^\prime\|_{L^\infty(-1,t)}\\
&&\times  \int_{-81/100}^{t} \left(1+\frac{\sqrt{t-s}}{|x|}\right)^{2\gamma} \left(\sqrt{t-s}+|x|\right)^{1-n}ds  \\
&\lesssim & \varepsilon^{\frac{n-4}{2}}\|\lambda\xi^\prime\|_{L^\infty(-1,t)}|x|^{3-n}.
\end{eqnarray*}
Here we have used  Lemma \ref{lem heat kernel integral 1} to deduce the second inequality.
\end{proof}

\begin{coro}
  If $ |x|\geq  K\lambda(t)/2$, then
\begin{equation}\label{estimate on phi 4,2}
 \overline{\phi}_4(x,t)\lesssim   K^{-\frac{n-4}{2}} \left(K\varepsilon\right)^{-\frac{n-2}{2}}\|\lambda\xi^\prime\|_{L^\infty(-1,t)}.
\end{equation}
\end{coro}

\begin{lem}[Estimate on $\overline{\phi}_5$]\label{lem estimate on phi 5,2}
If $  |x|\geq  K\lambda(t)/2$, then
\begin{equation}\label{estimate on phi 5,1}
 \overline{\phi}_5(x,t)  \lesssim   \varepsilon^{\frac{n-4}{2}}|x|^{3-n} \left\|\left(\frac{a}{\lambda},\lambda a^\prime-\mu_0\frac{a}{\lambda}, \lambda \lambda^\prime,\lambda\xi^\prime \right)\right\|_{L^\infty(-1,t)}.
\end{equation}
\end{lem}
\begin{proof}
For $|y|\geq L\lambda(s)$,
\[\lambda(s)^{-\frac{n+2}{2}}e^{-c|y|/\lambda(s)}\lesssim \lambda(s)^{\frac{n-4}{2}}|y|^{1-n}.\]
Hence
\begin{eqnarray*}
   &&  \partial_t\overline{\phi}_5-\Delta\overline{\phi}_5 - \left[C+\frac{\delta}{|x|^{ 2}}\right]\overline{\phi}_5+\xi^\prime(t)\cdot \nabla\overline{\phi}_5\\
    &\lesssim& \left[\frac{|a(t)|}{\lambda(t)}+\Big|\lambda(t)a^\prime(t)-\mu_0\frac{a(t)}{\lambda(t)}\Big| + |\lambda(t)\lambda^\prime(t)|+|\lambda(t)\xi^\prime(t)|\right]\lambda(t)^{\frac{n-4}{2}}|x|^{1-n}\chi_{B_{L\lambda(t)}^c}.
\end{eqnarray*}
Then we can proceed as in the previous lemma to conclude.
\end{proof}

\begin{coro}\label{coro outer 5}
  If $  |x|\geq  K\lambda(t)/2$, then
\begin{equation}\label{estimate on phi 5,2}
 \overline{\phi}_5(x,t) \lesssim K^{-\frac{n-4}{2}} \left(K\varepsilon\right)^{-\frac{n-2}{2}} \left\|\left(\frac{a}{\lambda},\lambda a^\prime-\mu_0\frac{a}{\lambda}, \lambda \lambda^\prime,\lambda\xi^\prime \right)\right\|_{L^\infty(-1,t)}.
\end{equation}
\end{coro}

Combining Corollary \ref{coro outer 2}-Corollary \ref{coro outer 5} with Lemma \ref{lem estimate outer 1}, proceeding as in the proof of Proposition \ref{prop inner to outer} (to estimate $|\nabla\phi_1|$-$|\nabla\phi_5|$), we find two constants $C(K,L)$ and $\sigma(K,L)\ll 1$ such that, for any $t\in(-81/100,81/100)$,
\begin{equation}\label{inner to outer 2}
  \mathcal{O}(t)\leq C(K,L)\varepsilon^{\frac{n-2}{2}-\gamma}+\sigma(K,L)\sup_{-1<s<t}\mathcal{D}(s).
\end{equation}
Combining this inequality with Proposition \ref{prop decay estimate}, we obtain (compare this with \eqref{decay relation 4})
\begin{equation}\label{decay relation 5}
  \mathcal{D}(t)\leq \frac{1}{2} \sup_{[t-C\varepsilon^2,t+C\varepsilon^2]}\mathcal{D}(s)+C\varepsilon^{\frac{n-2}{2}-\gamma}.
\end{equation}
Here we have used the fact that, by \eqref{bound on parameters 1} and the definition of inner time variable $\tau$ in Section \ref{sec inner}, we have
\begin{equation}\label{change of time variables}
  \frac{d\tau}{dt}\sim \varepsilon^{-2}.
\end{equation}

Similar to Section \ref{sec Harnack for lambda}, an iteration of \eqref{decay relation 5} from $r=1$ to $r=8/9$ gives \eqref{inner estimate 2}. This finishes the proof of Proposition \ref{prop inner estimate 2}. Plugging \eqref{inner estimate 2} into Lemma \ref{lem estimate on phi 2,2}-Lemma \ref{lem estimate on phi 5,2}, we get Proposition \ref{prop outer estimate 2}.

\section{Improved estimates on $\phi$}\label{sec regularity on phi}
\setcounter{equation}{0}

From now on we need to write down the dependence on $i$ explicitly. For $u_i$, the decomposition in Section \ref{sec decomposition} reads as
\begin{equation}\label{decomposition}
   \phi_i(x,t):=u_i(x,t)-W_{\xi_i(t),\lambda_i(t)}(x)-a_i(t)Z_{0,\xi_i(t),\lambda_i(t)}(x).
\end{equation}
The parameters satisfy   the assumptions \eqref{determination of the scaling parameter size}-\eqref{bound on parameters 2}. In particular,
\[ \varepsilon_i:=\lambda_i(0), \quad \xi_i(0)=0.\]
 In this section we  prove some uniform (in $\varepsilon_i$) estimates on $\phi_i$.

For simplicity of notation, we will denote
\[ W_i(x,t):=W_{\xi_i(t),\lambda_i(t)}(x), \quad Z_{\ast,i}(x,t):=\left(Z_{j,\xi_i(t),\lambda_i(t)}(x)\right)_{j=0,\cdots,n+1}.\]

\subsection{Uniform $L^\infty$ bound}\label{subsec uniform bound}
The first one is a uniform $L^\infty$  bound on $\phi_i$.
\begin{prop}\label{prop zeroth order on phi}
There exists a universal constant $C$ (independent of $i$) such that
  \begin{equation}\label{improved estimate}
   |\phi_i(x,t)|\leq C \quad \mbox{in} \quad  Q_{6/7}.
  \end{equation}
\end{prop}

First by \eqref{outer estimate 3}, we get
\begin{equation}\label{bound on u}
  u_i(x,t)\lesssim \left(\frac{\varepsilon_i}{\varepsilon_i^2+|x-\xi_i(t)|^2}\right)^{\frac{n-2}{2}}+|x-\xi_i(t)|^{-\gamma}.
\end{equation}
Hence $\overline{\phi}_{i,1}$ satisfies, for some constant $C_0>0$,
\begin{equation}\label{outer 1,2}
   \partial_t\overline{\phi}_{i,1}-\Delta\overline{\phi}_{i,1}+\xi_i^\prime(t)\cdot\nabla\overline{\phi}_{i,1}\leq C_0\left[\varepsilon_i^2|x|^{-4}\chi_{B_{L\varepsilon_i}^c}+|x|^{-(p-1)\gamma}\right]\overline{\phi}_{i,1}.
\end{equation}

As in the proof of Lemma \ref{lem estimate outer 1}, for each $\varepsilon>0$, consider the problem
\begin{equation}\label{singular eqn}
 \left\{\begin{aligned}
&-\Delta w_\varepsilon =C_0 \varepsilon^2|x|^{-4}\chi_{B_{L\varepsilon}^c}w_\varepsilon \quad & \mbox{in } ~~ B_1,\\
&w_\varepsilon=1 \quad & \mbox{on } ~~ \partial B_1.
\end{aligned}\right.
\end{equation}

\begin{lem}\label{lem estimate on weight}
There exists a  radially symmetric solution $w_\varepsilon$ of \eqref{singular eqn}. Moreover, if $L$ is universally large, for any $\varepsilon>0$,
\[ 1\leq w_\varepsilon\leq 2 \quad \mbox{in } B_1.\]
\end{lem}
\begin{proof}
Consider the initial value problem
\[
 \left\{\begin{aligned}
&w^{\prime\prime}(s)+(n-2)w^\prime(s)+C_0e^{-2s}w(s)=0 \quad   \mbox{in } ~~ [\log L,+\infty) ,\\
&w(\log L)=1, \quad    w^\prime(\log L)=0.
\end{aligned}\right.
\]
Global existence and uniqueness of the solution follows from  standard ordinary differential equation theory.

If $w>0$ in $[\log L, s_0]$ for some $s_0>\log L$, then  $w^\prime<0$, and consequently, $0<w<1$ in $(\log L,s_0)$. Integrating the equation of $w$ we get
\begin{eqnarray}\label{gradient for weight fct}
 w^\prime(s)&=& -C_0e^{-(n-2)s}\int_{\log L}^{s}e^{(n-4)\tau}w(\tau)d\tau \\
   &\geq & -\frac{C_0}{n-4}e^{-2s}. \nonumber
\end{eqnarray}
Hence, if $L$ is large enough,
\[w(s_0)\geq 1-\frac{C_0}{n-4}L^{-2}\geq \frac{1}{2}.\]
This holds for any $s_0$ provided that $w>0$ in $[\log L, s_0]$, so $w(s)\geq 1/2$ for any $s\geq \log L$.

For each $\varepsilon>0$, the function
\[w_\varepsilon(x):=\frac{w\left(\log(\varepsilon^{-1}|x|)\right)}{w\left(\log\varepsilon^{-1}\right)}\]
 satisfies all of the requirements.
\end{proof}

\begin{coro}\label{coro gradient for weight fct}
  There exists a constant $C$ independent of $\varepsilon$ such that
  \[|\nabla w_\varepsilon|\leq C \quad \mbox{in }~~B_1.\]
\end{coro}
\begin{proof}
  By definition, we have
  \[|\nabla w_\varepsilon(x)|=\frac{w^\prime\left(\log(\varepsilon^{-1}|x|)\right)}{w\left(\log\varepsilon^{-1}\right)}\frac{\varepsilon}{|x|}.\]
If $|x|\leq L\varepsilon$, $w^\prime=0$ and there is nothing to prove. For $|x|\geq L\varepsilon$, by \eqref{gradient for weight fct},
\[ \left|w^\prime\left(\log(\varepsilon^{-1}|x|)\right)\right|\leq \frac{C_0}{n-4}.\]
Because $w\left(\log\varepsilon^{-1}\right)\geq 1/2$, we get
\[|\nabla w_\varepsilon(x)|\leq  \frac{2C_0}{n-4}. \qedhere\]
\end{proof}

The function $\varphi_i:=\overline{\phi}_{i,1}/w_{\varepsilon_i}$ satisfies
\[ \partial_t\varphi_i- w_{\varepsilon_i}^{-2}\mbox{div}\left(w_{\varepsilon_i}^2\nabla\varphi_i\right)+\xi_i^\prime\cdot\nabla\varphi_i\leq \left[C_0|x|^{-(p-1)\gamma}-\xi_i^\prime\cdot\nabla w_{\varepsilon_i}\right]\varphi_i.\]
By Lemma \ref{lem estimate on weight}, this is a uniformly parabolic equation. Because $\gamma$ is very small, $|x|^{-(p-1)\gamma}\in L^{2n}(B_1)$. By \eqref{Lip assumption} and  Corollary \ref{coro gradient for weight fct}, $\xi_i^\prime\cdot\nabla w_{\varepsilon_i}$ are uniformly bounded in $L^\infty(Q_{6/7})$. Then standard Moser iteration gives
\begin{lem}
There exists a universal constant $C$ independent of $i$ such that
\begin{equation}\label{estimate on phi 1,2}
  \sup_{Q_{6/7}}\varphi_i\lesssim \int_{Q_1}\varphi_i.
\end{equation}
\end{lem}
Combining this lemma with Lemma \ref{lem estimate on weight} and the definition of $\overline{\phi}_{i,1}$ (see \eqref{transformation of outer components}), we get
\[ \|\phi_{i,1}\|_{L^\infty(Q_{6/7})}\leq C.\]
Starting from this estimate, following the iteration argument in Section \ref{sec inner and outer 2}, we deduce that
\begin{equation}\label{estimate on D 2}
g(5/6)\lesssim \varepsilon_i^{\frac{n-2}{2}}.
\end{equation}
Substituting this estimate into Lemma \ref{lem estimate on phi 2,2}-Lemma \ref{lem estimate on phi 5,2}, and noting that now the heat kernel for $\partial_t-\Delta-\xi_i^\prime\cdot\nabla-V_i$ enjoys the standard Gaussian bound
\[ G_i(x,t;y,s)\lesssim (t-s)^{-\frac{n}{2}}e^{-c\frac{|x-y|^2}{t-s}},\]
i.e. we can take $\gamma=0$, we obtain
\begin{lem}
For any  $x\in B_{K\varepsilon}^c$ and $t\in(-5/6,5/6)$, we have
  \begin{equation*}
 |\phi_{i,2}(x,t)| + |\phi_{i,3}(x,t)| + |\phi_{i,4}(x,t)| +|\phi_{i,5}(x,t)|\lesssim_{K} \frac{\varepsilon^{n-4}}{|x|^{n-4}}.
\end{equation*}
\end{lem}

Combining \eqref{estimate on phi 1,2}, \eqref{estimate on D 2} with this lemma, we finish the proof of Proposition \ref{prop zeroth order on phi}.

Once we know that $\phi_i$ are uniformly bounded in $Q_1$ (after a scaling of the domain), the above estimates are improved to
\begin{prop}\label{prop uniform estimate}
\begin{eqnarray*}
   &\|\phi_{i,1}\|_{L^\infty(Q_{5/6})}\lesssim \|\phi_i\|_{L^\infty(Q_1)}, \\
  & g(5/6)\lesssim \varepsilon_i^{\frac{n-2}{2}}\left(\|\phi_i\|_{L^\infty(Q_1)}+e^{-c\varepsilon_i^{-2}}\right),
\end{eqnarray*}
and for any $(x,t)\in Q_{5/6}$,
   \[ |\phi_{i,2}(x,t)| + |\phi_{i,3}(x,t)| + |\phi_{i,4}(x,t)| +|\phi_{i,5}(x,t)|\lesssim_{K} \frac{\varepsilon^{n-4}}{|x|^{n-4}}\left(\|\phi_i\|_{L^\infty(Q_1)}+e^{-c\varepsilon_i^{-2}}\right).\]
\end{prop}
Here the term $e^{-c\varepsilon_i^{-2}}$ appears because when we apply Lemma \ref{lem representation of a} to estimate  $a/\lambda$, there is a boundary term
\[ \frac{a_i(1)}{\lambda_i(1)}e^{-c \left[\tau(1)-\tau(t)\right]},\]
where, because $d\tau/dt\sim \varepsilon_i^{-2}$,
\[ \tau(1)-\tau(t)\geq c\varepsilon_i^{-2}, \quad \mbox{for any } ~~ t<25/36.\]

Because $\xi_i(0)=0$, by the estimate of $g(5/6)$ in Proposition \ref{prop uniform estimate}, now \eqref{bound on parameters 3} is upgraded to
\begin{equation}\label{bound on parameters 4}
  |\xi_i(t)|\lesssim \varepsilon_i^{\frac{n-4}{2}}, \quad \mbox{for any} ~~ t\in[-25/36,25/36].
\end{equation}

\subsection{Gradient estimates}
In this subsection we establish a  uniform $C^{1+\theta, (1+\theta)/2}$ estimate on $\phi_i$.

With the estimates in Proposition \ref{prop uniform estimate}, now we have
\begin{lem}\label{lem improvement on RHS in error eqn}
 In $Q_{5/6}$,
  \begin{equation}\label{improvement of error eqn}
  \left|\partial_t\phi_i-\Delta\phi_i\right|\lesssim \left(\frac{\varepsilon_i^2}{\varepsilon_i^4+|x|^4}+1\right)\left(\left\|\phi_i\right\|_{L^\infty(Q_1)}+e^{c\varepsilon_i^{-2}}\right) .
\end{equation}
\end{lem}
\begin{proof}
By \eqref{estimate on D 2}, now \eqref{bound on parameters 4} can be upgraded to
\begin{equation}\label{bound on parameters 5}
  |\xi_i(t)|\lesssim \varepsilon_i^{\frac{n-4}{2}}, \quad \mbox{for any } ~~ t\in[-25/36,25/36].
\end{equation}
Then by the definition of $\phi_i$, we have
\begin{eqnarray*}
  u_i(x,t) &\lesssim & W_i(x,t)+|a_i(t)|Z_{0,i}(x,t)+|\phi_i(x,t)| \\
   &\lesssim & \left(\frac{\varepsilon_i}{\varepsilon_i^2+|x|^2}\right)^{\frac{n-2}{2}}+ e^{-c|x|/\varepsilon_i}+1.
\end{eqnarray*}
Therefore
\begin{eqnarray*}
 &&\left|u_i^p-W_i^p\right|  \\
 &\lesssim& \left(u_i^{p-1}+W_i^{p-1}\right)\left(|\phi_i|+|a_i(t)|Z_{0,i}\right) \\
  &\lesssim &  \left[\left(\frac{\varepsilon_i}{\varepsilon_i^2+|x|^2}\right)^{\frac{n-2}{2}}+ e^{-c|x|/\varepsilon_i}+1\right]^{\frac{4}{n-2}}\left(1+ e^{-c|x|/\varepsilon_i}\right)\left(\left\|\phi_i\right\|_{L^\infty(Q_1)}+e^{c\varepsilon_i^{-2}}\right) \\
  &\lesssim &\left(\frac{\varepsilon_i^2}{\varepsilon_i^4+|x|^4}+1\right)\left(\left\|\phi_i\right\|_{L^\infty(Q_1)}+e^{c\varepsilon_i^{-2}}\right).
\end{eqnarray*}

By \eqref{estimate on D 2}, we have
\begin{eqnarray*}
   \left|\lambda_i^\prime(t) Z_{n+1,i}(x,t)\right| & \lesssim & \varepsilon_i^{-2}\left(1+\frac{|x-\xi_i(t)|}{\varepsilon_i}\right)^{2-n}
 \left(\left\|\phi_i\right\|_{L^\infty(Q_1)}+e^{c\varepsilon_i^{-2}}\right)\\
    &\lesssim &  \frac{\varepsilon_i^{n-4}}{\left(\varepsilon_i+|x|\right)^{n-2}}\left(\left\|\phi_i\right\|_{L^\infty(Q_1)}+e^{c\varepsilon_i^{-2}}\right)\\
    &\lesssim &\frac{\varepsilon_i^2}{\varepsilon_i^4+|x|^4}\left(\left\|\phi_i\right\|_{L^\infty(Q_1)}+e^{c\varepsilon_i^{-2}}\right).
\end{eqnarray*}
Similar estimates hold for $(-a_i^\prime+\mu_0a_i\lambda_i^{-2})Z_{0i}$, $\xi_{ji}^\prime\cdot Z_{ji}$, $j=1,\cdots, n$ and $a_i\partial_tZ_{0i}$.
\end{proof}

\begin{lem}\label{lem gradient estimate on error fct}
 For any $(x,t)\in Q_{4/5}$,
   \begin{equation}\label{gradient estimate on error fct}
    |\nabla\phi_i(x,t)|\lesssim \left(\frac{\varepsilon_i^2}{\varepsilon_i^3+|x|^3}+1\right)\left(\left\|\phi_i\right\|_{L^\infty(Q_1)}+e^{c\varepsilon_i^{-2}}\right) .
  \end{equation}
\end{lem}
\begin{proof}
We consider three cases separately.

{\bf Case 1.} If $1/4\leq |x|\leq 4/5$, this estimate (in fact, a bound on the Lipschitz seminorm) follows from standard interior gradient estimates.

{\bf Case 2.} If $|x|\leq \varepsilon_i$, by looking at the inner equation \eqref{inner eqn} and using Proposition \ref{prop uniform estimate}, we obtain
\[|\nabla\varphi_i(y,\tau)|\lesssim \varepsilon_i^{\frac{n-2}{2}}\left(\left\|\phi_i\right\|_{L^\infty(Q_1)}+e^{c\varepsilon_i^{-2}}\right).\]
Scaling back to $\phi_i$, this is
\[|\nabla\phi_i(x,t)|\lesssim \varepsilon_i^{-1}\left(\left\|\phi_i\right\|_{L^\infty(Q_1)}+e^{c\varepsilon_i^{-2}}\right).\]
In fact, this gives a full Lipschitz estimate in space-time.

{\bf Case 3.} Finally, we consider the remaining case where $\varepsilon_i\leq |x|\leq 1/4$. Take the decomposition $\phi_i=\phi_{h,i}+\phi_{n,i}$, where
  \[
    \left\{\begin{aligned}
&\partial_t\phi_{h,i}-\Delta\phi_{h,i}=0, \quad  & \mbox{in }~~ Q_{3/4},\\
&\phi_{h,i}=\phi_i, \quad  & \mbox{on }~~ \partial^pQ_{3/4}.
\end{aligned}\right.
 \]

By standard interior gradient estimates for heat equation, we have
\begin{equation}\label{gradient estimate from boundary}
  \|\nabla\phi_{h,i}\|_{L^\infty(Q_{2/3})}\lesssim  \|\phi_i\|_{L^\infty(\partial^pQ_{3/4})}.
\end{equation}

Next, for the non-homogeneous part $\phi_{n,i}$, by the heat kernel representation, we have
\begin{eqnarray*}
  |\nabla\phi_{n,i}(x,t)| &=& \left|\int_{-\frac{81}{100}}^{t}\int_{B_{\frac{3}{4}}} \nabla_x G(x,t;y,s)\left(\partial_t-\Delta\right)\phi_i(y,s)dyds \right| \\
  &\lesssim&\left(\left\|\phi_i\right\|_{L^\infty(Q_1)}+e^{c\varepsilon_i^{-2}}\right) \int_{-\frac{9}{16}}^{t}\int_{B_{\frac{3}{4}}}\frac{|x-y|}{(t-s)^{\frac{n+2}{2}}} e^{-c\frac{|x-y|^2}{t-s}}\left(\frac{\varepsilon_i^2}{\varepsilon_i^4+|y|^4}+1\right) \\
   &\lesssim&\left(\left\|\phi_i\right\|_{L^\infty(Q_1)}+e^{c\varepsilon_i^{-2}}\right)\int_{-\frac{9}{16}}^{t}\int_{B_{\frac{3}{4}}}(t-s)^{-\frac{n+1}{2}} e^{-c\frac{|x-y|^2}{t-s}}\left(\frac{\varepsilon_i^2}{|y|^4}+1\right) \\
    &\lesssim& \left(\left\|\phi_i\right\|_{L^\infty(Q_1)}+e^{c\varepsilon_i^{-2}}\right)\int_{-\frac{9}{16}}^{t} (t-s)^{-\frac{1}{2}}\left[ \frac{\varepsilon_i^2}{\left(\sqrt{t-s}+|x|\right)^4}+1\right] \\
       &\lesssim& \left(\frac{\varepsilon_i^2}{|x|^3}+1\right) \left(\left\|\phi_i\right\|_{L^\infty(Q_1)}+e^{c\varepsilon_i^{-2}}\right).
\end{eqnarray*}
Combining this estimate with \eqref{gradient estimate from boundary}, we get \eqref{gradient estimate on error fct} in this case.
\end{proof}

Next we extend this Lipschitz estimate in spatial variables to the full Lipschitz estimate in space-time variables. For this we need a technical lemma transforming the spatial Lipschitz estimate  into a full Lipschitz estimate in space-time.
\begin{lem}\label{lem Lip estimate}
Given a constant $\sigma>0$, suppose $\psi\in C^\infty(Q_1^-)$ satisfies
\begin{enumerate}
    \item $|\nabla\psi|\leq\sigma$ in $Q_1^-$;
  \item $|\partial_t\psi-\Delta\psi|\leq \sigma$ in $Q_1^-$.
\end{enumerate}
Then the Lipschitz seminorm (with respect to the parabolic distance) of $\psi$ in $Q_{1/2}^-$ satisfies
\[ |\psi|_{Lip(Q_{1/2}^-)}\lesssim \sigma.\]
\end{lem}
\begin{proof}
  We need only to prove that, for any $x\in B_{1/2}$ and $-1/4<t_1<t_2<0$,
  \begin{equation}\label{half Holder in time}
    |\psi(x,t_1)-\psi(x,t_2)|\lesssim \sigma \sqrt{t_2-t_1}.
  \end{equation}

Denote $h:=\sqrt{t_2-t_1}$ and define
\[ \widetilde{\psi}(y,s):=\frac{1}{h}\left[\psi(x+hy,t_2+h^2s)-\psi(x,t_2)\right].\]
It satisfies
\begin{enumerate}
   \item $|\nabla\widetilde{\psi}|\leq\sigma$ in $Q_2^-$;
  \item  for any  $(y,s)\in Q_2^-$,
  \begin{equation}\label{scaled heat equation}
    |\partial_t\widetilde{\psi}(y,s)-\Delta\widetilde{\psi}(y,s)|\leq h\sigma;
  \end{equation}
  \item because $\widetilde{\psi}(0,0)=0$, by integrating (1) we obtain
    \begin{equation}\label{bound at time 0}
      \sup_{y\in B_2}|\widetilde{\psi}(y,0)|\leq 2\sigma.
    \end{equation}
\end{enumerate}
 Fix a function $\eta\in C_0^\infty(B_2)$. Multiplying \eqref{scaled heat equation} and integrating by parts, we obtain
  \begin{eqnarray*}
    \left|\frac{d}{ds}\int_{B_2}\widetilde{\psi}\eta\right| &\leq & h\sigma\int_{B_2}|\eta|+\int_{B_2}|\nabla\widetilde{\psi}||\nabla\eta| \\
   &\lesssim& \sigma.
  \end{eqnarray*}
  Combining this inequality with \eqref{bound at time 0}, we obtain
  \[ \left|\int_{B_2}\widetilde{\psi}(y,s)\eta(y)\right|\lesssim \sigma \quad \mbox{ for any} ~~ s\in(-4,0).\]
  With the Lipschitz estimate in (1), this implies that
  \begin{equation}\label{bound at time s}
      \sup_{y\in B_2}|\widetilde{\psi}(y,s)|\lesssim \sigma \quad \mbox{ for any} ~~ s\in(-4,0).
    \end{equation}
In particular,
\[ |\widetilde{\psi}(0,-1)|\lesssim \sigma.\]
Scaling back to $\psi$, this is \eqref{half Holder in time}.
\end{proof}

\begin{prop}\label{prop Lip estimate on error fct}
 For any $(x,t)\in Q_{3/4}$,
   \begin{equation}\label{Lip estimate on error fct}
    |\phi_i|_{Lip(Q_{|x|/4}(x,t))}\lesssim \left(\frac{\varepsilon_i^2}{\varepsilon_i^3+|x|^3}+1\right)\left(\left\|\phi_i\right\|_{L^\infty(Q_1)}+e^{c\varepsilon_i^{-2}}\right) .
  \end{equation}
\end{prop}
\begin{proof}
  When $|x|<\varepsilon_i$ and $1/4<|x|<3/4$, this was already covered in the proof of Lemma \ref{lem gradient estimate on error fct}, see Case 1 and Case 2 therein.

If $\varepsilon_1<|x|<1/4$, denote $r:=|x|/2$ and let
\[ \widetilde{\phi}_i(y,s):=\phi_i\left(x+r y, t+r^2s\right).\]
By Lemma \ref{lem improvement on RHS in error eqn} and Lemma \ref{lem gradient estimate on error fct}, it satisfies the assumptions in Lemma \ref{lem Lip estimate} with the constant
\[ \sigma\lesssim \left(\frac{\varepsilon_i^2}{r^2}+r\right)\left(\left\|\phi_i\right\|_{L^\infty(Q_1)}+e^{c\varepsilon_i^{-2}}\right) .\]
Hence this lemma implies that
\[  |\widetilde{\phi}_i |_{Lip(Q_{1/2}^-)}\lesssim \left(\frac{\varepsilon_i^2}{r^2}+r\right)\left(\left\|\phi_i\right\|_{L^\infty(Q_1)}+e^{c\varepsilon_i^{-2}}\right) .\]
Scaling back to $\phi_i$, this is \eqref{Lip estimate on error fct} in this case.
\end{proof}

\subsection{Estimate of time derivative}\label{sec Schauder}

In this subsection we establish the following  estimate on  $\partial_t\phi_i$. We can also obtain some estimates on $\nabla^2\phi_i$, but we do not need it.
\begin{prop}\label{prop Schauder estimate on error fct}
 For any $(x,t)\in Q_{2/3}$,
   \begin{equation}\label{Schauder estimate on error fct}
    |\partial_t\phi_i(x,t)|\lesssim  \left(\frac{\varepsilon_i^2}{\varepsilon_i^4+|x|^4}+\frac{1}{\varepsilon_i^2+|x|}\right)\left(\left\|\phi_i\right\|_{L^\infty(Q_1)}+e^{c\varepsilon_i^{-2}}\right) .
  \end{equation}
\end{prop}
\begin{proof}
 Take an arbitrary sequence of points $(x_i,t_i)\in Q_{2/3}$. Denote
$r_i:=|x_i|/2$.

{\bf Case 1.} First assume
\[\limsup_{i\to+\infty}\frac{r_i}{\varepsilon_i}<+\infty.\]
In this case we work in the inner coordinates introduced in Section \ref{sec inner}. By \eqref{orthogonal in inner} and \eqref{inner eqn with cut-off}, we have the following representation formulas:
\[
  \left\{\begin{aligned}
&\frac{\dot{a}_i(\tau)-\mu_0a_i(\tau)}{\lambda_i(\tau)}=\int_{\R^n}Z_0(y)E_{K,i}(y,\tau)dy,\\
&\frac{\dot{\xi}_i(\tau)}{\lambda_i(\tau)}=-\int_{\R^n}Z_j(y)E_{K,i}(y,\tau)dy, \quad j=1,\cdots, n,\\
& \frac{\dot{\lambda}_i(\tau)}{\lambda_i(\tau)}=-\int_{\R^n}Z_{n+1}(y)E_{K,i}(y,\tau)dy.
\end{aligned}\right.
\]
By the form of $E_K$ in Section \ref{sec inner}, and the smallness of
\[ \left|\frac{\dot{a}_i -\mu_0a_i }{\lambda_i}\right|+\left|\frac{\dot{\xi}_i}{\lambda_i}\right|+\left|\frac{\dot{\lambda}_i}{\lambda_i}\right|,\]
 the above three equations are solved as
\begin{equation}\label{solving parameter eqn}
  \left(\frac{\dot{a}_i(\tau)-\mu_0a_i(\tau)}{\lambda_i(\tau)},\frac{\dot{\xi}_i(\tau)}{\lambda_i(\tau)},\frac{\dot{\lambda}_i(\tau)}{\lambda_i(\tau)}\right)
=\mathcal{J}\left(\nabla\varphi_i(\tau),\varphi_i(\tau), \frac{a_i(\tau)}{\lambda_i(\tau)}\right),
\end{equation}
where $\mathcal{J}=(\mathcal{J}_0,\cdots, \mathcal{J}_{n+1})$ is a vector valued, nonlinear (but smooth) integral operator.
Plugging this back into  \eqref{inner eqn}, we get
\begin{eqnarray*}
 \partial_\tau\varphi_i-\Delta_y\varphi_i
 & =&\left(W+\varphi+a_iZ_0\right)^p-W^p+\mathcal{J}\cdot Z  \\
 &+&\left(\mathcal{J}_1,\cdots,\mathcal{J}_n\right)\cdot\nabla\varphi+\mathcal{J}_{n+1}\left(y\cdot\nabla\varphi_i+\frac{n-2}{2}\varphi_i\right)\\
 &+&\frac{a_i}{\lambda_i}\left[\left(\mathcal{J}_1,\cdots,\mathcal{J}_n\right)\cdot\nabla Z_0+\mathcal{J}_{n+1}\left(y\cdot\nabla Z_0+\frac{n}{2}Z_0\right)\right].
\end{eqnarray*}
Starting from the $L^\infty$ bound of $\varphi_i$ (by rescaling the estimates in Proposition \ref{prop uniform estimate}) and  bootstrapping parabolic estimates,  we get
\begin{equation}\label{Schauder for inner}
  \|\varphi_i\|_{C^{2+\theta,1+\theta/2}(B_{2K}\times [\tau-1,\tau]) }\lesssim \varepsilon_i^{\frac{n-2}{2}}, \quad \mbox{for any}~~ \tau.
\end{equation}
A byproduct of this estimate is the validity of \eqref{Schauder estimate on error fct} in this case.

{\bf Case 2.} Substituting \eqref{Schauder for inner} into \eqref{solving parameter eqn} gives, for any $\tau_1<\tau_2$,
\begin{eqnarray}\label{Holder estimate of parameters 1}
\left\|\left(\frac{\dot{a}_i-\mu_0a_i}{\lambda_i},\frac{\dot{\xi}_i}{\lambda_i},\frac{\dot{\lambda}_i}{\lambda_i}\right)\right\|_{C^{\theta/2}([\tau_1,\tau_2])}
 \lesssim  \varepsilon_i^{\frac{n-2}{2}} \left(\left\|\phi_i\right\|_{L^\infty(Q_1)}+e^{c\varepsilon_i^{-2}}\right).
\end{eqnarray}
After scaling back to the original coordinates, by noting that $d\tau/dt\sim \varepsilon_i^{-2}$, this estimate is transformed into
\begin{eqnarray}\label{Holder estimate of parameters 2}
\left\|\left(a_i^\prime-\mu_0\frac{a_i}{\lambda_i^2},\xi_i^{\prime},\lambda_i^{\prime}\right)\right\|_{C^{\theta/2}([-3/4,3/4])}
 \lesssim  \varepsilon_i^{\frac{n-4}{2}-\theta} \left(\left\|\phi_i\right\|_{L^\infty(Q_1)}+e^{c\varepsilon_i^{-2}}\right).
\end{eqnarray}

 By Proposition \ref{prop uniform estimate}, \eqref{Holder estimate of parameters 2} and Proposition \ref{prop Lip estimate on error fct}, we deduce that the terms in the right hand side of \eqref{error eqn} are uniformly bounded in $C^{\theta,\theta/2}\left((B_1\setminus B_{1/8})\times(-1,1)\right)$. Hence standard Schauder estimates for heat equation gives
\[ \sup_{\left((B_{2/3}\setminus B_{1/4})\times(-4/9,4/9)\right)}|\partial_t\phi_i|\lesssim \left\|\phi_i\right\|_{L^\infty(Q_1)}+e^{c\varepsilon_i^{-2}}.\]

{\bf Case 3.}  In this case we assume
  \[ \lim_{i\to+\infty}\left(r_i+\frac{\varepsilon_i}{r_i}\right)=0.\]
Define
\[\widetilde{u}_i(x,t):=r_i^{\frac{n-2}{2}}u_i\left(r_i x,t_i+r_i^2t\right).\]

The decomposition in Section \ref{sec decomposition} can be transferred to $\widetilde{u}_i$ as
\begin{equation}\label{decomosition scaled}
  \widetilde{u}_i=W_{\widetilde{\xi}_i,\widetilde{\lambda}_i}+\widetilde{a}_iZ_{0,\widetilde{\xi}_i,\widetilde{\lambda}_i}+\widetilde{\phi}_i,
\end{equation}
where
\[
  \left\{\begin{aligned}
&\widetilde{\xi}_i(t):=r_i^{-1} \xi_i(t_i+r_i^2t),\\
&\widetilde{\lambda}_i(t):=r_i^{-1} \lambda_i(t_i+r_i^2t),\\
& \widetilde{a}_i(t):=r_i^{-1}a_i(t_i+r_i^2t),\\
& \widetilde{\phi}_i(x,t):=r_i^{(n-2)/2}\phi_i(r_ix, t_i+r_i^2t).
\end{aligned}\right.
\]
Because for any $t\in[-3/4,3/4]$,
\[ r_i>2K\varepsilon_i>K\lambda_i(t), \]
the orthogonal relation in Proposition \ref{prop orthogonal decompostion} still holds between $\widetilde{\phi}_i$ and $Z_{j,\widetilde{\xi}_i,\widetilde{\lambda}_i}$, $j=0,\cdots, n+1$.

Furthermore, the error equation \eqref{error eqn} now reads as
\begin{eqnarray}\label{error eqn scaled}
  \partial_t\widetilde{\phi}_i-\Delta\widetilde{\phi}_i&=&\left(W_{\widetilde{\xi}_i,\widetilde{\lambda}_i}
  +\widetilde{a}_iZ_{0,\widetilde{\xi}_i,\widetilde{\lambda}_i}+\widetilde{\phi}_i\right)^p
  -W_{\widetilde{\xi}_i,\widetilde{\lambda}_i}^p-pW_{\widetilde{\xi}_i,\widetilde{\lambda}_i}^{p-1}\widetilde{a}_iZ_{0,\widetilde{\xi}_i,\widetilde{\lambda}_i} \nonumber \\
  &+&\left(-\widetilde{a}_i^\prime+\mu_0\frac{\widetilde{a}_i}{\widetilde{\lambda}_i^2},\widetilde{\xi}_i^\prime ,\widetilde{\lambda}_i^\prime \right)\cdot Z_{\widetilde{\xi}_i,\widetilde{\lambda}_i}-\widetilde{a}_i\partial_tZ_{0,\widetilde{\xi}_i,\widetilde{\lambda}_i}.
\end{eqnarray}

By \eqref{bound on parameters 5} and \eqref{Holder estimate of parameters 2}, we get the following  estimates for those terms in the right hand side of \eqref{error eqn scaled}:
\begin{equation}\label{estimate of LHS scaled}
  \left\{\begin{aligned}
   &  \|W_{\widetilde{\xi}_i,\widetilde{\lambda}_i}\|_{L^\infty((B_2\setminus B_{1/4})\times(-4,4))}\lesssim \left(\frac{\varepsilon_i}{r_i}\right)^{\frac{n-2}{2}};\\
   & \|\widetilde{a}_iZ_{0,\widetilde{\xi}_i,\widetilde{\lambda}_i}\|_{L^\infty((B_2\setminus B_{1/4})\times(-4,4))}\lesssim \varepsilon_i^{\frac{n-2}{2}} e^{-c\frac{r_i}{\varepsilon_i}}\left(\left\|\phi_i\right\|_{L^\infty(Q_1)}+e^{c\varepsilon_i^{-2}}\right);\\
   & \|\widetilde{\phi}_i\|_{L^\infty((B_1\setminus B_{1/4})\times(-4,4))}\lesssim r_i^{\frac{n-2}{2}}\left(\left\|\phi_i\right\|_{L^\infty(Q_1)}+e^{c\varepsilon_i^{-2}}\right);\\
  & \left\|\left(\widetilde{a}_i^\prime-\mu_0 \widetilde{a}_i \widetilde{\lambda}_i^{-2},\widetilde{\xi}_i^{\prime},\widetilde{\lambda}_i^{\prime}\right)\cdot Z_{\widetilde{\xi}_i,\widetilde{\lambda}_i}\right\|_{L^\infty((B_2\setminus B_{1/4})\times(-4,4))}\lesssim \frac{  \varepsilon_i^{n-4}}{r_i^{\frac{n}{2}-3}}\left(\left\|\phi_i\right\|_{L^\infty(Q_1)}+e^{c\varepsilon_i^{-2}}\right);\\
  & \left\|\widetilde{a}_i\partial_tZ_{0,\widetilde{\xi}_i,\widetilde{\lambda}_i}\right\|_{C^{\theta,\theta/2}((B_2\setminus B_{1/4})\times(-4,4))}\lesssim \varepsilon_i^{n-2}  e^{-c\frac{r_i}{\varepsilon_i}}\left(\left\|\phi_i\right\|_{L^\infty(Q_1)}+e^{c\varepsilon_i^{-2}}\right).
\end{aligned}\right.
\end{equation}

These estimates imply that in $\left(B_2\setminus B_{1/4}\right)\times(-4,4)$,
\begin{equation}\label{estimate of u scaled}
  \widetilde{u}_i\lesssim  \left(\frac{\varepsilon_i}{r_i}\right)^{\frac{n-2}{2}}+r_i^{\frac{n-2}{2}}\left(\left\|\phi_i\right\|_{L^\infty(Q_1)}+e^{c\varepsilon_i^{-2}}\right)
\end{equation}
and
\begin{eqnarray*}
  &&\left| \left(W_{\widetilde{\xi}_i,\widetilde{\lambda}_i}
  +\widetilde{a}_iZ_{0,\widetilde{\xi}_i,\widetilde{\lambda}_i}+\widetilde{\phi}_i\right)^p
  -W_{\widetilde{\xi}_i,\widetilde{\lambda}_i}^p-pW_{\widetilde{\xi}_i,\widetilde{\lambda}_i}^{p-1}\widetilde{a}_iZ_{0,\widetilde{\xi}_i,\widetilde{\lambda}_i} \right|\\
  &\lesssim& \left[W_{\widetilde{\xi}_i,\widetilde{\lambda}_i}^{p-1}+\widetilde{u}_i^{p-1}\right] \left( \left|\widetilde{\phi_i}\right|+\left| \widetilde{a}_iZ_{0,\widetilde{\xi}_i,\widetilde{\lambda}_i} \right|\right)\\
  &\lesssim& \left(\frac{\varepsilon_i^2}{r_i^2}+r_i^2\right)r_i^{\frac{n-2}{2}} \left(\left\|\phi_i\right\|_{L^\infty(Q_1)}+e^{c\varepsilon_i^{-2}}\right) .
\end{eqnarray*}

By Proposition \ref{prop Lip estimate on error fct} and a rescaling, we know that the Lipschitz seminorm of $\widetilde{\phi}_i$ t in $\left(B_2\setminus B_{1/4}\right)\times(-4,4)$ is bounded by
\[ r_i^{\frac{n}{2}}\left(\frac{\varepsilon_i^2}{r_i^3}+1\right)\left(\left\|\phi_i\right\|_{L^\infty(Q_1)}+e^{c\varepsilon_i^{-2}}\right) .\]
By these estimates and standard $W^{2,p}$ estimates for heat equation, we obtain
\begin{eqnarray}\label{integral estimate for phi t}
  \int_{-3}^{3}\int_{B_{9/5}\setminus B_{1/3}}|\partial_t\widetilde{\phi}_i| &\lesssim & \left[r_i^{\frac{n}{2}}\left(\frac{\varepsilon_i^2}{r_i^3}+1\right)
  +\varepsilon_i^{\frac{n-2}{2}}\left(\frac{\varepsilon_i}{r_i}\right)^{\frac{n-3}{2}}
  +r_i^{\frac{n-2}{2}}\left(\frac{\varepsilon_i^2}{r_i^2}+r_i^2\right)\right] \nonumber \\
  && \quad \quad \times \left(\left\|\phi_i\right\|_{L^\infty(Q_1)}+e^{c\varepsilon_i^{-2}}\right).
\end{eqnarray}

Because
\[\left(\partial_t -\Delta\right)\partial_t\widetilde{u}_i=p\widetilde{u}_i^{p-1}\partial_t\widetilde{u}_i,\]
and $\widetilde{u}_i$ are uniformly bounded in $\left(B_{9/5}\setminus B_{1/3}\right)\times(-3,3)$, we get
\begin{eqnarray}\label{estimate of u t}
   &&\left\|\partial_t\widetilde{u}_i\right\|_{L^\infty((B_{3/2}\setminus B_{1/2})\times(-1,1))} \nonumber\\
  &\lesssim &  \int_{-3}^{3}\int_{B_{9/5}\setminus B_{1/3}}|\partial_t\widetilde{u}_i|\\
  &\lesssim&\left(\int_{-3}^{3}\int_{B_{9/5}\setminus B_{1/3}}|\partial_t\widetilde{\phi}_i|\right) + \|\partial_tW_{\widetilde{\xi}_i,\widetilde{\lambda}_i}\|_{L^\infty((B_2\setminus B_{1/4})\times(-4,4))}\nonumber\\
  && \quad \quad \quad \quad +\left\|\partial_t\left(\widetilde{a}_iZ_{0,\widetilde{\xi}_i,\widetilde{\lambda}_i}\right)\right\|_{L^\infty((B_2\setminus B_{1/4})\times(-4,4))}, \nonumber
\end{eqnarray}
where we have used the expansion from \eqref{decomosition scaled},
\[\partial_t\widetilde{u}_i=\partial_tW_{\widetilde{\xi}_i,\widetilde{\lambda}_i}+\partial_t\left(
\widetilde{a}_iZ_{0,\widetilde{\xi}_i,\widetilde{\lambda}_i}\right)+\partial_t\widetilde{\phi}_i.\]

Similar to \eqref{estimate of LHS scaled}, we also have
\[
 \left\{\begin{aligned}
   &  \|\partial_tW_{\widetilde{\xi}_i,\widetilde{\lambda}_i}\|_{L^\infty((B_2\setminus B_{1/4})\times(-4,4))}\lesssim \varepsilon_i^{n-4} r_i^{-\frac{n}{2}-3} \left(\left\|\phi_i\right\|_{L^\infty(Q_1)}+e^{c\varepsilon_i^{-2}}\right);\\
   &\left\|\partial_t\left(\widetilde{a}_iZ_{0,\widetilde{\xi}_i,\widetilde{\lambda}_i}\right)\right\|_{L^\infty((B_2\setminus B_{1/4})\times(-4,4))}\lesssim \varepsilon_i^{\frac{n-2}{2}} e^{-c\frac{r_i}{\varepsilon_i}}\left(\left\|\phi_i\right\|_{L^\infty(Q_1)}+e^{c\varepsilon_i^{-2}}\right).
\end{aligned}\right.
\]

Substituting these estimates and \eqref{integral estimate for phi t} into \eqref{estimate of u t}, we obtain
\begin{eqnarray*}
    &&\left\|\partial_t\widetilde{\phi}_i\right\|_{L^\infty((B_{3/2}\setminus B_{1/2})\times(-1,1))}\\
    &\lesssim &  \left\|\partial_t\widetilde{u}_i\right\|_{L^\infty((B_{3/2}\setminus B_{1/2})\times(-1,1))}\\
    &+&  \|\partial_tW_{\widetilde{\xi}_i,\widetilde{\lambda}_i}\|_{L^\infty((B_2\setminus B_{1/4})\times(-4,4))}+\left\|\partial_t\left(\widetilde{a}_iZ_{0,\widetilde{\xi}_i,\widetilde{\lambda}_i}\right)\right\|_{L^\infty((B_2\setminus B_{1/4})\times(-4,4))}\\
  &\lesssim&  \int_{-3}^{3}\int_{B_{9/5}\setminus B_{1/3}}|\partial_t\widetilde{\phi}_i|\\
  &+&  \|\partial_tW_{\widetilde{\xi}_i,\widetilde{\lambda}_i}\|_{L^\infty((B_2\setminus B_{1/4})\times(-4,4))}+\left\|\partial_t\left(\widetilde{a}_iZ_{0,\widetilde{\xi}_i,\widetilde{\lambda}_i}\right)\right\|_{L^\infty((B_2\setminus B_{1/4})\times(-4,4))}\\
  &\lesssim& \left[r_i^{\frac{n}{2}}\left(\frac{\varepsilon_i^2}{r_i^3}+1\right)
  +\varepsilon_i^{\frac{n-2}{2}}\left(\frac{\varepsilon_i}{r_i}\right)^{\frac{n-3}{2}}
  +r_i^{\frac{n-2}{2}}\left(\frac{\varepsilon_i^2}{r_i^2}+r_i^2\right)\right]
\left(\left\|\phi_i\right\|_{L^\infty(Q_1)}+e^{c\varepsilon_i^{-2}}\right) \\
&+&\left[  \frac{\varepsilon_i^{n-4}}{r_i^{\frac{n}{2}-3}}+\varepsilon_i^{\frac{n-2}{2}} e^{-c\frac{r_i}{\varepsilon_i}}\right] \left(\left\|\phi_i\right\|_{L^\infty(Q_1)}+e^{c\varepsilon_i^{-2}}\right).
\end{eqnarray*}
Scaling back to $\phi_i$, we get \eqref{Schauder estimate on error fct}.
\end{proof}

\section{Linearization of Pohozaev identity}\label{sec Pohozaev}
\setcounter{equation}{0}

 Given a smooth function $v$ and a sphere $\partial B_r$, define the Pohozaev invariant as
\[\mathcal{P}_v(r):= \int_{\partial B_r}\left[\frac{|\nabla v|^2 }{2}-\left(\partial_rv\right)^2-\frac{n-2}{2r}v\partial_rv\right].\]

Choose of a sequence of $r_i$ satisfying
\[ \lim_{i\to+\infty}\frac{r_i}{\varepsilon_i}=+\infty.\]
Multiplying \eqref{eqn} by $x\cdot\nabla u_i(x,0)$ and integrating in $B_{r_i}$, after integrating by parts we obtain a Pohozaev  identity
\begin{equation}\label{Pohozaev}
\mathcal{P}_{u_i(0)}(r_i)-\int_{\partial B_{r_i}}\frac{u_i(0)^{p+1}}{p+1} =-\frac{1}{r_i}\int_{B_{r_i}}\partial_tu_i(0)\left[x\cdot\nabla u_i(0)+\frac{n-2}{2}u_i(0)\right].
\end{equation}
Substitute the decomposition
\eqref{decomposition} into this identity, and take a expansion. Let us estimate each term in this expansion. We will see that the zeroth order terms cancel with each other, because these terms form exactly the Pohozaev invariant of the bubble. The next order term then gives us some information.

\subsection{The left hand side}
For the left hand side, we have the expansion
\begin{eqnarray*}
&&\mathcal{P}_{u_i(0)}(r_i)-\int_{\partial B_{r_i}}\frac{u_i(0)^{p+1}}{p+1} \\
 &=& \mathcal{P}_{W_i(0)}(r_i)-\int_{\partial B_{r_i}}\frac{W_i(0)^{p+1}}{p+1} +\mathcal{P}_{\phi_i(0)}(r_i)+\mathcal{P}_{a_i(0) Z_{0,i}(0)}(r_i)\\
&+&\underbrace{\int_{\partial B_{r_i}}\nabla W_i(0)\cdot\nabla\phi_i(0)-2\partial_rW_i(0)\partial_r\phi_i(0)-\frac{n-2}{2r_i}
\left[W_i(0)\partial_r\phi_i(0)+\phi_i(0)\partial_rW_i(0)\right]}_{\mbox{Main order term}}\\
&-&\frac{1}{p+1}\int_{\partial B_{r_i}}\left[u_i(0)^{p+1}- W_i(0)^{p+1}\right]+\mbox{cross terms involving } a_i(0)Z_{0,i}(0).
\end{eqnarray*}

Let us estimate each term in this expansion.
\begin{enumerate}
  \item Because $W_i(0)$ is a smooth solution of \eqref{stationary eqn}, it satisfies the standard Pohozaev identity
  \[\mathcal{P}_{W_i(0)}(r_i)-\int_{\partial B_{r_i}}\frac{W_i(0)^{p+1}}{p+1} =0.\]

  \item  By Proposition \ref{prop zeroth order on phi} and Proposition \ref{prop Lip estimate on error fct}, we get
  \[ \mathcal{P}_{\phi_i(0)}(r_i)=O\left(r_i^{n-3}\right)\left(\left\|\phi_i\right\|_{L^\infty(Q_1)}+e^{c\varepsilon_i^{-2}}\right)^2.\]
Here  the factor $\left(\left\|\phi_i\right\|_{L^\infty(Q_1)}+e^{c\varepsilon_i^{-2}}\right)$ is kept for later use.

  \item By estimates on $a_i$ and $\lambda_i$  in Proposition \ref{prop uniform estimate}, on $\partial B_{r_i}$ we have
\begin{equation}\label{estimate of Z0 on sphere}
 \left\{\begin{aligned}
 &|a_i(t)Z_{0,i}(t)|\lesssim e^{-cr_i /\varepsilon_i}\left(\left\|\phi_i\right\|_{L^\infty(Q_1)}+e^{c\varepsilon_i^{-2}}\right), \\
 & \left|\nabla\left(a_i(t)Z_{0,i}(t)\right)\right|\lesssim \varepsilon_i^{-1}e^{-cr_i/\varepsilon_i}\left(\left\|\phi_i\right\|_{L^\infty(Q_1)}+e^{c\varepsilon_i^{-2}}\right).
 \end{aligned}\right.
\end{equation}
By these estimates we get
  \[ \left|\mathcal{P}_{a_i(0)Z_{0,i}(0)}(r)\right|\lesssim \varepsilon_i^{-2}r_i^{n-1}e^{-cr_i/\varepsilon_i} \left(\left\|\phi_i\right\|_{L^\infty(Q_1)}+e^{c\varepsilon_i^{-2}}\right)^2.\]

\item First by integrating \eqref{gradient estimate on error fct}, for any $x\in\partial B_{r_i}$, we have
    \[\phi_i(x,0)=\frac{1}{|\partial B_\rho|}\int_{\partial B_{\rho}}\phi_i(0)+O\left(\frac{\varepsilon_i^2}{r_i^2}+\rho\right)\left(\left\|\phi_i\right\|_{L^\infty(Q_1)}+e^{c\varepsilon_i^{-2}}\right).\]
Next,   because $\xi_i(0)=0$, on   $\partial B_{r_i}$, by Proposition \ref{prop uniform estimate} and \eqref{bound on parameters 5},  we obtain
 \[
 \left\{\begin{aligned}
& W_i(0)\lesssim \varepsilon_i^{\frac{n-2}{2}}r_i^{2-n} ,\\
& \partial_rW_i(0)= -c(n)\varepsilon_i^{\frac{n-2}{2}}r_i^{1-n}+O\left(\varepsilon_i^{\frac{n+2}{2}}r_i^{-1-n}\right),  \\
& \nabla W_i(0)=\partial_rW_i(0)\frac{x}{|x|}, \\
&   |\nabla\phi_i(t)|\lesssim \left(\varepsilon_i^2 r_i^{-3}+1\right)\left(\left\|\phi_i\right\|_{L^\infty(Q_1)}+e^{c\varepsilon_i^{-2}}\right).
\end{aligned}\right.
\]
 Therefore the main order term equals
  \begin{eqnarray*}
   &&-c(n)\left[\frac{1}{|\partial B_\rho|}\int_{\partial B_{\rho}}\phi_i+O\left(\frac{\varepsilon_i^2}{r_i^2}+\rho\right)\left(\left\|\phi_i\right\|_{L^\infty(Q_1)}+e^{c\varepsilon_i^{-2}}\right)\right]
   \frac{\varepsilon_i^{\frac{n-2}{2}}}{r_i}
 \\
 & +&O\left(\varepsilon_i^{\frac{n-2}{2}}\right)\left(\left\|\phi_i\right\|_{L^\infty(Q_1)}+e^{c\varepsilon_i^{-2}}\right).
\end{eqnarray*}

   \item On $\partial B_{r_i}$, we have
 \[ W_i(0)\lesssim  \varepsilon_i^{\frac{n-2}{2}}r_i^{2-n}, \quad  |a_i Z_{0,i}(0)|\lesssim e^{-cr_i/\varepsilon_i} \left(\left\|\phi_i\right\|_{L^\infty(Q_1)}+e^{c\varepsilon_i^{-2}}\right).\]
 Therefore
  \[ \left|\int_{\partial B_{r_i}}\left[u_i(t)^{p+1}- W_i(t)^{p+1}\right]\right|
 \lesssim \left(\varepsilon_i^{\frac{n+2}{2}}r_i^{-3}+r_i^{n-1}\right)\left(\left\|\phi_i\right\|_{L^\infty(Q_1)}+e^{c\varepsilon_i^{-2}}\right). \]

\item Finally, by \eqref{estimate of Z0 on sphere} and estimates of $W_i(0)$ and $\phi_i(0)$, the cross terms involving $a_i Z_{0,i}(0)$ are of the order
\[ O\left(\varepsilon_i^{\frac{n-4}{2}}+ r_i^{n-3}+\varepsilon_i^{-1}r_i^{n-2}\right)e^{-cr_i/\varepsilon_i}
\left(\left\|\phi_i\right\|_{L^\infty(Q_1)}+e^{c\varepsilon_i^{-2}}\right).\]
\end{enumerate}

Putting these estimates together, we see that the left hand side of \eqref{Pohozaev} equals
\begin{eqnarray}\label{LHS of Pohozaev}
  &&-c(n)\left[\frac{1}{|\partial B_\rho|}\int_{\partial B_{\rho}}\phi_i+O\left(\frac{\varepsilon_i^2}{r_i^2}+\rho\right)\left(\left\|\phi_i\right\|_{L^\infty(Q_1)}+e^{c\varepsilon_i^{-2}}\right)\right]
   \frac{\varepsilon_i^{\frac{n-2}{2}}}{r_i} \nonumber
 \\
 & +&O\left(r_i^{n-3}+\varepsilon_i^{-2}r_i^{n-1}e^{-cr_i/\varepsilon_i} +\varepsilon_i^{\frac{n-2}{2}}\right)\left(\left\|\phi_i\right\|_{L^\infty(Q_1)}+e^{c\varepsilon_i^{-2}}\right)\\
 &+&\left(\varepsilon_i^{\frac{n+2}{2}}r_i^{-3}+r_i^{n-1}+\varepsilon_i^{\frac{n-4}{2}}e^{-cr_i/\varepsilon_i}\right)\left(\left\|\phi_i\right\|_{L^\infty(Q_1)}+e^{c\varepsilon_i^{-2}}\right). \nonumber
\end{eqnarray}

\subsection{The right hand side}\label{subsection RHS of Pohozaev}

By \eqref{decomposition}, we get
\begin{equation}\label{local decomposition of time derivative 1}
  \partial_tu_i=\left(a_{i}^\prime,\xi_{i}^\prime,\lambda_{i}^\prime\right)\cdot Z_{\ast,i}+\partial_t\phi_{i}+a_{i} \partial_tZ_{0,i}
\end{equation}
and
\begin{eqnarray}\label{local decomposition of scaling derivative 1}
  x\cdot\nabla u_i+\frac{n-2}{2}u_i&=&\lambda_{i}Z_{n+1,i}+\left(x\cdot\nabla\phi_{i}+\frac{n-2}{2}\phi_{i}\right)\\
  &+&a_{i}\left(x\cdot Z_{0,i}+\frac{n-2}{2}Z_{0,i}\right). \nonumber
\end{eqnarray}
Therefore for any $r$, we have
\begin{eqnarray*}
   && \int_{B_r}\partial_tu_i(x,0)\left[x\cdot\nabla u_i(x,0)+\frac{n-2}{2}u_i(x,0)\right] dx\\
 &=& \lambda_{i}(0)\left(a_{i}^\prime(0),\xi_{i}^\prime(0),\lambda_{i}^\prime(0)\right)\cdot\int_{B_r}  Z_{\ast,i}(x,0)Z_{n+1,i}(x,0)
dx\\
 &+&\left(a_{i}^\prime(0),\xi_{i}^\prime(0),\lambda_{i}^\prime(0)\right)\cdot\int_{B_r} Z_{
 \ast,i}(x,0)\left[x\cdot\nabla\phi_{i}(x,0)+\frac{n-2}{2}\phi_{i}(x,0)\right]dx\\
  &+&a_{i}(0)\left(a_{i}^\prime(0),\xi_{i}^\prime(0),\lambda_{i}^\prime(0)\right)\cdot\int_{B_r} Z_{\ast,i}(x,0)\left[x\cdot Z_{0,i}(x,0)+\frac{n-2}{2}Z_{0,i}(x,0)\right]dx\\
 &+&\lambda_{i}(0)\int_{B_r}Z_{n+1,i}(x,0)\partial_t\phi_{i}(x,0)dx\\
 &+&\int_{B_r}\partial_t\phi_i(x,0)
 \left[x\cdot\nabla\phi_{i}(x,0)+\frac{n-2}{2}\phi_{i}(x,0)\right]dx\\
 &+&a_{i}(0)\int_{B_r}\partial_t\phi_i(x,0) \left[x\cdot Z_{0,i}(x,0)+\frac{n-2}{2}Z_{0,i}(x,0)\right]dx\\
 &+& a_{i}(0) \lambda_{i}(0)\int_{B_r} Z_{n+1,i}(x,0)\partial_tZ_{0,i}(x,0)dx\\
 &+&a_{i}(0)\int_{B_r} \partial_tZ_{0,i}(x,0)\left[x\cdot\nabla\phi_{i}(x,0)+\frac{n-2}{2}\phi_{i}(x,0)\right]dx\\
 &+&a_{i}(0)^2 \int_{B_r}\partial_t Z_{0,i}(x,0)\left[x\cdot Z_{0,i}(x,0)+\frac{n-2}{2}Z_{0,i}(x,0)\right]dx\\
 &=:&\mathrm{I}+\mathrm{II}+\mathrm{III}+\mathrm{IV}+\mathrm{V}+\mathrm{VI}+\mathrm{VII}+\mathrm{VIII}+\mathrm{IX}.
\end{eqnarray*}
Let us take $r=r_i$ and estimate each term one by one.
\begin{enumerate}
  \item By \eqref{bound on parameters 1}, Proposition \ref{prop uniform estimate} and the orthogonal relation between $Z_0$ and $Z_{n+1}$,
  \[ \left| \lambda_i(0) a_i^\prime(0)\int_{B_{r_i}} Z_{0,i}(x,t)Z_{n+1,i}(x,0)dx\right|\lesssim \varepsilon_i^{\frac{n-2}{2}}e^{-cr_i/\varepsilon_i}\left(\left\|\phi_i\right\|_{L^\infty(Q_1)}+e^{c\varepsilon_i^{-2}}\right) .\]

 Similarly, for each $j=1,\cdots, n$,
  \[ \left| \lambda_i(0) \xi_{j,i}^\prime(0)\int_{B_{r_i}} Z_{j,i}(x,t)Z_{n+1,i}(x,0)dx\right|\lesssim \varepsilon_i^{\frac{3n}{2}-4}r_i^{3-n}\left(\left\|\phi_i\right\|_{L^\infty(Q_1)}+e^{c\varepsilon_i^{-2}}\right) .\]

  By Lemma \ref{lem estiamte of parameters}, \eqref{bound on parameters 1} and Proposition \ref{prop uniform estimate},
  \[  \left|\lambda_i(0) \lambda_i(0)^\prime\int_{B_{r_i}}Z_{i,n+1}(x,0)^2dx\right|\lesssim K^{-\frac{n-2}{2}}\varepsilon_i^{\frac{n-2}{2}}\left(\left\|\phi_i\right\|_{L^\infty(Q_1)}+e^{c\varepsilon_i^{-2}}\right) .\]

  Putting these three estimates together we obtain
\[|\mathrm{I}|\lesssim \left[ e^{-cr_i/\varepsilon_i}+\left(\frac{\varepsilon_i}{r_i}\right)^{n-3}+ K^{-\frac{n-2}{2}}\right]\varepsilon_i^{\frac{n-2}{2}}\left(\left\|\phi_i\right\|_{L^\infty(Q_1)}+e^{c\varepsilon_i^{-2}}\right) .\]

  \item By Proposition \ref{prop uniform estimate}, Proposition \ref{prop zeroth order on phi} and Lemma \ref{lem gradient estimate on error fct}, we get
  \[|\mathrm{II}|\lesssim   \varepsilon_i^{n-4}r_i^2\left(\left\|\phi_i\right\|_{L^\infty(Q_1)}+e^{c\varepsilon_i^{-2}}\right)^2.\]

 \item By \eqref{bound on parameters 1} and \eqref{estimate on D 2},
 \[ |\mathrm{III}|\lesssim   \varepsilon_i^{n-2}\left(\left\|\phi_i\right\|_{L^\infty(Q_1)}+e^{c\varepsilon_i^{-2}}\right) ^2. \]

 \item To estimate $\mathrm{IV}$,  first note that because $r_i\gg \varepsilon_i$, the orthogonal condition \eqref{orthogonal condition} for $\phi_i$ reads as
 \[ \int_{B_{r_i}}\phi_i(x,t)\eta_{in}(x,t)Z_{n+1,i}(x,t)dx=0.\]
 Differentiating this equation in $t$ gives
 \begin{eqnarray*}
  &&  \left|\int_{B_{r_i}}Z_{n+1,i}(x,0)\partial_t\phi_{i}(x,0) \eta_{in}(x,0)dx \right|\\
  &\leq & \left| \int_{B_{r_i}}\phi_i(x,0) \partial_t\eta_{in}(x,0) Z_{n+1,i}(x,0)dx\right|+\left|\int_{B_{r_i}} \phi_i(x,0) \eta_{in}(x,0) \partial_tZ_{n+1,i}(x,0) dx\right|\\
  &\lesssim & K^2\varepsilon_i^{n-3}\left(\left\|\phi_i\right\|_{L^\infty(Q_1)}+e^{c\varepsilon_i^{-2}}\right).
 \end{eqnarray*}
 On the other hand, a direct calculation using Proposition \ref{prop Schauder estimate on error fct} gives
 \begin{eqnarray*}
 &&\left|\int_{B_{r_i}}Z_{n+1,i}(x,0)\partial_t\phi_{i}(x,0) \left[1-\eta_{in}(x,0)\right]dx \right|  \\
 &\lesssim&   \left(K^{-2}+r_i\right)\varepsilon_i^{\frac{n-4}{2}}\left(\left\|\phi_i\right\|_{L^\infty(Q_1)}+e^{c\varepsilon_i^{-2}}\right).
 \end{eqnarray*}
Putting these two estimates together we get
\[ |\mathrm{IV}|
   \lesssim  \left[\left(K^{-2}+r_i\right)\varepsilon_i^{\frac{n-2}{2}}+K^2\varepsilon_i^{n-2}\right]\left(\left\|\phi_i\right\|_{L^\infty(Q_1)}+e^{c\varepsilon_i^{-2}}\right).
 \]

 \item By  Proposition \ref{prop zeroth order on phi}, Lemma \ref{lem gradient estimate on error fct}  and Proposition \ref{prop Schauder estimate on error fct},
\[ |\mathrm{V}|\lesssim \left(\varepsilon_i^2r_i^{n-4}+r_i^{n-1}\right)\left(\left\|\phi_i\right\|_{L^\infty(Q_1)}+e^{c\varepsilon_i^{-2}}\right) ^2.\]

 \item By Proposition \ref{prop uniform estimate} and Proposition \ref{prop Schauder estimate on error fct},
\[ |\mathrm{VI}|\lesssim \varepsilon_i^{n-2}\left(\left\|\phi_i\right\|_{L^\infty(Q_1)}+e^{c\varepsilon_i^{-2}}\right) ^2.\]

 \item By Proposition \ref{prop uniform estimate},
\[ |\mathrm{VII}|\lesssim \varepsilon_i^{n-2}\left(\left\|\phi_i\right\|_{L^\infty(Q_1)}+e^{c\varepsilon_i^{-2}}\right)^2.\]

 \item By Proposition \ref{prop uniform estimate}, Proposition \ref{prop zeroth order on phi} and Lemma \ref{lem gradient estimate on error fct},
\[ |\mathrm{VIII}|\lesssim \varepsilon_i^{\frac{3n}{2}-3}\left(\left\|\phi_i\right\|_{L^\infty(Q_1)}+e^{c\varepsilon_i^{-2}}\right)^2.\]

 \item By Proposition \ref{prop uniform estimate},
\[|\mathrm{IX}|\lesssim \varepsilon_i^{2n-3}\left(\left\|\phi_i\right\|_{L^\infty(Q_1)}+e^{c\varepsilon_i^{-2}}\right)^2.\]
\end{enumerate}

Putting these estimates together, we get
\begin{equation}\label{RHS of Pohozaev}
  \mbox{RHS}  \lesssim r_i^{-1}\left[\left(K^{-2}+r_i\right)\varepsilon_i^{\frac{n-2}{2}}+\varepsilon_i^2r_i^{n-4}+r_i^{n-1}\right]
  \left(\left\|\phi_i\right\|_{L^\infty(Q_1)}+e^{c\varepsilon_i^{-2}}\right).
\end{equation}

\section{A weak form of Schoen's Harnack inequality}\label{sec Schoen's Harnack inequality}
\setcounter{equation}{0}

In this section, we establish a weak form of Schoen's Harnack inequality, which then finishes the proof of Theorem \ref{thm no bubble towering}. In the Yamabe problem, this Harnack inequality was first introduced in Schoen \cite{Schoen-course}, see also Li \cite{LiYanYan-Harnack} and Li-Zhang \cite{Li-Zhang1} for a proof using the method of moving plane and moving sphere. The following proof instead is mainly a consequence of the Pohozaev identity calculation in Section \ref{sec Pohozaev}.

Combining \eqref{LHS of Pohozaev} and \eqref{RHS of Pohozaev} in the previous section, we get
\begin{eqnarray}\label{recursive relation}
  &&\frac{1}{|\partial B_\rho|}\int_{\partial B_{\rho}}\phi_i(0) \nonumber\\
   & \lesssim&\left(\rho+\frac{\varepsilon_i^2}{r_i^2}+ \frac{r_i^{n-2}}{\varepsilon_i^{(n-2)/2}}+\frac{\varepsilon_i^2}{r_i^2}+K^{-2}+r_i+\frac{r_i^{n-3}}{\varepsilon_i^{\frac{n}{2}-3}}+e^{-c\frac{r_i}{\varepsilon_i}}\right) \left(\left\|\phi_i\right\|_{L^\infty(Q_1)}+e^{c\varepsilon_i^{-2}}\right).\nonumber
\end{eqnarray}

By choosing $r_i=\varepsilon_i^{(n-1)/n}$,  we obtain a constant $\delta(n)>0$ such that
\begin{equation}\label{recursive relation 1}
 \frac{1}{|\partial B_\rho|}\int_{\partial B_{\rho}}\phi_i(0)\lesssim \left(\rho+K^{-2}+\varepsilon_i^{\delta(n)}\right)\left(\left\|\phi_i\right\|_{L^\infty(Q_1)}+e^{c\varepsilon_i^{-2}}\right).
\end{equation}

On $\partial B_\rho$, $u_i(0)\to u_\infty(0)$ uniformly. In view of the estimates \eqref{determination of the scaling parameter size}-\eqref{bound on parameters 2}, we deduce that $\phi_i(0)\to u_\infty(0)$ uniformly on $\partial B_\rho$, too. Passing to the limit in \eqref{recursive relation 1}, we obtain
\begin{equation}\label{limit condition from Pohozaev}
   \frac{1}{|\partial B_\rho|}\int_{\partial B_{\rho}}u_\infty(x,0)\lesssim \left(\rho+K^{-2}\right)\|u_\infty\|_{L^\infty(Q_1)}\lesssim \rho+K^{-2}.
\end{equation}
  Since $u_\infty$ is a smooth, nonnegative solution of \eqref{eqn}, if $u_\infty$ is not identically zero, by the standard Harnack ineqaulity,
\begin{equation}\label{Harnack for limit fct I}
  \inf_{Q_{1/4}}u_\infty>0.
\end{equation}
On the other hand, if we have chosen $K$ large enough at the beginning (in Proposition \ref{prop orthogonal decompostion}) and taken $\rho$ arbitrarily small,  \eqref{limit condition from Pohozaev} contradicts  \eqref{Harnack for limit fct I}. \footnote{We can also avoid the use of Harnack inequality, and use only the strong maximum principle. For this approach, we need to take a sequence of $K\to+\infty$ in Proposition \ref{prop orthogonal decompostion}. This changes the error function $\phi_i$, but it does not affect the argument. This is because both \eqref{limit condition from Pohozaev} and \eqref{Harnack for limit fct I} involve only $u_\infty$, which is the weak limit of $u_i$ and does not depend on the construction of $\phi_i$.}

 Hence in the setting of Section \ref{sec setting I}, we must have
$u_\infty=0$. This finishes the proof of Theorem \ref{thm no bubble towering}.

For the study of bubble clusterings (see Section \ref{sec bubble cluster} below), we need  a quantitative version of the above qualitative description. If the solution exists for a sufficiently large time, an iteration of the following proposition backwardly in time will lead to an quantitative upper bound on the error function, see Section \ref{sec bubble cluster} for details. In particular, if the solution exists globally in time (e.g. a solution independent of time), we can recover Schoen's Harnack inequality.
\begin{prop}\label{prop weak Schoen Harnack}
  There exist two universal constants $\varepsilon_0$ and $M_0$ so that the following holds. Suppose $u_\varepsilon$ is a positive solution of \eqref{eqn} in $Q_1$, satisfying {\bf(II.a')-(I.d')} and the decomposition \eqref{decomposition} with parameters $(a_\varepsilon(t),\xi_\varepsilon(t),\lambda_\varepsilon(t))$, where
  \[\lambda_\varepsilon(0)=\varepsilon, \quad \xi_\varepsilon(0)=0.\]
  If $\varepsilon\leq \varepsilon_0$ and
  \begin{equation}\label{larger than natural}
    \|\phi_\varepsilon\|_{L^\infty(Q_1)}\geq M_0\varepsilon^{\frac{n-2}{2}},
  \end{equation}
then
\[ \|\phi_\varepsilon\|_{L^\infty(B_1\times(-1/4,1/8))}\leq \frac{1}{2}\|\phi_\varepsilon\|_{L^\infty(Q_1)}.\]
\end{prop}
\begin{proof}
We prove this proposition by contradiction. Assume
\begin{itemize}
  \item $u_i$ is a sequence of solutions satisfying  {\bf (II.a)-(II.c)} and \eqref{Lip assumption};
  \item  the decomposition given in \eqref{decomposition} holds, where
\[ \varepsilon_i:=\lambda_i(0)\to 0, \quad \xi_i(0)=0;\]
  \item $\phi_i$ satisfies
  \begin{equation}\label{larger than natural 2}
  \lim_{i\to+\infty} \frac{\|\phi_i\|_{L^\infty(Q_1)}}{\varepsilon_i^{\frac{n-2}{2}}}=+\infty,
  \end{equation}
     but
     \begin{equation}\label{larger than natural not true}
   \|\phi_i\|_{L^\infty(B_1\times(-1/4,1/8))}>\frac{1}{2}\|\phi_i\|_{L^\infty(Q_1)}.
  \end{equation}
\end{itemize}
We show that this leads to a contradiction.

{\bf Step 1.} Denote $\delta_i:=\|\phi_i\|_{L^\infty(Q_1)}$. Let $\widetilde{\phi}_i:=\phi_i/\delta_i$. In this step we prove
\begin{equation}\label{quanlitative for decay}
  \widetilde{\phi}_i\to 0 \quad \mbox{uniformly in any compact set of} ~~ Q_1.
\end{equation}

First by Proposition \ref{prop Lip estimate on error fct}, for any $\rho>0$, $\widetilde{\phi}_i$ are uniformly Lipschitz in $(B_{2/3}\setminus B_\rho)\times (-4/9,4/9)$. Hence after passing to a subsequence, we may assume $\widetilde{\phi}_i$ converges to a limit $\widetilde{\phi}$, uniformly in any compact set of $(B_{2/3}\setminus\{0\})\times(-4/9,4/9)$.

Because
\[ u_i=\phi_i+O\left(\varepsilon_i^{\frac{n-2}{2}}|x|^{2-n}\right),\]
in view of \eqref{larger than natural 2}, $u_i/\delta_i$ also converges uniformly to $\widetilde{\phi}$ in any compact set of $(B_{2/3}\setminus\{0\})\times(-4/9,4/9)$. As a consequence,
\begin{equation}\label{limit nonnegative}
  \widetilde{\phi}\geq 0 \quad \mbox{in } \left(B_{2/3}\setminus\{0\}\right)\times(-4/9,4/9).
\end{equation}

Dividing \eqref{eqn} by $\delta_i$, and then letting $i\to +\infty$, by the above uniform convergence of $u_i/\delta_i$, we get
\[\partial_t\widetilde{\phi}-\Delta\widetilde{\phi}=0 \quad \mbox{in } \left(B_{2/3}\setminus\{0\}\right)\times(-4/9,4/9).\]
Since $\left|\widetilde{\phi}\right|\leq 1$ in $Q_1$, removable singularity theorem for heat equation implies that $\widetilde{\phi}$ is smooth and satisfies the heat equation in $Q_{2/3}$.

By \eqref{limit nonnegative} and the strong maximum principle, either $\widetilde{\phi}\equiv 0$ or $\widetilde{\phi}>0$ everywhere in $Q_{2/3}$.
We claim that the first case must happen.

Indeed, dividing both sides of \eqref{recursive relation} by $\delta_i$ and letting $i\to +\infty$, by \eqref{larger than natural 2} and the above uniform convergence of $\widetilde{\phi}_i$, we obtain
\[\frac{1}{|\partial B_\rho|}\int_{\partial B_{\rho}}\widetilde{\phi}(x,0)\lesssim  \rho+K^{-2}.\]
Sending $\rho\to0$ and $K\to+\infty$ as before, we get $\widetilde{\phi}(0,0)=0$. By the strong maximum principle (or Harnack inequality), $\widetilde{\phi}\equiv 0$ in $Q_{2/3}$.

{\bf Step 2.} Choose a small radius $\rho$. By results obtained in Step 1, for all $i$ large,
\begin{equation}\label{smallness improved outside}
  |\phi_i|\ll \delta_i \quad \mbox{on} ~~ \partial^pQ_{2/3}\setminus \left(B_\rho\times\{-2/3\}\right).
\end{equation}
On $B_\rho\times\{-2/3\}$, we have the trivial bound
\begin{equation}\label{}
  |\phi_i|\leq \delta_i.
\end{equation}

As in Subsection \ref{subsec uniform bound}, we obtain\\
{\bf Claim.} If $\rho$ is sufficiently small, there exists a constant $\sigma(\rho)\ll 1$ (independent of $\varepsilon_i$) such that
\begin{equation}\label{}
   |\phi_{i,1}|\leq \sigma(\rho)\delta_i  \quad \mbox{in } Q_{7/12}.
\end{equation}
Here $\phi_{i,1}$ denotes the decomposition of the outer component taken in Section \ref{sec outer}.

By this claim, repeating the estimates of  other terms in the decomposition of the outer component, $\phi_{i,2}$-$\phi_{i,5}$, and the corresponding estimate for the inner component $\phi_{in,i}$ in Section \ref{sec inner and outer 2}, we obtain
\begin{equation}\label{improvement in Q half}
  |\phi_i|\lesssim \sigma(\rho)\delta_i  \quad \mbox{in } Q_{1/2}.
\end{equation}

{\bf Step 3.} By \eqref{improvement in Q half}, we get
\[ u_i\lesssim \varepsilon_i^{\frac{n-2}{2}}+\sigma(\rho)\delta_i \quad \mbox{in } \left(B_{1/2}\setminus B_{1/4}\right)\times(-1/4,1/4).\]
Because $u_i$ satisfies the standard parabolic Harnack inequality in $\left(B_1\setminus B_{1/8}\right)\times (-1,1)$, we get
\[ u_i\lesssim \varepsilon_i^{\frac{n-2}{2}}+\sigma(\rho)\delta_i \quad \mbox{in } \left(B_1\setminus B_{1/2}\right)\times(-1/4,1/8).\]
By the definition of $\phi_i$, we get
\[-\varepsilon_i^{\frac{n-2}{2}}  \lesssim \phi_i  \lesssim \varepsilon_i^{\frac{n-2}{2}}+\sigma(\rho)\delta_i \quad \mbox{in } \left(B_1\setminus B_{1/2}\right)\times(-1/4,1/8).\]
Combining these two estimates with \eqref{improvement in Q half}, for all $i$ sufficiently large, we have
\[ \left|\widetilde{\phi}_i\right|\leq 1/4  \quad \mbox{in }  B_1\times(-1/4,1/8).\]
This is a contradiction with \eqref{larger than natural not true}.
\end{proof}

\section{A conditional exclusion of bubble clustering}\label{sec bubble cluster}
\setcounter{equation}{0}

In this section, we work in the following settings. (This will appear in a suitable rescalings of bubble clustering from Theorem \ref{thm bubble clustering convergence} in Part \ref{part many bubbles}.)
\begin{enumerate}
  \item There exist two sequences $R_i$ and $T_i$, both diverging to $+\infty$ as $i\to+\infty$;

  \item $u_i$ is a sequence of positive, smooth solution of \eqref{eqn} in $\Omega_i:=B_{R_i}\times(-T_i,T_i)$;

   \item  as $i\to +\infty$,
  \begin{equation}\label{time derivative to 0 II}
    \int_{-T_i}^{T_i}\int_{B_{R_i}}\partial_tu_i(x,t)^2dxdt\to0;
  \end{equation}

   \item there exists an $N\in\mathbb{N}$ such that for each  $i$ and $t\in(-T_i,T_i)$, in any compact set of $\R^n$, we have the bubble decomposition
  \begin{equation}\label{bubble decomposition at t=0}
    u_i(x,t)=\sum_{j=1}^{N}W_{\xi_{ij}^\ast(t),\lambda_{ij}^\ast(t)}(x)+o_i(1),
  \end{equation}
  where $o_i(1)$ are measured in $H^1_{loc}(\R^n)$;

  \item as $i\to+\infty$,
  \begin{equation}\label{uniform bound on bubble height}
  \max_{j=1,\cdots, N}\sup_{-T_i<t<T_i}\lambda_{ij}^\ast(t)\to 0;
  \end{equation}

  \item for any $t\in(-T_i,T_i)$,
  \begin{equation}\label{assumption on bubble location}
 \min_{1\leq j \neq k \leq N}|\xi_{ij}^\ast(t)-\xi_{ik}^\ast(t)|\geq 2,
  \end{equation}
and at $t=0$,
  \begin{equation}\label{normalization on bubble distance}
    \xi_{i1}^\ast(0)=0, \quad |\xi_{i2}^\ast(0)|=1;
  \end{equation}

\item after relabelling indices,  assume for some $N^\prime\leq N$ and any $j=1,\cdots, N^\prime$,
\[\xi_{ij}^\ast(0)\to P_j, \quad \mbox{as}~~ i\to+\infty,\]
while for any $j=N^\prime+1,\cdots, N$,
\[|\xi_{ij}^\ast(0)|\to +\infty, \quad \mbox{as}~~ i\to+\infty,\]
\end{enumerate}
By \eqref{normalization on bubble distance}, $P_1=0$ and $P_2\in\partial B_1$, so $N^\prime\geq 2$.

The main result of  this section is
\begin{prop}\label{prop exclusion of bubble clusetering}
  Under the above assumptions, we have
  \begin{equation}\label{exclusion of bubble clustering}
   T_i\leq 100 \max_{1\leq j \leq N^\prime}|\log\lambda_{ij}^\ast(0)|.
  \end{equation}
\end{prop}

We prove this proposition by contradiction, so assume \eqref{exclusion of bubble clustering} does not hold, that is,
\begin{equation}\label{exclusion of bubble clustering 2}
   T_i>100 \max_{1\leq j \leq N^\prime}|\log\lambda_{ij}^\ast(0)|.
\end{equation}
With this bound in hand,  we can iterate Proposition \ref{prop weak Schoen Harnack} backwardly in time, leading to an optimal upper bound on the error function as Schoen's Harnack inequality in Yamabe problem. This bound allows us to define a Green function from $u_i$. This will be done in Subsection \ref{subsec construction of Green fct}. Then in Subsection \ref{subsec local Pohozaev}, still similar to the treatment in Yamabe problem, we employ Pohozaev identity again to give a sign restriction on the next order term in the expansion of this Green function at a pole. This then leads to a contradiction with the assumption that there is one bubble located at $P_1$ and $P_2$ respectively.

\subsection{Construction of Green function}\label{subsec construction of Green fct}

Under the above assumptions, for any $i,j$ and $t\in(-T_i+1,T_i-1)$, Proposition \ref{prop orthogonal decompostion} can be applied to $u_i$ in $Q_1(\xi_{ij}^\ast(t),t)$. We denote the corresponding parameters by $\xi_{ij}(t)$, $\lambda_{ij}(t)$ and $a_{ij}(t)$, and the error function
\begin{equation}\label{local decomposition}
  \phi_{ij}(x,t):= u_i(x,t)- W_{\xi_{ij}(t),\lambda_{ij}(t)}(x)-a_{ij}(t) Z_{0,\xi_{ij}(t),\lambda_{ij}(t)}(x).
\end{equation}

Denote $\varepsilon_{ij}:=\lambda_{ij}(0)$, $j=1,\cdots, N^\prime$ and $\varepsilon_i=\max_{1\leq j\leq N^\prime}\varepsilon_{ij}$. By Proposition \ref{prop orthogonal decompostion} and \eqref{exclusion of bubble clustering 2}, we get
\begin{equation}\label{exclusion of bubble clustering 3}
 T_i>50 \max_{1\leq j \leq N^\prime}|\log\varepsilon_{ij}|.
\end{equation}

\begin{lem}
  For each $j=1,\cdots, N^\prime$ and any $t\in(-50|\log\varepsilon_{ij}|, 50|\log\varepsilon_{ij}|)$,
  \begin{equation}\label{Harnack II}
   \frac{1}{2}\varepsilon_{ij}\leq  \lambda_{ij}(t)\leq 2\varepsilon_{ij}.
  \end{equation}
\end{lem}
\begin{proof}
Applying Proposition \ref{prop uniform estimate}  to each $\phi_{ij}$ in $Q_1(\xi_{ij}(t),t)$, we obtain
 \begin{equation}\label{differential eqn for lambda}
   \left|\lambda_{ij}^\prime(t)\right|\lesssim \lambda_{ij}(t)^{\frac{n-4}{2}}.
 \end{equation}
 Define
\[ \widetilde{T}_{ij}:=\sup\left\{T: ~~ T\leq 50|\log\varepsilon_{ij}|, ~~ \lambda_{ij}(t)\leq 2\varepsilon_{ij} \quad \mbox{in } ~[0,T]\right\}.\]
Integrating \eqref{differential eqn for lambda} on $[0,\widetilde{T}_{ij}]$  leads to
\begin{equation}\label{Harnack for lambda II}
 \sup_{t\in[0,\widetilde{T}_{ij}]}\lambda_{ij}(t)\leq \varepsilon_{ij} e^{C\varepsilon_{ij}^{(n-6)/2} |\log\varepsilon_{ij}|}\leq \varepsilon_{ij} e^{C\varepsilon_{ij}^{(n-6)/4} }<\frac{3}{2}\varepsilon_{ij}.
\end{equation}
Hence we must have $\widetilde{T}_{ij}=50|\log\varepsilon_{ij}|$.  This gives the upper bound of \eqref{Harnack II} in the positive side. The lower bound and the negative side follow in the same way.
\end{proof}

For each $t\in(-50|\log\varepsilon_{ij}|, 50|\log\varepsilon_{ij}|)$, applying Proposition \ref{prop weak Schoen Harnack} to $Q_1(\xi_{ij}(t),t)$, we obtain
\begin{equation}\label{iteration of Harnack decay}
 \sup_{Q_1(\xi_{ij}(t),t)}|\phi_{ij}|\leq M_0\varepsilon_{i}^{\frac{n-2}{2}}+\frac{1}{2}\sup_{Q_1(\xi_{ij}(t+1),t+1)}|\phi_{ij}|.
\end{equation}
For any $R>1$ fixed, an iteration of this estimate from $t=50|\log\varepsilon_{ij}|$ to any $t\in[-R,R]$ leads to
\begin{equation}\label{Schoen's Harnack 1}
  \sup_{B_1(\xi_{ij}(0))\times(-R,R)}|\phi_{ij}|\lesssim \varepsilon_{i}^{\frac{n-2}{2}}.
\end{equation}
This bound depends only on the constant $M_0$ in Proposition \ref{prop weak Schoen Harnack}, and it is independent of $R$.

Substituting \eqref{Schoen's Harnack 1} and  estimates of parameters in Proposition \ref{prop uniform estimate} into \eqref{local decomposition}, we get
\[ \sup_{B_1(\xi_{ij}(0))\times(-R,R)}u_i\lesssim \varepsilon_{i}^{\frac{n-2}{2}}.\]
 In $\left(B_{7/4}(\xi_{ij}(0))\setminus B_{3/4}(\xi_{ij}(0))\right)\times(-R,R)$, by applying standard Harnack inequality to $u_i$,  \eqref{Schoen's Harnack 1} is extended to
\begin{equation}\label{Schoen's Harnack 2}
  \sup_{B_{3/2}(\xi_{ij}(0))\times(-R,R)}|\phi_{ij}|\lesssim \varepsilon_{i}^{\frac{n-2}{2}}.
\end{equation}
This estimate gives us the expansion
\begin{equation}\label{expansion near isotaled simple pt}
 u_i(x,t)= \varepsilon_{ij}^{-\frac{n-2}{2}}W\left(\frac{x-\xi_{ij}(t)}{\varepsilon_{ij}}\right)+O\left(\varepsilon_{i}^{\frac{n-2}{2}}\right) \quad \mbox{in} ~~   B_{3/2}(\xi_{ij}(0))\times(-R,R).
\end{equation}
Here we note that, by Proposition \ref{prop uniform estimate},
\[ |\xi_{ij}^\prime(t)|\lesssim \lambda_{ij}(t)^{\frac{n-4}{2}}\lesssim \varepsilon_i^{\frac{n-4}{2}}.\]
Hence
\begin{equation}\label{variation of bubble center}
 |\xi_{ij}(t)-\xi_{ij}(0)|\lesssim_R  \varepsilon_i^{\frac{n-4}{2}}, \quad \mbox{in} ~~ [-R,R].
\end{equation}

Assume
\[\lim_{i\to+\infty}\frac{\varepsilon_{ij}}{\varepsilon_i}=m_j\in[0,1] \quad \mbox{for any }~~ j=1,\cdots, N^\prime.\]
By the definition of $\varepsilon_i$, there exists at least one $j$ satisfying $m_j=1$.

It is directly verified that for some dimensional constant $c(n)>0$,
\begin{equation}\label{construction of Diracs}
  \varepsilon_{i}^{-\frac{n-2}{2}}u_i^pdxdt \rightharpoonup c(n)m_j^{\frac{n-2}{2}} \delta_{P_j}\otimes dt
\end{equation}
weakly as Radon measures in $B_{3/2}(P_j)\times(-R,R)$.

Set
\[\widehat{u}_{i}:=\varepsilon_{i}^{-\frac{n-2}{2}}u_i.\]
It satisfies
\begin{equation}\label{finding of the heat eqn}
  \partial_t\widehat{u}_{i}-\Delta\widehat{u}_{i}=  \varepsilon_{i}^{-\frac{n-2}{2}}u_i^p=\varepsilon_{i}^2\widehat{u}_i^p.
\end{equation}
By \eqref{Schoen's Harnack 2}, in any compact set of $\left(B_{3/2}(P_j)\setminus\{P_j\}\right)\times\R$, $\widehat{u}_{i}$ are uniformly bounded.
Because $u_i$ satisfies the standard parabolic Harnack inequality in any compact set of $\left(\R^n\setminus\cup_{j=1}^{N^\prime}\{P_j\}\right)\times\R$, we deduce that $\widehat{u}_{i}$ are also uniformly bounded in any compact set of $\left(\R^n\setminus\cup_{j=1}^{N^\prime}\{P_j\}\right)\times\R$. Then by standard parabolic regularity theory, $\widehat{u}_{i}$ converges to $\widehat{u}_\infty$ smoothly in any compact set of $\left(\R^n\setminus\cup_{j=1}^{N^\prime}\{P_j\}\right)\times\R$.

By \eqref{construction of Diracs},  $\widehat{u}_\infty$ satisfies
\begin{equation}\label{limiting heat eqn}
  \partial_t\widehat{u}_\infty-\Delta\widehat{u}_\infty=c(n)\sum_{j=1}^{N^\prime}m_j^{\frac{n-2}{2}}\delta_{P_j}\otimes dt  \quad \mbox{in} ~~  \R^n \times\R.
\end{equation}

\begin{lem}\label{lem comparison of heights}
  For each $j=1\cdots, N^\prime$,
  \[m_j>0.\]
\end{lem}
\begin{proof}
Because there is one $m_j$ satisfying $m_j=1$, $\widehat{u}_\infty\neq 0$. By the strong maximum principle,
\begin{equation}\label{nontrivial caloric fct}
  \widehat{u}_\infty>0 \quad    \mbox{in} ~~ \left(\R^n\setminus\cup_{j=1}^{N^\prime}\{P_j\}\right)\times\R.
\end{equation}
By \eqref{expansion near isotaled simple pt}, for each $j=1,\cdots, N^\prime$,
\[\widehat{u}_i\lesssim\left(\frac{\varepsilon_{ij}}{\varepsilon_i}\right)^{\frac{n-2}{2}} \quad \mbox{in} ~~  \left( B_{3/2}(P_j)\setminus B_{1/2}(P_j)\right)\times(-1,1).\]
If $m_j=0$, we would have
\[ \widehat{u}_\infty=\lim_{i\to+\infty}\widehat{u}_i=0  \quad \mbox{in} ~~   \left( B_{3/2}(P_j)\setminus B_{1/2}(P_j)\right)\times(-1,1).\]
This is a contradiction with \eqref{nontrivial caloric fct}.
\end{proof}

\begin{lem}
For any  $(x,t)\in \left(\R^n\setminus\cup_{j=1}^{N^\prime}\{P_j\}\right)\times\R$,
 \begin{equation}\label{sums of Green functions}
  \widehat{u}_\infty(x,t)\equiv c(n)\sum_{j=1}^{N^\prime}m_j^{\frac{n-2}{2}}|x-P_j|^{2-n}.
 \end{equation}
\end{lem}
\begin{proof}
Take two arbitrary $s<t$. Take a sequence of cut-off functions $\zeta_R\in C_0^\infty(\R^n)$ such that
\[
   \left\{\begin{aligned}
&0\leq \zeta_R\leq 1,\\
&\zeta_R\equiv 1 \quad \mbox{in} ~~ B_R\setminus \cup_{j=1}^{N^\prime} B_{1/R}(P_j),\\
& \zeta_R\equiv 0 \quad \mbox{in} ~~ B_{2R}^c \cup \cup_{j=1}^{N^\prime} B_{1/2R}(P_j).
\end{aligned}\right.
\]
By the comparison principle, for any $x\in\R^n\setminus\{P_1,\cdots, P_{N^\prime}\}$,
\begin{eqnarray*}
   \widehat{u}_\infty(x,t) &\geq &  \int_{\R^n} \widehat{u}_\infty(y,s)\zeta_R(y)G(x-y,t-s) dy \\
  &+&c(n)\sum_{j=1}^{N^\prime}m_j^{\frac{n-2}{2}}\int_{s}^{t}G(x-P_j,t-s)ds.
\end{eqnarray*}
Here $G$ denotes the standard heat kernel on $\R^n$.

Letting $R\to+\infty$, 
by the monotone convergence theorem we obtain
\[  \widehat{u}_\infty(x,t)\geq \int_{\R^n} \widehat{u}_\infty(y,s)G(x-y,t-s) dy+c(n)\sum_{j=1}^{N^\prime}m_j^{\frac{n-2}{2}}\int_{s}^{t}G(x-P_j,t-s)ds.\]
Letting $s\to-\infty$, we get
\[  \widehat{u}_\infty(x,t)\geq c(n)\sum_{j=1}^{N^\prime}m_j^{\frac{n-2}{2}}|x-P_j|^{2-n}.\]
Thus their difference
\[\widehat{v}_\infty(x,t):=\widehat{u}_\infty(x,t)-c(n)\sum_{j=1}^{N^\prime}m_j^{\frac{n-2}{2}}|x-P_j|^{2-n}\]
is a positive caloric function on $\R^n\times\R$. By \cite{Lin2019ancient} or \cite{Widder1963heat}, there exists a nonnegative Radon measure $\nu$ on $\{(\xi,\lambda):\lambda=|\xi|^2\}\subset \R^n\times\R$ such that
\[\widehat{v}_\infty(x,t)=\int_{\{\lambda=|\xi|^2\}}e^{\lambda t+\xi\cdot x}d\nu(\xi,\lambda).\]
Because \eqref{expansion near isotaled simple pt} holds for any $R$, $\widehat{u}_\infty$ is bounded in $\left(B_1(P_1)\setminus B_{1/2}(P_1)\right)\times\R$, so $\widehat{v}_\infty$ is also bounded $\left(B_1(P_1)\setminus B_{1/2}(P_1)\right)\times\R$. This is possible only if $\nu=0$, and \eqref{sums of Green functions} follows.
\end{proof}

By this lemam, near $P_1=0$, there exists a smooth harmonic function $h$ such that
\[\widehat{u}_\infty(x)=c(n)m_1|x|^{2-n}+h(x).\]
Because  $P_2\in\partial B_1$,   we deduce that
\begin{equation}\label{positive harmonic}
  h(0)>0.
\end{equation}

\subsection{Local Pohozaev invariants}\label{subsec local Pohozaev}

By the expansion of $\widehat{u}_\infty(x)$ near $0$, in particular, \eqref{positive harmonic}, and a direct calculation, we find a dimensional constant $c(n)>0$ such that for all $r$ small,
\begin{equation}\label{Pohozaev asymptotics}
  \mathcal{P}_{\widehat{u}_\infty}(r)=\frac{c(n)m_1 h(0)}{r}+O(1).
\end{equation}

On the other hand, because $\widehat{u}_\infty$ comes from $u_i$, we claim that
\begin{lem}\label{lem limiting Pohozaev}
For all $r$ small,
 \begin{equation}\label{limiting Pohozaev}
 \mathcal{P}_{\widehat{u}_\infty}(r)\lesssim K^{-2}+r.
 \end{equation}
\end{lem}
\begin{proof}
In \eqref{Pohozaev}, choose $r_i$ to be a fixed, small  $r>0$.  Multiply both sides of \eqref{Pohozaev} by $\varepsilon_i^{2-n}$. Let us analyse the convergence of both sides.

{\bf Step 1. The left hand side of \eqref{Pohozaev}.}

 By the  convergence of $\widehat{u}_i$ in $C^\infty_{loc}\left(\left(\R^n\setminus\cup_{j=1}^{N^\prime}\{P_j\}\right)\times\R\right)$, we get
\begin{equation}\label{limiting Pohozaev 1}
\varepsilon_i^{2-n}\mathcal{P}_{u_i(0)}(r)=\mathcal{P}_{\widehat{u}_i(0)}(r)\to \mathcal{P}_{\widehat{u}_\infty(0)}(r), \quad \mbox{as } ~~ i\to+\infty
\end{equation}
and
\begin{equation}\label{limiting Pohozaev 2}
\varepsilon_i^{2-n}\int_{\partial B_r}u_i^{p+1}\lesssim \varepsilon_i^2\to0, \quad \mbox{as } ~~ i\to+\infty.
\end{equation}
Hence the left hand side converges to $\mathcal{P}_{\widehat{u}_\infty}(r)$.

{\bf Step 2. The right hand side of \eqref{Pohozaev}.}

For simplicity of notations, denote the parameters $(a_{i1}, \xi_{i1}, \lambda_{i1})$ just by $(a_i,\xi_i,\lambda_i)$, and $\phi_{i1}$, the error function in $Q_1$, just by $\phi_i$. Plugging the $L^\infty$ estimate  \eqref{Schoen's Harnack 1} into Proposition \ref{prop uniform estimate}, we get the following estimates:
\begin{equation}\label{estimate on parameters from Harnack 1}
 \sup_{t\in(-1,1)} |a_{i}(t)|\lesssim \varepsilon_i^{n-1},
\end{equation}
\begin{equation}\label{estimate on parameters from Harnack 2}
 \sup_{t\in(-1,1)}\left(\left|a_{i}^\prime(t)-\mu_0\frac{a_{i}(t)}{\lambda_{i}(t)^2}\right|+\left|\xi_{i}^\prime(t)\right|+
 \left|\lambda_{i}^\prime(t)\right|\right)\lesssim K^{-\frac{n-2}{2}}\varepsilon_i^{n-3},
\end{equation}
\begin{equation}\label{estimate on error fct from Harnack 1}
  |\nabla\phi_{i}(x,t)|\lesssim \frac{\varepsilon_i^{(n+2)/2}}{\varepsilon_i^3+|x|^3}+\varepsilon_i^{\frac{n-2}{2}} \quad \mbox{for any} ~~ (x,t)\in Q_1,
\end{equation}
and
\begin{equation}\label{estimate on error fct from Harnack 2}
  |\partial_t\phi_{i}(x,t)|\lesssim \frac{\varepsilon_i^{(n+2)/2}}{\varepsilon_i^4+|x|^4}+\frac{\varepsilon_i^{(n-2)/2}}{\varepsilon_i+|x|} \quad \mbox{for any} ~~ (x,t)\in Q_1.
\end{equation}
The factor $K^{-(n-2)/2}$ in \eqref{estimate on parameters from Harnack 2} comes from Lemma \ref{lem estiamte of parameters}.

As in Subsection \ref{subsection RHS of Pohozaev}, we expand the right hand side of \eqref{Pohozaev} into nine terms, $\mathrm{I}$-$\mathrm{IX}$.
Let us estimate them one by one.
\begin{enumerate}
  \item By \eqref{estimate on parameters from Harnack 2} and the orthogonal relation between $Z_0$ and $Z_{n+1}$,
  \[ \left| \lambda_i(0) a_i^\prime(0)\int_{B_r} Z_{0,i}(x,0)Z_{n+1,i}(x,0)dx\right|\lesssim \varepsilon_i^{n-2}e^{-cr/\varepsilon_i}.\]
Similarly, for each $j=1,\cdots, n$,
  \[ \left| \lambda_i(0) \xi_{i}^\prime(0)\int_{B_r} Z_{j,i}(x,0)Z_{n+1,i}(x,0)dx\right|\lesssim \varepsilon_i^{2n-5}r^{3-n},\]
and
  \[  \left|\lambda_i(0) \lambda_i^\prime(0)\int_{B_r}Z_{n+1,i}(0)^2\right|\lesssim K^{-\frac{n-2}{2}}\varepsilon_i^{n-2}.\]
Combining these three estimates, we obtain
\[|\mathrm{I}|\lesssim \left[K^{-\frac{n-2}{2}}+\left(\frac{\varepsilon_i}{r}\right)^{n-3}+e^{-cr/\varepsilon_i}\right]\varepsilon_i^{n-2}.\]

  \item By \eqref{Schoen's Harnack 1}, \eqref{estimate on parameters from Harnack 2} and \eqref{estimate on error fct from Harnack 1},
  \[ |\mathrm{II}| \lesssim \varepsilon_i^{2n-6}r^2.\]

 \item By \eqref{estimate on parameters from Harnack 1}  and \eqref{estimate on parameters from Harnack 2},
 \[ |\mathrm{III}| \lesssim \varepsilon_i^{2n-4}.\]

 \item As in the treatment of $\mathrm{IV}$ in Subsection \ref{subsection RHS of Pohozaev}, we get
 \[ |\mathrm{IV}|\lesssim K^2\varepsilon_i^{\frac{3n}{2}-3}+\left(K^{-2}+r\right)\varepsilon_i^{n-2}.\]

 \item By \eqref{Schoen's Harnack 1}, \eqref{estimate on parameters from Harnack 1} and \eqref{estimate on parameters from Harnack 2},
\[ |\mathrm{V}|\lesssim \varepsilon_i^nr^{n-4}+\varepsilon_i^{n-2}r^{n-1}.\]

 \item By \eqref{estimate on parameters from Harnack 1}  and \eqref{estimate on error fct from Harnack 2},
\[ |\mathrm{VI}|\lesssim \varepsilon_i^{2n-4}.\]

 \item By \eqref{estimate on parameters from Harnack 1}  and \eqref{estimate on parameters from Harnack 2},
\[|\mathrm{VII}|\lesssim \varepsilon_i^{2n-4}.\]

 \item By \eqref{estimate on parameters from Harnack 1}, \eqref{estimate on parameters from Harnack 2} and \eqref{estimate on error fct from Harnack 1},
\[ |\mathrm{VIII}|\lesssim \varepsilon_i^{3n-6}.\]

 \item By \eqref{estimate on parameters from Harnack 1} and \eqref{estimate on parameters from Harnack 2}
\[ |\mathrm{IX}|\lesssim \varepsilon_i^{3n-6}.\]
\end{enumerate}
Putting these estimates together and using \eqref{limiting Pohozaev 1}, we obtain \eqref{limiting Pohozaev}.
\end{proof}
 Combining this lemma with \eqref{Pohozaev asymptotics}, after letting $r\to0$, we deduce that $h(0)=0$. This  is a contradiction with \eqref{positive harmonic}, so \eqref{exclusion of bubble clustering 2} cannot be true. The proof of Proposition \ref{prop exclusion of bubble clusetering} is thus complete.

\begin{appendices}

\section{Linearization estimates around bubbles}\label{sec bubbles}
\setcounter{equation}{0}

In this appendix, we collect some estimates about the linearized equation around the standard bubble. This is mainly used to study the inner problem in Section \ref{sec inner}.

Recall that the standard positive bubble is
\[
  W(x):= \left(1+\frac{|x|^2}{n(n-2)}\right)^{-\frac{n-2}{2}}.
\]
Concerning eigenvalues and eigenfunctions for the linearized operator $-\Delta-pW^{p-1}$, we have (see for example \cite[Proposition 2.2]{Collot2017ground})
\begin{thm}\label{thm spectrum}
  \begin{itemize}
    \item[(i)] There exists one and only one negative eigenvalue for $-\Delta-pW^{p-1}$, denoted by $-\mu_0$, for which there exists a unique (up to a constant), positive, radially symmetric and exponentially decaying eigenfunction $Z_0$.
    \item[(ii)] There exist exactly $(n+1)$-eigenfunctions $Z_i$ in $L^\infty(\R^n)$ corresponding to eigenvalue $0$, given by
    \[
    \left\{\begin{aligned}
&Z_i=\frac{\partial W}{\partial x_i}, \quad i=1,\cdots, n,\\
&     Z_{n+1}=\frac{n-2}{2}W+x\cdot\nabla W.
\end{aligned}\right.
 \]
    \end{itemize}
\end{thm}
\begin{rmk}
  The following decay as $|x|\to\infty$ holds for these eigenfunctions:
  \[
  \left\{\begin{aligned}
&Z_0(x)\lesssim e^{-c|x|},\\
&|Z_i(x)|\lesssim |x|^{1-n}, \quad i=1,\cdots, n,\\
&     |Z_{n+1}(x)|\lesssim |x|^{2-n}.
\end{aligned}\right.
\]
\end{rmk}
Throughout the paper $Z_0$ is normalized so that
\[\int_{\R^n}Z_0^2=1.\]
For any $\xi\in\R^n$ and $\lambda\in\R^+$, in accordance with the scalings for $W$, define
\[Z_{i,\xi,\lambda}(x):=\lambda^{-\frac{n}{2}}Z_i\left(\frac{x-\xi}{\lambda}\right).\]
This scaling preserves the $L^2(\R^n)$ norm of $Z_i$.

Next we study the linearized parabolic equation around $W$. The first one is a nondegeneracy result.
\begin{prop}[Nondegeneracy]\label{prop nondegeneracy}
Suppose $\alpha>2$ and $\varphi\in L^\infty(-\infty,0;\mathcal{X}_\alpha)$ is a solution of
 \begin{equation}\label{linear eqn}
  \partial_t\varphi-\Delta\varphi=pW^{p-1}\varphi, \quad \mbox{in } ~~ \R^n\times\R^-.
\end{equation}
Then there exist $(N+2)$ constants $c_0,\cdots, c_{n+1}$ such that
\[\varphi(x,t)\equiv c_0e^{\mu_0t}Z_0(x)+\sum_{i=1}^{n+1}c_iZ_i(x).\]
\end{prop}
\begin{proof}
  By standard parabolic regularity theory and the $O(|x|^{-4})$ decay of $W^{p-1}$ at infinity, we deduce that $\varphi\in L^\infty(-\infty,0;\mathcal{X}_\alpha^{2+\theta})$ and $\partial_t\varphi\in L^\infty(-\infty,0;\mathcal{X}_{2+\alpha})$.
Because $\alpha>2$, we may assume
\[ \int_{\R^n}\varphi(x,t)Z_i(x)dx\equiv 0, \quad \forall t<0, \quad i=0,\cdots, n+1.\]
Then $\partial_t\varphi$ satisfy these orthogonal conditions, too. Since it is also a solution of \eqref{linear eqn}, by the method in \cite[Section 2]{Collot2017ground} (in particular,  the coercivity estimate in \cite[Lemma 2.3]{Collot2017ground}), we deduce that $\partial_t\varphi\equiv 0$. Therefore $\varphi$ is a stationary solution of \eqref{linear eqn}. Since $\varphi\in\mathcal{X}_\alpha$ and it is orthogonal to $Z_i$, by Theorem \ref{thm spectrum}, $\varphi\equiv 0$.
\end{proof}

Now we state two a priori estimates for this linearized equation.

\begin{lem}\label{lem linear decay estimate 1}
  Suppose $\alpha^\prime>\max\{2,n/2\}$,  $\varphi_0\in\mathcal{X}_{\alpha^\prime}$. If $\varphi_0$ satisfies the orthogonal condition
  \begin{equation}\label{ortghogonal condition model 0}
\int_{\R^n} \varphi_0(x)  Z_i(x)dx=0, \quad \forall  i=0,\cdots, n+1,
\end{equation}
   then there exists a unique $\varphi\in L^\infty_{loc}\left(\R^+;\mathcal{X}_{\alpha^\prime}\right)$ solving the equation
\begin{equation}\label{linear eqn 1}
   \left\{\begin{aligned}
&    \partial_t\varphi-\Delta \varphi=pW^{p-1}\varphi   \quad \mbox{in}~~ \R^n\times(0,2T_D),\\
&  \varphi(x,0)=\varphi_0(x)  \quad \mbox{on}~~ \R^n,
\end{aligned}\right.
  \end{equation}
Moreover, $\varphi$ satisfies the orthogonal condition
\begin{equation}\label{ortghogonal condition model 1}
\int_{\R^n} \varphi(x,t)  Z_i(x)dx=0, \quad \forall t>0, \quad \forall i=0,\cdots, n+1.
\end{equation}
and the decay property
\begin{equation}\label{local uniform convergence}
 \varphi(\cdot,\cdot+t)\to 0 \quad \mbox{in} ~~ C_{loc}^{(1+\theta)/2}(\R;C_{loc}^{1,\theta}(\R^n)) ~~ \mbox{as} ~~  t\to+\infty.
\end{equation}
\end{lem}
\begin{proof}
Existence of a global solution  $\varphi$ to the problem \eqref{linear eqn 1} follows from standard parabolic theory, if we note that the heat semigroup $e^{t\Delta}$ is uniformly bounded as an operator from $\mathcal{X}_{\alpha^\prime}$ into itself.

By standard parabolic regularity theory, $\partial_t\varphi$, $\nabla\varphi$ and $\Delta\varphi$  belong to $L^\infty_{loc}\left(\R^+;\mathcal{X}_{\alpha^\prime}\right)$. The orthogonal condition \eqref{ortghogonal condition model 1} then follows by testing \eqref{linear eqn 1} with $Z_i$.

Testing \eqref{linear eqn 1} with $\varphi$ and then applying Theorem \ref{thm spectrum}, we obtain
\begin{equation}\label{Lyapunov functional 1}
 \frac{1}{2} \frac{d}{dt}\int_{\R^n}\varphi(x,t)^2dx=-\int_{\R^n}\left[|\nabla\varphi(x,t)|^2-pW(x)^{p-1}\varphi(x,t)^2\right]dx\leq 0.
\end{equation}
Testing   \eqref{linear eqn 1} with $\partial_t\varphi$, we also obtain
\begin{equation}\label{Lyapunov functional 2}
 \frac{1}{2} \frac{d}{dt}\int_{\R^n}\left[|\nabla\varphi(x,t)|^2-pW(x)^{p-1}\varphi(x,t)^2\right]dx=-\int_{\R^n}\partial_t\varphi(x,t)^2dx\leq 0.
\end{equation}
Combining these two identities with Proposition \ref{prop nondegeneracy}, we deduce that
\[ \lim_{t\to+\infty}\|\varphi(t)\|_{H^1(\R^n)}=0.\]
Then  \eqref{local uniform convergence} follows by an application of  standard parabolic regularity theory.
\end{proof}

The second one is inspired by \cite[Lemma 5.1]{wei-signchanging}.
\begin{lem}\label{lem linear decay estimate 2}
If $\alpha>2$, there exists a positive constant $C(\alpha)$ so that the following holds. Given $T>1$,  assume $\varphi\in L^\infty\left(0,T;\mathcal{X}_{\alpha}\right)$, $a_i\in L^\infty([0,T])$ and $E\in L^\infty\left(0,T;\mathcal{X}_{2+\alpha}\right)$  solve  the problem
\begin{equation}\label{linear eqn 2}
    \left\{\begin{aligned}
&    \partial_t\varphi-\Delta \varphi=pW^{p-1}\varphi+ \sum_{i=0}^{n+1}a_i Z_i+E   \quad \mbox{in}~~ \R^n\times(0,T),\\
&  \varphi(x,0)=0,
\end{aligned}\right.
  \end{equation}
and $\varphi$ satisfies the orthogonal condition
\begin{equation}\label{ortghogonal condition model 2}
\int_{\R^n} \varphi(x,t)  Z_i(x)dx=0, \quad \forall t\in[0,T], \quad \forall i=0,\cdots, n+1.
\end{equation}
Then
\[  \left\{\begin{aligned}
&\sum_{i=0}^{n+1}|a_i(t)|\leq C\|E(t)\|_{2+\alpha}, \quad \forall t \in[0,T],\\
&    \|\varphi\|_{C^{(1+\theta)/2}\left(0,T;\mathcal{X}^{1+\theta}_\alpha\right)}\leq C(\alpha)\|E\|_{L^{\infty}\left(0,T;\mathcal{X}_{2+\alpha}\right)}.
\end{aligned}\right.
\]
\end{lem}
\begin{proof}
First, for each $i$, multiplying \eqref{linear eqn 2} by $Z_i$ and integrating on $\R^n$, we obtain
\begin{equation}\label{form of a}
  a_{i}(t)=-\frac{\int_{\R^n}E(x,t)Z_i(x)}{\int_{\R^n}Z_i(x)^2}.
\end{equation}
Because $\alpha>2$, this gives the estimate on $a_i(t)$.

For the second estimate, by standard parabolic regularity theory, it suffices to prove
\begin{equation}\label{linear decay estimate}
   \|\varphi\|_{L^{\infty}\left(0,T;\mathcal{X}_\alpha\right)}\leq C(\alpha)\|E\|_{L^{\infty}\left(0,T;\mathcal{X}_{2+\alpha}\right)}.
\end{equation}
We will argue by contradiction.  Assume  there exists  a sequence of $T_k>1$, a sequence of  $\varphi_k\in L^\infty\left( 0,T_k;\mathcal{X}_{\alpha}\right)$, $a_{k,i}\in  L^\infty\left( [0,T_k]\right)$ and $E_k \in L^{\infty}\left( 0,2T_k;\mathcal{X}_{2+\alpha}\right)$, satisfying \eqref{linear eqn 2} and \eqref{ortghogonal condition model 2}, with
\begin{equation}\label{normalization of varphi}
  \|\varphi_k\|_{L^\infty\left(0,T_k;\mathcal{X}_{\alpha}\right)}=1
\end{equation}
  but
\begin{equation}\label{absurd assumption in Appendix A}
\|E_k\|_{L^{\infty}\left( 0,2T_k;\mathcal{X}_{2+\alpha}\right)}\leq \frac{1}{k}.
\end{equation}

Take a $t_k\in(0,T_k)$ with $\|\varphi_k(t_k)\|_\alpha\geq 1/2$ and a point $x_k\in\R^n$ such that
\begin{equation}\label{weigthed norm attained at finite pt}
  \left(1+|x_k|\right)^\alpha|\varphi_k(x_k,t_k)|\geq 1/4.
\end{equation}

We claim that
\begin{sublem}
$\limsup_{k\to\infty}|x_k|<+\infty$.
\end{sublem}
 \begin{proof}
By the heat kernel representation,
\begin{eqnarray*}
 \varphi(x,t_k)&=&\int_{0}^{t_k} \int_{\R^n} \left[4\pi(t_k-t)\right]^{-\frac{n}{2}}e^{-\frac{|x-y|^2}{4(t_k-t)}}\\
 && \quad \quad \quad \quad \times \left[pW(y)^{p-1}\varphi_k(y,t)+ \sum_{i=0}^{n+1}a_{k,i}(t) Z_i(y)+E_k(y,t)\right]dydt.
 \end{eqnarray*}
Following the calculation in   \cite[Appendix B]{Wei-Yan}, we get
\begin{eqnarray*}
   && \left|\int_{0}^{t_k} \left[4\pi(t_k-t)\right]^{-\frac{n}{2}}e^{-\frac{|x-y|^2}{4(t_k-t)}} pW(y)^{p-1}\varphi_k(y,t)dydt \right| \\
   &\lesssim & \int_{\R^n}|x-y|^{2-n} \left(1+|y|\right)^{-4-\alpha}\\
   &\lesssim& \left(1+|x|\right)^{-2-\alpha}
\end{eqnarray*}
and
\begin{eqnarray*}
   && \left|\int_{0}^{t_k} \left[4\pi(t_k-t)\right]^{-\frac{n}{2}}e^{-\frac{|x-y|^2}{4(t_k-t)}} \left[ \sum_{i=0}^{n+1}a_{k,i}(t) Z_i(y)+E_k(y,t)\right]dydt \right| \\
   &\lesssim & \frac{1}{k}\int_{\R^n}|x-y|^{2-n} \left(1+|y|\right)^{-2-\alpha}\\
   &\lesssim& \frac{1}{k}\left(1+|x|\right)^{-\alpha}.
\end{eqnarray*}
Putting these two estimates together we get
\[ |\varphi_k(x_k,t_k)|\leq C\left(1+|x_k|\right)^{-2-\alpha}+\frac{1}{k}\left(1+|x_k|\right)^{-\alpha}.\]
Combining this inequality with \eqref{weigthed norm attained at finite pt} we finish the proof of this sub-lemma.
 \end{proof}

Next we divide the proof into two cases.

{\bf Case 1.} $t_k\to +\infty$.

Let
\[\widetilde{\varphi}_k(x,t):=\varphi_k(x,t_k+t).\]
An application of standard parabolic regularity theory shows that $\widetilde{\varphi}_k$ are uniformly bounded in $C^{1+\theta,(1+\theta)/2}_{loc}(\R^n\times (-\infty,0])$. After passing to a subsequence, $\widetilde{\varphi}_k$ converges to a limit $\widetilde{\varphi}_\infty$, which satisfies the following conditions.
\begin{itemize}
  \item Passing to the limit in \eqref{normalization of varphi} gives
\begin{equation}\label{normalization of varphi 1}
  \|\widetilde{\varphi}_\infty\|_{L^\infty\left( -\infty,0;\mathcal{X}_\alpha\right)}\leq 1.
\end{equation}
\item Since $\alpha>2$, we can pass to the limit in \eqref{ortghogonal condition model 2} for $\varphi_k$, obtaining
\begin{equation}\label{ortghogonal condition 2}
  \int_{\R^n}\widetilde{ \varphi}_\infty(x,t)  Z_i(x)dx=0, \quad  \quad \mbox{for any} ~ i=0,\cdots, n+1, \quad t\leq 0.
\end{equation}
\item Passing to the limit in \eqref{linear eqn 2} for $\widetilde{\varphi}_k$ and noting \eqref{absurd assumption in Appendix A} as well as the estimate on $a_{k,i}$, we see $\widetilde{\varphi}_\infty$ is a solution of
\eqref{linear eqn}.
  \end{itemize}
 By  Proposition \ref{prop nondegeneracy}, these three conditions imply that $\widetilde{\varphi}_\infty\equiv 0$.

On the other hand, passing to the limit in \eqref{weigthed norm attained at finite pt} leads to
  \begin{equation*}\label{nonzero 1}
\left(1+|x_\infty|\right)^\alpha|\widetilde{\varphi}_\infty(x_\infty,0)|\geq 1/4.
\end{equation*}
This is a contradiction with the fact that $\widetilde{\varphi}_\infty\equiv 0$.

{\bf Case 2.} $t_k\to t_\infty$.

As in the previous case, now $\varphi_k$ itself converges to a limit $\varphi_\infty$, which solves the equation
\begin{equation*}
    \left\{\begin{aligned}
&    \partial_t\varphi_\infty-\Delta \varphi_\infty=pW^{p-1}\varphi_\infty   \quad \mbox{in}~~ \R^n\times(0,t_\infty),\\
&  \varphi_\infty(x,0)=0.
\end{aligned}\right.
  \end{equation*}
  Passing to the limit in \eqref{normalization of varphi} still gives
\begin{equation}\label{normalization of varphi 1}
  \|\varphi_\infty\|_{L^\infty\left( 0,t_\infty;\mathcal{X}_\alpha\right)}\leq 1.
\end{equation}
By standard parabolic theory, we also get $\varphi_\infty\equiv 0$. Then we get the same contradiction as in Case 1.
\end{proof}

\section{Estimates on some integrals}\label{sec estimate of integrals}
\setcounter{equation}{0}

In this appendix we give some technical integral estimates involving the heat kernel associated to the outer equation in Section \ref{sec outer}.

\begin{lem}\label{lem heat kernel integral 1}
 Assume $0<\nu<n$, $0\leq \gamma<n-\nu$, $t>s$. Then
\[\int_{\R^n}\left(t-s\right)^{-\frac{n}{2}}e^{-c\frac{|x-y|^2}{t-s}} \left(1+\frac{\sqrt{t-s}}{|y|}\right)^\gamma |y|^{-\nu}dy\lesssim \left(|x|+\sqrt{t-s}\right)^{-\nu}\left(1+\frac{\sqrt{t-s}}{|x|}\right)^{\gamma}. \]
\end{lem}
\begin{proof}
After a change of variables $\hat{x}:=x/\sqrt{t-s}$, $\hat{y}:=y/\sqrt{t-s}$, this integral is transformed into
\[
  (t-s)^{-\nu/2} \int_{\R^n}e^{-c|\hat{x}-\hat{y}|^2} \left(1+\frac{1}{|\hat{y}|}\right)^\gamma |\hat{y}|^{-\nu}d\hat{y}.
\]
To estimate the integral in this formula, we consider two cases separately.

{\bf Case 1.} If $|\hat{x}|\leq 1$, this  integral is bounded by a universal constant.

{\bf Case 2.} If $|\hat{x}|\geq 1$, we divide this integral into two parts, $\mathrm{I}$ outside $B_{|\hat{x}|/3}(0)$, and $\mathrm{II}$ in $B_{|\hat{x}|/3}(0)$.

Outside $B_{|\hat{x}|/3}(0)$,
  \[
      \left(1+\frac{1}{|\hat{y}|}\right)^\gamma|\hat{y}|^{-\nu} \lesssim \left(1+\frac{1}{|\hat{x}|}\right)^\gamma |\hat{x}|^{-\nu}.
 \]
Hence
\begin{eqnarray*}
    \mathrm{I}&\lesssim&\left(1+\frac{1}{|\hat{x}|}\right)^\gamma |\hat{x}|^{-\nu} \int_{B_{|\hat{x}|/3}(0)^c}e^{-c|\hat{x}-\hat{y}|^2} d\hat{y} \\
  &\lesssim & (t-s)^{\nu/2}|x|^{-\nu}\left(1+\frac{\sqrt{t-s}}{|x|}\right)^{\gamma}.
\end{eqnarray*}

For $\mathrm{II}$, by noting that for  $\hat{y}\in B_{|\hat{x}|/3}(0)$,
\[e^{-c|\hat{x}-\hat{y}|^2}\lesssim e^{-c|\hat{x}|^2/9}.\]
we obtain
\begin{eqnarray*}
    \mathrm{II}&\lesssim&e^{-c|\hat{x}|^2/9} \int_{B_{|\hat{x}|/3}(0)}|\hat{y}|^{-\nu-\gamma} d\hat{y} \\
  &\lesssim & e^{-c\frac{|x|^2}{t-s}}\\
  &\lesssim &  (t-s)^{\nu/2}|x|^{-\nu}.
\end{eqnarray*}

Adding $\mathrm{I}$ and $\mathrm{II}$ together,  we get
\[\int_{\R^n}e^{-c|\hat{x}-\hat{y}|^2} \left(1+\frac{1}{|\hat{y}|}\right)^\gamma |\hat{y}|^{-\nu}d\hat{y} \lesssim  (t-s)^{\nu/2} |x|^{-\nu}\left(1+\frac{\sqrt{t-s}}{|x|}\right)^{\gamma}. \qedhere\]
\end{proof}

\begin{lem}\label{lem heat kernel integral 2}
 Assume   $t>s$,   $c_1$, $c_2$, $L$ and $\lambda$ are four positive constants. Then
\begin{eqnarray*}
     &&  \int_{B_{L\lambda}^c}\left(t-s\right)^{-\frac{n}{2}}e^{-c_1\frac{|x-y|^2}{t-s}} \left(1+\frac{\sqrt{t-s}}{|y|}\right)^\gamma e^{-c_2\frac{|y|}{\lambda}}dy\\
   &\lesssim & e^{-\frac{c_2}{3}\frac{|x|}{\lambda}}\left(1+\frac{\sqrt{t-s}}{|x|}\right)^{\gamma}
   +e^{-\frac{c_2L}{2}-\frac{c_1}{9}\frac{|x|^2}{t-s}} \left(\frac{\lambda}{\sqrt{t-s}}\right)^{n-\gamma} \left(1+\frac{\lambda}{\sqrt{t-s}}\right)^{\gamma}.
  \end{eqnarray*}
\end{lem}
\begin{proof}
After a change of variables $\hat{x}:=x/\sqrt{t-s}$, $\hat{y}:=y/\sqrt{t-s}$, this integral is transformed into
  \[ \int_{B_{L\lambda/\sqrt{t-s}}^c}e^{-c_1|\hat{x}-\hat{y}|^2} \left(1+\frac{1}{|\hat{y}|}\right)^\gamma e^{-c_2\frac{\sqrt{t-s}}{\lambda}|\hat{y}|}d\hat{y}.
  \]
  We divide this integral into two parts, $\mathrm{I}$ outside $B_{|\hat{x}|/3}(0)$, and $\mathrm{II}$ in $B_{|\hat{x}|/3}(0)$.

Outside $B_{|\hat{x}|/3}(0)$,
 \[
    \left\{\begin{aligned}
&e^{-c_2\frac{\sqrt{t-s}}{\lambda}|\hat{y}|}\leq e^{-\frac{c_2}{3}\frac{\sqrt{t-s}}{\lambda}|\hat{x}|}, \\
&     \left(1+\frac{1}{|\hat{y}|}\right)^\gamma \lesssim \left(1+\frac{1}{|\hat{x}|}\right)^\gamma .
\end{aligned}\right.
 \]
Therefore
\begin{eqnarray*}
    \mathrm{I}&\lesssim& e^{-\frac{c_2}{3}\frac{\sqrt{t-s}}{\lambda}|\hat{x}|}\left(1+\frac{1}{|\hat{x}|}\right)^\gamma \int_{B_{|\hat{x}|/3}(0)^c}e^{-c_1|\hat{x}-\hat{y}|^2} d\hat{y} \\
  &\lesssim &  e^{-\frac{c_2}{3}\frac{|x|}{\lambda}}\left(1+\frac{\sqrt{t-s}}{|x|}\right)^{\gamma} .
\end{eqnarray*}

In $B_{|\hat{x}|/3}(0)$,
\[e^{-c_1|\hat{x}-\hat{y}|^2}\leq e^{-c_1|\hat{x}|^2/9}.\]
Hence
\begin{eqnarray*}
    \mathrm{II}&\lesssim& e^{-c_1|\hat{x}|^2/9} \int_{B_{|\hat{x}|/3}\setminus B_{L\lambda/\sqrt{t-s}}}  \left(1+\frac{1}{|\hat{y}|}\right)^\gamma e^{-c_2\frac{\sqrt{t-s}}{\lambda}|\hat{y}|}d\hat{y} \\
  &\lesssim &  e^{-\frac{c_2L}{2}-\frac{c_1}{9}\frac{|x|^2}{t-s}}\left(\frac{\lambda}{\sqrt{t-s}}\right)^{n}\left(1+\frac{\sqrt{t-s}}{\lambda}\right)^\gamma\\
   &\lesssim &  e^{-\frac{c_2L}{2}-\frac{c_1}{9}\frac{|x|^2}{t-s}} \left( \frac{\lambda}{\sqrt{t-s}}\right)^{n}\left(1+\frac{\lambda}{\sqrt{t-s}}\right)^{ \gamma}. \qedhere
\end{eqnarray*}
\end{proof}

\end{appendices}

\newpage


\part{Energy concentration in the general case}\label{part many bubbles}

\section{Setting}\label{sec setting II}
\setcounter{equation}{0}

In this part, we still consider a sequence of \emph{smooth, positive solutions} $u_i$ to the nonlinear heat equation
\eqref{eqn} (with $p=(n+2)/(n-2)$) in $Q_1$, but now satisfying the following three assumptions.
\begin{description}
\item [(IIII.a) Weak limit] $u_i$ converges weakly to $u_\infty$ in $L^{p+1}(Q_1)$, and  $\nabla u_i$ converges weakly to $\nabla u_\infty$ in $L^2(Q_1)$. Here $u_\infty$ is a \emph{smooth solution} of \eqref{eqn} in $Q_1$.

  \item[(IIII.b) Energy concentration behavior] there exists an $N\geq 2$ such that
  \[ \left\{\begin{aligned}
& |\nabla u_i|^2dxdt\rightharpoonup |\nabla u_\infty|^2dxdt+N\Lambda\delta_0\otimes dt,\\
& u_i^{p+1}dxdt\rightharpoonup u_\infty^{p+1}dxdt+N\Lambda\delta_0\otimes dt,
\end{aligned}\right.
\]
 weakly as Radon measures.

 \item[(IIII.c) Convergence of time derivatives] as $i\to\infty$, $\partial_tu_i$ converges to $\partial_tu_\infty$ strongly in $L^2(Q_1)$.
 \end{description}
The only difference with the assumptions in Part \ref{part one bubble} is {\bf (IIII.b)}, where now we do not assume there is only one bubble, in other words, now we are in \emph{the higher multiplicity case} instead of \emph{the multiplicity one case}.

The main result in this part is
\begin{thm}\label{thm bubble clustering convergence}
After passing to a subsequence, the followings hold for $u_i$.
  \begin{enumerate}
  \item  For all $i$   and any $t\in[-9/16,9/16]$,  there exist exactly $N$ local maximal point of $u_i(\cdot,t)$ in the interior of $B_1(0)$.

  Denote these points by $\xi_{ij}^\ast(t)$ ($j=1,\cdots, N$) and let
$\lambda_{ij}^\ast(t):=u_i(\xi_{ij}^\ast(t),t)^{-\frac{2}{n-2}}$.

\item Both $\lambda_{ij}^\ast$ and $\xi_{ij}^\ast$ are continuous functions on $[-9/16,9/16]$.

\item There exists a constant $C_2$ such that, for all $i$ and any $x\in B_1$,
\begin{eqnarray}\label{scaling invariant estimate III}
  && \min_{1\leq j \leq N}|x-\xi_{ij}^\ast(t)|^{\frac{n-2}{2}}u_i(x,t)+\min_{1\leq j \leq N}|x-\xi_{ij}^\ast(t)|^{\frac{n}{2}}|\nabla u_i(x,t)|\\
   && \quad \quad \quad \quad +\min_{1\leq j \leq N}|x-\xi_{ij}^\ast(t)|^{\frac{n+2}{2}}\left(|\nabla^2 u_i(x,t)|+|\partial_tu_i(x,t)|\right)\leq C_2. \nonumber
\end{eqnarray}

  \item For each $j=1,\cdots, N$, as $i\to \infty$,
\[\lambda_{ij}^\ast(t)\to0, \quad \xi_{ij}^\ast(t)\to 0,  \quad \mbox{uniformly on } [-9/16,9/16],\]
and  the function
\[u_{ij}^t(y,s):=\lambda_{ij}^\ast(t)^{\frac{n-2}{2}}u\left(\xi_{ij}^\ast(t)+\lambda_{ij}^\ast(t)y , t+\lambda_{ij}^\ast(t)^2s \right),\]
converges to $W(y)$ in $C^\infty_{loc}( \R^n\times\R)$.

  \item  For any $1\leq j\neq k \leq N$ and $t\in[-9/16,9/16]$,
\begin{equation}\label{bubbles separating II}
   \lim_{i\to+\infty}\frac{|\xi_{ij}^\ast(t)-\xi_{ik}^\ast(t)|}{\max\left\{\lambda_{ij}^\ast(t),\lambda_{ik}^\ast(t)\right\}}=+\infty.
\end{equation}

\item  $u_\infty\equiv 0$.
\end{enumerate}
\end{thm}

\begin{rmk}
Item (5) says there is no bubble towering. Item (6) will   be used (in an inductive way) to prove this property.
\end{rmk}

The proof of this theorem uses an induction argument on $N$. This part is organized in the following way.
\begin{enumerate}
  \item In Section \ref{sec preliminary III}, we establish several preliminary results.
  \item In Section \ref{sec bubble tree construction}, we find all bubbles under the inductive assumption that Theorem \ref{thm bubble clustering convergence} holds up to $N-1$.
  \item In Section \ref{sec Lipschitz hypothesis}, we prove the Lipschitz hypothesis \eqref{Lip assumption} in Part \ref{part one bubble}, under the assumptions {\bf (II.a)-(II.c)}. This also provides   the first step of our inductive argument.
 \item In Section \ref{sec bubble cluster}, we finish the proof of Theorem \ref{thm bubble clustering convergence} by employing Proposition \ref{prop weak Schoen Harnack} in Part \ref{part one bubble}.
\end{enumerate}

\section{Preliminaries}\label{sec preliminary III}
\setcounter{equation}{0}

Before starting the bubble tree construction, we recall several technical results.

The first one is a uniform Morrey space estimate on $u_i$, which is just a special case of Corollary \ref{coro Morrey} in Part \ref{part energy concentration}
\begin{lem}\label{lem Morrey bound}
There exists a constant $M>0$ such that for each $i$ and  $Q_r(x,t)\subset Q_{3/4}$,
\begin{equation*}
\int_{Q_r(x,t)}\left(|\nabla u_i|^2+u_i^{p+1}\right) \leq Mr^2.
\end{equation*}
\end{lem}
The second one is the non-concentration estimate of $\partial_tu_i$ in Lemma \ref{lem nonconcentration of time derivative} from Part \ref{part one bubble}, which still holds in the current setting (thanks to {\bf (IIII.c)}).

The following lemma is about the convergence of rescalings of $u_i$.
\begin{lem}\label{lem convergence of rescaled functions}
Given a sequence $(x_i,t_i)\in Q_{1/2}$, $r_i\to 0$, define
\begin{equation}\label{definition of rescaling}
   u_i^{r_i}(x,t):=r_i^{\frac{n-2}{2}} u_i\left(x_i+r_i x, t_i+r_i^2 t\right).
\end{equation}
Then we have
\begin{enumerate}
  \item  After passing to a subsequence, $u_i^{r_i}$ converges weakly to $0$ or $W_{\xi,\lambda}$ for some $(\xi,\lambda)\in\R^n\times\R^+$.
    \item If the defect measure associated to $u_i^{r_i}$ is nontrivial, it must be of the form
\[\sum_{j}k_j\Lambda\delta_{P_j}\otimes dt,\]
where  $P_j$ are distinct points in $\R^n$ and
\[ \sum_{j}k_j\leq M/\Lambda, \quad k_j\in\mathbb{N}.\]
\item  If there is no defect measure, then $u_i^{r_i}$ converges in $C^\infty_{loc}(\R^n\times \R)$.
\end{enumerate}

\end{lem}
\begin{proof}
As in Part \ref{part energy concentration}, $u_i^{r_i}$ converges weakly to a nonnegative solution of \eqref{eqn} in $L^{p+1}_{loc}(\R^n\times\R)$. By  Lemma \ref{lem nonconcentration of time derivative} (non-concentration of time derivatives), $\partial_tu_i^{r_i}$ converges to $0$ in $L^2_{loc}(\R^n\times\R)$. Hence this weak limit is independent of time. By Caffarelli-Gidas-Spruck \cite{Caffarelli-Gidas-Spruck}, it must be $0$ or $W_{\xi,\lambda}$ for some $(\xi,\lambda)\in\R^n\times\R^+$. Furthermore, the defect measure is also independent of time, that is, there exists a Radon measure $\mu_0$ on $\R^n$ such that the defect measure equals $\mu_0\otimes dt$. By Lemma \ref{lem Morrey bound}, for any $R>0$,
\[  \int_{-R^2}^{R^2}\int_{B_R}d\mu_0 dt \leq MR^2.\]
Hence
\begin{equation}\label{measure bound}
  \mu_0(\R^n)\leq M.
\end{equation}

By Lemma \ref{lem energy quantization I}, there exist finitely many distinct points $P_j\in \R^n$ and constants $k_j\in\mathbb{N}$ such that
\[ \mu_0= \sum_{j}k_j\Lambda\delta_{P_j}.\]

Finally, if there is no defect measure, the smooth convergence of $u_i^{r_i}$ follows from a direct application of the $\varepsilon$-regularity theorem, Theorem \ref{thm ep regularity II}, by noting the smoothness of its weak limit.
\end{proof}

We also need a technical result about the blow up of each time slice of $u_i$. This will be used to find the first bubble in the next section.
\begin{lem}\label{lem blow up at all time}
  For each $t\in[-9/16,9/16]$,
  \[\lim_{i\to+\infty}\max_{B_1}u_i(x,t)=+\infty.\]
\end{lem}
\begin{proof}
For each $u_i$ and $(x,t)\in Q_{3/4}$, define $\rho_i(x,t)$ to be the unique (if it exists) solution to the equation
\[  \Theta_{r^2}(x,t;u_i)=  \varepsilon_\ast.\]
Here $\varepsilon_\ast$ is the small constant in the $\varepsilon$-regularity theorem, Theorem \ref{thm ep regularity I}.

By {\bf(IIII.b)}, if $x\neq 0$, there is a uniform, positive lower bound for $\rho_i(x,t)$. By {\bf(IIII.a)} and {\bf (IIII.b)}, for any $s>0$ fixed,
\[ \lim_{i\to+\infty}\Theta_s(0,t;u_i)=\left(4\pi\right)^{-\frac{n}{2}}\frac{\Lambda}{n}N+o_s(1).\]
Therefore
\[\lim_{i\to+\infty}\rho_i(0,t)=0.\]
As a consequence, $\rho_i(\cdot,t)$ has an interior minimal point in $B_1$, say $x_i(t)$. Denote $r_i:=\rho_i(x_i(t),t)$.
Define $u_i^{r_i}$ as in \eqref{definition of rescaling}, with base point at $(x_i(t),t)$.

{\bf Claim.} For some $(\xi,\lambda)\in\R^n\times\R^+$, $u_i^{r_i}$ converges to $W_{\xi,\lambda}$ in $C^\infty_{loc}(\R^n\times\R)$.

By this claim we get
\[\lim_{i\to+\infty} \rho_i(x_i(t),t)^{\frac{n-2}{2}} u_i(x_i(t),t)=W_{\xi,\lambda}(0)>0,\]
and the proof of this lemma is complete.

{\bf Proof of the Claim.} By the definition of $r_i$ and the scaling invariance of $\Theta_s(\cdot)$, for any $y\in B_{r_i^{-1}}$,
\[\Theta_1(y,0;u_i^{r_i})\leq\varepsilon_\ast.\]
Moreover, the equality is attained at $y=0$, that is,
 \begin{eqnarray}\label{normalization condition of Theta}
 \varepsilon_\ast&=&\int_{\R^n} \left[\frac{|\nabla u_i^{r_i}(y,-1)|^2}{2}-\frac{u_i^{r_i}(y,-1)^{p+1}}{p+1}\right]\psi\left(\frac{y-x_i(t)}{r_i}\right)^2 G(y,1)dy \nonumber\\
  &+&\frac{1}{2(p-1)}\int_{\R^n} u_i^{r_i}(y,-1)^2\psi\left(\frac{x-x_i(t)}{r_i}\right)^2 G(y,1)dy+Ce^{-cr_i^{-2}}.
\end{eqnarray}

By Theorem \ref{thm ep regularity II}, $u_i^{r_i}$ are uniformly bounded in $C^2(B_{r_i^{-1}-1}\times(-\delta_\ast,\delta_\ast))$. This then implies that there is no defect measure appearing when we apply Lemma \ref{lem convergence of rescaled functions} to $u_i^{r_i}$. In fact, if the defect measure is nontrivial, by Lemma \ref{lem convergence of rescaled functions}, it has the form
\[\sum_{j}k_j\Lambda\delta_{P_j}\otimes dt, \quad P_j\in\R^n, ~~ k_j\in\mathbb{N}.\]
Then by Lemma \ref{lem concentration of time slice}, for a.e. $t\in(-\delta_\ast,\delta_\ast)$, $u_i^{r_i}(\cdot, t)$ should develop nontrivial Dirac measures at $P_j$. This is a contradiction with the  uniform $C^2(B_{r_i^{-1}-1}\times(-\delta_\ast,\delta_\ast))$ regularity of $u_i^{r_i}$.

Since there is no defect measure,  Lemma \ref{lem convergence of rescaled functions} implies  that $u_i^{r_i}$ converges to $u_\infty$ in $C^\infty_{loc}(\R^n\times\R)$. 
 In view of the Liouville theorem in Caffarelli-Gidas-Spruck   \cite{Caffarelli-Gidas-Spruck}, it suffices to show that $u_\infty\neq 0$. (By the strong maximum principle, either $u_\infty\equiv 0$ or $u_\infty>0$ everywhere.)

First by the monotonicity formula (Proposition \ref{prop monotoncity formula}), we have
 \begin{eqnarray*}
 \varepsilon_\ast&\leq&\int_{-2}^{-1}(-s)^{\frac{p+1}{p-1}}\int_{\R^n} \left[\frac{|\nabla u_i^{r_i}(y,s)|^2}{2}-\frac{u_i^{r_i}(y,s)^{p+1}}{p+1}\right]\psi\left(\frac{y-x_i(t)}{r_i}\right)^2 G(y,-s)\nonumber\\
  &+&\frac{1}{2(p-1)}\int_{-2}^{-1}(-s)^{\frac{2}{p-1}}\int_{\R^n} u_i^{r_i}(y,s)^2\psi\left(\frac{x-x_i(t)}{r_i}\right)^2 G(y,-s)+Ce^{-cr_i^{-2}/4}.
\end{eqnarray*}
By the $C^\infty_{loc}(\R^n\times\R)$ convergence of $u_i^{r_i}$, the Morrey space bound in Lemma \ref{lem Morrey bound} for $u_i^{r_i}$ and the exponential decay of $G(y,-s)$ as $|y|\to+\infty$, we can let $i\to+\infty$ in the above inequality to get
 \begin{eqnarray*}
 \varepsilon_\ast&\leq&\int_{-2}^{-1}(-s)^{\frac{p+1}{p-1}}\int_{\R^n} \left[\frac{|\nabla u_\infty(y,s)|^2}{2}-\frac{u_\infty(y,s)^{p+1}}{p+1}\right]  G(y,-s)\nonumber\\
  &+&\frac{1}{2(p-1)}\int_{-2}^{-1}(-s)^{\frac{2}{p-1}}\int_{\R^n}u_\infty(y,s)^2 G(y,-s).
\end{eqnarray*}
Hence it is impossible that $u_\infty\equiv 0$.
\end{proof}

\section{Bubble tree construction}\label{sec bubble tree construction}
\setcounter{equation}{0}

In this section, we   construct bubbles by finding local maximal points of $u_i(\cdot,t)$. The construction is divided into six steps. During the course of this construction, we will also prove Theorem \ref{thm bubble clustering convergence} except the last point (6), under the inductive assumption that this theorem holds when the multiplicity is not larger than $N-1$ ($N\geq2$).

{\bf Step 1. Construction of the first maximal point.} By {\bf(IIII.a)} and Lemma \ref{lem blow up at all time}, for each $t$, $\max_{B_1}u_i(x,t)$ is attained at an interior point, say $\xi_{i1}^\ast(t)$. Denote
\[ \lambda_{i1}^\ast(t):=u_i\left(\xi_{i1}^\ast(t),t\right)^{-\frac{2}{n-2}}.\]
By Lemma \ref{lem blow up at all time},
\[ \lim_{i\to+\infty}\lambda_{i1}^\ast(t)=0.\]
Let
\[ u_{i1}(y,s):=\lambda_{i1}^\ast(t)^{\frac{n-2}{2}} u_i\left(\xi_{i1}^\ast(t)+\lambda_{i1}^\ast(t)y, t+\lambda_{i1}^\ast(t)^2 s\right).\]
It satisfies the following conditions.
\begin{itemize}
    \item By Lemma \ref{lem Morrey bound}, for any $R>0$,
    \begin{equation}\label{Morrey bound III.2}
    \int_{Q_R}\left[|\nabla u_{i1}|^2+u_{i1}^{p+1}\right]\leq MR^2.
    \end{equation}

 \item By  Lemma \ref{lem nonconcentration of time derivative}, for any $R>0$,
    \begin{equation}\label{Morrey bound III.3}
   \lim_{i\to+\infty}\int_{Q_R}|\partial_tu_{i1}|^2=0.
    \end{equation}

    \item At $s=0$,
    \begin{equation}\label{max at 0}
\max_{|y|\leq \lambda_{i1}^\ast(0)^{-1}/2}u_{i1}(y,0)=u_{i1}(0,0)=1.
    \end{equation}
\end{itemize}

Using these conditions we show
\begin{lem}\label{lem smooth convergence of scalings}
 As $i\to+\infty$, $u_{i1}$ are uniformly bounded in $C^\infty_{loc}(\R^n\times\R)$, and it converges to $W$ in  $C^\infty_{loc}(\R^n\times\R)$.
\end{lem}
\begin{proof}
We  use Lemma \ref{lem convergence of rescaled functions} to study the convergence of $u_{i1}$. We need only to show that there is no defect measure  appearing. Once this is established,  $u_{i1}$ would converge in $C^\infty_{loc}(\R^n\times\R)$. Then by \eqref{max at 0}, this limit must be $W$.

Assume by the contrary that the defect measure is nonzero.  By Lemma \ref{lem convergence of rescaled functions}, it has the form
\[\sum_{j}k_j\Lambda\delta_{P_j}\otimes dt, \quad P_j\in\R^n, ~~ k_j\in\mathbb{N}.\]
Then applying Lemma \ref{lem blow up at all time} to $u_{i1}$ at $t=0$, we obtain
\[\lim_{i\to+\infty}\max_{B_1(P_j)}u_{i1}(x,0)=+\infty.\]
This is a contradiction with \eqref{max at 0}.
\end{proof}
By this lemma, there exists a sequence $R_{i1}\to+\infty$ such that
\begin{equation}\label{decay away from bubble 1}
  \lim_{i\to+\infty} \sup_{y\in B_{R_{i1}}} |y|^{\frac{n-2}{2}}u_{i1}(y,0)<+\infty.
\end{equation}

{\bf Step 2. Iterative construction.} Suppose  $\xi_{i1}^\ast(t)$,$\cdots$,$\xi_{i,j-1}^\ast(t)$ have been constructed.  If
 \[\max_{x\in B_1}\min_{1\leq k \leq j-1}|x-\xi_{ik}^\ast(t)|^{\frac{n-2}{2}}u_i(x,t)\]
are uniformly bounded in $B_1$, we stop at this step. Otherwise,
\begin{equation}\label{inductive assumption}
\lim_{i\to \infty}\max_{x\in B_1} \min_{1\leq k \leq j-1}|x-\xi_{ik}^\ast(t
)|^{\frac{n-2}{2}}u_i(x,t)=+\infty.
\end{equation}
By {\bf(IIII.a)}, this function has an interior maximal point, say $\xi_{ij}^\ast(t)$. Define $\lambda_{ij}^\ast(t)$, $u_{ij}$ as in Step 1. The estimates \eqref{Morrey bound III.2} and \eqref{Morrey bound III.3} still hold for $u_{ij}$, while \eqref{max at 0} now reads as
\begin{equation}\label{max at 0 new}
   \left\{\begin{aligned}
&u_{ij}(0,0)=1  ,\\
& u_{ij}(y,0)\leq \frac{\min_{1\leq k\leq j-1}|\xi_{ij}^\ast(t)-\xi_{ik}^\ast(t)|^{\frac{n-2}{2}}}{\min_{1\leq k\leq j-1}|\xi_{ij}^\ast(t)-\xi_{ik}^\ast(t)+\lambda_{ij}^\ast(t)y|^{\frac{n-2}{2}}},
\end{aligned}\right.
    \end{equation}
 the right hand side of which converges to $1$ uniformly in any compact set of $\R^n$, by noting that we have
\[ \min_{1\leq k \leq j-1}|\xi_{ij}^\ast(t)-\xi_{ik}^\ast(t)|^{\frac{n-2}{2}} u_i\left(\xi_{ij}^\ast(t),t\right)=+\infty.\]
(This follows by letting $x=\xi_{ij}^\ast(t)$ in \eqref{inductive assumption}.)

We can argue as in Step 1 to deduce that $u_{ij}$ converges to $W$ in  $C^\infty_{loc}(\R^n\times\R)$. Moreover, there exists a sequence $R_{ij}\to+\infty$ such that
\begin{equation}\label{decay away from bubble 2}
  \lim_{i\to+\infty} \sup_{y\in B_{R_{ij}}} |y|^{\frac{n-2}{2}}u_{ij}(y,0)<+\infty.
\end{equation}

For any $k<j$, combining \eqref{inductive assumption} (evaluated at $x=\xi_{ij}^\ast(t)$) and \eqref{decay away from bubble 1} (with $j$ replaced by $k$) together, we obtain \eqref{bubbles separating II}.

{\bf Step 3. $\xi_{ij}^\ast(t)$ are local maximal points of $u_i(\cdot,t)$.} Because $0$ is the maximal point of $W$ and the Hessian $\nabla^2W(0)$ is strictly negative definite, by the above smooth convergence of $u_{ij}$ in  $C^\infty_{loc}(\R^n\times\R)$, we deduce that $\xi_{ij}^\ast(t)$ are local maximal points of $u_i(\cdot,t)$.

The strict concavity of $u_i$ near $\xi_{ij}^\ast(t)$ (with the help of the implicit function theorem) also implies the continuous dependence of $\xi_{ij}^\ast(t)$ with $t$. The continuity of $\lambda_{ij}^\ast(t)$ then follows from its definition.

{\bf Step 4. } Fix a large positive constant $R$. For each $t\in(-9/16,9/16)$, let
\[\Omega_i(t):= \cup_{j} B_{R\lambda_{ij}^\ast(t)}\left(\xi_{ij}^\ast(t)\right).\]
By \eqref{bubbles separating II}, for all $i$ large enough, these balls are disjoint from each other, and all of them are contained in $B_{1/2}$. Hence the above iterative construction of $\xi_{ij}^\ast(t)$ must stop in finitely many steps. (At this stage, we do not claim any uniform in $i$ bound on the number of these points.)

By our construction, there exists a constant $C_2$ such that  for any $i$ and $x\in B_1\setminus \Omega_i(t)$,
\begin{equation}\label{decay away from bubble 3}
  u_i(x,t)\leq C_2 \max_{j}|x-\xi_{ij}^\ast(t)|^{-\frac{n-2}{2}}.
\end{equation}
Then \eqref{scaling invariant estimate III} follows by applying standard parabolic regularity theory.

{\bf Step 5. }  In this step and the next one we show that there is no bubble other than those constructed in the previous steps, and the number of bubbles is exactly $N$.

In this step we consider the case that there are at least two bubbles.
For each $t$, let
\begin{equation}\label{diameter of bubble set}
   d_i(t):=\max_{j\neq k}|\xi_{ij}^\ast(t)-\xi_{ik}^\ast(t)|.
\end{equation}
Without loss of generality, assume this is attained between $\xi_{i1}^\ast(t)$ and $\xi_{i2}^\ast(t)$. By \eqref{bubbles separating II},
\begin{equation}\label{distance and bubble scale}
  \lim_{i\to+\infty}\frac{\max_{j}\lambda_{ij}^\ast(t)}{d_i(t)}=0.
\end{equation}

\begin{lem}\label{lem concentration of Theta 1}
  For any $\sigma\in(0,1)$, there exists an $R(\sigma)$ such that for each $i,j$ and $t$,
  \[\Theta_{R(\sigma)^2d_i(t)^2}\left(\xi_{ij}^\ast(t),t;u_i\right)\geq \left(4\pi\right)^{-\frac{n}{2}}\frac{\Lambda}{n}\left(N-\sigma\right).\]
\end{lem}
\begin{proof}
  Assume by the contrary, there exists a $\sigma\in(0,1)$, a sequence of $r_i$ satisfying
  \begin{equation}\label{absurd assumption II}
    \lim_{i\to+\infty} \frac{d_i(t)}{r_i}=0,
  \end{equation}
  but
\[\limsup_{i\to+\infty}\Theta_{r_i^2}\left(\xi_{ij}^\ast(t),t;u_i\right)\leq \left(4\pi\right)^{-\frac{n}{2}} \frac{\Lambda}{n}\left(N-\sigma\right).\]

By the weak convergence of $u_i$ and {\bf (III.a)} and {\bf(III.b)}, for any $s>0$ fixed,
\[\lim_{i\to+\infty}\Theta_s\left(\xi_{ij}^\ast(t),t;u_i\right)=\left(4\pi\right)^{-\frac{n}{2}}\frac{\Lambda}{n}N+\Theta_s(0,t;u_\infty).\]
Because $u_\infty$ is smooth,
\[\lim_{s\to0}\Theta_s(0,t;u_\infty)=0.\]
Therefore by choosing $s$   sufficiently small, for all $i$ large
\begin{equation}\label{limit of Theta at fixed scale}
\left(4\pi\right)^{-\frac{n}{2}} \frac{\Lambda}{n}N\leq \Theta_s\left(\xi_{ij}^\ast(t),t;u_i\right)\leq \left(4\pi\right)^{-\frac{n}{2}} \frac{\Lambda}{n}N+\sigma.
\end{equation}
By  the monotonicity and continuity of $\Theta_s(\cdot;u_i)$  in $s$, we can enlarge $r_i$ so that
\begin{equation}\label{choice of a large scale}
\Theta_{r_i^2}\left(\xi_{ij}^\ast(t),t;u_i\right)=\left(4\pi\right)^{-\frac{n}{2}}\frac{\Lambda}{n}\left(N-\sigma\right).
\end{equation}
By this choice, \eqref{absurd assumption II} still holds, while in view of \eqref{limit of Theta at fixed scale}, we must have $r_i\to0$.

 Define $u_i^{r_i}$ as in \eqref{definition of rescaling}, with respect to the base point $(\xi_{ij}^\ast(t),t)$ for some $j$. A scaling of \eqref{choice of a large scale} gives
  \begin{eqnarray}\label{scaling of Theta}
&& \left(4\pi\right)^{-\frac{n}{2}} \frac{\Lambda}{n}\left(N-\sigma\right) \nonumber\\
 &=& \Theta_{1}\left(0,0;u_i^{r_i}\right)  \\
  &=&\int_{\R^n} \left[\frac{|\nabla u_i^{r_i}(y,-1)|^2}{2}-\frac{u_i^{r_i}(y,-1)^{p+1}}{p+1}\right]\psi\left(\frac{y-\xi_{ij}^\ast(t)}{r_i}\right)^2 G(y,1)dy \nonumber \\
  &+&\frac{1}{2(p-1)}\int_{\R^n} u_i^{r_i}(y,-1)^2\psi\left(\frac{y-\xi_{ij}^\ast(t)}{r_i}\right)^2 G(y,1)dy+Ce^{-cr_i^{-2}}.\nonumber
\end{eqnarray}

By \eqref{decay away from bubble 3} and \eqref{absurd assumption II}, for any $\rho>0$, if $i$ is large enough,
\[u_i^{r_i}(y,0)\leq C_2|y|^{-\frac{n-2}{2}}, \quad \mbox{for any} ~~ y\in B_{r_i^{-1}/2}\setminus B_\rho.\]
Arguing as in the proof of Lemma \ref{lem smooth convergence of scalings}, we deduce  that the defect measure of $u_i^{r_i}$ is $N^\prime\delta_0\otimes dt$, for some $N^\prime\in\mathbb{N}$. Letting $i\to+\infty$ in \eqref{scaling of Theta}, we see $N^\prime\leq N-1$.

By our inductive assumption, $\nabla u_i^{r_i}$ converges to $0$ weakly in $L^2_{loc}(\R^n\times\R)$, $u_i^{r_i}$ converges to $0$ weakly in $L^{p+1}_{loc}(\R^n\times \R)$ and strongly in $C_{loc}(\R; L^2_{loc}(\R^n))$. Then by Lemma \ref{lem energy quantization I}, both $|\nabla u_i^{r_i}|^2dxdt$ and $\left(u_i^{r_i}\right)^{p+1}dxdt$ converges to $N^\prime\delta_0\otimes dt$ weakly as Radon measures. These convergence imply that the right hand side of \eqref{scaling of Theta} converges to
\[\left(4\pi\right)^{-\frac{n}{2}} N^\prime\Lambda/n.\] (Here we need to use the monotonicity formula, integrate in time $s$ and then argue as in the last step of the proof of Lemma \ref{lem blow up at all time}.) This is a contradiction with \eqref{choice of a large scale}.
\end{proof}

For $(y,s)\in Q_{d_i(t)^{-1}/2}$, define
\begin{equation}\label{a special rescaling}
 \widetilde{u}_i(y,s):=d_i(t)^{\frac{n-2}{2}} u_i\left(\xi_{i1}^\ast(t)+d_i(t)y, t+d_i(t)^2s\right)
\end{equation}
and
\[ \widetilde{\xi}_{ij}^\ast(t):=\frac{\xi_{ij}^\ast(t)-\xi_{i1}^\ast(t)}{d_i(t)}\]
Then we have
\begin{itemize}
  \item $\widetilde{\xi}_{i1}^\ast(t)=0$;
  \item $|\widetilde{\xi}_{i2}^\ast(t)|=1$;
  \item for each $j$, $|\widetilde{\xi}_{ij}^\ast(t)|\leq 1$.
\end{itemize}
After passing to a subsequence, we may assume each $\widetilde{\xi}_{ij}^\ast(t)$ converges to a point $\widetilde{\xi}_j^\ast(t)\in \overline{B_1}$, and $\widetilde{\xi}_{i2}^\ast(t)$ converges to a point $\widetilde{\xi}_2^\ast(t)\in \partial B_1$. (These limiting points need not to be distinct.)

By \eqref{distance and bubble scale}, $\widetilde{u}_i$ develops bubbles at each $\xi_j^\ast(t)$. By a scaling of \eqref{decay away from bubble 3},
$\widetilde{u}_i$ does not develop any bubble outside $\cup_j\{\xi_j^\ast(t)\}$. Therefore the defect measure of $\widetilde{u}_i$ is
\[\Lambda\sum_{j}\delta_{\xi_j^\ast(t)}\otimes dt=\Lambda\sum_{k}m_k\delta_{P_k}.\]
In the above, we write these $\widetilde{\xi}_j^\ast(t)$ as distinct points $P_k$, each one with multiplicity $m_k\in\mathbb{N}$. There are at least two different points, $P_1=0$ and $P_2\in\partial B_1$. Therefore for each $k$, $m_k\leq N-1$. By our inductive assumption, Theorem \ref{thm bubble clustering convergence} holds for $\widetilde{u}_i$. In particular, $\widetilde{u}_i$ converges to $0$ weakly in $L^{p+1}_{loc}(\R^n\times\R)$, and there is no bubble towering for $\widetilde{u}_i$.

By \eqref{limit of Theta at fixed scale} and Proposition \ref{prop monotoncity formula}, we   have
\[\limsup_{i\to+\infty}\Theta_{R(\sigma)^2d_i(t)^2}\left(\xi_{i1}^\ast(t),t;u_i\right)\leq \left(4\pi\right)^{-\frac{n}{2}} \frac{\Lambda}{n}\left(N+\sigma\right).\]
Combining this inequality with Lemma \ref{lem concentration of Theta 1}, we get
\begin{eqnarray}\label{no energy loss}
&& \left(4\pi\right)^{-\frac{n}{2}} \frac{\Lambda}{n} \left(N-\sigma\right) \nonumber\\
&\leq&R(\sigma)^{\frac{n}{2}}\int_{\R^n} \left[\frac{|\nabla \widetilde{u}_i(x,-1)|^2}{2}-\frac{ \widetilde{u}_i(x,-1)^{p+1}}{p+1}\right]\psi\left(\frac{x-\xi_{i1}^\ast(t)}{r_i}\right)^2 G(x,R(\sigma)^2)dx \nonumber\\
  &+&\frac{R(\sigma)^{\frac{n-2}{2}}}{2(p-1)}\int_{\R^n}  \widetilde{u}_i(x,-1)^2\psi\left(\frac{x-\xi_{i1}^\ast(t)}{r_i}\right)^2 G(x,R(\sigma)^2)dx+Ce^{-cd_i(t)^{-2}}\\
  &\leq& \left(4\pi\right)^{-\frac{n}{2}}  \frac{\Lambda}{n} \left(N+\sigma\right). \nonumber
\end{eqnarray}
Letting $i\to+\infty$, we deduce that the defect measure from $\widetilde{u}_i$ satisfies
\[ N-\sigma \leq \sum_{k}m_k  e^{-\frac{|P_k|^2}{4R(\sigma)^2}}\leq N+\sigma.\]
Letting $\sigma\to0$, which  implies that $R(\sigma)\to+\infty$, we get
\[ \sum_{k}m_k=N.\]
By Theorem \ref{thm bubble clustering convergence} and our inductive assumption, for all $i$ large, there are exactly $N$ bubbles located at $\widetilde{\xi}_{ij}^\ast(t)$ for $\widetilde{u}_i(\cdot,0)$. Coming back to $u_i$, this says for each $t$, there are exactly $N$ bubbles located at $\xi_{ij}^\ast(t)$,  $j=1,\cdots, N$.

{\bf Step 6. }
In this step we consider the case where there is only one $\xi_{ij}^\ast(t)$ at time $t$. We show that this is impossible, because we have assumed the multiplicity $N\geq 2$.

In this case, \eqref{decay away from bubble 3} reads as
\begin{equation}\label{decay away from bubble 4}
  u_i(x,t)\leq C_2 |x-\xi_{i1}^\ast(t)|^{-\frac{n-2}{2}}, \quad \mbox{for any} ~~ x\in B_1.
\end{equation}

Under this assumption,   we expect that there are $N$ bubbles towering at $\xi_{i1}^\ast(t)$. We will determine the scale for the lowest bubble, and then perform a rescaling at this scale. This gives a sequence of solutions weakly converging to some nontrivial $W_{\xi,\lambda}$ and exhibiting at most $N-1$ bubbles. By our inductive assumption, this is impossible.

 More precisely, similar to Lemma \ref{lem concentration of Theta 1}, we have
\begin{lem}\label{lem concentration of Theta 2}
  For any $\sigma\in(0,1)$, there exists an $R(\sigma)$ such that
  \begin{equation}\label{concentration of Theta 2}
    \Theta_{R(\sigma)^2\lambda_{i1}^\ast(t)^2}\left(\xi_{i1}^\ast(t),t;u_i\right)\geq \left(4\pi\right)^{-\frac{n}{2}}\frac{\Lambda}{n}\left(N-\sigma\right).
  \end{equation}
\end{lem}
\begin{proof}
Assume by the contrary that \eqref{concentration of Theta 2} does not hold. Arguing as in the proof of Lemma \ref{lem concentration of Theta 1}, we find a sequence of $r_i\to0$, satisfying
\begin{equation}\label{choice of a large scale 3}
 \lim_{i\to+\infty}\frac{\lambda_{i1}^\ast(t)}{r_i}=0,
\end{equation}
and
\begin{equation}\label{choice of a large scale 2}
  \Theta_{r_i^2}\left(\xi_{i1}^\ast(t),t;u_i\right)=\left(4\pi\right)^{-\frac{n}{2}} \frac{\Lambda}{n}\left(N-\sigma\right).
\end{equation}
Define $u_i^{r_i}$ as before, with base point at $(\xi_{i1}^\ast(t),t)$. By \eqref{decay away from bubble 4}, the defect measure for $u_i^{r_i}$ is $N^\prime\delta_0\otimes dt$. By \eqref{choice of a large scale 2}, $N^\prime\leq N-1$, while by \eqref{choice of a large scale 3}, $N^\prime\geq 1$.  By our inductive assumption, $u_i^{r_i}$ converges weakly to $0$. Then as in the proof of Lemma \ref{lem concentration of Theta 1}, we get
\[  \Theta_{r_i^2}\left(\xi_{i1}^\ast(t),t;u_i\right)=\Theta_{1}\left(0,0;u_i^{r_i}\right)\to  \left(4\pi\right)^{-\frac{n}{2}}\frac{\Lambda}{n}N^\prime, \quad \mbox{as } i\to+\infty.\]
This is a contradiction with \eqref{choice of a large scale 2}, because $N^\prime\leq N-1$ and $\sigma\in(0,1)$.
\end{proof}
\begin{rmk}
  The scale $r_i$ satisfying \eqref{choice of a large scale 2} is the scale of the lowest bubble.
\end{rmk}

Choose a $\sigma\in(0,1)$ and set $r_i:= R(\sigma)\lambda_{i1}^\ast(t)$ according to the previous lemma. Define $u_i^{r_i}$ as before, with respect to the base point $(\xi_{i1}^\ast(t),t)$. By definition,
\[ u_i^{r_i}(0,0)=\max_{y\in B_{r_i^{-1}/2}(0)}u_i^{r_i}(y,0)=R(\sigma)^{-\frac{n-2}{2}}.\]
Similar to Lemma \ref{lem smooth convergence of scalings}, this implies that $u_i^{r_i}$ converges to $W_{0,R(\sigma)}$ in $C^\infty_{loc}(\R^n\times\R)$. Arguing as in the proof of Lemma \ref{lem blow up at all time}, we deduce that
\begin{eqnarray*}
&& \Theta_{R(\sigma)^2\lambda_{i1}^\ast(t)^2}\left(\xi_{i1}^\ast(t),t;u_i\right)\\
&=&R(\sigma)^{\frac{p+1}{p-1}}\int_{\R^n} \left[\frac{|\nabla \widetilde{u}_i(x,-1)|^2}{2}-\frac{ \widetilde{u}_i(x,-1)^{p+1}}{p+1}\right]\psi\left(\frac{x-\xi_{ij}^\ast(t)}{r_i}\right)^2 G(x,R(\sigma)^2)  \\
  &+&\frac{1}{2(p-1)}R(\sigma)^{\frac{2}{p-1}}\int_{\R^n}  \widetilde{u}_i(x,-1)^2\psi\left(\frac{x-\xi_{ij}^\ast(t)}{r_i}\right)^2 G(x,R(\sigma)^2) +Ce^{-cR(\sigma)^{-2}r_i^{-2}}\\
  &\to& R(\sigma)^{\frac{p+1}{p-1}}\int_{\R^n} \left[\frac{|\nabla W_{0,R(\sigma)}(x)|^2}{2}-\frac{ W_{0,R(\sigma)}(x)^{p+1}}{p+1}\right]  G(x,R(\sigma)^2) dx\\
  &+&\frac{1}{2(p-1)}R(\sigma)^{\frac{2}{p-1}}\int_{\R^n}  W_{0,R(\sigma)}(x)^2 G(x,R(\sigma)^2)dx\\
&<& \left(4\pi\right)^{-\frac{n}{2}}\frac{\Lambda}{n}.
\end{eqnarray*}
This is a contradiction with Lemma \ref{lem concentration of Theta 2}.

\section{A remark on the Lipschitz hypothesis in Part \ref{part one bubble}}\label{sec Lipschitz hypothesis}
\setcounter{equation}{0}

In this section, under the hypothesis {\bf(II.a)-(II.c)} in Part \ref{part one bubble}, we prove the Lipschitz hypothesis \eqref{Lip assumption}.

First we claim that
Proposition \ref{prop blow up profile I} and Lemma \ref{lem scaling invariant estimate I} still hold, even now we do not assume \eqref{Lip assumption}. This follows from the general analysis in the previous section. Here we give more details.
\begin{proof}[Proof of the claim]
The proof is divided into three steps.

 {\bf Step 1.} By Lemma \ref{lem blow up at all time}, for any $t\in(-1,1)$, $u_i(\cdot,t)$ blows up as $i\to+\infty$. This allows us to find the first blow up point $\xi_{i}^\ast(t)$.

 {\bf Step 2.} With $u_i^{r_i}$ defined as in \eqref{definition of rescaling} (with $x_i\to0$), we claim that
  \begin{itemize}
    \item   either there is no defect measure and $u_i^{r_i}$ converges in $C^\infty_{loc}(\R^n\times\R)$,
    \item   or there exists a point    $P\in\R^n$ such that
 \[|\nabla u_i^{r_i}(x,t)|^2dxdt\rightharpoonup \Lambda\delta_P\otimes dt.\]
  \end{itemize}

First, similar to \eqref{limit of Theta at fixed scale}, for any $\sigma\in(0,1)$, there exists an $s$ such that
\[
 \left(4\pi\right)^{-\frac{n}{2}} \frac{\Lambda}{n}\leq  \lim_{i\to+\infty}\Theta_s\left(\xi_{i}^\ast(t),t;u_i\right)\leq \left(4\pi\right)^{-\frac{n}{2}} \frac{\Lambda}{n}+\sigma.
\]
Then because $r_i\to0$, by the monotonicity formula we get, for any $R>0$,
   \begin{eqnarray*}
&& R^{\frac{p+1}{p-1}}\int_{\R^n} \left[\frac{|\nabla u_i^{r_i}(x,-R^2)|^2}{2}-\frac{ u_i^{r_i}(x,-R^2)^{p+1}}{p+1}\right]\psi\left(\frac{x-x_i}{r_i}\right)^2 G(x,R^2)dx  \\
  &&\quad\quad +\frac{R^{\frac{p-1}{2}}}{2(p-1)}\int_{\R^n}  u_i^{r_i}(x,-R^2)^2\psi\left(\frac{x-x_i}{r_i}\right)^2 G(x,R^2)dx+Ce^{-cR^2r_i^{-2}} \\
 &=& \Theta_{R^2r_i^2}(\xi_{i}^\ast(t),t;u_i)\\
&\leq& \Theta_s\left(\xi_{i}^\ast(t),t;u_i\right)\\
  &\leq & \left(4\pi\right)^{-\frac{n}{2}}\frac{\Lambda}{n}+\sigma.
\end{eqnarray*}

Denote the weak limit of  $u_i^{r_i}$ by $u_\infty$ (which, by Lemma \ref{lem convergence of rescaled functions}, is $0$ or $W_{\xi,\lambda}$ for some $(\xi,\lambda)\in\R^n\times\R^+$), the defect measure associated to them by $\Lambda\sum_{j}k_j\delta_{P_j}\otimes dt$ (with notations as in Lemma \ref{lem convergence of rescaled functions}). Passing to the limit in the above inequality leads to
\[R^{-2}\int_{-2R^2}^{-R^2}\Theta_s(0,0;u_\infty)ds+\left(4\pi\right)^{-\frac{n}{2}}\frac{\Lambda}{n} R^{-2} \int_{-2R^2}^{-R^2}\sum_j k_je^{\frac{|P_j|^2}{4s}} \leq \left(4\pi\right)^{-\frac{n}{2}}\frac{\Lambda}{n}+\sigma.\]
By the monotonicity formula and the smoothness of $u_\infty$, we deduce that
\[\Theta_s(0,0;u_\infty)\geq \lim_{s\to0}\Theta_s(0,0;u_\infty)=0.\]
Hence we have
\[R^{-2} \int_{-2R^2}^{-R^2}\sum_j k_je^{\frac{|P_j|^2}{4s}}\leq 1+O(\sigma).\]
Letting $R\to+\infty$, we deduce that there exists at most one blow up point, whose multiplicity is exactly $1$.

If there is no defect measure, then Lemma \ref{lem convergence of rescaled functions} implies that $u_i^{r_i}$ converges to $u_\infty$ in $C^\infty_{loc}(\R^n\times\R)$.

Finally, we show that if the defect measure is nontrivial, then $u_\infty\equiv 0$. Assume this is not the case. By Lemma \ref{lem convergence of rescaled functions}, $u_\infty=W_{\xi,\lambda}$ for some $(\xi,\lambda)\in\R^n\times\R^+$.  By the weak convergence of $|\nabla u_i^{r_i}|^2dxdt$, there exists a sequence $R_i$ satisfying
\[R_i\to+\infty \quad \mbox{and} \quad R_ir_i\to0,\]
 such that
the defect measure associated to $u_i^{R_ir_i}$ is $2\Lambda\delta_0\otimes dt$. This is a contradiction with the claim that the multiplicity of this defect measure is $1$.

{\bf Step 3.} Denote $\lambda_i^\ast(t):=u_i(\xi_i^\ast(t),t)$. Similar to Lemma \ref{lem smooth convergence of scalings}, we deduce that
$u_i^{\lambda_i^\ast(t)}$ (with base point at $(\xi_i^\ast(t),t)$) converges to $W$ in $C^\infty_{loc}(\R^n\times\R)$. On the other hand, if the sequence $r_i$ satisfies
 \[r_i\to 0 \quad \mbox{and} \quad \frac{r_i}{\lambda_i^\ast(t)}\to+\infty,\]
 then by results in Step 2, we deduce that the sequence $u_i^{r_i}$ satisfies
 \[ |\nabla u_i^{r_i}(x,t)|^2dxdt\rightharpoonup \Lambda\delta_0\otimes dt.\]
 Combining this fact with Theorem \ref{thm ep regularity I}, we deduce that
  \begin{eqnarray}\label{scaling invariant estimate 16}
  && |x-\xi_{i}^\ast(t)|^{\frac{n-2}{2}}u_i(x,t)+|x-\xi_{i}^\ast(t)|^{\frac{n}{2}}|\nabla u_i(x,t)|\\
   && \quad \quad \quad \quad +|x-\xi_{i}^\ast(t)|^{\frac{n+2}{2}}\left(|\nabla^2 u_i(x,t)|+|\partial_tu_i(x,t)|\right)\leq C_2. \nonumber \qedhere
\end{eqnarray}
\end{proof}

Now we come to the proof of the Lipschitz hypothesis \eqref{Lip assumption}.
Assume there exist a sequence of smooth solutions $u_i$ to \eqref{eqn}, satisfying {\bf (II.a-II.c)}, but
\begin{equation}\label{16.1}
  \max_{|t|\leq 1/2}\int_{B_1}\left|\partial_tu_i(x,t)\right|^2dx\to+\infty.
\end{equation}
For each $i$, take a $t_i\in(-1,1)$ attaining
\[ \max_{|t|\leq 1}\left(1-|t|\right) \int_{B_1}\left|\partial_tu_i(x,t)\right|^2dx.\]

Denote
\[\lambda_i^{-2}:= \int_{B_1}\left|\partial_tu_i(x,t_i)\right|^2dx.\]
By \eqref{12.1}, as $i\to+\infty$,
\begin{equation}\label{16.2}
 \frac{1-|t_i|}{\lambda_i^2}\geq \frac{1}{2} \max_{|t|\leq 1/2}\int_{B_1}\left|\partial_tu_i(x,t)\right|^2dx\to +\infty.
\end{equation}
As a consequence,
\[\lim_{i\to+\infty}\lambda_i=0.\]

If $|t-t_i|<(1-|t_i|)/2$, by the definition of $t_i$, we get
\begin{equation}\label{16.3}
   \int_{B_1}\left|\partial_tu_i(x,t)\right|^2dx\leq 2\lambda_i^{-2}.
\end{equation}

At time $t_i$, there exists a unique maximal point of $u_i(\cdot, t_i)$ in $B_1$, denoted by $\xi_i(t_i)$.
Define
\[ \widetilde{u}_i(x,t):=\lambda_i^{\frac{n-2}{2}} u_i\left(\xi_i(t_i)+\lambda_i x, t_i+\lambda_i^2t\right).\]
By a scaling, we have
\begin{equation}\label{16.4}
   \left\{\begin{aligned}
& \int_{B_{\lambda_i^{-1}}}\left|\partial_t\widetilde{u}_i(x,0)\right|^2dx=1,\\
& \int_{B_{\lambda_i^{-1}}}\left|\partial_t\widetilde{u}_i(x,t)\right|^2dx\leq  2, \quad \mbox{for any} ~|t|<\frac{1-|t_i|}{2\lambda_i^2}
\end{aligned}\right.
\end{equation}

On the other hand, we claim that if $i$ is large,
\begin{equation}\label{16.5}
\int_{B_{\lambda_i^{-1}}}\left|\partial_t\widetilde{u}_i(x,0)\right|^2dx\leq \frac{1}{2}.
\end{equation}
This contradiction implies that \eqref{16.1} cannot be true. In other words, if $u_i$ satisfies {\bf (II.a-II.c)}, then there must exist a constant $L$ such that
\[ \limsup_{i\to+\infty}\sup_{|t|\leq 1/2}\int_{B_1}\partial_tu_i(x,t)^2dx\leq L,\]
that is, the Lipschitz hypothesis \eqref{Lip assumption} holds.

To prove \eqref{16.5}, first note that for $u_i$, by \eqref{scaling invariant estimate 16} we have
\[ \left|\partial_tu_i(x,t_i)\right|\lesssim |x-\xi_i(t_i)|^{-\frac{n+2}{2}}.\]
For $\widetilde{u}_i$, this reads as
\[ \left|\partial_t\widetilde{u}_i(x,0)\right|\lesssim |x|^{-\frac{n+2}{2}}.\]
Therefore there exists a large $R$ (independent of $i$) such that
\begin{equation}\label{12.6}
  \int_{B_{\lambda_i^{-1}}\setminus B_R}\left|\partial_t\widetilde{u}_i(x,0)\right|^2dx \leq C\int_{B_R^c}|x|^{-n-2}\leq \frac{1}{4}.
\end{equation}
After establishing this inequality, \eqref{16.5} will follow from
\begin{lem}
  For all $i$ large,
  \begin{equation}\label{12.7}
  \int_{B_R}\left|\partial_t\widetilde{u}_i(x,0)\right|^2dx \leq   \frac{1}{4}.
\end{equation}
\end{lem}
\begin{proof}

{\bf Case 1.} $\widetilde{u}_i$ converges in $C^\infty_{loc}(\R^n\times\R)$.

By Lemma \ref{lem nonconcentration of time derivative},
\[ \lim_{i\to+\infty}\int_{-1}^{1}\int_{B_R}|\partial_t\widetilde{u}_i|^2=0.\]
Hence by the smooth convergence of $\widetilde{u}_i$, we deduce that $\partial_t\widetilde{u}_i$ converges to $0$ in $C_{loc}(\R^n\times\R)$ and the conclusion follows.

{\bf Case 2.} $\widetilde{u}_i$ does not converge  in $C^\infty_{loc}(\R^n\times\R)$.

In this case, in view of the estimates in \eqref{16.4},  Theorem \ref{thm no bubble towering} is applicable to $\widetilde{u}_i$, which gives
\begin{equation}\label{defect measure III}
 |\nabla\widetilde{u}_i|^2dxdt\rightharpoonup \Lambda\delta_0\otimes dt.
\end{equation}
By \eqref{local decomposition of time derivative 1}, estimates on the scaling parameters etc. in Proposition \ref{prop uniform estimate} and the estimate on $\partial_t\phi_i$ in  Proposition \ref{prop Schauder estimate on error fct}, there exists an $\rho>0$ such that for all $i$ large,
\begin{eqnarray*}
  \int_{B_\rho}\left|\partial_t\widetilde{u}_i(x,0)\right|^2dx &\leq & C \int_{B_\rho} |x|^{-4}+o(1)  \\
    &\leq &C\rho^{n-4}+o(1)\leq 1/8.
\end{eqnarray*}

Next, in view of \eqref{defect measure III}, we deduce that as $i\to+\infty$,
$\partial_t\widetilde{u}_i(x,0)\to 0$ uniformly in $B_R\setminus B_\rho$. Hence for all $i$ large, we also have
\[ \int_{B_R\setminus B_\rho}\left|\partial_t\widetilde{u}_i(x,0)\right|^2dx \leq   \frac{1}{8}.\]
Putting these two inequalities together we get \eqref{12.7}.
\end{proof}

\section{Exclusion of bubble towering}\label{sec bubble towering II}
\setcounter{equation}{0}

In this section, we prove Item (6) in Theorem \ref{thm bubble clustering convergence}. If $N=1$, in view of the analysis in the previous section, this is exactly Theorem \ref{thm no bubble towering}, so here we consider the   case  $N\geq 2$.

We will prove this by a contradiction argument. Recall that we have assumed  $u_\infty$ is smooth, so it is a classical solution of \eqref{eqn}. If  $u_\infty\neq 0$, by standard Harnack inequality,
\begin{equation}\label{Harnack for limit fct}
  \inf_{Q_{7/8}}u_\infty>0.
\end{equation}

First we note that
\begin{lem}\label{lem lower bound from weak limit}
  There exists a constant $c_\infty>0$, depending only on $\inf_{Q_{7/8}}u_\infty$, such that
  \[ u_i\geq c_\infty \quad \mbox{in } \quad Q_{6/7}.\]
\end{lem}
\begin{proof}
  Because
  \[\partial_tu_i-\Delta u_i>0 \quad \mbox{in} \quad Q_1,\]
  the following weak Harnack inequality holds for $u_i$: for any $(x,t)\in Q_{6/7}$,
  \[ u_i(x,t)\geq c\int_{Q_{1/100}^-(x,t-1/50)}u_i.\]
Because $u_i\to u_\infty$ in $L^2_{loc}(Q_1)$, by \eqref{Harnack for limit fct} we get
\[\lim_{i\to+\infty}\int_{Q_{1/100}^-(x,t-1/50)}u_i=\int_{Q_{1/100}^-(x,t-1/50)}u_\infty \geq c  \inf_{Q_{7/8}}u_\infty>0.\]
The conclusion follows by combining these two inequalities.
\end{proof}

For each $t\in(-1,1)$, let
\[ \rho_i(t):=\min_{1\leq j\neq k\leq N}|\xi_{ij}^\ast(t)-\xi_{ik}^\ast(t)|.\]
Set $t_0=0$ and $\rho_0=\rho_i(t_0)$, which is very small if $i$ is large enough. Fix a large index $i$. For each $k\in\mathbb{N}$, set
\[ t_k:= \sup\left\{t:  \rho_i(s)\geq \frac{1}{2}\rho_i(t_{k-1}) \quad \mbox{for any} ~~s\in [t_{k-1},t]\right\},\]
and
\begin{equation}\label{induction on bubble distance}
  \rho_k:=\rho_i(t_k)=2^{-1}\rho_i(t_{k-1})=\cdots=2^{-k}\rho_0.
\end{equation}
For each $k$, assume $\rho_i(t_k)$ is attained between $\xi_{i1}^\ast(t_k)$ and $\xi_{i2}^\ast(t_k)$. Consider
\[ u_i^k(x,t):= \rho_k^{\frac{n-2}{2}} u_i\left(\xi_{i1}^\ast(t_k)+\rho_k x, t_k+\rho_k^2 t\right).\]
By Lemma \ref{lem lower bound from weak limit}, we have
\[ u_i^k(x,0)\geq c_\infty \rho_k^{\frac{n-2}{2}}, \quad \mbox{for} ~~ x\in B_1\setminus B_{1/2}.\]
By our construction, for any $t\in[0,(t_{k+1}-t_k)/\rho_k^2]$, the distance between different bubble points of $u_i^k$ is not smaller than $1/2$. This allows us to iteratively apply Proposition \ref{prop weak Schoen Harnack} (backwardly in time), which gives
\[ u_i^k(x,0)\lesssim e^{-c\frac{t_{k+1}-t_k}{\rho_k^2}}, \quad \mbox{for} ~~ x\in B_1\setminus B_{1/2}.\]
Combining these two inequalities, we obtain
\[ t_{k+1}-t_k\lesssim \rho_k^2|\log\rho_k|\lesssim \rho_k.\]
Combining this inequality with \eqref{induction on bubble distance}, we get
\[ t_\infty:=\lim_{k\to\infty}t_k\lesssim \sum_{k=0}^{\infty}\rho_k\lesssim \rho_0\ll 1.\]
This then implies $\rho_i(t_\infty)=0$ (by the continuity of $\rho_i(t)$, which follows from the continuity of $\xi_{ij}^\ast(t)$). This is a contradiction. In other words, we must have $u_\infty\equiv 0$.

\newpage


\part{Analysis of first time singularity}\label{part first time singularity}

\section{Setting}\label{sec setting IV}
\setcounter{equation}{0}

In this part we assume $u\in C^\infty(Q_1^-)$ is a \emph{smooth} solution of \eqref{eqn}, satisfying
\begin{equation}\label{integral bound up to blow up time}
  \int_{-1}^{t}\int_{B_1} \left(|\nabla u|^2+|u|^{p+1}\right)\leq  K(T-t)^{-K}
\end{equation}
for some constant $K$. In particular, $u$ may not be smoothly extended to $B_1\times\{0\}$.

Define
\[\mathcal{R}(u):=\left\{a\in B_1: \exists ~ r>0  \mbox{  such  that } u\in L^\infty(Q_r^-(a,T))\right\}\]
to be the regular set, and $\mathcal{S}(u):=B_1\setminus\mathcal{R}(u)$ to be the set of blow up points. By standard parabolic estimates, if $a\in\mathcal{R}(u)$, $u$ can be extended smoothly up to $t=0$ in a small backward parabolic cylinder centered at $(a,0)$.

The main result of this part is
\begin{thm}\label{thm local analysis for first time singularity}
  If $n\geq 7$, $p=\frac{n+2}{n-2}$ and $u>0$, then there exists a constant $C$ such that
\begin{equation}\label{interior Type I}
  \|u(t)\|_{L^\infty(B_{1/2})}\leq C(T-t)^{-\frac{1}{p-1}}.
\end{equation}
\end{thm}
We will also show how to get Theorem \ref{main result for Cauchy} and Theorem \ref{main result for Cauchy-Dirichlet} from this theorem.
The proof of this theorem consists of the following four steps.
\begin{enumerate}
  \item After some preliminary estimates about the monotonicity formula and Morrey space bounds in Section \ref{sec integral estimates}, we perform the tangent flow analysis in Section \ref{sec tangent flow analysis} at a possible singular point $(a,T)$. Here we mainly use results from Part \ref{part energy concentration}. This tangent flow analysis  shows that Type II blow up points are isolated in $\mathcal{S}(u)$, see Section \ref{sec alternative}, which allows us to apply results in Part \ref{part one bubble} and Part \ref{part many bubbles}.
  \item In Section \ref{sec bubble cluster IV}, we apply Theorem \ref{thm bubble clustering convergence} in Part \ref{part many bubbles} and Proposition \ref{prop exclusion of bubble clusetering} in Section \ref{sec bubble cluster} to exclude the bubble clustering formation.
  \item In Section \ref{sec completion of proof}, we  exclude the formation of only one bubble. Here we mainly use the differential inequality on the scaling parameter $\lambda$ derived in Part \ref{part one bubble}.
  \item Finally, we use a stability result on Type I blow ups to derive the Type I estimate \eqref{interior Type I local}.
\end{enumerate}

\section{Some integral estimates}\label{sec integral estimates}
\setcounter{equation}{0}

In this section and the next one, we assume only  $p>1$, and $u$ needs not to be positive. This section is devoted to showing a Morrey space bound  on $u$, which will be used in the next section to  perform the tangent flow analysis. This is similar to Section \ref{sec tools} in Part \ref{part energy concentration}.

Denote the standard heat kernel on $\R^n$ by $G$.  Take a function $\psi\in C^\infty(B_{1/4})$, which satisfies $0\leq\psi\leq 1$, $\psi\equiv 1$ in $
B_{1/8}$ and $|\nabla\psi|+|\nabla^2\psi|\leq C$. Take an arbitrary point $a\in B_{1/2}$.  For any $s\in(0,1)$, set
\begin{eqnarray*}
 \Theta_s(a)&:= &s^{\frac{p+1}{p-1}}\int_{B_1}\left[\frac{|\nabla u(y,-s)|^2 }{2}-\frac{|u(y,-s)|^{p+1}}{p+1}\right]G(y-a,s)\psi(y)^2 dy\\
   &+&\frac{1}{2(p-1)}s^{\frac{2}{p-1}}\int_{B_1}u(y,-s)^2 G(y-a,s)\psi(y)^2 dy+Ce^{-cs^{-1}}.
\end{eqnarray*}

The following is another localized version of the monotonicity formula, which is a little different from Proposition \ref{prop monotoncity formula}.
\begin{prop}[Localized monotonicity formula II]\label{prop monotoncity formula IV}
For any $0<s_1<s_2<1$,
\begin{eqnarray*}
 && \Theta_{s_2}(x)-\Theta_{s_1}(x)\\
  &\geq & \int_{s_1}^{s_2} \tau^{\frac{2}{p-1}-1}\int_{B_1}\Big|(-\tau)\partial_tu(y,-\tau)+\frac{u(y,-\tau)}{p-1}+\frac{y\nabla u(y,-\tau)}{2}\Big|^2\\
&&  \quad \quad \quad \quad  \quad \quad  \times G(y-x,-\tau)\psi(y)^2dy d\tau.
\end{eqnarray*}
\end{prop}
This almost monotonicity formula allows us to define
\[\Theta(x):=\lim_{s\to 0}\Theta_s(x).\]
As in Part \ref{part energy concentration}, we still have
\begin{lem}\label{lem u.s.c. IV}
  $\Theta$ is nonnegative and it is upper semi-continuous in $x$.
\end{lem}

\begin{prop}\label{prop Morrey IV}
  For any $x\in B_{1/2}$, $r<1/4$ and $\theta\in(0,1/2)$, there exists a constant $C(\theta,u)$ such that
\begin{equation}\label{Morrey IV}
  r^{-2}\int_{Q_r^-(x,-\theta r^2)}\left(|\nabla u|^2+|u|^{p+1}\right)+\int_{Q_r^-(x,-\theta r^2)}|\partial_t u|^2 \leq C(\theta,u).
\end{equation}
\end{prop}

By the $\varepsilon$-regularity theorem   (see Theorem \ref{thm ep regularity II}), we also get
\begin{prop}
  $\mathcal{R}(u)=\{\Theta=0\}$ and $\mathcal{S}(u)=\{\Theta\geq \varepsilon_{\ast}\}$.
\end{prop}

\section{Tangent flow analysis, II}\label{sec tangent flow analysis}
\setcounter{equation}{0}

In this section, we perform the tangent flow analysis for $u$. This is similar to the one in Section \ref{sec tangent flow}, with the only difference that now we work   in backward parabolic cylinders, instead of the full parabolic cylinders. Some special consequences will also follow from the assumption that $p=(n+2)/(n-2)$, which will be used in the next section to classify Type I and Type II blow up points.

Take an arbitrary point $a\in B_{1/2}$. For any $\lambda>0$ sufficiently small, define
\[u_\lambda(x,t):=\lambda^{\frac{2}{p-1}}u(a+\lambda x, T+\lambda^2 t).\]

By Proposition \ref{prop Morrey IV}, for any $R>0$ and $\theta\in(0,1/2)$, we have
\begin{equation}\label{4.1}
 R^{-2}\int_{Q_R^-(0,-\theta R^2)}\left[|\nabla u_\lambda|^2+|u_\lambda|^{p+1}\right]+\int_{Q_R^-(0,-\theta R^2)}|\partial_t u_\lambda|^2 \leq C(\theta,u),
\end{equation}
where $C(\theta,u)$ is the constant in Proposition \ref{prop Morrey IV}. Therefore, as in Section \ref{sec defect measure}, for any $\lambda_i\to 0$, we can subtract a  subsequence $\lambda_i$ (not relabelling) such that
\begin{itemize}
  \item  $u_{\lambda_i}$ converges to  $u_\infty$, weakly in  $L^{p+1}_{loc}(\R^n\times\R^-)$ and strongly in $L^q_{loc}(\R^n\times\R^-)$ for any $q<p+1$;
  \item $\nabla u_{\lambda_i}$ converges to  $\nabla u_\infty$ weakly in $L^2_{loc}(\R^n\times\R^-)$;
  \item $\partial_tu_{\lambda_i}$ converges to  $\partial_tu_\infty$ weakly in $L^2_{loc}(\R^n\times\R^-)$;
  \item there exist two Radon measures $\mu$ and $\nu$ such that in any compact set of $\R^n\times \R^-$,
  \[  \left\{\begin{aligned}
&|\nabla u_i|^2 dxdt \rightharpoonup |\nabla u_\infty|^2dxdt+\mu,\\
& |u_i|^{p+1}dxdt \rightharpoonup |u_\infty|^{p+1}dxdt+\mu,\\
&\ |\partial_t u_i|^2 dxdt\rightharpoonup |\partial_tu_\infty|^2dxdt+\nu
\end{aligned}\right.
\]
weakly as Radon measures.
\end{itemize}

Passing to the limit in the equation for $u_{\lambda_i}$, we deduce that $u_\infty$ is a weak solution of \eqref{eqn} in $\R^n\times\R^-$. Passing to the limit in \eqref{4.1} gives
\begin{equation}\label{global Morrey bound}
  R^{-2}\int_{Q_R^-(0,-\theta R^2)}\left[|\nabla u_\infty|^2+|u_\infty|^{p+1}\right]+\int_{Q_R^-(0,-\theta R^2)}|\partial_t u_\infty|^2 \leq C(\theta,u),
\end{equation}
where $C(\theta,u)$ is the constant in Proposition \ref{prop Morrey IV}.

Similar to Lemma \ref{lem backwardly self-similar} in Part \ref{part energy concentration}, we deduce that both $u_\infty$ and $\mu$ are  backwardly self-similar.

Results obtained so far hold for any $p>1$. If   $p=(n+2)/(n-2)$, we can say more about $(u_\infty,\mu)$.
 \begin{prop}\label{prop tangent flow structure}
 If $p=(n+2)/(n-2)$, then the followings hold.
 \begin{enumerate}
   \item  Either $u_\infty\equiv 0$ or $u_\infty\equiv \pm\left[-(p-1)t\right]^{-\frac{1}{p-1}}$.
   \item   There exists a constant $M>0$ such that
  \[\mu=M\delta_0\otimes dt.\]
   \item  $\nu=0$.
 \end{enumerate}
 \end{prop}
\begin{proof}
A rescaling of  \eqref{global Morrey bound} gives
\begin{equation}\label{polynomial growth condition}
  \int_{B_1(x)}\left(|\nabla w_\infty|^2+|w_\infty|^{p+1}\right)\leq C,  \quad \forall x\in\R^n.
\end{equation}
 Because $p=(n+2)/(n-2)$, by Theorem \ref{thm Liouville for backward self-similar sol.},  we obtain (1).

Because   $\mu$ is self-similar, by Theorem \ref{thm energy quantization I}, we find a discrete set $\{\xi_j\}\subset\R^n$,  $M_j>0$ such that
\begin{equation}\label{tangent flow structure 2}
 \mu=\mu_tdt, \quad \mu_t=\sum_{j}M_j \delta_{\sqrt{-t}\xi_j}.
\end{equation}

  Because  $u_{\lambda_i}$ are all smooth in $\R^n\times\R^-$, the energy inequalities \eqref{energy inequality I} for them are in fact identities.  Letting $\lambda_i\to0$, we see the limiting energy inequality  \eqref{energy identity limit} for $(u_\infty,\mu)$ is also an identity.  Furthermore, because $u_\infty$ is smooth in $\R^n\times\R^-$,  it also satisfies the energy identity. From these facts we deduce that
  \[ \partial_t\mu=-n\nu \quad \mbox{in the distributional sense in }~~ \R^n\times\R^-.\]
Combining this relation with \eqref{tangent flow structure 2}, we get (2) and (3).
\end{proof}

Furthermore, if $u$ is positive, this proposition implies that the blowing up sequence $u_\lambda$ satisfy the assumptions {\bf (III.a-III.c)} in Part \ref{part many bubbles}.

\section{No coexistence of Type I and Type II blow ups}\label{sec alternative}
\setcounter{equation}{0}

From now on it is always assumed that $u$ is a positive solution, $p=(n+2)/(n-2)$ and $n\geq 7$.  In this section we classify Type I and Type II blow up points, and give some qualitative description of the set of Type I and Type II blow up points.

\begin{lem}\label{lem alternative}
  Either $u_\infty=0$ or $\mu=0$.
\end{lem}
\begin{proof}
If $\mu\neq 0$, by the positivity, there exists an $N\in\mathbb{N}$ such that the constant in \eqref{tangent flow structure 2} satisfies
\[ M=N\Lambda.\]
Since $u_\infty$ is smooth in $\R^n\times\R^-$, {\bf (III.a-III.c)} are satisfied by any blow up sequence $u_{\lambda_i}$ at $(a,T)$. Then an application of  Theorem \ref{thm bubble clustering convergence} implies that $u_\infty=0$. Alternatively, if $u_\infty\neq 0$, then $\mu=0$.
\end{proof}
A direct consequence of this lemma is
\begin{prop}\label{prop dichotomy}
Any point $a\in B_{1/2}$ belongs to one of the following three classes:
\begin{description}
  \item [Regular point]   $\Theta(a)=0$;
  \item [Type I blow up point] $\Theta(a)=n^{-1}\left(\frac{n-2}{4}\right)^{\frac{n}{2}}$;
  \item [Type II blow up point]   $\Theta(a)=n^{-1}\left(4\pi\right)^{-\frac{n}{2}}N\Lambda$ for some $N\in\mathbb{N}$.
\end{description}
\end{prop}
Since all of the possible tangent flows form a discrete set (classified according to the value of $\Theta(a)$), we obtain
\begin{coro}
  The tangent flow of $u$ at any point $a\in B_{1/2}$ is unique.
\end{coro}

Next we study the sets of Type I and Type II blow up points.
\begin{lem}\label{lem open for Type I}
  The set of Type I blow up points is relatively open in $\mathcal{S}(u)$.
\end{lem}
\begin{proof}
It is directly verified that
 \begin{equation}\label{difference between Type I and II}
 n^{-1}\left(\frac{n-2}{4}\right)^{\frac{n}{2}}<n^{-1}\left(4\pi\right)^{-\frac{n}{2}}\Lambda.
 \end{equation}
The conclusion then follows from the upper semi-continuity of $\Theta$, see Lemma \ref{lem u.s.c. IV}.
\end{proof}

\begin{lem}\label{lem isolated for Type II}
  The set of Type II blow up points is isolated in $\mathcal{S}(u)$.
\end{lem}
\begin{proof}
  Assume $x_0$ is Type II, that is, $\Theta(x_0)=n^{-1}\left(4\pi\right)^{-\frac{n}{2}}N\Lambda$ for some $N\in\mathbb{N}$. Assume by the contrary that there exists a sequence of $x_i\in\mathcal{S}(u)$ converging to $x_0$.

Denote
\[ \lambda_i:=|x_i-x_0|, \quad  \widehat{x}_i:=\frac{x_i-x_0}{\lambda_i}.\]
Define the blow up sequence
\[ u_i(x,t):=\lambda_i^{\frac{2}{p-1}}u(x_0+\lambda_i x, T+\lambda_i^2 t).\]
Then the tangent flow analysis together with the fact that $x_0$ is Type II implies that
\[ |\nabla u_i|^2dxdt\rightharpoonup N\Lambda \delta_0\otimes dt.\]
As a consequence, there exists a fixed, small $s>0$ such that
\[\Theta_s\left(\widehat{x}_i;u_i\right)\leq \varepsilon_\ast/2.\]
However, by the monotonicity formula (Proposition \ref{prop monotoncity formula IV}), the fact that $x_i\in\mathcal{S}$ and the scaling invariance of $\Theta_s$, we always have
\[ \Theta_s\left(\widehat{x}_i;u_i\right)=\Theta_{\lambda_i^2s}(x_i;u)\geq \Theta(x_i;u_i)\geq  \varepsilon_\ast.\]
This is a contradiction.
\end{proof}

\section{Exclusion of bubble clustering}\label{sec bubble cluster IV}
\setcounter{equation}{0}

In this section we exclude Type II blow ups with bubble clustering. From now on we assume there exists a Type II blow up point. By Lemma \ref{lem isolated for Type II}, this blow up point is isolated in $\mathcal{S}(u)$. Hence after a translation and a scaling, we may assume $u\in C^\infty(\overline{Q_1^-}\setminus\{(0,0)\})$, and $(0,0)$ is a Type II blow up point.

Under these assumptions, Theorem \ref{thm bubble clustering convergence} is applicable to any rescaling of $u$ at $(0,0)$. Combining this theorem with a continuity argument in time, we obtain
\begin{prop}\label{prop blow up profile IV}
There exists an $N\in\mathbb{N}$ so that for all $t$ sufficiently close to $0$, the following results hold.
  \begin{enumerate}
  \item  There exist exactly $N$ local maximal point of $u(\cdot,t)$ in the interior of $B_1(0)$.

  Denote this point by $\xi_{j}^\ast(t)$ ($j=1,\cdots, N$) and let
$\lambda_{j}^\ast(t):=u(\xi_{ij}^\ast(t),t)^{-\frac{2}{n-2}}$.

\item Both $\lambda_{j}^\ast$ and $\xi_{j}^\ast$ depend continuously on $t$.

\item There exists a constant $K$ such that, for all $i$ and any $x\in B_1$,
\begin{equation}\label{scaling invariant estimate IV}
    u(x,t)\leq K \max_{1\leq j \leq N}|x-\xi_{j}^\ast(t)|^{-\frac{n-2}{2}}.
\end{equation}

  \item
As $t\to 0$,
\[\frac{\lambda_{j}^\ast(t)}{\sqrt{-t}}\to0, \quad \frac{|\xi_{j}^\ast(t)|}{\sqrt{-t}}\to 0,\]
and  the function
\[u_{j}^t(y,s):=\lambda_{j}^\ast(t)^{\frac{n-2}{2}}u\left(\xi_{j}^\ast(t)+\lambda_{j}^\ast(t)y , t+\lambda_{j}^\ast(t)^2s \right),\]
converges to $W(y)$ in $C^\infty_{loc}( \R^n\times\R)$.

  \item  For any $1\leq j\neq k \leq N$ and $t\in[-9/16,0)$,
\begin{equation}\label{bubbles separating}
   \lim_{t\to 0}\frac{|\xi_{j}^\ast(t)-\xi_{k}^\ast(t)|}{\max\left\{\lambda_{j}^\ast(t),\lambda_{k}^\ast(t)\right\}}=+\infty.
\end{equation}
\end{enumerate}
\end{prop}

The main result of  this section  is
\begin{prop}\label{prop no bubble clustering}
  The multiplicity $N$ in Proposition \ref{prop blow up profile IV} must be $1$.
\end{prop}

The proof is by a contradiction argument, so assume $N\geq 2$. For each $t\in(-1,0)$, denote
\[\rho(t):=\min_{1\leq k \neq \ell\leq N} |\xi_k^\ast(t)-\xi_\ell^\ast(t)|.\]
By Point (4) in Proposition \ref{prop blow up profile IV}, for any $t<0$, $\rho(t)>0$, and
\begin{equation}\label{inter-disctance convergence}
  \lim_{t\to0} \rho(t) =0.
\end{equation}
Then by the continuity of $\rho(t)$ (thanks to the continuity of $\xi_j^\ast(t)$), there exists a large positive integer $i_0$ so that the following is well-defined.
For each $i\geq i_0$, define $t_i$ to be the first time satisfying $\rho(t)=2^{-i}$. In other words,
\begin{equation}\label{choice of time}
 \rho(t_i)=2^{-i} \quad \mbox{and} \quad  \rho(t)>2^{-i} ~~ \mbox{in}~~ [-1,t_i).
\end{equation}

On the other hand, we claim that
\begin{lem}\label{lem 30.1}
 For all $i$ large,
  \begin{equation}\label{30.2}
\rho(t_{i+1})\geq \frac{2}{3}\rho(t_i).
  \end{equation}
\end{lem}
This is clearly a contradiction with the definition of  $\rho(t_i)$. The proof of Proposition \ref{prop no bubble clustering} is thus complete. (This contradiction in fact shows that if $N\geq 2$, $\xi_i^\ast(t)$ cannot converges to the same point as $t\to 0^-$, that is, \eqref{inter-disctance convergence} does not hold.)

The proof of Lemma \ref{lem 30.1} uses two facts: the first one is the reduction equation for scaling parameters, which will be applied in a forward in time manner; the second one is the weak form of Schoen's Harnack inequality, Proposition \ref{prop weak Schoen Harnack}, which will applied in a backward in time manner.
\begin{proof}[Proof of Lemma \ref{lem 30.1}]
Without loss of generality, assume
 \[\rho(t_{i+1})=|\xi_1^\ast(t_{i+1})-\xi_2^\ast(t_{i+1})|=2^{-i-1}.\]
Choose $\widetilde{t}_i$ to be the first time in $[t_i,t_{i+1}]$ satisfying
\[ |\xi_1^\ast(t)-\xi_2^\ast(t)|=2^{-i},\]
that is,
\[ |\xi_1^\ast(\widetilde{t}_i)-\xi_2^\ast(\widetilde{t}_i)|=2^{-i}, \quad \mbox{and} \quad |\xi_1^\ast(t)-\xi_2^\ast(t)|>2^{-i}
 ~~ \mbox{for any} ~~ t\in[t_i,\widetilde{t}_i).\]
 This is well defined, if we note the continuous dependence of $|\xi_1^\ast(t)-\xi_2^\ast(t)|$ with respect to $t$, and the fact that
 \[  |\xi_1^\ast(t_i)-\xi_2^\ast(t_i)|\geq \rho(t_i)=2^{-i}.\]

Let
\[u_i(x,t):=2^{-\frac{n-2}{2}i} u\left(\xi_1^\ast(\widetilde{t}_i)+2^{-i} x, \widetilde{t}_i+4^{-i} t\right).\]
It is well defined for $x\in B_{2^{i-1}}$ and $t\in(-4^{i-1}, - 4^i\widetilde{t}_i)$.

For each $t\in(-4^{i-1}, - 4^i\widetilde{t}_i)$, $u_i$ has $N$ bubbles, located at
\[ \xi^\ast_{ij}(t):=2^i \left[\xi_j^\ast\left(\widetilde{t_i}+4^it\right)-\xi_1^\ast\left(\widetilde{t}_i\right)\right],\]
 with bubble scale
 \[ \lambda^\ast_{ij}(t):=2^i\lambda_j^\ast\left(\widetilde{t_i}+4^it\right).\]

Denote
\[ T_i:= 4^i \left(t_{i+1}-\widetilde{t}_i\right).\]

By the definition of $\widetilde{t}_i$ and $t_{i+1}$,  for any $1\leq j\neq k\leq N$, we have
\begin{equation}\label{distance lower bound 2}
  |\xi^\ast_{ij}(t)-\xi^\ast_{ik} (t)|\geq \frac{1}{2}, \quad \forall t\in \left[-4^{i-1},T_i\right].
\end{equation}
We also have the normalization condition
\begin{equation}\label{distance lower bound 3}
|\xi^\ast_{i1}(0)-\xi^\ast_{i2}(0)|=1.
\end{equation}
Moreover, by Point (5) in Proposition \ref{prop blow up profile IV}, as $i\to+\infty$, if $|\xi_{ij}^\ast(t)|$ does not escape to infinity, then
\[ \lambda_{ij}^\ast(t) \to 0 \quad \mbox{uniformly in any compact set of}~~ \R.\]
Then Proposition \ref{prop exclusion of bubble clusetering} is applicable if $i$ is large enough, which gives
  \begin{equation}\label{estimate of time leap}
    T_i  \leq 100 \min\left\{\left|\log \lambda_{i1}^\ast(0)\right|, \left|\log \lambda_{i2}^\ast(0)\right|\right\}.
  \end{equation}

For any $t\in[0, T_i]$,   Proposition  \ref{prop orthogonal decompostion} can be applied to $u_i$ in $Q_{1/2}(\xi^\ast_{i1}(t),t)$ and $Q_{1/2}(\xi^\ast_{i2}(t),t)$. Denote the scaling parameter by $ \lambda_{i1}(t)$ and $ \lambda_{i2}(t)$.
By Proposition \ref{prop orthogonal decompostion},
\begin{equation}\label{changing the scaling parameter}
    \lambda_{i1}(t)=\left[1+o(1)\right] \lambda^\ast_{i1}(t), \quad |\xi_{i1}^\ast(t)-\xi_{i1}(t)|=o(\lambda^\ast_{i1}(t)).
\end{equation}
The reduction equation for $ \lambda_{i1}(t)$ (see Proposition \ref{prop uniform estimate}) gives
\begin{equation}\label{differential Harnack for lambda}
   \left| \lambda_{i1}^\prime(t)\right|\lesssim   \lambda_{i1}(t)^{\frac{n-4}{2}}.
\end{equation}
For any $t\in(0,T_i]$, integrating \eqref{differential Harnack for lambda} on $[0, t]$   leads to
\[
  \lambda_{i1}(t)^{-\frac{n-6}{2}}\geq\lambda_{i1}(0)^{-\frac{n-6}{2}}-CT_i.
\]
Plugging \eqref{estimate of time leap} and \eqref{changing the scaling parameter} into this inequality, after a simplification we get
 \begin{eqnarray}\label{Harnack for lambda IV}
 \lambda_{i1}(t)&\leq & \lambda_{i1}(0)\left[1+C\lambda_{i1}(0)^{\frac{n-6}{2}}\left|\log \lambda_{i1}(0)\right|\right]\\
&\leq  & 2\lambda_{i1}(0), \nonumber
\end{eqnarray}
where we have used the fact that $ \lambda_{i1}(0)\ll 1$ in the last step.

The reduction equation for $\xi_{i1}(t)$ is
\[ \left|\xi_{i1}^\prime(t)\right|\lesssim   \lambda_{i1}(t)^{\frac{n-4}{2}}.\]
Integrating this on $[0,T_i]$ and using \eqref{Harnack for lambda IV} and \eqref{estimate of time leap} again, we get
\[
 |\xi_{i1}(T_i)-\xi_{i1}(0)|\lesssim \lambda_{i1}(0)^{\frac{n-4}{2}}\left|\log\lambda_{i1}(0)\right|\ll 1.
\]
The same estimates hold for $\xi_{i2}$. Combining these two estimates with the second equality in \eqref{changing the scaling parameter} (and a similar estimate for $\xi_{i2}$), we obtain
\[ |\xi_{i1}^\ast(T_i)-\xi_{i2}^\ast(T_i)|\geq |\xi_{i1}^\ast(0)-\xi_{i2}^\ast(0)|-\frac{1}{3}=\frac{2}{3}|\xi_{i1}^\ast(0)-\xi_{i2}^\ast(0)|.\]
Scaling back to $u$, this is
\[ \rho(t_{i+1})=|\xi_1(t_{i+1})-\xi_2(t_{i+1})|\geq \frac{2}{3}|\xi_1(\widetilde{t}_i)-\xi_2(\widetilde{t}_i)|=\frac{2}{3}\rho(t_i).\qedhere\]
\end{proof}

\section{Exclusion of blow up with only one bubble}\label{sec completion of proof}
\setcounter{equation}{0}

In this section, we  prove
\begin{prop}\label{prop no Type II blow up}
  Under the setting of Theorem \ref{thm local analysis for first time singularity}, any point $a\in\mathcal{S}(u)$ must be Type I, in the sense that
\[\Theta(a)=n^{-1}\left(\frac{n-2}{4}\right)^{\frac{n}{2}}.\]
\end{prop}
In view of Proposition \ref{prop no bubble clustering}, we need to consider the remaining case in Proposition \ref{prop blow up profile IV}, that is, when there is only one bubble, $N=1$.
Under this assumption, Proposition \ref{prop blow up profile IV} now reads as
\begin{prop}\label{prop blow up profile with one bubble}
For any $t$ sufficiently close to $0$, there exists a unique maxima point of $u(\cdot,t)$ in  $B_1(0)$. Denote this point by $\xi^\ast(t)$ and let
$\lambda^\ast(t):=u(\xi^\ast(t),t)^{-\frac{2}{n-2}}$.

As $t\to 0$,
\[\lambda^\ast(t)\to0, \quad |\xi^\ast(t)|\to 0,\]
and   the function
\[u^t(y,s):=\lambda^\ast(t)^{\frac{n-2}{2}}u\left(\xi^\ast(t)+\lambda^\ast(t)y , t+\lambda^\ast(t)^2s  \right),\]
converges to $W(y)$ in $C^\infty_{loc}( \R^n\times \R)$.
\end{prop}
Next, applying Proposition \ref{prop orthogonal decompostion} (with the notation of $\eta_K$ etc. as used in Part \ref{part one bubble}) to a suitable rescaling of $u$ at $(\xi^\ast(t),t)$ gives
\begin{prop}[Orthogonal condition]\label{prop orthogonal decompostion III}
For any $t$ sufficiently close $0$, there exists a unique $(\xi(t),\lambda(t),a(t))\in\R^n\times \R_+\times\R$ with
\begin{equation}\label{deviation form max pt III}
\frac{|\xi(t)-\xi^\ast(t)|}{\lambda(t)}+\Big|\frac{\lambda(t)}{\lambda^\ast(t)}-1\Big|+\Big|\frac{a(t)}{\lambda(t)}\Big|=o(1),
\end{equation}
such that for any $i=0,\cdots, n+1$,
\begin{equation}\label{orthogonal III}
  \int_{B_1}\left[u(x,t)-W_{\xi(t),\lambda(t)}(x)-a(t)Z_{0,\xi(t),\lambda(t)}(x)\right]\eta_K\left(\frac{x-\xi(t)}{\lambda(t)}\right) Z_{i,\xi(t),\lambda(t)}(x)dx =0.
\end{equation}
\end{prop}
Starting with this decomposition, the analysis in Part \ref{part one bubble} is applicable to
\[ \widetilde{u}^t(y,s):=(-t)^{\frac{n-2}{4}}u\left(\sqrt{-t}y, -ts\right), \quad \forall t\in(-1/4,0).\]
This  gives an ordinary differential inequality for $\lambda(t)$:
\begin{equation}\label{inequality for lambda 1}
  \left|\lambda^\prime(t)\right|\lesssim (-t)^{\frac{n-10}{4}}\lambda(t)^{\frac{n-4}{2}}.
\end{equation}
This inequality can be rewritten as
\begin{equation}\label{inequality for lambda 2}
  \left|\frac{d}{dt}\lambda(t)^{-\frac{n-6}{2}}\right|\lesssim (-t)^{\frac{n-10}{4}}
\end{equation}
Because $n\geq 7$,
\[\frac{n-10}{4}>-1.\]
Integrating \eqref{inequality for lambda 2} leads to
\[\lim_{t\to 0}\lambda(t)^{-\frac{n-6}{2}}<+\infty.\]
Thus as $t\to 0^-$, $\lambda(t)$ does not converges to $0$. By \eqref{deviation form max pt III} and Proposition \ref{prop blow up profile with one bubble},
\[ \limsup_{t\to 0^-}\sup_{B_1}u(x,t)<+\infty.\]
Hence $u$ does not blow up at time $t=0$. This is a contradiction.

\section{Proof of main results}\label{sec proof of main results}
\setcounter{equation}{0}

In this section we first finish the proof of Theorem \ref{thm local analysis for first time singularity}, and then show how Theorem \ref{main result for Cauchy}, Theorem \ref{main result for Cauchy-Dirichlet} and Corollary \ref{coro energy collapsing} follow from this theorem.

\subsection{Proof of Theorem \ref{thm local analysis for first time singularity}}

Combining   Proposition \ref{prop no Type II blow up} with the tangent flow analysis in Section \ref{sec tangent flow analysis}, we deduce that for any  $a\in\mathcal{S}(u)$,
  \begin{equation}\label{weak Type I}
    \limsup_{t\to T}(T-t)^{\frac{1}{p-1}}u\left(a+\sqrt{T-t}y,t+(T-t)s\right)=(p-1)^{-\frac{1}{p-1}}(-s)^{\frac{1}{p-1}}
  \end{equation}
 uniformly in any compact set of  $\R^n\times\R^-$.

However, this is only a qualitative description and we cannot obtain \eqref{interior Type I} from it. Now we show how to obtain a quantitative estimate. We will mainly use the stability of Type I blow ups, as characterized in \eqref{difference between Type I and II}. (For other notions of stability of Type I blow ups, see Merle-Zaag \cite{Merle-Zaag-stability}, Kammerer -Merle-Zaag \cite{Merle-Zaag-dynamical},   Collot-Merle--Rapha\"{e}l \cite{Collot2017stability}, Collot-Rapha\"{e}l-Szeftel \cite{Collot2019stability}.)

\begin{lem}
For each $a\in B_{1/2}$,  there exist two  constants $C(a)>1$ and $\rho(a)<1/8$ such that for any $0<s<\rho(a)^2$,
\begin{equation}\label{interior Type I local}
  \sup_{x\in B_{\rho(a)}(a)}u(x,T-s)\leq C(a)s^{-\frac{1}{p-1}}.
\end{equation}
\end{lem}
\begin{proof}
If $a\in\mathcal{R}(u)$, by definition, there exists an $\rho(a)>0$ such that
$u\in L^\infty(Q_{\rho(a)}^-(a,T))$. Then \eqref{interior Type I local} holds trivially by choosing a suitable constant $C(a)$.

Next, assume $a\in\mathcal{S}(u)$.
By \eqref{difference between Type I and II}, we can take a small $\varepsilon>0$ so that
\begin{equation}\label{gap between Type I and II}
 n^{-1}\left(\frac{n-2}{4}\right)^{\frac{n}{2}}+4\varepsilon< n^{-1}\left(4\pi\right)^{-\frac{n}{2}}\Lambda.
\end{equation}
By \eqref{weak Type I}, there exists an $r(a)>0$ such that
\begin{equation}\label{close to Type I}
  \Theta_{r(a)^2}(a)<n^{-1}\left(\frac{n-2}{4}\right)^{\frac{n}{2}}+\varepsilon.
\end{equation}
Next, combining \eqref{weak Type I} with Proposition \ref{prop Morrey IV} and the monotonicity formula (Proposition \ref{prop monotoncity formula IV}), we find another   constant  $\rho(a)<r(a)$ such that
\[\Theta_{r(a)^2}(x)\leq \Theta_{r(a)^2}(a)+\varepsilon, \quad \forall  x\in B_{2\rho(a)}(a).\]
Substituting this inequality into \eqref{close to Type I} and noting \eqref{gap between Type I and II}, we get
\begin{equation}\label{control on monotonicity quantity}
 \Theta_{r(a)^2}(x)< n^{-1}\left(4\pi\right)^{-\frac{n}{2}}\Lambda-2\varepsilon, \quad \forall  x\in B_{2\rho(a)}(a).
\end{equation}

In the following, we argue by contradiction. Assume with $\rho(a)$ defined as above, there does not exist such a constant $C(a)$ so that
\eqref{interior Type I local} can hold, that is, there exists a sequence of $x_k\in B_{\rho(a)}(a)$, $s_k\to 0$ with
\begin{equation}\label{absurd assumption IV}
  u(x_k,T-s_k)\geq k s_k^{-\frac{1}{p-1}}.
\end{equation}
Define the blow up sequence
\[ u_k(x,t):=s_k^{\frac{1}{p-1}}u(x_k+ \sqrt{s_k}y, T+ s_k t).\]
As in Section \ref{sec tangent flow analysis},  by applying results in Part \ref{part energy concentration}, we may assume $u_k$ converges weakly to $u_\infty$, and
\[ |\nabla u_k|^2dxdt \rightharpoonup |\nabla u_\infty|^2dxdt+\mu\]
weakly as Radon measures in $\R^n\times\R^-$. By \eqref{absurd assumption IV},
\[ u_k(0,-1)\geq k.\]
Hence $u_k$ cannot converge smoothly to $u_\infty$. On the other hand, we claim that $u_k$  does converge smoothly to $u_\infty$. This contradiction then finishes the proof of this lemma.

{\bf Proof of the claim.} We first prove $\mu=0$. By scaling \eqref{control on monotonicity quantity} and passing to the limit, we get
\begin{equation}\label{control on monotonicity quantity, 2}
  \Theta_s(x,0;u_\infty,\mu)\leq  n^{-1}\left(4\pi\right)^{-\frac{n}{2}}\Lambda-2\varepsilon, \quad \forall  x\in\R^n, \quad s>0.
\end{equation}

By Corollary \ref{coro relation between densities}, for $\mu$-a.e. $(x,t)$,
\begin{equation}\label{control on monotonicity quantity, 3}
    \Theta(x,t;u_\infty,\mu)\geq  n^{-1}\left(4\pi\right)^{-\frac{n}{2}}\Lambda.
 \end{equation}
On the other hand, if $s$ is large enough, by the Morrey space bound as in \eqref{Morrey scaled}, we get
\begin{equation}\label{control on monotonicity quantity, 4}
\Theta_s(x,t;u_\infty,\mu)\leq \Theta_s(x,0;u_\infty,\mu)+\varepsilon.
 \end{equation}
By the monotonicity of $\Theta_s(\cdot)$, these three inequalities lead to a contradiction unless $\mu=0$.

Next we show that $u_\infty\in C^\infty(\R^n\times\R^-)$. First, because $\mu=0$, as in Theorem \ref{thm Marstrand}, we deduce that $u_\infty$ is a suitable weak solution of \eqref{eqn} in $\R^n\times\R^-$.
If there is a singular point of $u_\infty$, say $(x,t)$, by the analysis in Subsection \ref{subsec proof of energy quantization} and Theorem \ref{thm energy quantization I}, any tangent flow of $u_\infty$ at $(x,t)$ would be of the form
\[ N_1\Lambda\delta_0\otimes dt\lfloor_{\R^-}+N_2\Lambda\delta_0\otimes dt\lfloor_{\R^+},\]
where $N_1\geq 1$, $N_2\geq 0$. Then by Corollary \ref{coro relation between densities}, we get
\[\Theta(x,t;u_\infty)\geq n^{-1}\left(4\pi\right)^{-\frac{n}{2}}\Lambda.\]
Using this inequality to replace \eqref{control on monotonicity quantity, 3} and aruging as above,   we get a contradiction with \eqref{control on monotonicity quantity, 2}. In conclusion, the singular set of $u_\infty$ is empty.

Now $\mu=0$ implies that $u_k$ converges to $u_\infty$ strongly. The fact that $u_\infty$ is smooth, implies that, for any $(x,t)\in\R^n\times\R^-$, we can apply the $\varepsilon$-regularity theorem, Theorem \ref{thm ep regularity II}, to $u_k$ in a sufficiently small cylinder $Q_r^-(x,t)$. This   implies that  $u_k$ converges  to $u_\infty$ in $C^\infty(\R^n\times\R^-)$.
\end{proof}

These balls $\{B_{\rho(a)}(a)\}$ form a covering of $B_{1/2}$. Take a finite sub-cover of $B_{1/2}$ from $\{B_{\rho(a)}(a)\}$, $\{B_{\rho_i}(a_i) \}$. For
\[ t\geq T-\min_i \rho_i^2,\]
\eqref{interior Type I local} is a direct consequence of \eqref{interior Type I}, if we choose the constant in \eqref{interior Type I local} to be
\[ C:= \max_{i}C(a_i).\]
Since \eqref{interior Type I} holds trivially for $ t\leq T-\min_i \rho_i^2$, we finish the proof of Theorem \ref{thm local analysis for first time singularity}.

\subsection{Cauchy-Dirichlet problems}\label{subsec Cauchy-Dirichlet}

Define
\[H(t):=\int_{\Omega}u(x,t)^2dx, \quad E(t):=\int_{\Omega}\left(\frac{|\nabla u(x,t)|^2}{2}-\frac{|u(x,t)|^{p+1}}{p+1}\right)dx.\]
The following estimate is   Blatt-Struwe's \cite[Lemma 6.3]{Struwe1}.
\begin{lem}\label{lem integral estiamte}
There exists a constant $C$ such that for any $t\in(T/2,T)$,
\[
\left\{\begin{aligned}
& H(t)\leq C(T-t)^{-\frac{2}{p-1}},\\
& \int_{T/2}^{t}\int_{\Omega}\left[|\nabla u(y,s)|^2+|u(y,s)|^{p+1} \right]dyds\leq C(T-t)^{-\frac{2}{p-1}}.
\end{aligned}\right.
\]
\end{lem}
With this estimate in hand, we know that $u$ locally satisfies the assumptions in Section \ref{sec setting IV}. Theorem \ref{main result for Cauchy-Dirichlet} then follows from Theorem \ref{thm local analysis for first time singularity} and a covering argument.

\subsection{Cauchy problem}\label{subsec Cauchy problem}

Now we come to the proof of Theorem \ref{main result for Cauchy}. As in \cite{Giga04},
by the regularizing effect of parabolic equations, we may assume $u_0\in C^2(\R^n)$.

 For each $a\in\R^n$ and $s\in[0,T)$, define
\begin{eqnarray*}
 \Theta_s(a)&:= &s^{\frac{p+1}{p-1}}\int_{\R^n}\left[\frac{|\nabla u(y,T-s)|^2 }{2}-\frac{|u(y,T-s)|^{p+1}}{p+1}\right]G(y-x,s) dy\\
   &+&\frac{1}{2(p-1)}s^{\frac{2}{p-1}}\int_{\R^n}u(y,T-s)^2 G(y-x,s)dy.
\end{eqnarray*}
Then we have (see the summarization in \cite[Section 3.2]{Giga04})
\begin{enumerate}
  \item $\Theta_s(a)$ is non-decreasing in $s$;
  \item for any $a\in\R^n$, $r\in(0,\sqrt{T}/2)$ and $\theta\in(0,1/2)$,
  \[  \left\{\begin{aligned}
& \int_{Q_r(a, T-\theta r^2)}|\nabla u|^2+|u|^{p+1} \leq C(\theta)\max\left\{\Theta_T(a), \Theta_T(a)^{\frac{2}{p+1}}\right\}r^2,\\
& \int_{Q_r(a, T-\theta r^2)}|\partial_t u|^2  \leq C(\theta)\max\left\{\Theta_T(a), \Theta_T(a)^{\frac{2}{p+1}}\right\}.
\end{aligned}\right.
\]
\end{enumerate}
With these estimates in hand, we can repeat the proof of Theorem \ref{thm local analysis for first time singularity} to get Theorem \ref{main result for Cauchy}.

\subsection{Energy collapsing and the blow up of $L^{p+1}$ norm}
In this subsection, we prove Corollary \ref{coro energy collapsing}.

Recall that we have assumed that there exists a blow up point  $a\in\Omega$. We define the blow up sequence $u_\lambda$ at $(a,T)$ as in Section \ref{sec tangent flow analysis}. Combining Proposition \ref{prop tangent flow structure} with Proposition \ref{prop no Type II blow up}, we see
\[ u_\lambda(x,t)\to u_\infty(x,t)=(p-1)^{-\frac{1}{p-1}}(-t)^{-\frac{1}{p-1}} \quad \mbox{in} ~~ C^\infty_{loc}(\R^n\times\R^-).\]
Hence there exists a universal constant $c>0$ such that
\[
  \lim_{\lambda\to0}  \int_{Q^-_{\lambda}(a,T-\lambda^2)}|\partial_tu|^2= \lim_{\lambda\to0}\int_{Q^-_1(0,-1)}|\partial_t u_{\lambda}|^2\geq \int_{Q_1(0,-1)}|\partial_t u_{\infty}|^2\geq c.
\]
Taking a   subsequence  $\lambda_i$ so that $Q^-_{\lambda_i}(a,T-\lambda_i^2)$ are  disjoint from each other. Then
\[\lim_{t\to T}\int_{0}^{t}\int_{\Omega}|\partial_tu|^2\geq \sum_{i=1}^{+\infty}\int_{Q^-_{\lambda_i}(a,T-\lambda_i^2)}|\partial_tu|^2=+\infty.\]
Then by the energy identity for $u$, we get \eqref{energy collapsing}.

In the same way, there exists a universal constant $c>0$ such that, for any $R>1$,
\[\lim_{\lambda\to0}  \int_{B_{\lambda R}(a,T-\lambda^2)}u(x,t)^{p+1}dx= \lim_{\lambda\to0}\int_{B_R}  u_{\lambda}(x,t)^{p+1}dx\geq cR^n.\]
Therefore
\[\lim_{t\to T} \int_{\Omega}u(x,t)^{p+1}dx\geq \lim_{t\to T}  \int_{B_{R\sqrt{T-t}}(a,T)}u(x,t)^{p+1}dx\geq cR^n.\]
Because $R$ can be arbitrarily large, \eqref{blow up of p+1 norm} follows.

\bibliography{parabolic-bubbling-2021-01-17}
\bibliographystyle{plain}
\end{document}